%% file: arxiv_version.tex
\documentclass[10pt]{article}
\usepackage[left=1in,right=1in,top=1in,bottom=1in]{geometry}

\usepackage{xr}
\usepackage[utf8]{inputenc}

\newif\ifanonymize
\anonymizefalse

\usepackage{amsmath,amssymb,amsthm,mathabx,mathtools,bm,booktabs,bbm}
\mathtoolsset{showonlyrefs,showmanualtags}
\usepackage{enumitem}
\usepackage[round]{natbib}
\usepackage[colorlinks=true,linkcolor=blue,citecolor=blue,urlcolor=blue]{hyperref}

\makeatletter
\newcommand{\numcite}[1]{%
  [\begingroup
   \renewcommand{\NAT@sep}{,}%
   \citenum{#1}%
   \endgroup]}
\makeatother

\usepackage{algorithm}
\usepackage[noend]{algpseudocode}

\algnewcommand\algorithmicoutput{\textbf{Output:}}
\algnewcommand\Output{\item[\algorithmicoutput]}

\newtheorem{theorem}{Theorem}

\newtheorem{lemma}{Lemma}

\newtheorem{remark}{Remark}

\usepackage{etoolbox}

\apptocmd{\appendix}{%
  \setcounter{theorem}{0}%
  \renewcommand{\thetheorem}{S.\arabic{theorem}}%
}{}{}

\usepackage[dvipsnames]{xcolor}

\newcommand{\R}{\mathbb R}
\newcommand{\Pbb}{\mathbb P}
\newcommand{\Ebb}{\mathbb E}
\newcommand{\op}{\mathrm{op}}
\newcommand{\sgn}{\operatorname{sgn}}
\newcommand{\diag}{\operatorname{diag}}
\newcommand{\ip}[2]{\langle #1,#2\rangle}
\newcommand{\norm}[1]{\lVert #1\rVert}
\newcommand{\convp}{\xrightarrow{\mathrm{p}}}
\newcommand{\col}{\mathrm{col}}
\newcommand{\row}{\mathrm{row}}

\newcommand{\simiid}{\overset{\mathrm{i.i.d}}{\sim}}
\newcommand{\Id}{\mathbb I}
\newcommand{\sphere}[1]{\mathbb S^{#1}}
\newcommand{\wh}{\widehat}

\newcommand{\wt}{\widetilde}
\newcommand{\calL}{\mathcal L}

\newcommand{\KL}{\mathrm{KL}}
\newcommand{\TV}{\mathrm{TV}}

\DeclareMathOperator{\argmin}{\arg\min}
\DeclareMathOperator{\dnorm}{\mathcal N}
\renewcommand{\sp}{\mathrm{span}}

\newcommand{\calT}{\mathcal T}                  
\newcommand{\trg}{T}                   
\newcommand{\thetaT}{\theta_{\trg}}             
\newcommand{\nT}{n_{\trg}}
\newcommand{\sigmaT}{\sigma_{\trg}}

\newcommand{\Xtar}{X^{(\trg)}}
\newcommand{\ztar}{z^{(\trg)}}
\newcommand{\Ztar}{Z^{(\trg)}}
\newcommand{\thetamcT}[1]{\theta_{\trg,#1}}
\newcommand{\ThetaT}{\Theta_{\trg}}
\newcommand{\ept}{\varepsilon^{(\trg)}}
\newcommand{\hatztar}{\wh{z}^{(\trg)}}
\newcommand{\DeltaT}{\Delta_{\trg}}

\newcommand{\hatXtar}{\wh{X}^{(\trg)}}
\newcommand{\hatZtar}{\wh{Z}^{(\trg)}}

\newcommand{\rmqv}{\mathrm{v}^{(\trg)}}
\newcommand{\rmqg}{\mathrm{g}^{(\trg)}}

\newcommand{\rmT}{\mathsf{T}}

\newcommand{\calS}{\mathcal S}
\newcommand{\src}{S}
\newcommand{\thetaS}{\theta_{\src}}
\newcommand{\nS}{n_{\src}}
\newcommand{\nSm}[1]{n_{\src_{#1}}}
\newcommand{\sigmaS}{\sigma_{\src}}
\newcommand{\sigmaSm}[1]{\sigma_{\src_{#1}}}

\newcommand{\Xsrc}{X^{(\src)}}
\newcommand{\Xsrcm}[1]{X^{(\src_{#1})}}
\newcommand{\zsrc}{z^{(\src)}}
\newcommand{\zsrcm}[1]{z^{(\src_{#1})}}

\newcommand{\Zsrcm}[1]{Z^{(\src_{#1})}}
\newcommand{\thetamS}[1]{\theta_{\src_{#1}}}
\newcommand{\thetamSS}[2]{\theta_{\src_{#1},#2}}

\newcommand{\ThetamS}[1]{\Theta_{\src_{#1}}}
\newcommand{\eps}{\varepsilon^{(\src)}}
\newcommand{\epsm}[1]{\varepsilon^{(\src_{#1})}}

\newcommand{\DeltaS}{\Delta_{\src}}
\newcommand{\DeltaSm}[1]{\Delta_{\src_{#1}}}
\newcommand{\hatthetaS}{\wh{\theta}_S}
\newcommand{\hatthetaSm}[1]{\wh{\theta}_{S_{#1}}}

\newcommand{\hatThetaS}[1]{\wh{\Theta}_{S_{#1}}}

\newcommand{\Pest}{\wh{\mathrm P}}
\newcommand{\spclust}{\wh{\mathrm{S}}_{\mathrm{clust}}}
\newcommand{\clustmodule}{\mathsf{TSClust}}
\newcommand{\matchmodule}{\mathsf{HMatch}}

\title{Transfer Learning in High-Dimensional Clustering: Minimax Thresholds and Applications in Single-Cell Data}
\ifanonymize
\author{Anonymous Authors}
\else
\author{
  Abhinav Chakraborty \\
  \small Columbia University \\
  \small \texttt{ac4662@columbia.edu}
  \and
  Sagnik Nandy \\
  \small Ohio State University \\
  \small \texttt{nandy.15@osu.edu}
}
\fi
\date{}

\begin{document}
\maketitle

\begin{abstract}
Clustering is a fundamental problem in statistics, with applications across many scientific disciplines. In many modern applications involving clustering, the primary dataset (the target data) is accompanied by related datasets (the source data). Transferring information from such sources may improve clustering accuracy in the target, making transfer learning for clustering practically important. Despite recent progress, the conditions under which source data improve target clustering remain unclear in high-dimensional settings, even for the canonical Gaussian mixture model. In this paper, we study the clustering problem in a two-community Gaussian mixture model where relatedness is captured by the geometric alignment of the target and source cluster means. We develop a minimax-optimal transfer-assisted clustering procedure and characterize, up to logarithmic factors, the phase transition for consistent target clustering in terms of the signal-to-noise ratios, sample sizes, ambient dimension, and degree of alignment between the datasets.  
The technique is also extended to adaptively choose between the target-only or the source assisted clustering depending on the target signal strength. Furthermore, we also extend our techniques to accommodate multiple communities and and multiple source datasets. 
Extensive simulations and an analysis of a human lung single-cell RNA-sequencing atlas demonstrate the practical effectiveness of our methods.
\end{abstract}

\section{Introduction}

Clustering is an integral component of diverse pipelines in modern data analysis, with applications ranging from single-cell genomics \citep{Kiselev2017,Stuart2019} and image segmentation \citep{Balafar2011-vh,peng2021medical} to spatial transcriptomics \citep{zhao2021bayesspace}. For example, clustering the gene-expression profiles of sequenced cells is an important intermediate step in the clustering based cell-type annotation pipeline implemented in \texttt{Seurat} \citep{Stuart2019}. Similarly, neuroimaging studies use clustering to group subjects with neurological or neurodevelopmental disorders according to high-dimensional
voxel-level measurements or functional-connectivity profiles. In these and other applications, the feature dimension is often comparable to, or larger than, the number of observations. This high-dimensional regime is substantially more challenging since the noise accumulates across coordinates, pairwise distances concentrate, and consequently stronger signal may be required to distinguish
genuine cluster separation from random variation
\citep{CaiZhang2018,LofflerZhangZhou2021,Ndaoud2022}.

Fortunately, the primary dataset of interest (referred as the \emph{target data}) is often not the only relevant dataset available. In many applications, one observes auxiliary \emph{source} datasets whence information can be transferred to assist in clustering the target data. Although the cluster centers in such datasets may not be identical to the target data, a source with a sufficiently strong and well-aligned cluster centers can effectively supplement the limited target sample size, thereby alleviating the concerns arising from the
adverse dimension-to-sample-size ratio in the target dataset. But the utility of these additional information depends on the source sample size and signal strength, as well as the strength of alignment between the target and source cluster centers. Using a poorly aligned or low quality source might deteriorate the performance of the clustering algorithm leading to what is called \emph{negative transfer} \citep{zhang2022survey}.

\begin{figure}[tb]
    \centering
    \includegraphics[width=\textwidth]
    {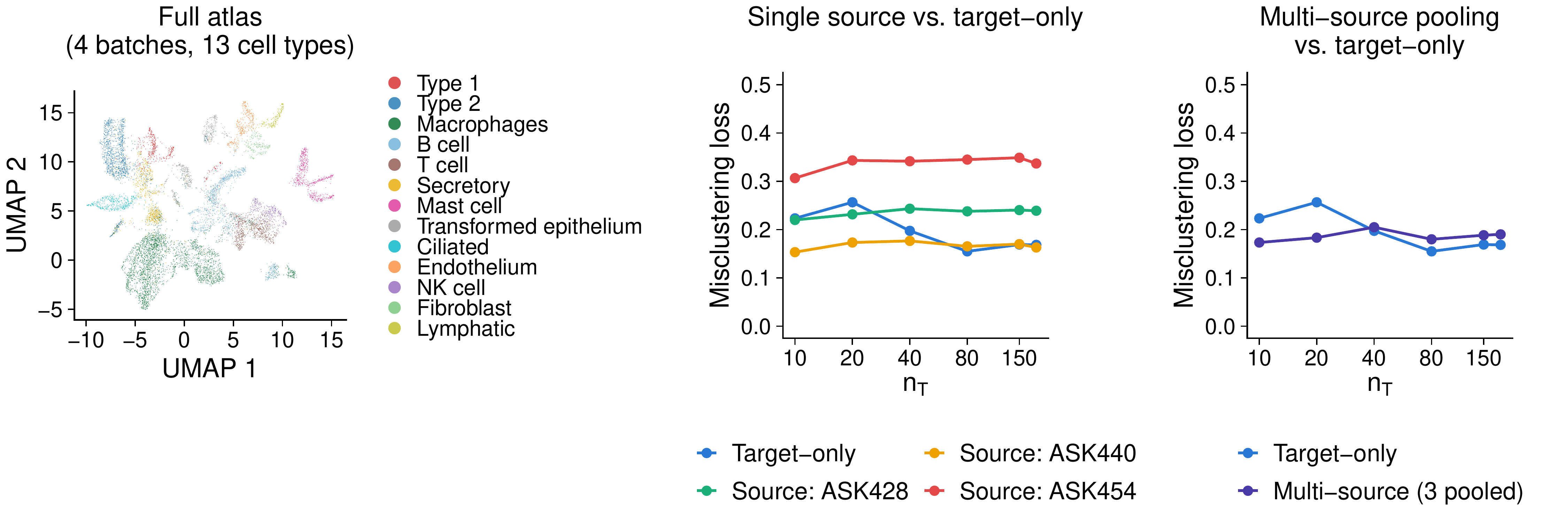}
    \caption{%
    Transfer-assisted clustering on the human lung atlas data from \cite{Vieira_Braga2019-nl},
    illustrated using the NK-cell versus T-cell contrast.
    \textbf{Left:} UMAP representation of the full atlas, containing $9{,}941$
    cells from four Dropseq batches and $13$ annotated cell types.
    \textbf{Center and right:} misclustering loss of our spectral procedures depending on target, and source datasets (to be described in Section~\ref{sec:oracle_two_cluster}) as the
    \texttt{ASK452} target batch is proportionally subsampled from
    $n_T=10$ to its full NK/T-cell sample size of $184$. Each point is
    averaged over $30$ independent replications of the experiment. The candidate source batches used for clustering
    are the patients \texttt{ASK428}, \texttt{ASK440}, and
    \texttt{ASK454}. The figure illustrates that the benefit of transfer
    depends jointly on the target sample size and on the choice and combination of source datasets.}
\label{fig:intro_lung_motivation}
\end{figure}
 
To illustrate the tension between the potential benefits of borrowing information and the risk of worsening performance when using bad source datasets, we consider the problem of separating NK cells from T cells in a log-normalized scRNA-seq dataset of human lung tissue from \citet{Vieira_Braga2019-nl}. The dataset contains the gene expression of $9{,}941$ cells measured over $5{,}000$ highly variable genes and is divided into four batches corresponding to four patients (namely, \texttt{ASK428}, \texttt{ASK440}, \texttt{ASK452}, \texttt{ASK454}). We use \texttt{ASK452} as the target and the remaining patients as candidate source datasets. 
As shown in the second panel of Figure~\ref{fig:intro_lung_motivation}, the benefit of transfer in clustering the target cells depends both on the source dataset used and on the number of available target observations. Some sources (\texttt{ASK440}) substantially improve upon target-only clustering, whereas others provide considerably smaller gains (\texttt{ASK428}), and the relative performance of the methods changes as the target sample size varies. Incorporating information from some sources like \texttt{ASK454} might even have detrimental effect on the clustering performance illustrating negative transfer. Combining multiple sources can also improve upon the use of a target-only estimator when the target sample size is small as seen in panel 3. This dataset is further studied in Section~\ref{sec:lung_atlas}.

Transfer learning has achieved broad success in supervised learning and domain
adaptation \citep{PanYang2010,Weiss2016,BenDavid2010}, and
transfer-assisted clustering has recently begun to receive attention \citep{wang2023transfergmm,tian2025robust}. The existing theoretical works are primarily concerned with a \emph{low-dimensional} regime in which the target
sample size is at least of the order of the ambient dimension and the signal-to-noise ratio is bounded below and above by a constant. Consequently, their
guarantees do not reveal how the target and source signal-to-noise ratios and the target-to-source alignment govern the success or failure of transfer in the regime when the feature dimension grows faster than the sample size,, leaving a gap in the literature. In this paper, we address this question in a high dimensional two community Gaussian mixture model setting, providing a fixed-threshold characterization of necessary conditions for successful transfer clustering. An algorithmic extension to a multi-cluster framework is also proposed and analyzed.  A detailed comparison with the existing literature on high dimensional clustering is deferred to Section~\ref{sec:related_literature}.

In the following, we outline our major contributions.

\begin{enumerate}
\item {\bf A spectral algorithm for transfer-assited clustering in two community GMMs.} In Section~\ref{sec:two_comm_model}, we consider a setting with one target and one source dataset, both generated from two-community Gaussian mixture models whose cluster directions are aligned but not necessarily identical. In Section~\ref{sec:cons_clust_one_source}, we propose a spectral procedure to borrow information from the source dataset to cluster the target observations when the target signal is weak and characterize sufficient conditions for consistent clustering. Our analysis reveals interesting interplay between the feature dimension, target and source sample sizes, their respective signal strengths and the alignment between the two datasets. In Section~\ref{sec:oracle_adaptive_two_cluster}, we further develop an adaptive procedure that switches from target-based to source-based clustering when the target signal is weak and demonstrate that we pay no additional cost for this adaptation. On the technical front, our theoretical analysis involves an interesting Gaussian anti-concentration argument which is non-standard in the analysis of clustering algorithms.

\item {\bf Necessary condition for transfer assisted clustering.} In Section~\ref{sec:necessary_cond}, we characterize necessary conditions for consistent clustering in a two-community GMM when one source dataset is observed in addition to the target data. The lower bound identifies the same target, source, alignment, and product scales that govern the sufficient conditions for the spectral procedure developed in Section~\ref{sec:oracle_two_cluster}. The necessary spectral and transfer thresholds are fixed-constant requirements, whereas our sufficient conditions impose divergence (and, in one regime, logarithmic factors). In the supplement, we extend the spectral algorithm and the corresponding lower-bound to two-community, multiple source settings (cf. Supplement~\ref{sec:two_clust_mult_source}).

\item {\bf Transfer based clustering in multi community framework.} In Section~\ref{sec:mult_clust_mult_source}, we consider a multi-community GMM framework with multiple independent source datasets in addition to the target dataset such that at least one source dataset is sufficiently aligned with the target dataset. We propose a procedure that estimates the source subspace, projects the target observations onto this estimated subspace and clusters the projected embeddings using a variant of relaxed \texttt{K-means} algorithm followed by Lloyd iterations. Under the assumption that at least one source dataset with strong signal strength is sufficiently aligned with the target data, we show that our procedure can help in consistent clustering even when the target signal falls below the information theoretic threshold for consistent clustering using target-only procedures.

\item {\bf Numerical experiments and analysis of the human lung atlas.} We investigate the finite-sample performance of our procedures through a range of synthetic experiments and compare them with the existing methods of \citet{tian2025robust} and \citet{wang_zhou}. To assess their performance in a real application, we apply our multi-community, multi-source methodology to an scRNA-seq data sourced from \citet{Vieira_Braga2019-nl}, comprising gene expression measurements in lung cells corresponding to four independently sequenced patients. Our experiments show that certain variants of our proposed procedure are competitive with the single-cell-specific methods from \citet{gdec_paper} and \citet{nmf_paper} in terms of clustering the sequenced cells into different celltypes. The analysis also reveals several salient features of the dataset, including substantial heterogeneity in the informativeness of the different source batches, sensitivity of transfer performance to the target sample size, and the possibility of comparatively inefficient clustering when the source structure is poorly aligned with the target. At the same time, the adaptive and pooled variants of our procedure provides effective safeguards against relying on an uninformative source and remain competitive across the four target batches.
\end{enumerate}

\subsection{Overview of main theoretical findings}
\label{sec:overview_main_regimes}
We now give an informal summary of our main theoretical findings. Focusing on clustering observations in a stylized two-community Gaussian mixture model with one auxiliary source dataset, we delineate four statistical regimes depending the model parameters where either no method succeeds in consistent clustering, target-only method succeeds, source-assisted method succeeds, or both methods succeed consistently recovering the true cluster labels. The formal statistical model and the loss function considered in this analysis are described in Section~\ref{sec:two_comm_model}. The regimes are determined by the target and source sample sizes $\nT$ and $\nS$, respectively, the corresponding signal strengths $\DeltaT$ and $\DeltaS$, respectively, and the absolute cosine similarity $\mu\in[0,1]$ between the target and source cluster mean directions governing the alignment between the datasets. We refer the readers to Section~\ref{sec:two_comm_model} for formal definitions of these parameters.
The target data alone are sufficient for consistent clustering in the
high-dimensional two-community GMM when the target signal satisfies \citep{Ndaoud2022}
\begin{align}
\label{eq:lower_bound_2_clust_delt}
    \DeltaT \gg
    \max\left\{1,\left(\frac{d}{\nT}\right)^{1/4}\right\}.
\end{align}
In this regime, no source information is needed. The more interesting case is
when the target signal lies below this threshold. We provide an algorithm that uses transfer assisted clustering to consistently recover the target labels provided the source is strong enough and
sufficiently aligned with the target. In other words, if $\DeltaS$ and $\mu$ satisfy
\begin{align}
\label{eq:overview_transfer_regime}
    \DeltaS \gg \left(\frac{d}{\nS}\right)^{1/4},
    \qquad
    \mu\cdot\DeltaT \gg 1,
    \qquad
    \mu\cdot\DeltaS\cdot\DeltaT \gg \sqrt{\frac{d}{\nS}},
\end{align}
up to logarithmic factors, our algorithm can leverage the source information to cluster the target observations. 
The first condition ensures that the source cluster direction can be learned
reliably. The second condition ensures that after projecting the target observations onto the space of estimated source signal the target clusters remain separable. The third condition captures the
additional high-dimensional cost of estimating the source direction before
using it for target clustering.

\begin{figure}[tb]
    \centering
    \includegraphics[scale=0.9]{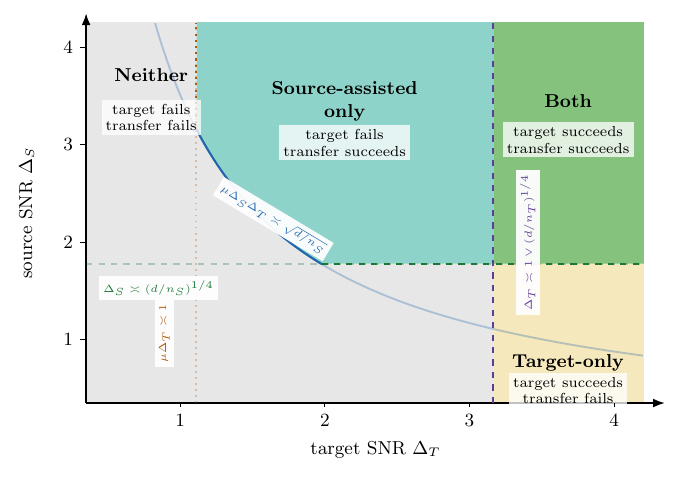}
    \caption{Schematic phase diagram in the two-community, one-source setting,
    shown as a function of the target signal-to-noise ratio $\DeltaT$ and the
    source signal-to-noise ratio $\DeltaS$ for fixed $d$, $\nT$, $\nS$, and
    $\mu$.  The regions indicate whether the target-only route, the
    source-assisted route, both routes, or neither route achieve consistent
    target clustering.  The diagram is intended as an asymptotic regime map and
    suppresses constants and logarithmic factors.}
    \label{fig:overview_transfer_phase_diagram}
\end{figure}

These inequalities reveal a transfer phase diagram in the
$(\DeltaT,\DeltaS)$ plane (cf. Figure~\ref{fig:overview_transfer_phase_diagram}). When \eqref{eq:lower_bound_2_clust_delt} holds (yellow and green regions in the diagram),
target-only clustering succeeds.  When \eqref{eq:lower_bound_2_clust_delt} fails
but \eqref{eq:overview_transfer_regime} holds, source information enables
consistent target clustering even though the target data alone are
insufficient (the blue region). The grey region represents the subcritical
regime in which our lower bound rules out consistent clustering.
The lower bound in Section~\ref{sec:necessary_cond} shows that any consistent procedure must cross fixed universal thresholds at the same target, source, alignment, and product scales.

This regime structure also explains the need for adaptivity. As mentioned before, the adaptive procedure developed in Section~\ref{sec:oracle_adaptive_two_cluster}
attains the same asymptotic success regimes represented in Figure~\ref{fig:overview_transfer_phase_diagram}, while avoiding reliance on a source when the target data already contain sufficient clustering signal. Similar results are also obtained for the multi-community, multi-source setting in Section~\ref{sec:mult_clust_mult_source}.

\subsection{Related literature}
\label{sec:related_literature}
The problem of clustering noisy data into groups of similar observations dates back to Pearson's study of finite mixtures of Gaussian distributions \citep{pearson1893contributions}. Subsequently, numerous model-free procedures were developed, including distance-based methods \citep{Steinhaus1956,Forgy1965,Lloyd1982} and hierarchical clustering algorithms \citep{Ward1963}. Another popular procedure for clustering in Gaussian mixture models is the EM algorithm proposed in \cite{dempster1977maximum}. These classical approaches largely focused on low-dimensional settings in which the feature dimension is much smaller than the sample size. Theoretical guarantees for the consistency of distance-based procedures such as Lloyd's algorithm \citep{lu2016statistical,gao2022iterative,chen2021optimal}, as well as for the EM algorithm \citep{dasgupta2007probabilistic,daskalakis2017ten,xu2016global}, have been extensively studied in this regime. Alternative approaches include spectral clustering \citep{von2007tutorial,vempala2004spectral} and methods based on semidefinite relaxations of the \texttt{K-means} objective \citep{giraud2019partial}.

In the high-dimensional regime, where the feature dimension is comparable to or much larger than the sample size, the problem becomes more nuanced and typically requires substantially stronger signal strength in the form of larger inter-cluster separation. The minimax thresholds for consistent clustering in high-dimensional Gaussian mixture models have been characterized in \citep{CaiZhang2018,LofflerZhangZhou2021,Ndaoud2022,fei2018hidden,giraud2019partial}. Information-theoretically optimal procedures in this setting are typically based on an iterative procedure with spectral initialization followed by Lloyd-type refinement, or on semidefinite relaxations of the \texttt{K-means} procedure. However, none of this literature considers the setting in which auxiliary source datasets containing related but non-identical clustering signals are available in addition to the primary target dataset. Therefore, the precise statistical costs and benefits of transferring information from related source datasets while clustering high dimensional observations remain unexplored.

The problem of transfer assisted clustering with high-dimensional data is practically relevant and has been explored in several scientific disciplines, including molecular genomics, where analysts routinely use \emph{label transfer} procedures \citep{Stuart2019} to annotate new query datasets using large, high-signal, expert-annotated atlases. These procedures typically project the target and source observations into the corresponding leading singular spaces and perform analysis using these transformed observations. It is also unclear whether such reduction is statistically optimal. Furthermore, such procedures also assume that the reference atlas is informative for the query dataset and do not explicitly account for the possibility of negative transfer. Nevertheless, transfer-assisted clustering has begun to receive attention in the statistical literature. In the low-dimensional regime, this problem was studied in \cite{gu2026adaptive}. In \cite{ma2026optimal} and \cite{baharav2025stacked}, the authors consider related problems involving shared subspace estimation. In the context of stochastic block models, \cite{chen2022global} and \cite{yang2025fundamental} study clustering in multilayer networks when a fraction of the labels are shared across layers. These problems are related to, but not directly comparable with, the setting studied here. In particular, their formulations are symmetric across layers and seek to recover the labels in all datasets, whereas our inferential target consists only of the labels in the target dataset. Moreover, their methods are specifically designed for square matrices and effectively operate in proportional asymptotic regimes in which the feature dimension is comparable to the sample size. They therefore do not directly address the high-dimensional regime considered here, where the feature dimension may be substantially larger than the sample size.

\citet{wang_zhou} provides an EM algorithm for transfer assisted clustering but provides limited theoretical analysis. \citet{wang2023transfergmm} study transfer learning for GMM clustering in a distributed peer-to-peer network. Their problem is related but their objective is different from ours.

\citet{tian2025robust} provide a closely related theoretical analysis of unsupervised transfer learning for GMMs. Their TL-GMM uses source tasks to estimate a common discriminant direction, then shrinks the target direction toward it via penalized EM, while tolerating a fraction of arbitrary outlier sources. Their bounds concern target-parameter estimation and excess out-of-sample misclustering risk relative to the Bayes classifier. Our objective instead is recovery of the observed target labels, which may be possible without estimating the full mixture parameters accurately enough for out-of-sample classification.
The theory of \citet{tian2025robust} also operates in a different statistical regime. It assumes that the target sample size is at least of the order of the
ambient dimension, that the target and informative sources have Mahalanobis separation bounded below by a constant, and that the EM iterates are suitably initialized. Its similarity condition controls the Euclidean distance between
the source and target discriminant directions. Thus, although its rates show how dimension, sample size, source closeness, and contamination affect estimation and excess risk in that regime, they do not cover a dimension-larger-than-target-sample, or weak-target-signal setting. Furthermore, they do not
give necessary and sufficient conditions for observed-sample label recovery in
terms of source signal strength and directional alignment. Our results address
these open questions. Our empirical analysis is also motivated by the study of
rate-optimal batch integration in \cite{cao2025modah}. Although the primary
objective there is to estimate cluster means and use them to correct for batch
effects, the analysis also relies on transferring information across batches
which motivated our problem formulation.

To the best of our knowledge, this is the first work to systematically study high-dimensional transfer-assisted clustering and to establish information-theoretic lower bounds characterizing when successful transfer is possible. We expect our analysis to provide theoretical justification for a broad class of transfer-assisted methods in unsupervised learning, including reference-based procedures in the analysis of scRNA-seq data such as \emph{label transfer}, and to offer a general statistical foundation for borrowing information across related datasets widely applied across different scientific disciplines.

\subsection{Notations}
We shall denote the set of natural numbers by $\mathbb N$ and the set of real numbers by $\R$. The set of $k$-dimensional vectors with real entries will be denoted by $\R^k$. The $k$-dimensional unit sphere will be denoted by $\mathbb S^{k-1}$. For any natural number $n \in \mathbb N$, the set $\{1,\ldots,n\}$ will be denoted by $[n]$. For two sequences $\{a_n\}$ and $\{b_n\}$, the notation $a_n \gg b_n$ (or, $a_n \ll b_n$) will imply $a_n/b_n \to \infty$ (or, $a_n/b_n \to 0$), as $n \to \infty$. Similarly, $a_n \gtrsim b_n$ (or, $a_n \lesssim b_n$) will imply there exists an absolute constants $C_0>0$ such that $a_n \ge C_0\,b_n$ (or, $a_n \le C_0\,b_n$) for all $n \ge n_0$ where $n_0 \in \mathbb N$ is sufficiently large. Finally, $a_n \asymp b_n$ implies that there exist constants $C_0,C_1>0$ such that $C_0\,b_n \le a_n \le C_1\,b_n$ for all $n \ge n_0$ where $n_0 \in \mathbb N$ is sufficiently large. By $X_n\convp X$, we mean that $X_n$ converges in probability to $X$. By $X_n\convp\infty$, we mean that, for every fixed $\mathrm M>0$, $\Pbb\bigl(|X_n|>\mathrm M\bigr)\to 1$. We write $X_n=o_p(1)$ if $X_n\convp 0$, and $X_n=O_p(1)$ if the sequence ${X_n}$ is asymptotically tight.

\section{Clustering in two-community Gaussian mixture models}
\label{sec:oracle_two_cluster}
To investigate the fundamental limits of transfer learning for clustering
in high-dimensional Gaussian mixture models, we begin with a stylized
two-community setting, which has been extensively studied in the clustering literature \citep{li2017minimax,Ndaoud2022}.

\subsection{Statistical model and problem objective}
\label{sec:two_comm_model}

We observe a target dataset $\calT := \{\Xtar_1, \ldots, \Xtar_{\nT}\} \subseteq \R^d$ and a source dataset $\calS := \{\Xsrc_1, \ldots, \Xsrc_{\nS}\} \subseteq \R^d$. The target observations follow the signal-plus-noise model
\begin{align}
    \label{eq:target_sample}
    \Xtar_j = \ztar_j \thetaT + \ept_j, \qquad j = 1, \ldots, \nT,
\end{align}
where $\ztar_j \in \{-1, 1\}$ are the unknown target labels and
$\ept_j \simiid \dnorm_d(0_d, \sigmaT^2 \Id_d)$.
The source observations satisfy an analogous model,
\begin{align}
    \label{eq:source_sample}
    \Xsrc_j = \zsrc_j \thetaS + \eps_j, \qquad j = 1, \ldots, \nS,
\end{align}
where $\zsrc_j \in \{-1, 1\}$ are the source labels and
$\eps_j \simiid \dnorm_d(0_d, \sigmaS^2 \Id_d)$.
Throughout, labels are treated as deterministic, and target and source
noise vectors are mutually independent.

Transferring information across datasets is useful only if the cluster means point in sufficiently similar directions. We quantify this alignment by the cosine similarity
\begin{align}
    \label{eq:alignment}
    \frac{|\ip{\thetaT}{\thetaS}|}{\norm{\thetaT}_2 \norm{\thetaS}_2} \geq \mu,
\end{align}
where $\mu \in [0, 1]$; the value $\mu = 1$ corresponds to perfect
alignment up to sign, and smaller values of $\mu$ permit greater
angular deviation between the target and source directions.
The target and source signal-to-noise ratios are defined as
\begin{align}
    \label{eq:def_delta_snr}
    \DeltaT := \frac{\norm{\thetaT}_2}{\sigmaT},
    \qquad
    \DeltaS := \frac{\norm{\thetaS}_2}{\sigmaS}.
\end{align}

While the aforementioned formulation appears stylized, yet Gaussian mixture models of similar flavor (sometimes with more than two components) are routinely used to model data in diverse disciplines including single cell genomics \citep{zhong2022empirical,cao2025modah}, neuroscience \citep{Nezhadmoghadam2021-wl}, and astronomy \citep{10.1093/mnras/stx1024}. Therefore, the models described in \eqref{eq:target_sample}-\eqref{eq:source_sample}, provide reasonable settings for exploring the statistical trade-offs involved in transfer assisted clustering.

Our goal is to recover the target cluster labels $\ztar :=
(\ztar_1, \ldots, \ztar_{\nT})^\top \in \{-1, 1\}^{\nT}$ from the combined data $(\calT, \calS)$, up to a global sign flip. This sign ambiguity is unavoidable since in a two-community Gaussian mixture model, labels are identifiable only up to a global sign flip \citep{li2017minimax,Ndaoud2022}. For an estimator $\hatztar := \wh{z}(\calT, \calS) \in \{-1, 1\}^{\nT}$, define the misclustering loss $\calL : \R^{\nT} \times \R^{\nT} \to [0,1]$ by
\begin{align}
    \label{eq:loss_func_2}
    \calL(\hatztar, \ztar) :=
    \frac{1}{\nT} \min_{s \in \{-1, 1\}}
    \sum_{j=1}^{\nT} \mathbbm{1}\{\hatztar_j \neq s \ztar_j\}.
\end{align}
We aim to characterize conditions on $\DeltaT,\DeltaS$ and $\mu$ involving $\nT,\nS$ and $d$, such that we can ensure $\calL(\hatztar, \ztar) \to 0$, as $\nT,\nS \to \infty$. 
To make the problem precise, we fix $\mu \in [0, 1]$ and $\DeltaT,\DeltaS > 0$, and consider the parameter class
\begin{align}
    \label{eq:param_class}
    \Omega = \Omega(\DeltaT, \DeltaS, \mu)
    := \biggl\{(\thetaT, \thetaS, \ztar, \zsrc) :
    &\; \thetaT, \thetaS \in \R^d,\;
    \ztar \in \{-1,1\}^{\nT},\;
    \zsrc \in \{-1,1\}^{\nS}, \\
    &\; \frac{\norm{\thetaT}_2}{\sigmaT} = \DeltaT,\;
    \frac{\norm{\thetaS}_2}{\sigmaS} = \DeltaS,\;
    \frac{|\ip{\thetaT}{\thetaS}|}{\norm{\thetaT}_2 \norm{\thetaS}_2}
    \geq \mu \biggr\}. \nonumber
\end{align}
We write a generic element of $\Omega(\DeltaT, \DeltaS, \mu)$ as
$(\theta_{all},z_{all}) := (\thetaT, \thetaS, \ztar, \zsrc)$. 
The minimax clustering risk can therefore be defined as
\begin{align}
    \label{eq:minimax_risk}
    \mathfrak{R}(\DeltaT, \DeltaS, \mu)
    := \inf_{\hatztar} \sup_{(\theta_{all},z_{all}) \in \Omega}
    \Ebb\bigl[\calL(\hatztar, \ztar)\bigr],
\end{align}
where the expectation is taken under the joint law of $(\calT, \calS)$
generated by \eqref{eq:target_sample}-\eqref{eq:source_sample}, and the
infimum is over all estimators of the cluster labels depending on $(\calT,
\calS)$. We say that \emph{consistent clustering of the target} is
achievable if
\begin{align}
\label{eq:consistent_clustering_definition}
    \mathfrak{R}(\DeltaT, \DeltaS, \mu) \to 0
    \quad \text{as } \nT, \nS, d \to \infty,
\end{align}
Therefore, our objective can be reduced to characterizing the set of $(\Delta_T,\Delta_S,\mu)$ where the above convergence holds.

\subsection{Consistent clustering with one source}
\label{sec:cons_clust_one_source}
To construct a consistent estimator of the cluster labels, we begin by proposing a meta algorithm that requires a target based estimator of the cluster labels of the target data (referred as \texttt{TARGET-CLUSTERER}), an estimator $\wh{\theta}_S$ of $\thetaS$ and an oracle that informs the analyst if \eqref{eq:lower_bound_2_clust_delt} holds.
\begin{algorithm}[tb]
\caption{Meta Algorithm for Transfer-Assisted Clustering with One Source}
\label{alg:meta_oracle_transfer_clustering_two_community}
\begin{algorithmic}[1]
\Require Target data $\calT=\{\Xtar_j:j\in\calT\}$; source data
$\calS=\{\Xsrc_j:j\in\calS\}$; a clustering procedure
\texttt{GOOD-CLUSTERER}; oracle indicator
$\mathfrak o\in\{\trg,\src\}$.
\If{$\mathfrak o=\trg$}
    \State Set
    $\hatztar\gets\texttt{GOOD-CLUSTERER}(\calT)$.
\Else
 \If{$d\gg\nS$}
   \State Get $\wh{z}^{(\src)}$ using $\texttt{GOOD-CLUSTERER}(\calS)$.
   \State Set $\hatthetaS \gets n^{-1}_\src \cdot \sum_{j=1}^{\nS}\wh z^{(\src)}_j\Xsrc_j$.
 \Else 
 \State Set $\hatthetaS \gets \wh v_\src$, where $\wh v_\src$ is the leading right singular vector of the source data matrix.
 \EndIf
    \For{$j\in\calT$}
        \State Set
        \[
            \hatztar_j
            \gets
            \sgn\!\left(
            \frac{\langle\hatthetaS,\Xtar_j\rangle}
            {\|\hatthetaS\|_2}
            \right).
        \]
    \EndFor
\EndIf
\Output $\hatztar\in\{-1,+1\}^{\nT}$.
\end{algorithmic}
\end{algorithm}
\paragraph{A meta algorithm.}
The meta algorithm is composed of two ingredients: an efficient clustering module \texttt{GOOD-CLUSTERER} and a estimator $\wh \theta_S$ of the source direction $\thetaS$. It has already been established in the literature that one can achieve consistent clustering of $\calT$ using target-only procedures if \eqref{eq:lower_bound_2_clust_delt} holds. The clustering module \texttt{GOOD-CLUSTERER} is chosen to ensure that the estimated cluster labels 
\begin{align}
    \label{eq:oracle_target_estimate}
    \hatztar:=\texttt{GOOD-CLUSTERER}~(\calT) \in \{-1,+1\}^{\nT},
\end{align}
are consistent for the true labels $\ztar$ when \eqref{eq:lower_bound_2_clust_delt} holds.
However, when the oracle suggests that \eqref{eq:lower_bound_2_clust_delt} does not hold but the following condition holds
\begin{align}
    \label{eq:lower_bound_3_clust_delt}
    \DeltaS \gg \left(\frac{d\,(\log \nS)^2}{\nS}\right)^{1/4},
    \qquad
    \mu\cdot\DeltaT \gg 1,
    \qquad
    \mu\cdot\DeltaS\cdot\DeltaT \gg \sqrt{\frac{d}{\nS}},
\end{align}
we avoid directly clustering the target data and instead use the estimated source direction $\wh\theta_S$ to reconstruct $\ztar$ in the following way.
\begin{align}
    \label{eq:oracle_source_estimate}
    \hatztar_j := \sgn\!\left(
    \frac{\langle \hatthetaS, \Xtar_j \rangle}{\|\hatthetaS\|_2}
    \right),
    \qquad j \in \calT.
\end{align}
The basic idea behind the source based procedure is to project the target data onto $\wh \theta_S$ and cluster them by the sign of the projection. 

\paragraph{A spectral choice for \texttt{GOOD-CLUSTERER}.} One may have multiple choices for \texttt{GOOD-CLUSTERER}, including spectral methods from \cite{LofflerZhangZhou2021} or \cite{Ndaoud2022}, and the SDP relaxation based procedure from \cite{giraud2019partial}. However, henceforth in our analysis and implementation, we shall use the spectral procedure outlined in \cite{Ndaoud2022} as \texttt{GOOD-CLUSTERER}. For the sake of completeness, we outline the procedure below. 

Define the hollowed target Gram matrix
\begin{equation}
\label{eq:define_b_cal_t}
    B_{\calT} = X_{\calT} X_{\calT}^\top
    - \diag(X_{\calT} X_{\calT}^\top),
\end{equation}
where $X_{\calT} = \ztar \thetaT^\top + E_{\calT} \in \R^{\nT \times d}$,
with $\ztar = (\ztar_1, \ldots, \ztar_{\nT}) \in \{-1,1\}^{\nT}$ and
$E_{\calT} \in \R^{\nT \times d}$ the matrix of noise rows
$\varepsilon_j^{(\trg)}$.
Let $\wh{u}^{(\trg)}$ be the leading eigenvector of $B_{\calT}$.
Initialize $\wh{\eta}^{(\trg)}_0 := \sgn(\wh{u}^{(\trg)})$ and iterate
\[
    \wh{\eta}^{(\trg)}_{t+1} = \sgn\!\left(B_{\calT}\, \wh{\eta}^{(\trg)}_t\right),
    \qquad t \geq 0.
\]
Then \texttt{GOOD-CLUSTERER} will cluster the observations as
\begin{align}
    \hatztar := \wh{\eta}^{(\trg)}_{\lfloor 3 \log \nT \rfloor}
    \in \{-1, 1\}^{\nT}.
\end{align}
The fundamental advantage of hollowing before computing the eigenvector is that it removes the order $(d/\nT)^{1/2}$ fluctuations of the diagonal entries of $E_{\calT}E_{\calT}^{\top}$. This allows the estimation of the source labels at the sharper signal scale
$\max\{1,(d/\nT)^{1/4}\}$ which is the minimax optimal rate for high-dimensional clustering.

\paragraph{A spectral estimator for $\wh\theta_S$.}
Consider the concatenated source data matrix 
\[
X_{\cal S}=\zsrc\thetaS^\top+E_{\cal S}\in\mathbb R^{\nS\times d},
\]
where the rows of $E_{\cal S} \in \R^{\nS \times d}$ are independent Gaussian noise vectors. Since the source signal matrix $\mathbb E[X_{\cal S}]=\zsrc\thetaS^\top,$ has rank one, with column space spanned by $\zsrc$ and row space spanned by $\thetaS$, a natural estimator of the source signal direction is the leading normalized right singular vector of $X_{\cal S}$. 
When $d\lesssim \nS$, we use this natural estimator as $\wh \theta_S$. In other words,
\begin{align}
    \label{eq:hat_q_i_ld}
    \hatthetaS := \wh{v}_{\calS} \in \R^d,
\end{align}
where $\wh{v}_{\calS}$ is the leading normalized right singular vector of $X_{\calS}$.

When $d\gg \nS$, however, directly estimating $\thetaS$ becomes difficult because noise accumulates across the $d$ coordinates. In this regime, estimating the source labels first can be more effective than directly estimating the source direction. We therefore apply the spectral clustering method \texttt{GOOD-CLUSTERER} to the hollowed Gram matrix
\[
B_{\cal S}=X_{\cal S}X_{\cal S}^{\top}-\diag\left(X_{\cal S}X_{\cal S}^{\top}\right)
\]
to obtain an estimate $\wh z^{(\src)}$ of the source labels, and then construct
\begin{align}
\label{eq:hat_q_i_hd}
\wh\theta_S=\frac{1}{\nS}\sum_{j=1}^{\nS}\wh z^{(\src)}_j\Xsrc_j \in \R^d.
\end{align}
Thus, we estimate the source direction directly when $d\gg \nS$ we estimate $\thetaS$ using an intermediate source clustering step. A natural alternative would be to use $\wh v_{\calS}$ as $\wh\theta_S$. Our current analysis of the resulting clustering procedure, however, yields the stronger sufficient condition $\DeltaS\gg(d/\nS)^{1/2}$, rather than the optimal clustering threshold $(d/\nS)^{1/4}$. It remains unclear whether this stronger requirement reflects an intrinsic limitation of the singular-vector-based estimator or merely a limitation of the present analysis. We therefore adopt the label-first construction above, for which the optimal source signal threshold can be established. However, using the estimator \eqref{eq:hat_q_i_hd} when $d \ll \nS$ is suboptimal as it requires the stronger source signal condition $\DeltaS=O(1)$ for consistent cluster recovery which can be avoided using $\wh v_{\calS}$. When $d \asymp \nS$, both the estimated source label based approach and the approach based on $\wh{v}_{\calS}$ achieve comparable performance. 

On the technical side, the consistency proof for the clustering procedure when $d\lesssim \nS$, with \eqref{eq:hat_q_i_ld} used as the estimator
$\wh\theta_S$, relies on an anti-concentration property of Haar-uniform vectors on the unit sphere. In contrast, the corresponding proof in the regime $d\gg \nS$, using \eqref{eq:hat_q_i_hd}, exploits a Gaussian
anti-concentration argument (cf. Lemma~\ref{lem:gaussian_anti_conc}). The use of such anti-concentration techniques appears to be uncommon in the theoretical analysis of clustering algorithms. The entire meta algorithm with the specific choices of the source direction estimators is summarized in
Algorithm~\ref{alg:meta_oracle_transfer_clustering_two_community}.

 Consider the following theorem characterizing the consistency of the resulting estimator.
\begin{theorem}
\label{thm:correct_choose}
Suppose $d \ge 4$ and we choose the clustering module \texttt{GOOD-CLUSTERER} in Algorithm~\ref{alg:meta_oracle_transfer_clustering_two_community} to be the estimate defined in \eqref{eq:oracle_target_estimate}.
Then if either of the following conditions hold 
\begin{enumerate}[label=\emph{(\Alph*)}]
    \item $\DeltaT\gg\max\!\left\{1,\left(\dfrac{d}{\nT}\right)^{1/4}\right\}$, \quad\emph{or}
    \item $\DeltaS \gg \left(\frac{d\,(\log \nS)^2}{\nS}\right)^{1/4}$,\quad
          $\mu\cdot\DeltaT\gg 1$,\quad\emph{and}\quad
          $\mu\cdot\DeltaS\cdot\DeltaT\gg\!\sqrt{\dfrac{d}{\nS}}$,
\end{enumerate}
the estimator $\hatztar \in \{-1,1\}^{\nT}$ obtained from Algorithm~\ref{alg:meta_oracle_transfer_clustering_two_community} satisfies $\Ebb\bigl[\calL(\hatztar, \ztar)\bigr] \to 0$.
\end{theorem}

\begin{algorithm}[tb]
\caption{Adaptive Transfer-Assisted Clustering for Two Communities}
\label{alg:adaptive_transfer_clustering_two_community}
\begin{algorithmic}[1]
\Require Target data $\calT$; source data $\calS$; target noise variance
$\sigmaT^2$ or its estimate $\wh\sigmaT^2$; threshold constant $C_0>0$.
\State Construct the target-based candidate
$\hatztar_{\trg}$ according to
\eqref{eq:oracle_target_estimate}.
\State Compute the validation statistic
$\wh{\rmT}(\hatztar_{\trg})$ according to
\eqref{eq:validation_statistic}.
\State Compute the threshold $\tau_n$ according to
\eqref{eq:validation_threshold}.
\If{$|\wh{\rmT}(\hatztar_{\trg})|>\tau_n$}
    \State Set $\hatztar\gets\hatztar_{\trg}$.
\Else
    \State Construct $\hatthetaS$ according to
    \eqref{eq:hat_q_i_hd} when $d\gg\nS$, and according to
    \eqref{eq:hat_q_i_ld} when $d\lesssim\nS$.
    \State Construct the source-based candidate
    $\hatztar_{\src}$ according to
    \eqref{eq:oracle_source_estimate}.
    \State Set $\hatztar\gets\hatztar_{\src}$.
\EndIf
\Output $\hatztar\in\{-1,1\}^{\nT}$.
\end{algorithmic}
\end{algorithm}

\begin{remark}
    Observe that the proof of the above theorem does not rely on the particular construction from \cite{Ndaoud2022}. It suffices to use any estimator $\wh z^{(\src)}$ of the source cluster labels as long as the estimator satisfies $\mathbb P\left(\exists\;s \in \{\pm 1\},\;\wh z^{(\src)}=s\zsrc\right) \to 1$, if the first condition of (B) in the foregoing theorem is satisfied. The rest of the proof leverages this exact recovery property and does not rely on specific properties of the estimator $\wh z^{(\src)}$.
\end{remark}

\subsection{Adaptively choosing between the target and the source}
\label{sec:oracle_adaptive_two_cluster}
In this section, we focus on constructing an estimator of $\ztar$ that adaptively chooses between the target based construction in \eqref{eq:oracle_target_estimate} and the source based construction in \eqref{eq:oracle_source_estimate} without relying on the oracle knowledge of whether \eqref{eq:lower_bound_2_clust_delt} holds. 

In that direction, consider the following validation statistic $\wh{\rmT}:\{-1,1\}^{\nT} \to \mathbb R$.
\begin{align}
    \label{eq:validation_statistic}
   \wh{\rmT}(\wt z):=\left\|\frac{1}{\nT}\sum_{j \in \nT}\wt z_j\Xtar_j\right\|^2_2-\frac{d}{\nT}\,\sigmaT^2.
\end{align}
Observe that in the foregoing statistic, the estimated cluster labels $\wt z$ might depend on $\calT$ and hence this validation statistics is different from conventional cross validation that depends on sample splitting. 

To simplify exposition, we shall slightly abuse notation to denote the estimator constructed in \eqref{eq:oracle_target_estimate} by $\hatztar_{\trg}$ and the estimator constructed in \eqref{eq:oracle_source_estimate} by $\hatztar_{\src}$.
We shall show that there exists an absolute constant $C_0>0$ such that the value of $|\wh{\rmT}(\hatztar_T)|$ exceeds the threshold $\tau_n$ defined as
\begin{align}
    \label{eq:validation_threshold}
   \tau_n:=C_0\cdot\sigmaT^2\cdot\left(1+\sqrt{\frac{d}{\nT}}\right)
\end{align}
with probability 1 when \eqref{eq:oracle_target_estimate} holds. On the contrary, if \eqref{eq:oracle_target_estimate} does not hold, the value of $|\wh{\rmT}(\hatztar_T)|$ is less than $\tau_n$ with probability 1. Therefore, the data adaptive selection procedure is given by
\begin{align}
\label{eq:grand_def_adaptive}
    \hatztar:=\begin{cases}
        \hatztar_{\trg} & \mbox{if $|\wh{\rmT}(\hatztar_{\trg})|>\tau_n$,}\\
        \hatztar_{\src} & \mbox{if $|\wh{\rmT}(\hatztar_{\trg})|\le\tau_n$.}
    \end{cases}
\end{align}

Consider the following theorem that states that the target candidate is chosen only if the target signal is strong. In other words, we show that the target candidate is chosen only if \eqref{eq:lower_bound_2_clust_delt} holds.

\begin{theorem}
\label{thm:choose_correct_adaptive}
Assume that $d \ge 4$. Then if \eqref{eq:lower_bound_2_clust_delt} holds, we have $ \smash{\liminf_{\nT \to \infty} \Pbb[\hatztar=\hatztar_{\trg}]=1}$,
for any choice of $C_0>0$. Furthermore, if \eqref{eq:lower_bound_2_clust_delt} does not hold, that is there exists $\mathrm{D}_0>0$ such that
\(
\smash{\DeltaT \le \mathrm{D}_0 \cdot (1+\left(d/n\right)^{1/4})}, 
\)
then one can appropriately choose $C_0>0$ depending on $\mathrm{D}_0$ in \eqref{eq:validation_threshold} to ensure $\smash{\liminf_{\nT \to \infty} \Pbb[\hatztar=\hatztar_{\src}]=1}.$
Furthermore, $\hatztar$ defined by \eqref{eq:grand_def_adaptive} satisfies $\Ebb\bigl[\calL(\hatztar, \ztar)\bigr] \to 0$.
\end{theorem}
The above theorem shows that the validation based procedure from Algorithm~\ref{alg:adaptive_transfer_clustering_two_community} adaptively chooses between the target and the source branches of Algorithm~\ref{alg:meta_oracle_transfer_clustering_two_community} depending on the strength of target signal without paying an extra cost in terms of signal strength. 
The proof of this theorem is presented in Supplement~\ref{sec:proof_thm_2} and depends on an uniform concentration of the validation statistic over all $\wt z \in \{-1,1\}^{\nT}$. This allows for handling data dependent $\wt z$ and hence we can use the estimated $\hatztar$ in validation step. 

It is worthwhile to mention that this adaptive selection procedure requires the knowledge of $\sigmaT^2$. In several applications, it is common to estimate the variance of the baseline noise $\sigmaT^2$ from the data and standardize the entire datasets before analysis (for example, see \cite{zhong2022empirical}). In (2.2) of \citet{zhong2022empirical}, the authors propose an estimate of $\sigmaT^2$ constructed as follows: Suppose $\wh{X} \in \R^{\nT \times d}$ be defined as
\begin{align}
    \wh{X}:=\argmin_{\mathrm X \in \R^{\nT \times d}:\,\operatorname{rank}(\mathrm X) \le 1}\|X_\calT-\mathrm X\|^2_F,
\end{align}
where $X_\calT$ is defined following \eqref{eq:define_b_cal_t}. Then we define
\begin{align}
\label{eq:consistent_plug_in}
        \wh \sigma^2_T:=\frac{1}{\nT \cdot d}\|X_\calT-\wh{X}\|^2_F.
\end{align}
We can show that $\wh \sigma^2_T$ satisfies
\begin{equation}
\label{eq:consistency_sigma}
    \left|\frac{\wh \sigma^2_T}{\sigmaT^2}-1\right| \le \mathrm R_0,
\end{equation}
for some absolute constant $\mathrm R_0>0$ (cf. Theorem~\ref{thm:sigmaT_hat_consistency}). The conclusions of Theorem~\ref{thm:choose_correct_adaptive} continues to hold if $\sigmaT^2$ is replaced in \eqref{eq:validation_statistic} and \eqref{eq:validation_threshold} by $\wh \sigma^2_T$.
This establishes that we can achieve completely data dependent selection between the target and source based estimators and yet achieve consistency in the sense of \eqref{eq:consistent_clustering_definition}. The entire procedure is summarized in Algorithm~\ref{alg:adaptive_transfer_clustering_two_community}.
    
One also needs to choose the constant $C_0>0$ in \eqref{eq:validation_threshold} appropriately to ensure that the adaptive branch chooses the correct estimator. We propose a heuristic parametric bootstrap based selection procedure for $C_0$ in Supplement~\ref{sec:bootstrap_calibration}.  

\begin{remark}
    The analysis in this section can be extended to handle the setting when there are more than one but finite number of independent sources. We develop the theoretical characterization of the regime where consistent clustering is achievable in the two community multi-source setting in Supplement~\ref{sec:two_clust_mult_source}.
\end{remark}

\section{Necessary condition for consistent clustering}
\label{sec:necessary_cond}
The preceding analysis identifies two routes to consistent target clustering:
the target data may be sufficiently informative on their own, or a sufficiently
strong and well-aligned source may compensate for a weak target signal. We now
investigate whether the requirements underlying these two routes are intrinsic
to the problem. To this end, we construct adversarial instances in
$\Omega(\DeltaT,\DeltaS,\mu)$ and show that the minimax risk in
\eqref{eq:minimax_risk} remains bounded away from zero when both the target-only
and source-assisted routes are below their respective fixed thresholds.
Recall the two-community model from section \ref{sec:two_comm_model}, specifically that
 \(
     v_{\trg}:=\thetaT/\|\thetaT\|_2,
     v_{\src}:=\thetaS/\|\thetaS\|_2,
 \)
 so that the alignment condition is $|\ip{v_{\trg}}{v_{\src}}|\ge \mu$ as in \eqref{eq:alignment}.
We first construct a pair of random direction vectors $v_{\trg},v_{\src}\in\R^d$
whose overlap satisfies the required alignment condition with high probability.
To this end, let
\[
\xi=(\xi_1,\ldots,\xi_d)\sim\operatorname{Unif}(\{\pm1\}^d),\qquad
\kappa_1,\ldots,\kappa_d \overset{\mathrm{iid}}{\sim}
    \operatorname{Rad}\left(\frac{1+\wt\mu}{2}\right),
\]
where $\operatorname{Rad}(p)$ denotes a random variable taking values in
$\{+1,-1\}$ with probability $p$ assigned to $+1$. Define
\begin{align}
\label{eq:construct_signal_adv}
v_{\trg}:=\frac{1}{\sqrt d}\xi, \qquad
v_{\src}:=\frac{1}{\sqrt d}(\xi_1\kappa_1,\ldots,\xi_d\kappa_d).
\end{align}
Here $\wt\mu\ge\mu$ is a surrogate alignment parameter, whose explicit
choice is given in Appendix~\ref{app:proof-necessary-cond}. Then
$\|v_{\trg}\|_2=\|v_{\src}\|_2=1$ almost surely, and the overlap between the
source and target directions is $\langle v_{\trg},v_{\src}\rangle:=d^{-1}\,\sum_{r=1}^d \kappa_r$,
which concentrates around $\wt\mu$. Since the parameter class
\eqref{eq:param_class} imposes the pointwise constraint
$|\ip{v_{\trg}}{v_{\src}}|\ge \mu$, we condition on the high-probability event $\mathcal E_{\mu}:=\left\{|\ip{v_{\trg}}{v_{\src}}|\ge \mu\right\}$
in the formal reduction. The surrogate bias is chosen so that the complement
of this event has a prescribed small universal probability, allowing the
surrogate Bayes lower bound to transfer to the deterministic parameter class.

Although the prior couples the target and source directions, the coupling is coordinatewise. Conditional on the latent variables, the $d$ ambient coordinates remain independent; after averaging, the resulting laws are mixtures, which are handled by the likelihood-ratio and replica calculations below. The formal reduction from the surrogate prior to the deterministic alignment class is stated and proved in Appendix~\ref{app:proof-necessary-cond} (cf. Lemma~\ref{lem:sec3_conditioning}).
Assume first that $\nT=2k$; the odd case is handled by discarding one target observation. We reduce our original parameter of interest $\ztar \in \{-1,+1\}^{\nT}$ to the following pairwise product vector 
$\tau=(\tau_1,\ldots,\tau_k)\in\{-1,+1\}^k$, where
\begin{equation}
    \label{eq:sec3_tau}
    \tau_j:=\ztar_{2j-1}\ztar_{2j},
    \qquad j=1,\ldots,k.
\end{equation}
One can construct an instance of $\ztar \in \{-1,+1\}^{\nT}$ and $\zsrc \in \{-1,+1\}^{\nS}$ using $\tau$ through the following construction.
\begin{align}
\label{eq:sec3_eta_param}
&\ztar_{2j-1}=\eta_j,\; \ztar_{2j}=\tau_j\eta_j,
\; \text{for}\, j\in [k],\quad \mbox{and} \quad\zsrc_1,\ldots,\zsrc_{\nS} \simiid \mathrm{Unif}\{-1,+1\}
\end{align}
where $\eta_1,\ldots,\eta_k \simiid \mathrm{Unif}\{-1,+1\}$. Observe that the pair $(v_\trg,v_\src,\ztar,\zsrc)$ generated according to \eqref{eq:construct_signal_adv} and \eqref{eq:sec3_eta_param} constitutes a valid parameter vector contained in $\Omega$ when $\mathcal E_{\mu}$ holds. Through the choice of $\wt \mu$ and conditioning described above, this construction yields the fixed-threshold obstructions stated in Theorem~\ref{thm:necessary_condition}.

Next, we consider the distributions of the joint law $P_\tau$ of the target and source observations $\calT$ and $\calS$, respectively, after marginalizing the randomness from $\zsrc,\xi,\eta$ and $\kappa$. 
Observe that $P_\tau$ is uniquely characterized by the parameter $\tau \in \{-1,+1\}^k$ and hence the family $\{P_\tau:\tau \in \{-1,+1\}^{k}\}$ constitutes a family of distributions indexed by the $k$-dimensional hypercube.

For any estimator $\hatztar\in\{-1,+1\}^{\nT}$ of $\ztar$, define the
corresponding empirical pairwise products by $\wh{\tau}_j:=\hatztar_{2j-1}\hatztar_{2j}$ where $j \in [k]$.
The following lemma shows that the clustering loss $\calL(\hatztar,\ztar)$
defined in \eqref{eq:loss_func_2} dominates the Hamming loss for the induced
pairwise products.
\begin{lemma}
\label{lem:sec3_loss_domination}
For any estimator $\hatztar\in\{-1,+1\}^{\nT}$ of $\ztar$, we have
\[
    \calL(\hatztar,\ztar)
    \ge
    \frac{1}{\nT}\sum_{j=1}^k
    \mathbbm{1}\{\wh{\tau}_j\neq\tau_j\}.
\]
\end{lemma}
This reduction allows us to apply the classical Assouad lemma
\citep{assouad1983deux,Yu1997} to lower bound the worst-case risk
$\mathfrak R(\DeltaT,\DeltaS,\mu)$ defined in \eqref{eq:minimax_risk}.
\begin{remark}
The loss $\calL(\hatztar,\ztar)$ is not lower bounded by the Hamming distance
between $\hatztar$ and $\ztar$, since $\calL$ is invariant under the global sign
flip $\hatztar\mapsto-\hatztar$. This is precisely why we pass to the pairwise
product parametrization: $\tau_j=\ztar_{2j-1}\ztar_{2j}$ is itself invariant
under a global sign flip, and hence the Hamming loss for $\tau$ is compatible
with the clustering loss. The proof of Lemma~\ref{lem:sec3_loss_domination} is
given in Appendix~\ref{app:proof-sec3}.
\end{remark}

In Assouad's lemma, one proceeds by considering the neighboring Assouad marginals defined as follows.
\begin{equation}
    \label{eq:sec3_assouad_marginals}
    P_{+j}:=2^{-(k-1)}\sum_{\tau:\tau_j=+1} P_\tau,
    \qquad
    P_{-j}:=2^{-(k-1)}\sum_{\tau:\tau_j=-1} P_\tau, \quad \mbox{for all $j \in [k]$.}
\end{equation}
By Lemma~\ref{lem:sec3_loss_domination} and the loss-domination form of
Assouad's lemma stated in Lemma~\ref{lem:sec3_loss_domination_assouad}
\citep{assouad1983deux,Yu1997}, 
\begin{align}
\label{eq:sec3_assouad_display}
\inf_{\hatztar}\sup_{\tau\in\{-1,+1\}^k}\Ebb_{P_\tau}\bigl[\calL(\hatztar,\ztar)\bigr]&\ge \inf_{\hatztar}\sup_{\tau\in\{-1,+1\}^k}\Ebb_{P_\tau}\left[\frac{1}{\nT}\sum_{j=1}^k\mathbbm{1}\{\wh{\tau}_j\neq\tau_j\}\right]\\
&\ge \frac{k}{2\nT}\Bigl(1-\max_{1\le j\le k}\TV(P_{+j},P_{-j})\Bigr)=
\frac{1}{4}\Bigl(1-\max_{1\le j\le k}\TV(P_{+j},P_{-j})\Bigr).
\end{align}
Here, $\TV(P,Q)$ denotes the total variation distance of the distributions $P$ and $Q$. Therefore, to lower bound $\inf_{\hatztar}\sup_{\tau\in\{-1,+1\}^k}\Ebb_{P_\tau}\bigl[\calL(\hatztar,\ztar)\bigr]$, it suffices to show that the total variation distances $\TV(P_{+j},P_{-j})$ are uniformly bounded away from 1 over $j \in [k]$.

The remaining task is to control the total variation distances between
the neighboring laws $P_{+j}$ and $P_{-j}$. This is the main technical step:
these laws are high-dimensional mixtures of Gaussian experiments rather than
single Gaussian measures, so their total variation distance must be controlled
through the structure of the mixture.

The reason this TV bound is sufficient is exactly the Assouad reduction in
\eqref{eq:sec3_assouad_display}. If, for some constant $\delta>0$,
\[
    \max_{1\le j\le k}\TV(P_{+j},P_{-j})\le 1-\delta,\quad\mbox{then, by \eqref{eq:sec3_assouad_display}, it implies}\quad
    \inf_{\hatztar}\sup_{\tau\in\{\pm1\}^k}
    \Ebb_{P_\tau}\bigl[\calL(\hatztar,\ztar)\bigr]
    \ge \frac{\delta}{4}.
\]
Thus a uniform TV bound away from one yields a constant lower bound on the
pair-product Hamming risk, and Lemma~\ref{lem:sec3_loss_domination} transfers
that lower bound to the original clustering loss.

\subsection{Bounding the total variation distance}
The remaining task is to control $\TV(P_{+j},P_{-j})$ for a fixed pair
index $j$. We use the following standard identity: if $P$ and $P'$ are both
dominated by a common measure $\nu$, then
\begin{equation}
    \label{eq:tv-density-identity}
    \TV(P,P')
    =
    \frac12\int\left|\frac{dP}{d\nu}-\frac{dP'}{d\nu}\right|\,d\nu
    =
    \frac12\Ebb_{\nu}\left|
    \frac{dP}{d\nu}-\frac{dP'}{d\nu}
    \right|.
\end{equation}
Thus the problem is reduced to choosing a useful dominating measure and
computing likelihood ratios.

For $a\in\{+1,-1\}$, write $P_{aj}$ for the neighboring Assouad marginal
obtained by fixing $\tau_j=a$ and averaging uniformly over
$\tau_{-j}$ and the latent variables $(\eta,\xi,\kappa,\zsrc)$; thus
$P_{+j}$ and $P_{-j}$ are the two cases $a=+1$ and $a=-1$, respectively.
Also let $z_a(\tau_{-j},\eta)$ denote the target label vector obtained by
fixing $\tau_j=a$ and using \eqref{eq:sec3_eta_param} on all pairs.

We take the dominating measure to be the pure-noise Gaussian law $Q$ on
$\R^{d\times \nT}\times\R^{d\times \nS}$. For each coordinate $r\in[d]$, write $x_r=(\Xtar_{1,r},\ldots,\Xtar_{\nT,r})\in\R^{\nT},$ and $y_r=(\Xsrc_{1,r},\ldots,\Xsrc_{\nS,r})\in\R^{\nS}$
for the target and source observations along that coordinate. Under $Q$, the
pairs $(x_r,y_r)$ are independent across $r\in[d]$ and satisfy $x_r\sim\dnorm(0,\sigmaT^2 \Id_{\nT})$, $y_r\sim\dnorm(0,\sigmaS^2 \Id_{\nS}),$ and $x_r\perp y_r.$
This choice removes all signal terms from the reference law. Conditional on the latent variables, the model differs from $Q$ only through Gaussian shifts, so the likelihood ratio factors over coordinates. The next lemma records the
resulting density.

\begin{lemma}
\label{lem:sec3_exact_tv_form}
For each fixed pair index $j$ and $a\in\{+1,-1\}$, the neighboring Assouad
marginal $P_{aj}$, viewed as a law on
$\R^{d\times \nT}\times\R^{d\times \nS}$, has density with respect to $Q$
given by
\begin{align}
    \frac{dP_{aj}}{dQ}
    =e^{-(\nT\DeltaT^2+\nS\DeltaS^2)/2}
    \Ebb_{\tau_{-j},\eta,\xi,\kappa,\zsrc}
    \Biggl[
        \prod_{r=1}^d
        \exp\Bigl\{
            \frac{\DeltaT}{\sigmaT\sqrt d}\xi_r
            \ip{z_a(\tau_{-j},\eta)}{x_r}
            +
            \frac{\DeltaS}{\sigmaS\sqrt d}\kappa_r\xi_r
            \ip{\zsrc}{y_r}
        \Bigr\}
    \Biggr]. \nonumber\\
    \label{eq:sec3_exact_density}
\end{align}
\end{lemma}
Combining the preceding lemma with \eqref{eq:tv-density-identity} gives
\begin{equation}
    \label{eq:sec3_tv_exact}
    \TV(P_{+j},P_{-j})
    =
    \frac12
    \Ebb_Q\left|
        \frac{dP_{+j}}{dQ}
        -
        \frac{dP_{-j}}{dQ}
    \right|.
\end{equation}

The density formula shows why the comparison is not a direct calculation
between two Gaussian laws: the expectation in \eqref{eq:sec3_exact_density}
averages over several unobserved random quantities. Directly bounding the TV
distance between these mixtures is difficult. In two of the bounds below we
upper bound this TV distance by comparing enlarged experiments that include
additional variables from the construction. If $\widetilde P_{+j}$ and
$\widetilde P_{-j}$ denote such enlarged laws and their data marginals are
$P_{+j}$ and $P_{-j}$, then data processing gives
\[
    \TV(P_{+j},P_{-j})
    \le
    \TV(\widetilde P_{+j},\widetilde P_{-j}).
\]
Different enlargements, or no enlargement at all, are useful in different
regimes. We use four bounds. The \emph{target-direction-revealed} route adds
$v_{\trg}$ and reduces the comparison to the tested target pair. The
\emph{revealed-source-direction} route adds the source direction, the source
labels, and the target signs outside the tested pair, giving a conditional
comparison suited to the regime controlled by $\tilde\mu\DeltaT$. The
\emph{doubly subcritical} route applies when
$\nT\DeltaT^4+\nS\DeltaS^4$ is small and controls the original mixture directly
by a second-moment calculation. Finally, the \emph{product-scale} route
conditions on the label variables and uses a row-wise KL bound, which keeps
the term $\tilde\mu^2\nS\DeltaT^2\DeltaS^2/d$. The next theorem records the
corresponding total-variation bounds needed for the Assouad step.

\begin{theorem}[Upper bounds for the relevant total variation distances]
\label{thm:sec3_tv_three_routes}
Fix a pair index $j$. The following total variation bounds hold.
\begin{enumerate}
    \item The target-direction-revealed route yields
    \begin{equation}
        \label{eq:sec3_tv_oracle1}
        \TV(P_{+j},P_{-j})^2
        \le
        C\,\DeltaT^4.
    \end{equation}

    \item For every fixed $\delta_0\in(0,1)$, whenever
    $\tilde\mu\in[0,1-\delta_0]$, the revealed-source-direction route gives
    \begin{equation}
        \label{eq:sec3_tv_oracle}
        \TV(P_{+j},P_{-j})^2
        \le
        C_{\delta_0}\left(
            \tilde\mu^2\DeltaT^2
            +
            \frac{\nT\DeltaT^4}{d}
        \right).
    \end{equation}

    \item There exist universal constants $c,C>0$ such that whenever
    $
        \nT\DeltaT^4+\nS\DeltaS^4\le c\,d,
    $
    the doubly subcritical route gives
    \begin{equation}
        \label{eq:sec3_tv_doubly}
        \TV(P_{+j},P_{-j})^2
        \le
        C\,\frac{\DeltaT^4}{d}.
    \end{equation}

    \item Whenever
    $\nT\DeltaT^2/d\le c^2$ and
    $\tilde\mu^2\nS\DeltaS^2/d\le c^2$, the product-scale route gives
    \begin{equation}
        \label{eq:sec3_tv_product}
        \TV(P_{+j},P_{-j})^2
        \le
        C\,\frac{\nT\DeltaT^4+\tilde\mu^2\nS\DeltaT^2\DeltaS^2}{d}.
    \end{equation}
\end{enumerate}
\end{theorem}

Combining the total-variation bounds in Theorem~\ref{thm:sec3_tv_three_routes}
with the Assouad reduction in \eqref{eq:sec3_assouad_display}, and then
transferring the resulting random-prior lower bound to the deterministic
alignment class via the conditioning argument of Appendix~\ref{app:proof-necessary-cond},
we obtain the following necessary condition for consistent clustering.

\begin{theorem}[Necessary condition for consistent clustering]
\label{thm:necessary_condition}
Consider the two-community model with a single source dataset.
If the minimax risk satisfies
\[
    \inf_{\hatztar}\sup_{(\thetaT,\thetaS,\ztar,\zsrc)\in\Omega(\DeltaT,\DeltaS,\mu)}
    \Ebb\bigl[\calL(\hatztar,\ztar)\bigr]\to 0
    \qquad\text{as }\nT,\nS,d\to\infty,
\]
then necessarily either
\begin{enumerate}[label=\emph{(\Alph*)}]
    \item $\displaystyle
          \DeltaT \ge c_1 \max\left\{1,\left(\dfrac{d}{\nT}\right)^{1/4}\right\}$;
          \quad\emph{or}
    \item $\displaystyle
          \DeltaS\ge c_2\left(\dfrac{d}{\nS}\right)^{1/4}$,\quad
          $\mu\cdot\DeltaT\ge c_3 $,\quad\emph{and}\quad
          $\displaystyle \mu\cdot\DeltaS\cdot\DeltaT\ge c_4\sqrt{\dfrac{d}{\nS}}$.
\end{enumerate}
for some constants $c_1,\ldots,c_4 >0$.
\end{theorem}
\begin{remark}
The necessary conditions in this section extend to the multi-source setting,
with the single source alignment replaced by the maximum alignment across
sources. We characterize the necessary conditions for consistent clustering
under this criterion in Supplement~\ref{sec:two_clust_mult_source}.
\end{remark}

\section{Extension to multi-cluster Gaussian mixture models with multiple source datasets}
\label{sec:mult_clust_mult_source}
In this section, we extend our transfer-assisted clustering procedure in two community Gaussian mixture models to a multi-community framework. Furthermore, in this section, we assume that we observe $m$ independent source datasets. The extension of the two community analysis to multi-source framework is presented in Supplement~\ref{sec:two_clust_mult_source} and the content of this section can also be viewed as a generalization of the content therein. In particular, we postulate a statistical model for the data generating mechanism, propose an algorithm to recover the cluster labels (up to a global permutation), and characterize conditions for consistent recovery of the partition. 

\paragraph{Statistical Model.}
Define the target cluster mean matrix
$\ThetaT:=[\thetamcT{1}\; \cdots \; \thetamcT{K}] \in \R^{d\times K}$, where $\thetamcT{k} \in \R^d$ for $k=1,\ldots,K$. We assume that the target observations $\calT:=\{\Xtar_1,\ldots,\Xtar_{\nT}\}\subseteq \R^d$ are generated as
\begin{align}
\label{eq:mult_clust_target_sample}
    \Xtar_j &= \ThetaT\Ztar_j+\ept_j,
    \qquad \Ztar_j\in\{0,1\}^K,\qquad (\Ztar_j)^\top \mathbbm 1_K=1,\\
    \ept_j &\simiid \dnorm_d(0,\sigmaT^2 \Id_d),
    \qquad j \in [\nT].
\end{align}
Next, for the $m$ independent source datasets
$\calS_1,\ldots,\calS_m$, define the corresponding source cluster mean matrices
$\ThetamS{i}:=[\thetamSS{i}{1}\; \cdots \; \thetamSS{i}{K}]\in \R^{d\times K}$, where $\thetamSS{i}{k}\in \R^d$ for $k \in [K]$ and $i \in [m]$. For each $i\in [m]$, the observations in
$\calS_i:=\{\Xsrcm{i}_1,\ldots,\Xsrcm{i}_{\nSm{i}}\}$ are generated as
\begin{align}
\label{eq:mult_clust_source_sample}
    \Xsrcm{i}_j &= \ThetamS{i}\Zsrcm{i}_j+\epsm{i}_j,
    \qquad \Zsrcm{i}_j\in\{0,1\}^K,\qquad (\Zsrcm{i}_j)^\top \mathbbm 1_K=1,\\
    \epsm{i}_j &\simiid \dnorm_d(0,\sigmaSm{i}^2 \Id_d),
    \qquad j \in [\nSm{i}].
\end{align}
The datasets $\calT,\calS_1,\ldots,\calS_m$ are assumed to be mutually independent. The target and source signal-to-noise ratios are defined as
\begin{align}
\label{eq:def_delta_multi_snr}
    \DeltaT:=\frac{\min_{a\neq b}\|\thetamcT{a}-\thetamcT{b}\|_2}{\sigmaT},
    \quad \mbox{and} \quad
    \DeltaSm{i}:=\frac{\min_{a\neq b}\|\thetamSS{i}{a}-\thetamSS{i}{b}\|_2}{\sigmaSm{i}},
    \quad \mbox{for $i=1,\ldots,m$.}
\end{align}
Observe that the target and source signal-to-noise ratios in the
two-community setup described in \eqref{eq:def_delta_snr} can be recovered from the foregoing definition, up to a factor of $2$, by taking the two target cluster centers to be $\thetaT$ and $-\thetaT$, and the two source cluster centers in $\calS$ to be $\thetaS$ and $-\thetaS$. 
Since the conclusions in Theorem~\ref{thm:correct_choose} are insensitive to constant multiplicative factors in the signal-to-noise ratios, we ignore this constant-factor difference between the two definitions of signal strength.

To enable source-assisted clustering in the low target-signal regime, we assume that there exists at least one source dataset with strong signal. The alignment of this source dataset with the target is quantified by a scalar parameter $\mu \in [0,1]$. Observe that there can exist settings all source datasets with sufficiently strong signal are not aligned with the target. In such settings, transfer learning does not bring any extra benefit. 
In mathematical terms, our assumption on the source datasets is as follows: there exists $\mu\in[0,1]$ and an $i_\star\in[m]$ such that 
\begin{align}
\label{eq:necessary_condition_d_gg_m_src}
& \min_{a \neq b}
    \frac{|\ip{\thetamcT{a}-\thetamcT{b}}{\thetamSS{i_\star}{a}-\thetamSS{i_\star}{b}}|}
    {\|\thetamcT{a}-\thetamcT{b}\|_2\,\|\thetamSS{i_\star}{a}-\thetamSS{i_\star}{b}\|_2}
    \ge \mu,\quad \DeltaSm{i_\star}\gg \left(\frac{d(\beta K)}{\nSm{i_\star}}\right)^{1/4}\sqrt{(\beta K)\vee \log\nSm{i}},\\
    & \mu\cdot\DeltaT \gg \max\left\{K^3\sqrt{\log K},1\right\},\quad
\text{and}\quad
\mu\cdot\DeltaSm{i_\star}\cdot\DeltaT\gg \max\left\{K^3\sqrt{\log K},1\right\}\cdot
    \sqrt{\frac{K(d+\log K)}{\nSm{i_\star}}}.
\end{align}
Furthermore, we assume that there exist constants $\kappa_{\src_1},\ldots,\kappa_{\src_m}>0$ such that, the source $i_\star\in[m]$ satisfying \eqref{eq:necessary_condition_d_gg_m_src}, also satisfies $\mathrm{rank}\left(\ThetamS{i_\star}\right)=K$ and 
\begin{align}
\label{eq:singular_value_lower}
    \sigma_K\left(\ThetamS{i_\star}\right) \ge \DeltaSm{i_\star}\cdot \sigmaSm{i_\star}.
\end{align}
where $\sigma_K$ denotes the $K$-th largest singular value. 
\begin{remark}
This full rank condition on the source cluster mean matrix mentioned above is slightly stronger than what is necessary for clustering in this setup. Indeed, when only the target dataset is considered, \citet{LofflerZhangZhou2021} develop a more delicate argument to relax such a full-rank condition. However, since our primary focus is on characterizing conditions under which information can be successfully transferred from source to target, we adopt this stronger condition.
\end{remark}

Observe that, $\Ztar$ is identifiable only up to a global permutation of its columns
\citep{LofflerZhangZhou2021} which motivates the following misclustering loss.
\begin{align}
\label{eq:loss_func_2m}
    \mathcal L_{\mathrm{mult}}(\hatZtar,\Ztar)
    := \min_{\pi\in\Pi_K}\frac{1}{\nT}\sum_{j=1}^{\nT}
    \mathbbm 1\bigl\{\hatZtar_j\neq \pi\circ \Ztar_j\bigr\},
\end{align}
where $\hatZtar\in\mathcal Z_{\nT,K}$ and
$\smash{\mathcal Z_{\nT,K}:=\{Z\in\{0,1\}^{\nT\times K}:Z\mathbbm 1_K=\mathbbm 1_{\nT}\}}$.
In the above definition, $\Pi_K$ denotes the set of permutations on $\{1,\ldots,K\}$ and for any $\pi\in\Pi_K$ and any $z\in\{0,1\}^K$, the notation $\pi\circ z$ denotes the vector obtained by permuting the entries of $z$ according to $\pi$. For the two community framework, since the cluster labels can be reduced to $\{-1,+1\}$, the foregoing loss function is equivalent to \eqref{eq:loss_func_2}.

We assume that there exists a constant $\beta\ge 1$ such that
\begin{align}
\label{eq:approximate_balance}
    \frac{\nT}{\beta K}
    \le \min_{c\in[K]}\sum_{j=1}^{\nT} \mathbbm 1\{\Ztar_j=e_c\}
    \le \max_{c\in[K]}\sum_{j=1}^{\nT} \mathbbm 1\{\Ztar_j=e_c\}
    \le \frac{\beta \nT}{K}.
\end{align}
Equivalently, each cluster contains a constant-order fraction of $\nT/K$. When the number of clusters is fixed, this is a mild assumption, requiring only that each community is represented in the sample with an asymptotically non-negligible proportion. When $K$ grows with $\nT$, however, it imposes a balance condition on the cluster sizes. Such conditions are widely adopted in the clustering literature \citep{LofflerZhangZhou2021,gao2022iterative}. Nevertheless, they are not essential in all formulations. Indeed, one can instead follow techniques from \cite{giraud2019partial} and explicitly track the dependence on the minimum cluster size; the same approach can be adapted to our analysis. Similar balance conditions are imposed on the informative source dataset $\calS_{i_\star}$,
\begin{align}
\label{eq:approximate_balance_src}
    \frac{\nSm{i_\star}}{\beta K}
    \le \min_{c\in[K]}\sum_{j=1}^{\nSm{i_\star}} \mathbbm 1\{\Zsrcm{i_\star}_j=e_c\}
    \le \max_{c\in[K]}\sum_{j=1}^{\nSm{i_\star}} \mathbbm 1\{\Zsrcm{i_\star}_j=e_c\}
    \le \frac{\beta \nSm{i_\star}}{K}.
\end{align}
Without loss of generality, the same constant $\beta\ge 1$ may be used in both \eqref{eq:approximate_balance} and \eqref{eq:approximate_balance_src}. We also assume that the number of communities $K^2/\nT \to 0$ and $K^2/\nSm{i_\star} \to 0$.

\paragraph{A procedure for consistent clustering.}
First suppose the target signal strength is strong, i.e.,
\begin{align}
\label{eq:trg_mult_c_1}
    \DeltaT \gg \sqrt{K}\;\max\left\{1,
    \left(\frac{d \cdot (\beta K)}{\nT}\right)^{1/4} \right\}.
\end{align}
Then we construct the estimate $\hatZtar_T$ of $\Ztar$ using the relaxed \texttt{K-means} procedure from \cite{giraud2019partial}.

On the other hand, if \eqref{eq:trg_mult_c_1} is not guaranteed to hold but the number of sources $m$ is fixed, we use the following source-based construction. First, we construct estimates $\hatThetaS{1},\ldots,\hatThetaS{m}$ of the source mean matrices
corresponding to the source datasets $\calS_1,\ldots,\calS_m$, respectively. Let $\Pest_\src=\wh{\mathrm U}_\src\wh{\mathrm U}^\top_\src$ be the orthogonal projector onto
\begin{align}
\label{eq:cal_r_b}
    \mathcal R:=\col(\hatThetaS{1})+\cdots+\col(\hatThetaS{m}),
\end{align}
where $\wh{\mathrm U}_\src \in \R^{d \times r}$ is an orthonormal basis of $\mathcal R$ and $r=\operatorname{rank}(\mathcal R)$. 
We then construct the projected observations
\begin{align}
\label{eq:mult_clust_project_data}
    \hatXtar_j:= \wh{\mathrm U}^\top_\src \Xtar_j, \quad \mbox{for $j=1,\ldots,\nT$,}
\end{align}
and cluster them using Algorithm~\ref{alg:tclust} to obtain the source-based estimator $\hatZtar_\src$. Note that the above construction do not need to know the particular source which satisfies \eqref{eq:necessary_condition_d_gg_m_src}. The pooling of source information through $\mathcal R$ ensures that the relevant information is adaptively transferred to the sufficiently aligned source asymptotically. Thus, for fixed $m$, not knowing which source is aligned introduces no additional asymptotic requirement in our consistency guarantee. This conclusion is specific to our problem setting : in more general multi-source parametric estimation, adapting to an unknown source--target bias configuration can entail an unavoidable statistical cost \citep{chakraborty2026statistical}. However, in finite sample setting, including poorly aligned or noisy sources might impact the performance of the algorithm (cf. Section~\ref{sec:mult_clust_sim}).

When oracle knowledge about the validity of \eqref{eq:trg_mult_c_1} is not available, we adaptively choose between the target-based estimator $\hatZtar_\trg$ and the source-based estimator $\hatZtar_\src$ using the following validation rule. Consider any matrix $\wt Z\in\mathcal Z_{\nT,K}$. For each cluster $a\in[K]$, define
\begin{align}
\label{eq:candiate_cluster_size_cluster_mean}
    \wt n_a(\wt Z):=\sum_{j=1}^{\nT}\wt Z_{ja}, \qquad
    \wt{\theta}^{(\trg)}_a(\wt Z):=\frac{1}{\wt n_a(\wt Z)}
    \sum_{j=1}^{\nT}\wt Z_{ja}\,\Xtar_j.
\end{align}
For any $\wt Z \in \mathcal Z_{\nT,K}$ such that $\wt n_a(\wt Z) \neq 0$ for any $a \in [K]$, define the validation statistic
\begin{align}
\label{eq:validation_stat_mult_clust}
    \wh{\mathsf{S}}(\wt Z):=
    \sum_{1 \le a \neq b \le K}\frac{\wt n_a(\wt Z)\,\wt n_b(\wt Z)}{2\nT^2}\,
    \|\,\wt{\theta}^{(\trg)}_a(\wt Z)-\wt{\theta}^{(\trg)}_b(\wt Z)\,\|^2_2
    -\frac{\sigmaT^2\,d\,(K-1)}{\nT}.
\end{align}
Define the validation threshold
\begin{align}
\label{eq:validation_threshold__mult_clust}
    \mathfrak{t}_n:= D_0\cdot K \cdot\sigmaT^2
    \left(\sqrt{\frac{dK}{\nT}}+1\right),
\end{align}
where $D_0>0$ is a sufficiently large constant. The final adaptive estimator is then defined as
\begin{align}
\label{eq:mult_clust_adaptive}
    \hatZtar:=\begin{cases}
        \hatZtar_\trg, & \mbox{if $|\wh{\mathsf{S}}(\hatZtar_\trg)| > \mathfrak{t}_n$ and $\wt n_a(\hatZtar_\trg) \neq 0$ for all $a \in [K]$,}\\
        \hatZtar_\src, & \mbox{otherwise.}
    \end{cases}
\end{align}
The above construction requires the knowledge $\sigmaT^2$ which can be estimated by $\wh \sigma^2_T$ as in \eqref{eq:consistent_plug_in} with $\wh X$ replaced by $\wh X_K$ where
\begin{align}
\label{eq:estimate_sigma_hat_per_mult}
    \wh X_K:=\argmin_{\mathrm X \in \R^{\nT \times d}:\,
    \operatorname{rank}(\mathrm X) \le K}\|X_\calT-\mathrm X\|^2_F.
\end{align}
One can choose $D_0$ using the bootstrap based heuristics from Supplement~\ref{sec:bootstrap_calibration} to get an estimate of $D_0$.

\begin{algorithm}[tb]
\caption{Adaptive Transfer-Assisted Projected Clustering}
\label{alg:projected_clustering_mult}
\begin{algorithmic}[1]
\Require
Target observations $\calT=\{\Xtar_j:j\in[\nT]\}\subseteq\R^d$;
source datasets $\calS_1,\ldots,\calS_m$;
number of clusters $K$; clustering module $\clustmodule$ (Algorithm~\ref{alg:tclust});
validation constant $D_0>0$.
\State Construct the target-based estimator $\hatZtar_\trg$ using the procedure of \citet{giraud2019partial}.
\For{$i=1,\ldots,m$}
    \If{$\nSm{i}\ll d$}
        \State Construct $\hatThetaS{i}$ as in
        \eqref{eq:construction_mult_src_d_gg_m}.
    \ElsIf{$4\le d\lesssim \nSm{i}$}
        \State Construct $\hatThetaS{i}$ as in
        \eqref{eq:construction_mult_d_ll_m}.
    \EndIf
\EndFor
\State Construct the source subspace $\mathcal R$ as in \eqref{eq:cal_r_b},
and let $\Pest^{(\src)}$ denote the corresponding orthogonal projector.
\State Construct the projected target observations
$\{\hatXtar_j:j\in[\nT]\}$ as in \eqref{eq:mult_clust_project_data}.
\State Apply the clustering module to the projected target observations:
\[
\hatZtar_\src=\clustmodule\bigl(\{\hatXtar_j:j\in[\nT]\},K,T_0=2\log\nT\bigr)\in \mathcal Z_{\nT,K}.
\]
\State Compute the validation statistic $\wh{\mathsf S}(\hatZtar_\trg)$
as in \eqref{eq:validation_stat_mult_clust} and the validation threshold $\mathfrak t_n$ as in
\eqref{eq:validation_threshold__mult_clust}.
\State Construct the final adaptive estimator $\hatZtar$ as in
\eqref{eq:mult_clust_adaptive}.
\Output Estimated target label matrix $\hatZtar\in\mathcal Z_{\nT,K}$.
\end{algorithmic}
\end{algorithm}

\paragraph{Estimation of source directions.}
Next, we describe the construction of $\hatThetaS{1},\ldots,\hatThetaS{m}$. Consider a source dataset $\calS_i$. If $\nSm{i}\ll d$, then we first apply the relaxed \texttt{K-means} procedure of \cite{giraud2019partial} to obtain an estimate $\wh Z^{(\src_i)}$ of $\Zsrcm{i}$ and construct the columns of $\hatThetaS{i}\in\R^{d\times K}$ as the empirical cluster means:
\begin{align}
\label{eq:construction_mult_src_d_gg_m}
    \wh{\theta}_{\src_i,k}
    :=\frac{\sum_{j=1}^{\nSm{i}}\wh Z^{(\src_i)}_{jk}\cdot \Xsrcm{i}_j}
    {\sum_{j=1}^{\nSm{i}}\wh Z^{(\src_i)}_{jk}},
    \quad \mbox{for $k=1,\ldots,K$,}
    \qquad
    \hatThetaS{i}:=\left[\wh{\theta}_{\src_i,1}\;\cdots\;\wh{\theta}_{\src_i,K}\right].
\end{align}

On the other hand, if 
$4\le d\lesssim \nSm{i}$, then we estimate the source signal subspace directly from the source data matrix $X_{\calS_i}\in\R^{\nSm{i}\times d}$ using
\begin{align}
\label{eq:construction_mult_d_ll_m}
    \hatThetaS{i}:=\wh V_{\calS_i}\in\R^{d\times K},
\end{align}
where $\wh V_{\calS_i}$ denotes the matrix whose columns are the top $K$ normalized right singular vectors of $X_{\calS_i}$. The estimators defined in \eqref{eq:construction_mult_src_d_gg_m} and \eqref{eq:construction_mult_d_ll_m} are then used to construct the source projector $\Pest^{(\src)}$, which is used to
project the target observations in \eqref{eq:mult_clust_project_data}. The complete procedure, including the adaptive selection between the target-based
and source-based estimators, is summarized in Algorithm~\ref{alg:projected_clustering_mult}.

\paragraph{Theoretical properties.}
To characterize theoretical properties of the above construction, we first assume that we have an oracle that informs us if \eqref{eq:trg_mult_c_1} holds. If it holds, then we use $\hatZtar_T$ defined in the discussion following \eqref{eq:trg_mult_c_1}. If it does not hold, the we use $\hatZtar_S$ defined in the discussion following \eqref{eq:mult_clust_project_data}. In other words,
\begin{align}
    \label{eq:mult_cluster_oracle}
    \hatZtar_{\mathrm{orc}}=\begin{cases}
        \hatZtar_T & \mbox{if \eqref{eq:trg_mult_c_1} holds,}\\
        \hatZtar_S & \mbox{if \eqref{eq:trg_mult_c_1} does not hold.}
    \end{cases}
\end{align}
In the following theorem, we characterize sufficient conditions for consistent
estimation of $\Ztar$ using this oracle based estimator. 

\begin{theorem}
\label{thm:consistency_mult_clust_final}
Suppose that either \eqref{eq:trg_mult_c_1} holds,
or $d \ge \max\{K+3,2K\}$ and \eqref{eq:necessary_condition_d_gg_m_src},~\eqref{eq:singular_value_lower}, and \eqref{eq:approximate_balance_src} holds for some $i_\star \in [m]$.
Then the oracle based estimator from \eqref{eq:mult_cluster_oracle} satisfies
$\mathbb E\left[\mathcal L_{\mathrm{mult}}(\hatZtar_{\mathrm{orc}},\Ztar)\right]\to 0.$
\end{theorem}
In the following theorem, we analysis Algorithm~\ref{alg:projected_clustering_mult} under a stronger inter-cluster separation condition.

\begin{theorem}
\label{thm:consistency_mult_clust_final_adaptive}
If \eqref{eq:trg_mult_c_1} holds
then the adaptive estimator $\hatZtar$ in \eqref{eq:mult_clust_adaptive} satisfies $\mathbb P\left[\hatZtar=\hatZtar_T\right] \to 1.$
Furthermore, if \eqref{eq:trg_mult_c_1} does not hold and there exists $\kappa>0$ such that
\begin{align}
\label{eq:non_diverging_energy_inter_cluster}
\max_{a \neq b}\frac{\|\thetamcT{a}-\thetamcT{b}\|_2}{\sigmaT} \le \kappa\cdot \min_{a \neq b}\frac{\|\thetamcT{a}-\thetamcT{b}\|_2}{\sigmaT} ,
\end{align}
then there exists a $D_0>0$ such that $\mathbb P\left[\hatZtar=\hatZtar_S\right] \to 1.$ In particular, if the conditions of Theorem~\ref{thm:consistency_mult_clust_final} holds along with \eqref{eq:non_diverging_energy_inter_cluster}, then the estimator $\hatZtar$ obtained from Algorithm~\ref{alg:projected_clustering_mult} with validation constant equal to $D_0$ achieves $\mathbb E\left[\mathcal L_{\mathrm{mult}}(\hatZtar,\Ztar)\right]\to 0.$
\end{theorem}

\begin{remark}
Observe that the condition \eqref{eq:non_diverging_energy_inter_cluster} essentially argues that segregating all clusters is almost equally difficult. Without this assumption, the adaptive procedure might choose the target arm when the clusters that are least separated do not satisfy the necessary conditions in Theorem~\ref{thm:consistency_mult_clust_final} yet there exist some directions that are well separated. Nevertheless, our empirical findings in Section~\ref{sec:mult_clust_sim} and the real data analysis in Section~\ref{sec:lung_atlas} that the performance of Algorithm~\ref{alg:projected_clustering_mult} is reasonable under a wide range of scenarios.
\end{remark}

\input{numerical_section.tex}

\input{lung_atlas.tex}

\paragraph{Use of generative AI.} The authors have used ChatGPT 5.5 and 5.6 for proof checking, copyediting text and brainstorming ideas. The major ideas underlying some proofs also developed from discussion with ChatGPT. The authors however verified correctness of all results, re-wrote the arguments and own complete responsibility. The authors have also used Claude code for simulations and data analysis. Once again, all codes were carefully checked by the authors for correctness.

\bibliographystyle{abbrvnat}
\bibliography{clustering_minimax}

\newpage

\appendix

\setcounter{equation}{0}
\setcounter{figure}{0}
\setcounter{table}{0}
\setcounter{algorithm}{0}
\setcounter{lemma}{0}
\setcounter{theorem}{0}

\renewcommand{\theequation}{S\arabic{equation}}
\renewcommand{\thefigure}{S\arabic{figure}}
\renewcommand{\thetable}{S\arabic{table}}
\renewcommand{\thetheorem}{S\arabic{theorem}}
\renewcommand{\theHtheorem}{S\arabic{theorem}}
\renewcommand{\thelemma}{S\arabic{lemma}}
\renewcommand{\theHlemma}{S\arabic{lemma}}
\renewcommand{\thealgorithm}{S\arabic{algorithm}}

\begin{figure}[!b]
    \centering
    \includegraphics[width=\textwidth]
    {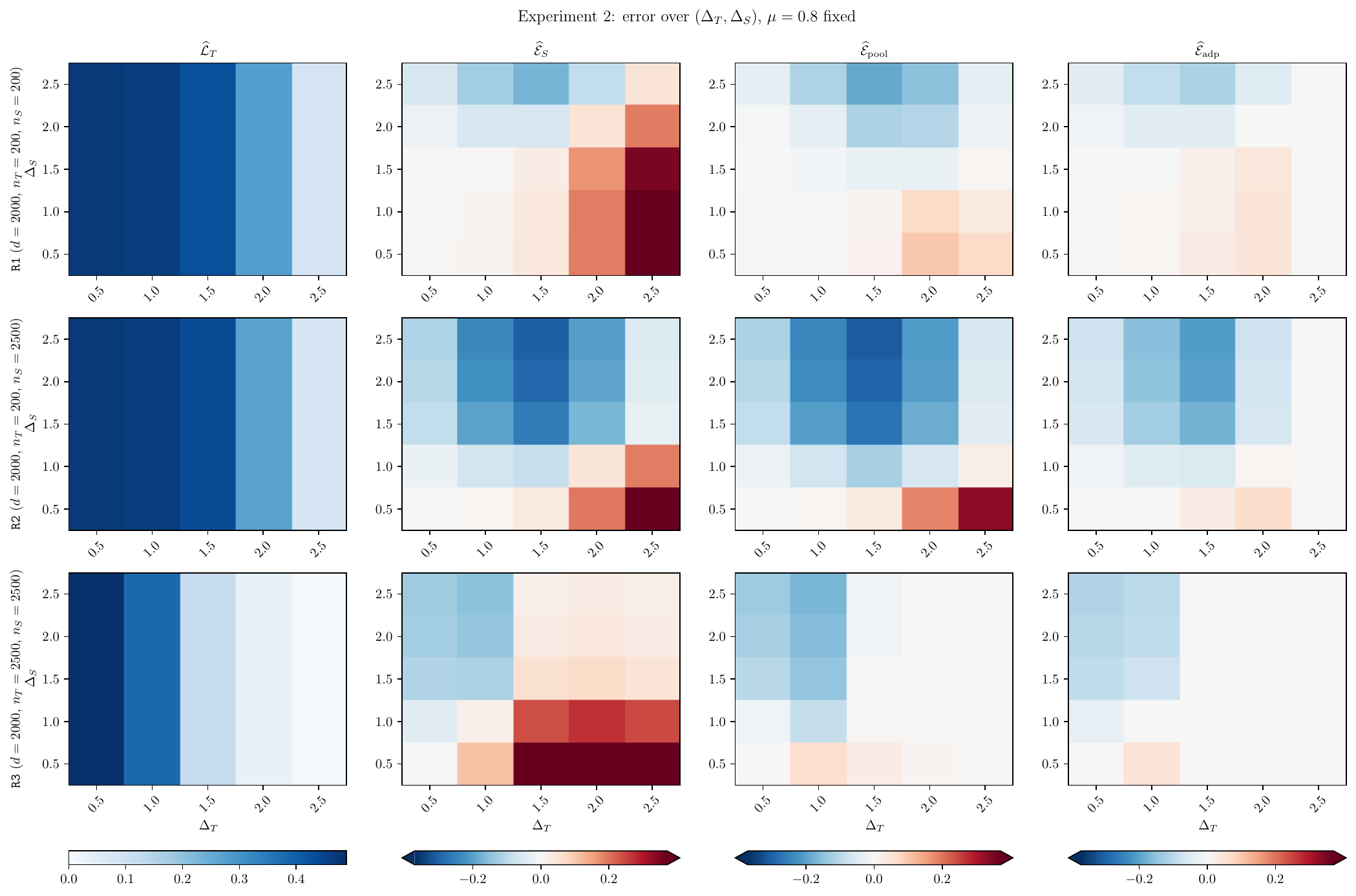}
    \caption{Target misclustering performance over the signal-strength grid
    $(\DeltaT,\DeltaS)$.
    The rows correspond to regimes \texttt{R1}, \texttt{R2}, and
    \texttt{R3}, respectively. The first column reports the target-only
    misclustering error, the second column reports $\wh{\mathcal E}_\src$, the third reports $\wh{\mathcal E}_{\mathrm{pool}}$ and the fourth column reports $\wh{\mathcal E}_{\mathrm{adp}}$.}
    \label{fig:error_heatmaps}
\end{figure}

\section{Additional numerical experiment and details about implementation}
In this section, we provide an additional numerical experiment exploring the dependence of misclustering on the target and source signal strengths. Furthermore, we describe a parametric-bootstrap calibration procedure used to select the absolute constants appearing in
\eqref{eq:validation_threshold} and \eqref{eq:validation_threshold__mult_clust}. When the target sample size is small, the adaptive selection rules in \eqref{eq:grand_def_adaptive} for the two-community setting and
\eqref{eq:grand_def_adaptive_mult_source} for the multi-community setting may be unstable in finite samples. This issue is particularly relevant in applications of transfer learning, where the primary motivation for borrowing information from the source datasets is often the limited size of the target sample. To provide an additional practical alternative in this regime, we also
propose an alternative estimator in this section that avoids the validation route. Finally, we also describe the clustering module $\clustmodule{}$ used in Algorithm~\ref{alg:projected_clustering_mult}.

\subsection{Numerical experiment to explore dependence of clustering performance on target and source signal strengths}
\label{sec:sim_sig_strength}
In this experiment, we study the dependence of the misclustering error on the target and source signal strengths $(\DeltaT,\DeltaS)$ in the regimes \texttt{R1}, \texttt{R2}, and \texttt{R3}. We adopt the same data generating mechanism as Section~\ref{sec:sim_align}, except we fix $\mu=0.8$ and vary the target and source signal strengths as follows:
\begin{align}
\label{eq:sig_strength_grid_sim}
(\DeltaT,\DeltaS) \in \{0.5,\,1.0,\,1.5,\,2.0,\,2.5\} \times \{0.5,\,1.0,\,1.5,\,2.0,\,2.5\}.
\end{align}
For the adaptive procedure the number of bootstrap iterations are reduced to 20 and the rest of the implementation mirrors Section~\ref{sec:sim_align}.
Let the average misclustering losses corresponding to the four methods be denoted by $\wh{\mathcal L}_\trg$ (for the target based estimator), $\wh{\mathcal L}_\src$ (for the source based estimator), $\wh{\mathcal L}_{\mathrm{pool}}$ (for the pooled estimator) and $\wh{\mathcal L}_{\mathrm{adp}}$ (for the adaptive estimator). We use $\wh{\mathcal L}_\trg$ as a baseline, and plot the relative excess errors $\wh{\mathcal E}_\src:=\wh{\mathcal L}_\src-\wh{\mathcal L}_\trg$, $\wh{\mathcal E}_{\mathrm{pool}}:=\wh{\mathcal L}_{\mathrm{pool}}-\wh{\mathcal L}_\trg$, $\wh{\mathcal E}_{\mathrm{pool}}:=\wh{\mathcal L}_{\mathrm{pool}}-\wh{\mathcal L}_\trg$ and $\wh{\mathcal E}_{\mathrm{adp}}:=\wh{\mathcal L}_{\mathrm{adp}}-\wh{\mathcal L}_\trg$ as a heatmap over the grid defined in \eqref{eq:sig_strength_grid_sim} in Figure~\ref{fig:error_heatmaps}.

The blue regions in the last three columns of Figure~\ref{fig:error_heatmaps} indicate an improvement over the target-only estimator, whereas red regions indicate a deterioration in performance.
In the first column we observe that $\wh{\mathcal L}_\trg$ reduces uniformly over all regimes as $\DeltaT$ increases, with faster improvement in regime \texttt{R3}, where $\nT$ is substantially larger. The source-only estimator provides substantial gains when $\DeltaS$ is large and $\DeltaT$ is realtively small (particularly, in \texttt{R2}). 
The regime of negative transfer is widest in \texttt{R3} where target data is more informative because of larger $\nT$.
The pooled estimator exhibits a similar pattern, but its losses relative to the target-only estimator are generally less severe than those of the source-only estimator. 
The adaptive estimator provides the most stable behavior across the parameter grid. It inherits much of the improvement offered by the source when transfer is beneficial, while remaining close to the target-only estimator when the target signal is sufficiently strong. This effect is especially evident in regime \texttt{R3}, where the adaptive procedure largely avoids the substantial negative transfer suffered by the source-only estimator. 

\subsection{Bootstrap calibration of the validation threshold}
\label{sec:bootstrap_calibration}

The validation thresholds in \eqref{eq:validation_threshold} and \eqref{eq:validation_threshold__mult_clust} contain absolute constants $C_0$ and $D_0$ whose values are not specified by the theoretical analysis. We describe a parametric-bootstrap procedure for calibrating these constants. 

For the two-community model, we define the boundary signal strength as
\begin{align}
\label{eq:bootstrap_boundary_signal}
\Delta_{\mathrm{bdry}}:=\max\left\{1,\left(\frac{d}{\nT}\right)^{1/4}\right\}.
\end{align}
Given an estimate of the target noise variance $\wh{\sigma}^2_T$ of $\sigmaT^2$, define the unit-constant threshold scale
\begin{align}
\label{eq:bootstrap_threshold_scale}
\wh q_{\nT,d}:=\wh\sigma^2_T\left(1+\sqrt{\frac{d}{\nT}}\right).
\end{align}
For each bootstrap replication $b=1,\ldots,B$, we independently generate a unit vector $u^{(b)}$ uniformly on the sphere in $\R^d$ and set
\begin{align}
\theta^{(b)}:=\wh\sigma_T\,\Delta_{\mathrm{bdry}}\,u^{(b)}.
\end{align}
We then generate independent labels
$z_1^{(b)},\ldots,z_{\nT}^{(b)}\simiid\operatorname{Unif}\{-1,1\}$ and
synthetic target observations
\begin{align}
\label{eq:bootstrap_synthetic_data}
X_j^{(b)}=z_j^{(b)}\theta^{(b)}+\varepsilon_j^{(b)},\qquad \varepsilon_j^{(b)}\simiid \dnorm_d(0,\wh\sigma^2_T\Id_d).
\end{align}
The random orientation of $\theta^{(b)}$ is immaterial under the isotropic Gaussian model, but avoids privileging any particular coordinate direction. Furthermore, this construction assumes that the cluster labels are approximately balanced with high probability which may not be the case in real data. However, in the absence reliable information about the proportion of different clusters in the sample, this is the best alternative.

Crucially, each bootstrap dataset is processed using exactly the same
target-based estimator and validation statistic as the observed target
dataset. In particular, let $\wh z^{(b)}_{\trg}$
denote the target-based labeling computed from 
$\{X_j^{(b)}:j\in\calT\}$ according to \eqref{eq:oracle_target_estimate}, and define
\begin{align}
\label{eq:bootstrap_statistic}
\mathsf T^{(b)}:=\left\|\frac{1}{\nT}\sum_{j=1}^{\nT}\hat z^{(b)}_{\trg,j}X_j^{(b)}\right\|_2^2-\frac{d}{\nT}\wh\sigma^2_T.
\end{align}
We normalize the resulting statistic by the deterministic scale in
\eqref{eq:bootstrap_threshold_scale} and set
\begin{align}
\label{eq:bootstrap_normalized_statistic}
\mathsf R^{(b)}:=\frac{|\mathsf T^{(b)}|}{\wh q_{\nT,d}}.
\end{align}
For a prescribed upper-tail probability $\alpha\in(0,1)$, the bootstrap
estimate of $C_0$ is
\begin{align}
\label{eq:bootstrap_C0}
\wh C_0:= \wh Q_{1-\alpha}\left(\mathsf R^{(1)},\ldots,\mathsf R^{(B)}\right),
\end{align}
where $\wh Q_{1-\alpha}$ denotes the empirical
$(1-\alpha)$-quantile. The calibrated validation threshold is therefore
\begin{align}
\label{eq:bootstrap_calibrated_threshold}
\wh\tau_n:=\wh C_0\cdot\wh\sigma^2_T\left(1+\sqrt{\frac{d}{\nT}}\right).
\end{align}
The adaptive procedure selects the target-based candidate whenever
\begin{align}
\left|\wh{\rmT}(\hatztar_{\trg})\right|>\widehat\tau_n,
\end{align}
and otherwise selects the source-based candidate. 

The parameters $\alpha$ controls the practical behavior of the selector. When $\alpha=0.5$ both source and target are placed on a equal footing. However, smaller values of $\alpha$ leads to a preference for the source and larger values of $\alpha$ leads to a preference for the target observations. Thus $\alpha$ is a tuning parameter that needs to be chosen by the user depending on domain knowledge about the datasets. In the absence of such domain knowledge, one should choose $\alpha=0.5$. 

For $K>2$, we calibrate $D_0$ using a symmetric $K$-component Gaussian mixture. In each bootstrap replication, the $K$ cluster centers are placed at the vertices of a randomly rotated regular simplex in a $(K-1)$-dimensional subspace of $\R^d$, with equal pairwise squared separation
\begin{align}
\delta_{\mathrm{bdry},K}^2:=\wh\sigma_T^2 \cdot K \cdot\max\left\{1,\,\sqrt{\frac{dK}{\nT}}\right\}.
\end{align}
The bootstrap labels are sampled independently and uniformly from $[K]$, and observations are generated with covariance $\wh\sigma_T^2\Id_d$. Thus, the calibration mixture has equal mixing proportions (and approximately balanced cluster sizes with high probability), affine rank $K-1$, and identical separation between every pair of components.

The same target-only relaxed \texttt{K-means} estimator and validation statistic from \eqref{eq:validation_stat_mult_clust} are applied to each bootstrap dataset. If $\mathsf S^{(b)}$ denotes the resulting statistic in replication $b$, define
\begin{align}
\label{eq:bootstrap_multicluster_scale}
\wh q_{\nT,d,K}:=\wh\sigma_T^2\cdot K \cdot\left\{\sqrt{\frac{dK}{\nT}}+1\right\},
\end{align}
and set
\begin{align}
\label{eq:bootstrap_D0}
\wh D_0:=\wh Q_{1-\alpha}\left(\frac{|\mathsf S^{(1)}|}{\wh q_{\nT,d,K}},\ldots,\frac{|\mathsf S^{(B)}|}{\wh q_{\nT,d,K}}\right).
\end{align}
The calibrated multi-community threshold is then
$\wh D_0\,\wh q_{\nT,d,K}$.

\subsection{Validation-free clustering through a pooled estimator} 
\label{sec:pooled_concat} 
As discussed in the preceding section, when the target sample size is small, the validation-based selection procedure may be unstable in finite samples. Moreover, the adaptive choice between the target- and source-based estimators in Algorithms~\ref{alg:adaptive_transfer_clustering_two_community} and~\ref{alg:projected_clustering_mult} can be sensitive to the tuning parameters $\alpha$ and $B$. We therefore propose the validation-free alternative described in Algorithm~\ref{alg:pooled_transfer_clustering}.

\begin{algorithm}[tb]
\caption{Pooled Transfer-Assisted Clustering}
\label{alg:pooled_transfer_clustering}
\begin{algorithmic}[1]
\Require Target data
$\calT=\{\Xtar_j:j\in\calT\}$; source datasets
$\calS_1,\ldots,\calS_m$; number of communities $K\geq 2$.
\State Form the pooled data matrix
\[
    X_{\mathrm{pool}}
    :=
    \begin{bmatrix}
        X^\top_{\calT}\;\cdots\;X^\top_{\calS_m}
    \end{bmatrix}^\top\in\R^{(\nT+\sum_{i=1}^m\nSm{i})\times d}.
\]
\If{$K=2$}
    \State Apply the target-based two-community clustering construction
    to $X_{\mathrm{pool}}$ and obtain
    \[
        \wh z_{\mathrm{pool}}
        \in
        \{-1,1\}^{\nT+\sum_{i=1}^m\nSm{i}}.
    \]
    \State Set $\wh z_{\calT}:= (\wh z_{\mathrm{pool},1},\ldots,
        \wh z_{\mathrm{pool},\nT})^\top.$
    \State {\bf Output:} $\wh z_{\calT}\in\{-1,1\}^{\nT}$.
\Else
    \State Apply the target-based $K$-community clustering construction
    to $X_{\mathrm{pool}}$ and obtain $\wh Z_{\mathrm{pool}} \in \{e_1,\ldots,e_K\}^{\nT+\sum_{i=1}^m\nSm{i}}.$
    \State Set $\wh Z_{\calT}:=\wh Z_{\mathrm{pool},\,1:\nT}.$
    \State {\bf Output:}
    $\wh Z_{\calT}\in\{e_1,\ldots,e_K\}^{\nT}$.
\EndIf
\end{algorithmic}
\end{algorithm}

In Algorithm~\ref{alg:pooled_transfer_clustering}, we concatenate the target and source data matrices and apply the corresponding target-based clustering procedure to the resulting pooled dataset. The estimated target labels are then obtained by restricting the pooled clustering solution to the observations from the target dataset. Algorithm~\ref{alg:pooled_transfer_clustering}
therefore provides a practical validation-free alternative to
Algorithms~\ref{alg:adaptive_transfer_clustering_two_community}
and~\ref{alg:projected_clustering_mult}. We do not pursue a theoretical
analysis of its misclustering error in the present work.

\subsection{Description of clustering module $\clustmodule{}$ used in Algorithm~\ref{alg:projected_clustering_mult}}
\label{sec:ts_clust_intro}
The clustering module $\clustmodule{}$ used in Algorithm~\ref{alg:projected_clustering_mult} is motivated from \cite{gao2022iterative}. The procedure combines a spectral initialization with Lloyd refinement. Given observations $\{\mathsf X_j:j\in[n]\}\subseteq\R^p$, we first form the data matrix $\mathsf X\in\R^{p\times n}$ and project each observation onto the subspace spanned by its top $K$ left singular vectors. This produces the $K$-dimensional spectral representations $\{\widehat{\mathsf M}_j:j\in[n]\}$. We then apply \texttt{K-Means++} \citep{arthur2007k} with $\varepsilon=O(\log K)$ to these representations to obtain an $\varepsilon$-approximate solution to the $K$-means objective and use the resulting assignments as the initial cluster labels. This approximation can be obtained in polynomial time.

Starting from this spectral initialization, the algorithm performs $T_0$ iterations of Lloyd's method in the original feature space. At each iteration, the cluster centers are recomputed by averaging the original observations currently assigned to each cluster, after which every observation is reassigned to its nearest updated center. The final output is the label vector $\widehat{\mathsf z}^{(T)}\in[K]^n$. Thus, the low-dimensional spectral embedding is used only to construct a suitable initialization, whereas the subsequent refinement operates directly on the original observations.

\begin{algorithm}[tb]
\caption{\(\mathsf{TClust}\): Spectral Initialization Followed by Lloyd Refinement}
\label{alg:tclust}
\begin{algorithmic}[1]
\Require
Data points \(\{\mathsf{X}_j:j\in[n]\}\subseteq\R^p\); number of clusters \(K\); number of Lloyd iterations \(T_0\).
\State Form the data matrix
\(\mathsf{X}=[\mathsf{X}_1,\ldots,\mathsf{X}_n]\in\R^{p\times n}\).
\State Compute the top \(K\) left singular vectors of \(\mathsf{X}\), and let
\(\widehat{\mathsf{U}}\in\R^{p\times K}\) denote the resulting orthonormal matrix.
\For{\(j=1,\ldots,n\)}
    \State Compute the spectral coordinate
    \(\widehat{\mathsf M}_j:=\widehat{\mathsf{U}}^\top \mathsf{X}_j\in\R^K\).
\EndFor
\State Find centers \(\widehat{\mathfrak m}_1,\ldots,\widehat{\mathfrak m}_K\in\R^K\)
and initial labels \(\widehat{\mathsf z}^{(0)}\in[K]^n\) satisfying
\[
    \sum_{j=1}^n
    \|\widehat{\mathsf M}_j-\widehat{\mathfrak m}_{\widehat{\mathsf z}^{(0)}_j}\|_2^2
    \le
    (1+\varepsilon)
    \min_{\substack{\mathfrak m_1,\ldots,\mathfrak m_K\in\R^K\\ z\in[K]^n}}
    \sum_{j=1}^n
    \|\widehat{\mathsf M}_j-\mathfrak m_{z_j}\|_2^2,
\]
where $\varepsilon=O(\log K)$, using the \texttt{K-Means++} algorithm from \citet{arthur2007k}.
\For{\(t=1,\ldots,T\)}
    \For{\(a=1,\ldots,K\)}
        \State Update the \(a\)-th center in the original feature space:
        \[
            \widehat{\mathsf c}_a^{(t)}
            :=
            \frac{
                \sum_{j=1}^n
                \mathbbm 1\{\widehat{\mathsf z}^{(t-1)}_j=a\}\mathsf X_j
            }{
                \sum_{j=1}^n
                \mathbbm 1\{\widehat{\mathsf z}^{(t-1)}_j=a\}
            }.
        \]
    \EndFor
    \For{\(j=1,\ldots,n\)}
        \State Reassign \(\mathsf X_j\) to the nearest updated center:
        \[
            \widehat{\mathsf z}^{(t)}_j
            :=
            \argmin_{a\in[K]}
            \|\mathsf X_j-\widehat{\mathsf c}_a^{(t)}\|_2^2.
        \]
    \EndFor
\EndFor
\Output Estimated labels \(\widehat{\mathsf z}^{(T)}\in[K]^n\).
\end{algorithmic}
\end{algorithm}

\section{Data pre-processing and implementation of benchmark methods in Section~\ref{sec:lung_atlas}}
\subsection{Data acquisition and preprocessing}
\label{sec:data_preprocess}
The lung cell sequencing data of
\citet{Vieira_Braga2019-nl} was deposited at Gene Expression Omnibus
accession
\href{https://www.ncbi.nlm.nih.gov/geo/query/acc.cgi?acc=GSE130148}{GSE130148}
(raw counts and per-cell annotations downloaded from the GEO
series-level supplementary directory,
\url{https://ftp.ncbi.nlm.nih.gov/geo/series/GSE130nnn/GSE130148/suppl/}). The 13 cell types were 
macrophages, Type~1 pneumocytes, Type~2 pneumocytes, B~cell, T~cell, secretory, mast cell, transformed epithelium, ciliated, endothelium, NK~cell, fibroblast, and lymphatic cells. 

The dataset consisted of four patient samples (ASK428, ASK440, ASK452, ASK454), all processed with the same sequencing technology. 
The raw count matrix comprised $10{,}360$ cells and $16{,}327$ genes, resolved into $13$ cell types. At the QC stage, a cell was retained if
its number of expressed genes fell between $\max(Q_{0.02},100)$ and $Q_{0.98}$ (batch-specific 2nd/98th percentiles), its total UMI count was between $200$ and the batch's 98th percentile, and its mitochondrial read fraction was below $25\%$. After the cell QC step, we retained a subset of $9{,}941$ cells. Next, we retained a gene if it was expressed in at least $10$ cells and in at least $2$ of the $4$ batches. Furthermore, all mitochondrial genes were removed. Finally, we retained only those genes that are expressed in at least $2\%$ of the QC-passed cells. Out of the $6{,}063$ QC filtered genes, we considered the $5000$ most highly variable genes. The raw counts corresponding to these $5{,}000$ HVGs were library-size normalized to $10^{4}$ total counts per cell and $\log(1+x)$-transformed. The $\log(1+x)$ reduces skewness typically observed in single cell data and makes it more amenable to modelling using Gaussian mixture models \citep{zhong2022empirical,cao2025modah}. This preprocessing yielded the data matrix $X \in \R^{9941 \times 5000}$ analyzed in Section~\ref{sec:lung_atlas}.

\subsection{Implementation of TL-GMM, NMF and GDEC procedures}
\label{sec:comparator_methods}
For the TL-GMM method of \cite{tian2025robust}, we use isotropic covariance matrices for the Gaussian mixture components, with a common variance across all components, consistent with the models in \eqref{eq:mult_clust_target_sample}--\eqref{eq:mult_clust_source_sample}. For each batch, the variance is estimated using \eqref{eq:estimate_sigma_hat_per_mult} and is held fixed throughout the EM iterations. This modification makes TL-GMM competitive with the four procedures proposed in this paper. We run the EM algorithm for $30$ iterations. To control the shrinkage of the target-based estimates toward the source-based estimates, TL-GMM uses two tuning parameters, $C_{\lambda_0}$ and $\kappa$. In the real-data analysis, we set $\kappa=1/3$, following the original paper, and select $C_{\lambda_0}$ using five-fold cross-validation over the cells.

For the implementation of the NMF procedure in \cite{nmf_paper} an NMF decomposition of rank $K=13$ is fit on the pooled source data with regularization parameters $\alpha = 10$, and $\ell_1 = 0.75$. The target data is reconstructed through this source dictionary, mixed with the target's own raw data, and the resulting representation is clustered. We used a mixing weight of $0.5$ between the source-reconstructed and raw target representations. 

For the GDEC method of \cite{gdec_paper}, we use a stacked denoising autoencoder with hidden-layer dimensions $500$-$500$-$2000$-$20$. The autoencoder is first trained on the pooled source data, and the learned weights are then used to initialize a target autoencoder, which is subsequently fine-tuned on the target data. A deep embedded clustering (DEC) head is attached to the fine-tuned encoder and self-trained on the resulting target embeddings to produce the final cluster assignments. We allow up to $200$ epochs for each of the source pretraining, source fine-tuning, target transfer-training, and DEC self-training stages, consistent with the range considered in the original paper. We use a batch size of $256$ and apply early stopping at each stage.

\section{Transfer assisted clustering in two cluster Gaussian mixture models with multiple source datasets}
\label{sec:two_clust_mult_source}
In this section, we extend the procedure developed in Section~\ref{sec:oracle_two_cluster} to the setting where there are multiple source datasets in addition to the target dataset. Suppose we observe $m=O(1)$ source datasets $\calS_1,\ldots,\calS_m \subseteq \R^d$. Let the data in $\calS_i$ for any $i \in [m]$ are distributed as
\begin{align}
    \label{eq:two_cluster_mult_source}
    \Xsrcm{i}:=\zsrcm{i}_j\thetamS{i}+\epsm{i}_j, \quad j=1,\ldots,\nSm{i},
\end{align}
where $\zsrcm{i}_1,\ldots,\zsrcm{i}_{\nSm{i}} \in \{-1,+1\}$ and $\epsm{i}_j \simiid \dnorm_d(0,\sigmaSm{i}^2\Id_d)$.
The target data points are distributed as described in \eqref{eq:target_sample}. While the target signal strength is defined as in \eqref{eq:def_delta_snr}, the source signal strengths are defined as
\begin{align}
    \label{eq:multi_source_snr}
    \DeltaSm{i}:=\frac{\|\thetamS{i}\|_2}{\sigmaSm{i}}, \quad \mbox{for $i \in [m]$.}
\end{align}
We assume that there exists at least one source that is sufficiently aligned with the target dataset. In other words, we assume that 
\begin{align}
\label{eq:multis_alignment}
    \max_{i=1}^{m} \frac{|\langle \thetamS{i},\thetaT\rangle|}{\|\thetamS{i}\|_2\,\|\thetaT\|_2} \ge \mu, \quad \mbox{where $\mu \in (0,1]$.}
\end{align}

\subsection{Consistent clustering in the multi-source setting}
In this section, we provide an algorithm for clustering the target data points under the oracle set-up where it is known if \eqref{eq:lower_bound_2_clust_delt} holds. We use the estimator $\hatztar$ defined in \eqref{eq:oracle_target_estimate}, if \eqref{eq:lower_bound_2_clust_delt} holds. When \eqref{eq:lower_bound_2_clust_delt} does not hold, we proceed with the following two-stage construction: First, we construct candidate directions $\hatthetaSm{1},\ldots,\hatthetaSm{m}$ from the $m$ sources. Again, the explicit construction of the source directions depend on whether $d \gg \nSm{i}$ or $4 \le d \lesssim \nSm{i}$ for each $i \in [m]$. We adopt the same constructions as in \eqref{eq:hat_q_i_hd} when $d \gg \nSm{i}$ or \eqref{eq:hat_q_i_ld} when $d \lesssim \nSm{i}$. The estimates $\hatthetaSm{1},\ldots,\hatthetaSm{m}$ are used to construct the estimated source signal subspace 
\begin{align}
\label{eq:combined_source_space}
\spclust^{(\src)}:=\sp\{\hatthetaSm{1},\ldots,\hatthetaSm{m}\}\subseteq \R^d.
\end{align}
Let $\widehat Q_{\src}\in\R^{d\times r_{\src}}$ be an orthonormal basis of
$\spclust^{(\src)}$, where
$r_{\src}:=\dim(\spclust^{(\src)})\le m$, and define the corresponding orthogonal projector $\Pest^{(\src)}:=\widehat Q_{\src}\widehat Q_{\src}^{\top}.$
The source-based procedure then proceeds in two steps:
\begin{enumerate}
    \item \textbf{Projection.}
    For each $j\in\calT$, compute the reduced-dimensional projected observation
    \begin{align}
    \label{eq:def_hat_x_tar_q_tar}
        \hatXtar_j:=\widehat Q_{\src}^{\top}\Xtar_j\in\R^{r_{\src}}.
    \end{align}
    \item \textbf{Clustering.}
    Construct the Gram matrix
    \begin{align}
    \label{eq:hat_sigma_t}
        \wh{\Sigma}_{\calT}:=\frac{1}{\nT}\sum_{j=1}^{\nT}\hatXtar_j(\hatXtar_j)^\top\in\R^{r_{\src}\times r_{\src}}.
    \end{align}
    Let $\wh v_{\trg}\in\R^{r_{\src}}$ be a unit-norm leading eigenvector of
    $\wh{\Sigma}_{\calT}$. The estimate $\hatztar \in \{-1,1\}^{\nT}$ is then defined as
    \begin{align}
    \label{eq:oracle_source_estimate_mult_source}
        \hatztar_j:=\sgn\left(\wh v_{\trg}^{\top}\hatXtar_j \right),\qquad j\in\calT.
    \end{align}
\end{enumerate}
The procedure is summarized in Algorithm~\ref{alg:oracle_transfer_clustering}. Observe that the aforementioned procedure reduces to the estimation strategy outlined in Section~\ref{sec:oracle_two_cluster} when only one source dataset is observed. Furthermore, also observe that, the construction does not require the knowledge of which source dataset is aligned with the target. It uses the pooling of source information through $\spclust^{(\src)}$ and in the following theorem we show that asymptotically it suffices to have one source dataset satisfying 
\begin{align}
\label{eq:correct_choose_alignment}
    \frac{|\langle \thetamS{i},\thetaT\rangle|}{\|\thetamS{i}\|_2\,\|\thetaT\|_2} \ge \mu,
\end{align}
which also satisfy the signal strength conditions
\begin{align}
\label{eq:necessary_condition_mult_source}
\DeltaSm{i} \gg \left(\frac{d\;(\log \nSm{i})^2}{\nSm{i}}\right)^{1/4},\quad\mu\cdot\DeltaT \gg 1\quad\text{and}\quad\mu\cdot\DeltaSm{i}\cdot\DeltaT \gg \sqrt{\frac{d}{\nSm{i}}},
\end{align}
for consistent clustering of the observations in $\calT$.

\begin{theorem}
    \label{thm:two_clust_multi_source}
Consider the estimator $\hatztar$ of $\ztar$ produced in Algorithm~\ref{alg:oracle_transfer_clustering}. If \eqref{eq:lower_bound_2_clust_delt} holds, or there exists $i \in [m]$ satisfying \eqref{eq:correct_choose_alignment} and \eqref{eq:necessary_condition_mult_source}, then for any $d \ge 4$, we have $\mathbb E\left[\calL\left(\hatztar,\ztar\right)\right]\to 0,$ where the loss $\calL(\cdot)$ is defined in \eqref{eq:loss_func_2}.
\end{theorem}
One can extend the oracle based procedure in Algorithm~\ref{alg:oracle_transfer_clustering} to a procedure that adaptively chooses between the target and source based estimators following the ideas outlined in \eqref{eq:grand_def_adaptive}. In other words, we consider the same validation statistic $\wh{\rmT}$ defined in \eqref{eq:validation_statistic} and the validation threshold $\tau_n$ defined in \eqref{eq:validation_threshold}. Again, let us denote the target based estimator defined in \eqref{eq:oracle_target_estimate} by $\hatztar_{\trg}$ and the estimator constructed in \eqref{eq:oracle_source_estimate_mult_source} by $\hatztar_{\src}$.
Then the adaptive estimator is defined as 
\begin{align}
\label{eq:grand_def_adaptive_mult_source}
    \hatztar:=\begin{cases}
        \hatztar_{\trg} & \mbox{if $|\wh{\rmT}(\hatztar_{\trg})|>\tau_n$,}\\
        \hatztar_{\src} & \mbox{if $|\wh{\rmT}(\hatztar_{\trg})|\le\tau_n$.}
    \end{cases}
\end{align}
Then the conclusion of Theorem~\ref{thm:choose_correct_adaptive} extends to the estimator defined in \eqref{eq:grand_def_adaptive_mult_source}. This extension to the current setup is possible since the validation principle works based on the property of the target based estimator $\hatztar_{\trg}$ which we do not change between the single source and multi-source setups. 

\subsection{Necessary conditions for consistent clustering}
To characterize the necessary conditions for consistent clustering in the current setting, we 
consider a stratified multi-source parameter space.
For each $i^\star\in[m]$, define
\begin{align}
\label{eq:param_class_mult_i}
    \Omega_{\mathrm{mult}}^{[i^\star]}
    :=
    \left\{
        (\theta_{\mathrm{mult}},z_{\mathrm{mult}})\in\Omega_{\mathrm{mult}}:
        \frac{|\ip{\thetaT}{\thetamS{i^\star}}|}
        {\norm{\thetaT}_2\norm{\thetamS{i^\star}}_2}\ge\mu
    \right\}.
\end{align}
We use the oracle-indexed minimax risk
\begin{align}
\label{eq:minimax_risk_mult_oracle_index}
    \mathfrak R_{\mathrm{or}}^{[i^\star]}
    :=
    \inf_{\hatztar_{i^\star}}
    \sup_{(\theta_{\mathrm{mult}},z_{\mathrm{mult}})
        \in\Omega_{\mathrm{mult}}^{[i^\star]}}
    \Ebb\bigl[\calL(\hatztar_{i^\star},\ztar)\bigr].
\end{align}
Here the estimator is allowed to use all $m$ source datasets and is also
given the index $i^\star$ of the informative and aligned source.  Thus this oracle risk makes the lower-bound
experiment easier than the adaptive problem, in which the identity of an
aligned informative source is unknown. Then the following theorem characterizes the necessary on the problem parameters $(\DeltaT,\DeltaSm{i_\star},\mu)$ for consistent clustering.

\begin{theorem}[Oracle-indexed multi-source necessary condition]
\label{thm:two_clust_multi_source_lb}
Consider the two-community model with $m=O(1)$ source datasets and fix
$i^\star\in[m]$. If $\mathfrak R_{\mathrm{or}}^{[i^\star]}\to0,$ as
$\nT,\nSm{i^\star},d\to\infty$,
then necessarily either
\begin{align}
\label{eq:multi_source_lb_condition}
&\DeltaT\gtrsim\max\!\left\{1,\left(\frac{d}{\nT}\right)^{1/4}\right\}, \quad \text{or}\\
\notag
&\DeltaSm{i^\star}\gtrsim
\left(\frac{d}{\nSm{i^\star}}\right)^{1/4},\qquad
\mu\DeltaT\gtrsim 1,\qquad\text{and}\qquad
\mu\DeltaSm{i^\star}\DeltaT
\gtrsim \sqrt{\frac{d}{\nSm{i^\star}}},
\end{align}
where the constants involved in the $\gtrsim$ statements are derived from Theorem~\ref{thm:necessary_condition}.
\end{theorem}
The proof embeds the corrected one-source hard experiment for the designated
index $i^\star$ and fixes every other source at a common nuisance distribution.
The full argument is given in Appendix~\ref{app:proof-multisource-lb}.

\begin{remark}
Observe that the above is an oracle lower bound which restricts to the class of estimators that is provided the information of the useful source. Therefore, failure even in this enlarged experiment implies failure when $i^\star$ is unknown.
The procedure in
Algorithm~\ref{alg:oracle_transfer_clustering}, however, is not indexed by
$i^\star$; it is consistent whenever there exists an index satisfying the
alignment and the corresponding divergence conditions.  In this sense the
upper procedure adapts over the family
$\{\Omega_{\mathrm{mult}}^{[i^\star]}:i^\star\in[m]\}$ without knowing the
active index.  The thresholds in Theorem~\ref{thm:two_clust_multi_source_lb} are fixed-constant requirements, whereas Theorem~\ref{thm:two_clust_multi_source} is stated in terms of divergence of the signal strength and includes a logarithmic factor for source-direction estimation. When $m=1$, the statement reduces to Theorem~\ref{thm:necessary_condition}.
\end{remark}

\begin{algorithm}[tb]
\caption{Oracle Transfer-Assisted Clustering}
\label{alg:oracle_transfer_clustering}
\begin{algorithmic}[1]
\Require Target data $\{\Xtar_j:j\in\calT\}$; source datasets
$\calS_1,\ldots,\calS_m$; oracle indicator $\mathfrak o\in\{\trg,\src\}$.
\If{$\mathfrak o=\trg$}
    \State Form $B_{\calT}=X_{\calT}X_{\calT}^\top-
    \diag(X_{\calT}X_{\calT}^\top)$ and let $\wh u^{(\trg)}$ be its leading eigenvector.
    \State Set $\wh\eta^{(\trg)}_0=\sgn(\wh u^{(\trg)})$ and iterate
    $\wh\eta^{(\trg)}_{t+1}=\sgn(B_{\calT}\wh\eta^{(\trg)}_t)$ for
    $t<\lfloor 3\log\nT\rfloor$.
    \State Set $\wh z_{\calT}=\wh\eta^{(\trg)}_{\lfloor 3\log\nT\rfloor}$.
\Else
    \State Estimate $\hatthetaSm{1},\ldots,\hatthetaSm{m}$ and let
    $\widehat Q_{\src}$ be an orthonormal basis of their span.
    \State Set $\widehat Y_j=\widehat Q_{\src}^\top\Xtar_j$ and
    $\widehat\Sigma_{\calT}=\nT^{-1}\sum_{j=1}^{\nT}\widehat Y_j\widehat Y_j^\top$.
    \State Let $\widehat v_{\calT}$ be the leading eigenvector of
    $\widehat\Sigma_{\calT}$ and set
    $(\wh z_{\calT})_j=\sgn(\widehat v_{\calT}^\top\widehat Y_j)$.
\EndIf
\Output $\wh z_{\calT}\in\{-1,1\}^{\nT}$.
\end{algorithmic}
\end{algorithm}

\section{Anti-concentration, results from random matrix theory and properties of Gaussian random vectors}
In this section, we present some results about Gaussian anti-concentration and some results from random matrix theory used in the proofs presented in the later sections.

\begin{lemma}[Gaussian anti-concentration]
\label{lem:gaussian_anti_conc}
Let $\mathrm Z \sim \dnorm(\nu,\sigma^2)$. Then we have the following.
\begin{enumerate}
\item For every $t>0$, there exists an absolute constant $C>0$ such that
\[
    \mathbb P(|\mathrm Z|\le t) \le C\frac{t}{\sigma}.
\]

\item If $|\nu|\ge \sigma$ and $0<t\le |\nu|/2$, then there exists an absolute constant $C>0$ such that
\[
    \mathbb P(|\mathrm Z|\le t) \le C\frac{t}{|\nu|}.
\]
\end{enumerate}
\end{lemma}

\begin{proof}
For the first bound, observe that the Gaussian density $\phi(\cdot)$, is bounded by $(2\pi)^{-1/2}$. Therefore
\[
    \mathbb P(|\mathrm Z|\le t)\le \frac{2t}{\sigma} \sup_x \phi((x-\nu)/\sigma)\le \wt C_1\frac{t}{\sigma}, \quad \mbox{where $\wt C_1:=2\cdot(2\pi)^{-1/2}$.}
\]
For the second bound, again observe that 
\[
    \mathbb P(|\mathrm Z|\le t)\le \frac{2t}{\sigma} \sup_{|x| \le t} \phi((x-\nu)/\sigma).
\]
If $|x|\le t\le |\nu|/2$, then $|x-\nu|\ge |\nu|/2$, so
\[
\frac{1}{\sigma} \phi((x-\nu)/\sigma)\le\frac{1}{\sigma\sqrt{2\pi}}\exp\left(-\frac{\nu^2}{8\sigma^2}\right).
\]
Therefore, there exists a constant $\wt C_2>0$
\[
    \mathbb P(|\mathrm Z|\le t)\le
    \wt C_2\frac{t}{\sigma}
    \exp\left(-\frac{\nu^2}{8\sigma^2}\right).
\]
Since the function $x\mapsto x e^{-x^2/8}$ is bounded on $[1,\infty)$, we can use $|\nu|\ge \sigma$ and adjust the constant $\wt C_2$ to obtain 
\[
    \frac{1}{\sigma}\exp\left(-\frac{\nu^2}{8\sigma^2}\right) \le \frac{\wt C_2}{|\nu|}.
\]
Choosing $C:=\max\{\wt C_1,\wt C_2\}$ proves both the claims.
\end{proof}

Next, we consider the following bound on the norm of a $d$-dimensional Gaussian random vector.

\begin{lemma} 
\label{lem:gaussian_vector_norm} 
Suppose $\mathsf Z\sim\dnorm_d(\theta,\sigma^2\Id_d),$ where $\theta\in\R^d$ and $\sigma>0$. Then, for every $t>0$, 
\begin{align} 
\label{eq:gaussian_vector_norm} 
\|\mathsf Z\|_2 \le \|\theta\|_2 + \sigma\left(\sqrt d+\sqrt{2t}\right) 
\end{align} 
with probability at least $1-e^{-t}$. In particular, there exists an absolute constant $C>0$ such that 
\[ 
\|\mathsf Z\|_2 \le \|\theta\|_2+C\sigma\sqrt d 
\] with probability at least $1-e^{-d}$. 
\end{lemma}

\begin{proof}
Write $\mathsf Z=\theta+\sigma \mathsf g,$ where $\mathsf g\sim\dnorm_d(0,\Id_d)$. By the triangle inequality, 
\[ 
\|\mathsf Z\|_2 \le \|\theta\|_2+\sigma\|\mathsf g\|_2. 
\]
Moreover, $\|\mathsf g\|_2^2$ follows a chi-squared distribution with $d$ degrees of freedom. By Lemma~1 of \citet{laurent2000adaptive}, for every $t>0$, 
\[ 
\mathbb P\left( \|\mathsf g\|_2^2 > d+2\sqrt{dt}+2t \right) \le e^{-t}. 
\]
Since $d+2\sqrt{dt}+2t \le \left(\sqrt d+\sqrt{2t}\right)^2,$ it follows that, with probability at least $1-e^{-t}$, 
\[ 
\|\mathsf g\|_2 \le \sqrt d+\sqrt{2t}. 
\] 
Combining the preceding two displays proves \eqref{eq:gaussian_vector_norm}. Taking $t=d$ yields 
\[ 
\|\mathsf Z\|_2 \le \|\theta\|_2+(1+\sqrt 2)\sigma\sqrt d 
\] 
with probability at least $1-e^{-d}$, proving the final claim.
\end{proof}

Finally, we recall the following property of Haar uniformly distributed vectors on $p$-dimensional spheres.

\begin{lemma}
\label{lem:haar_coordinate_density}
Let $\mathsf{V}$ be uniformly distributed on $\mathbb S^{p-1}$, the unit sphere in $\mathbb{R}^p$, with $p \geq 3$. Then for any fixed unit vector $u \in \mathbb{R}^p$, the random variable $u^\top \mathsf{V}$ has density
\begin{equation}
\label{eq:haar_density}
    f_p(x) = \frac{\Gamma\!\left(\frac{p}{2}\right)}{\sqrt{\pi}\,\Gamma\!\left(\frac{p-1}{2}\right)}\,(1-x^2)^{(p-3)/2}, \qquad x \in [-1,1].
\end{equation}
\end{lemma}

\begin{proof}
By rotational invariance of the Haar measure on \(\mathbb S^{p-1}\), for every
fixed unit vector \(u\in\mathbb R^p\),
\[
    u^\top \mathsf V\stackrel{d}{=}e_1^\top\mathsf V.
\]
It therefore suffices to determine the marginal density of the first
coordinate \(V_1\) of \(\mathsf V\).

Let \(g_1,\ldots,g_p\stackrel{\mathrm{iid}}{\sim}\dnorm_p(0,1)\), write
\(g=(g_1,\ldots,g_p)^\top\), and set
\[
    \mathsf V=\frac{g}{\|g\|_2}.
\]
By rotational invariance of the standard Gaussian distribution,
\(\mathsf V\) is uniformly distributed on \(\mathbb S^{p-1}\). Let
\[
    R^2:=\sum_{j=2}^p g_j^2\sim\chi^2_{p-1},
\]
which is independent of \(g_1\). Then
\[
    V_1=\frac{g_1}{\sqrt{g_1^2+R^2}}.
\]

Consider the change of variables
\[
    g_1=x\sqrt{r},\qquad
    R^2=(1-x^2)r,
\]
where \(x\in(-1,1)\) and \(r>0\). Since
\[
    \left|
    \frac{\partial(g_1,R^2)}{\partial(x,r)}
    \right|
    =
    \left|
    \begin{matrix}
        \sqrt r & \dfrac{x}{2\sqrt r}\\[1ex]
        -2xr & 1-x^2
    \end{matrix}
    \right|
    =\sqrt r,
\]
and the joint density of \((g_1,R^2)\) is
\[
    \frac{e^{-g_1^2/2}}{\sqrt{2\pi}}\,
    \frac{(R^2)^{(p-3)/2}e^{-R^2/2}}
    {2^{(p-1)/2}\Gamma\!\left(\frac{p-1}{2}\right)},
\]
the marginal density of \(V_1\) is
\begin{align*}
    f_{V_1}(x)
    &=
    \frac{(1-x^2)^{(p-3)/2}}
    {\sqrt{2\pi}\,2^{(p-1)/2}
    \Gamma\!\left(\frac{p-1}{2}\right)}
    \int_0^\infty r^{p/2-1}e^{-r/2}\,dr\\
    &=
    \frac{(1-x^2)^{(p-3)/2}}
    {\sqrt{2\pi}\,2^{(p-1)/2}
    \Gamma\!\left(\frac{p-1}{2}\right)}
    \,2^{p/2}\Gamma\!\left(\frac p2\right)\\
    &=
    \frac{\Gamma\!\left(\frac p2\right)}
    {\sqrt{\pi}\,\Gamma\!\left(\frac{p-1}{2}\right)}
    (1-x^2)^{(p-3)/2},
    \qquad x\in(-1,1).
\end{align*}
The values at \(x=\pm1\) are immaterial for the density. This proves
\eqref{eq:haar_density}.
\end{proof}

\section{Consistently estimating the target variance for adaptive estimator selection}
\label{sec:consistency_sigmaT_hat}
In this section, we rigorously justify the use of estimator constructed in \eqref{eq:consistent_plug_in} as a consistent plug-in for the oracle target noise variance $\sigmaT^2$ in \eqref{eq:validation_statistic} and \eqref{eq:validation_threshold}. One can prove similar properties of the estimator based on \eqref{eq:estimate_sigma_hat_per_mult}.

Recall the definition of the target data matrix $X_\calT$ following \eqref{eq:define_b_cal_t}. Let $s_1(X_\calT)\ge s_2(X_\calT)\ge\cdots\ge s_{\nT \wedge d}(X_\calT)$ denote the singular
values of $X_\calT$ in order, then by the Eckart-Young theorem \citep{EckartYoung1936},
\begin{align}
\label{eq:eckhart_young}
\wh\sigma^2_T=\frac{1}{\nT d}\sum_{j=2}^{\nT \wedge d}s_j(X_\calT)^2 .
\end{align}
Consider the following theorem.
\begin{theorem}
\label{thm:sigmaT_hat_consistency}
Assume $\min\{\nT,d\}\ge 4$. Under the target model
\begin{align}
    X_\calT = \ztar \thetaT^\top + E_{\calT},
\end{align}
where the rows of $E_{\calT}$ are iid
$\dnorm_d(0_d,\sigmaT^2\,\Id_d)$, the estimator $\wh\sigma^2_T$ defined in
\eqref{eq:consistent_plug_in} satisfies the following. Then,
with probability at least  $1-3/\nT$,
\begin{align}
\left|\frac{\wh\sigma^2_T}{\sigmaT^2}-1\right|=O\left(\sqrt{\frac{\log \nT}{\nT d}}+\frac{1}{\nT}+\frac{1}{d}\right).
\label{eq:sigmaT_hat_rate}
\end{align}
Consequently, if $\nT,d\to\infty$, then for every fixed constant $\mathrm D_0>0$,
\begin{align}
\lim_{\nT,d \to \infty}\Pbb\left[\left|\frac{\wh\sigma^2_T}{\sigmaT^2}-1\right|\le \mathrm D_0\right]=1.
\end{align}
\end{theorem}

\begin{proof}
Observe that one can re-write the target data matrix $X_\calT$ as 
\begin{align}
X_\calT=A+\sigmaT\,G,\qquad A:=\ztar\thetaT^\top,
\end{align}
where $G\in\R^{\nT\times d}$ has iid $\dnorm(0,1)$ entries. 

Since $A$ has rank at
most one, the singular-value interlacing inequality \citep{Franklin1993} gives
\begin{align}
    s_j(X_\calT)\le \sigmaT s_{j-1}(G),
    \qquad j\ge 2,
    \label{eq:sigmaT_rank_one_interlace_upper}
\end{align}
and
\begin{align}
    s_j(X_\calT)\ge \sigmaT s_{j+1}(G),
    \qquad j\ge 1.
    \label{eq:sigmaT_rank_one_interlace_lower}
\end{align}
Therefore, using \eqref{eq:eckhart_young}, we get
\begin{align}
    \wh\sigma^2_T
    =
    \frac{1}{\nT d}\sum_{j=2}^{\nT \wedge d}s_j(X_\calT)^2
    \le
    \frac{\sigmaT^2}{\nT d}\sum_{j=1}^{\nT \wedge d}s_j(G)^2
    =
    \frac{\sigmaT^2}{\nT d}\|G\|_F^2.
    \label{eq:sigmaT_hat_upper}
\end{align}
Similarly,
\begin{align}
    \wh\sigma^2_T
    =
    \frac{1}{\nT d}\sum_{j\ge 2}s_j(X_\calT)^2
    \ge
    \frac{\sigmaT^2}{\nT d}\sum_{j\ge 3}s_j(G)^2
    =
    \frac{\sigmaT^2}{\nT d}\|G\|_F^2
    -
    \frac{\sigmaT^2}{\nT d}\{s_1(G)^2+s_2(G)^2\}.
\end{align}
Since $s_2(G)\le s_1(G)$, this yields
\begin{align}
    \frac{\sigmaT^2}{\nT d}\|G\|_F^2
    -
    \frac{2\sigmaT^2}{\nT d}s_1(G)^2
    \le
    \wh\sigma^2_T
    \le
    \frac{\sigmaT^2}{\nT d}\|G\|_F^2.
    \label{eq:sigmaT_hat_sandwich}
\end{align}

Next, observe that $\|G\|_F^2\sim \chi^2_{\nT d}.$
By the Lemma~1 of \cite{laurent2000adaptive}, with
probability at least $1-2e^{-t}$,
\begin{align}
    1-2\sqrt{\frac{t}{\nT d}}
    \le
    \frac{\|G\|_F^2}{\nT d}
    \le
    1+2\sqrt{\frac{t}{\nT d}}
    +
    \frac{2t}{\nT d}.
    \label{eq:frob_conc_sigmaT}
\end{align}
Using Corollary 5.35 \cite{vershynin2012introduction}, we get
\begin{align}
    s_1(G)
    \le
    \sqrt{\nT}+\sqrt d+\sqrt{2t}
    \label{eq:op_conc_sigmaT}
\end{align}
with probability at least $1-e^{-t}$. Combining
\eqref{eq:sigmaT_hat_sandwich}, \eqref{eq:frob_conc_sigmaT}, and
\eqref{eq:op_conc_sigmaT}, and applying a union bound, gives probability
at least $1-3e^{-t}$. On this event,
\begin{align}
    \wh\sigma^2_T-\sigmaT^2
    \le
    \sigmaT^2\left\{
        2\sqrt{\frac{t}{\nT d}}
        +
        \frac{2t}{\nT d}
    \right\},
\end{align}
and
\begin{align}
    \wh\sigma^2_T-\sigmaT^2
    \ge
    -\sigmaT^2\left\{
        2\sqrt{\frac{t}{\nT d}}
        +
        \frac{2}{\nT d}
        \left(\sqrt{\nT}+\sqrt d+\sqrt{2t}\right)^2
    \right\}.
\end{align}
Taking $t=\log \nT$ in the foregoing inequality, we obtain
\begin{align}
\left|\frac{\wh\sigma^2_T}{\sigmaT^2}-1\right|
\le 2\sqrt{\frac{\log \nT}{\nT d}}+\frac{2\log \nT}{\nT d}+\frac{2}{\nT d}\left(\sqrt{\nT}+\sqrt d+\sqrt{2\log \nT}\right)^2
\end{align}
with probability at least $1-3/\nT$. 
Therefore,
\begin{align}
\left|\frac{\wh\sigma^2_T}{\sigmaT^2}-1\right|=
O\left(\sqrt{\frac{\log \nT}{\nT d}}+\frac{1}{\nT}+\frac{1}{d}\right)
\end{align}
with probability at least $1-3/\nT$. This automatically implies that for every
fixed $\mathrm D_0>0$,
\begin{align}
\Pbb\left[\left|\frac{\wh\sigma^2_T}{\sigmaT^2}-1\right|\le \mathrm D_0\right]\to 1,
\end{align}
whence the theorem follows.
\end{proof}

\section{Proofs of results in Section~\ref{sec:oracle_two_cluster}}
Let us consider the following lemma that will be used in the subsequent proofs.
\begin{lemma}
\label{lem:projected_diverge_source}
Suppose $d \ge 4$ and \eqref{eq:lower_bound_3_clust_delt} holds. Then the estimators of $\wh{\theta}_{\src}$ constructed in \eqref{eq:hat_q_i_hd} and \eqref{eq:hat_q_i_ld} satisfy
\begin{align}
    \frac{|\langle\wh \theta_{\src},\thetaT\rangle|}{\|\wh \theta_\src\|_2\cdot\sigmaT} \convp \infty, \quad \mbox{as $\nT,\nS \to \infty$}.
\end{align}
\end{lemma}

\subsection{Proof of Theorem~\ref{thm:correct_choose}}
First, suppose \eqref{eq:lower_bound_2_clust_delt} holds. Then by Theorem 7 of \cite{Ndaoud2022}, we have
\begin{align}
    \calL\left(\hatztar,\ztar\right) \convp 0, \quad \mbox{as $\nT \to \infty$.}
\end{align}
Since $\calL\left(\hatztar,\ztar\right) \in [0,1]$, therefore, for any $\eta>0$
\begin{align}
\label{eq:bound_expectation_expect}
    \mathbb E\left[\calL\left(\hatztar,\ztar\right)\right] \le \eta + \mathbb P\left(\calL\left(\hatztar,\ztar\right)>\eta\right).
\end{align}
Then taking $\nT \to \infty$ and $\eta \to 0$, we get $\mathbb E\left[\calL\left(\hatztar,\ztar\right)\right] \to 0,$ as $\nT \to \infty$.

Now suppose \eqref{eq:lower_bound_2_clust_delt} does not hold but \eqref{eq:lower_bound_3_clust_delt} holds and we use the source side estimator.
Let $\mathcal F_\src$ denote the sigma-field generated by $\calS$. Choose $\wh s=\operatorname{sgn}(\langle\wh \theta_\src,\theta_\trg\rangle).$ By Lemma~\ref{lem:projected_diverge_source}, clearly,
\begin{align}
\label{eq:no_zero_source}
\limsup_{\nT 
\to \infty}\Pbb\left[\|\wh \theta_\src\|_2=0\right] = 0.
\end{align}
Conditional on $\mathcal F_\src$ and $\|\wh \theta_\src\|_2\neq 0$, for every $i\in \calT$,
\begin{align}
\label{eq:normal_inner_product}
\frac{\langle\wh\theta_\src, \Xtar_j\rangle}{\|\wh\theta_\src\|_2}=\ztar_j \frac{\langle\wh \theta_\src,\theta_\trg\rangle}{{\|\wh\theta_\src\|_2}}+\frac{\langle\wh\theta_\src,\varepsilon_j^{(\trg)}\rangle}{{\|\wh\theta_\src\|_2}},\quad \mbox{where}\; \frac{\langle\wh\theta_\src,\varepsilon_j^{(\trg)}\rangle}{\|\wh\theta_\src\|_2}\sim \dnorm(0,\sigmaT^2).
\end{align}
For the ease of notation, define $\wh w_\src:=\wh\theta_\src/\|\wh\theta_\src\|_2$. Then, using \eqref{eq:normal_inner_product}, we get
\begin{align}
\mathbb P\left[\operatorname{sgn}(\langle\wh w_\src, \Xtar_j\rangle)\ne \wh s \ztar_j\mid \mathcal F_\src\right]
&=\mathbb P\left[\wh s \ztar_j\langle\wh w_\src, \Xtar_j\rangle<0\mid \mathcal F_\src\right] \notag \\
&=\mathbb P\left[|\langle\wh w_\src, \thetaT\rangle|+\wh s \ztar_j\langle\wh w_\src, \varepsilon_j^{(\trg)}\rangle<0\mid \mathcal F_\src\right] \notag \\
&=\Phi\left(-\frac{|\langle\wh w_\src, \thetaT\rangle|}{\sigmaT}\right)=\wh p_\src.
\end{align}
By Lemma~\ref{lem:projected_diverge_source},
\[
\frac{|\langle\wh w_\src,\thetaT\rangle|}{\sigmaT}\convp\infty.
\]
Hence, since $\Phi(-x)\to0$ as $x\to\infty$, $\wh p_\src \convp 0$.

Now define
\[
E_\src:=\frac{1}{\nT}\sum_{j\in \calT}\mathbbm{1}\bigl\{\operatorname{sgn}(\langle\wh w_\src, \Xtar_j\rangle)\ne \wh s \ztar_j\bigr\}.
\]
Conditional on $\mathcal F_\src$, the summands in $E_\src$ are independent Bernoulli random variables with common success probability $\wh p_\src$. Therefore, by Hoeffding's inequality, for every $t>0$,
\[
\mathbb P\left(E_\src-\wh p_\src>t\mid \mathcal F_\src\right)\le \exp(-2\nS t^2).
\]
Taking $t=\sqrt{(\log \nS)/\nS}$ gives
\[
\mathbb P\left[E_\src>\wh p_\src+\sqrt{\frac{\log \nS}{\nS}}~\Big|~\mathcal F_\src\right]\le \nS^{-2}.
\]
Therefore,
\[
E_\src\le \wh p_\src+\sqrt{\frac{\log \nS}{\nS}}, \quad \mbox{with probability at least $1-\nS^{-2}$.}
\]
Combining this with $\wh p_\src\convp 0$ yields $E_\src\convp0,$ conditioned on $\mathcal F_\src$ and $\|\wh \theta_\src\|_2 \neq 0$. Since the random variable in context is bounded, the above inequality and \eqref{eq:no_zero_source} also implies, 
\begin{align}
    \calL(\hatztar,\ztar):=\min_{s\in\{-1,1\}}\frac{1}{\nT}\sum_{j\in \calT}\mathbbm{1}\bigl\{\operatorname{sgn}(\langle\wh w_\src, \Xtar_j\rangle)\ne \wh s \ztar_j\bigr\} \convp 0, \quad \mbox{as $\nT,\nS \to \infty$.}
\end{align}
Proceeding as in \eqref{eq:bound_expectation_expect}, it follows that we get $\mathbb E\left[\calL\left(\hatztar,\ztar\right)\right] \to 0,$ as $\nT,\nS \to \infty$.

\subsection{Proof of Lemma~\ref{lem:projected_diverge_source} when $d \gg \nS$}
\label{sec:proof_s_6_d_gg_n}
When $d \gg \nS$, the estimated direction is constructed using \eqref{eq:hat_q_i_hd}. First, consider the set
\[
\mathcal E_\src:=\{\exists~s\in \{\pm 1\}~\text{such that}~\wh z^{(\src)}=s\zsrc\}.
\]
It can be showed that when $d \gg \nS$ and \eqref{eq:lower_bound_3_clust_delt} holds, using Theorem~8 of \citet{Ndaoud2022}, we can conclude that $\mathbb P(\mathcal E_\src) \to 1,$ as $\nS \to \infty$. 

On $\mathcal E_\src$, we can replace the estimated labels $\wh z^{(\src)}$ by the true labels $\zsrc$. Next, we focus on
\begin{align}
    \lim_{\nS \rightarrow \infty}\mathbb P\left[\|\wh \theta_\src\|_2 \le C_\src\left(\|\theta_\src\|_2+\sigmaS\sqrt{\frac{d}{\nS}}\right)\right].
\end{align}
Observe that on $\mathcal E_\src$, upon replacing the estimated labels $\wh z^{(\src)}$ by $s\,\zsrc$, we get that $\wh \theta_\src \overset{d}{=} s\,q_\src$, where
\[
q_\src \sim \dnorm_d\left(\theta_\src,\frac{\sigmaS^2}{\nS}\Id_d\right).
\]
Using Lemma~1 of \cite{laurent2000adaptive}, we have a constant $C_\src>0$
\begin{align}
\label{eq:laurent_massart}
    \left\|q_\src-\theta_\src\right\|_2  \le C_\src\cdot\sigmaS\cdot\sqrt{\frac{d}{\nS}}, \quad \mbox{with probability greater than $1-e^{-d}$.}
\end{align}
Since $\mathbb P(\mathcal E_\src) \to 1,$ by \eqref{eq:laurent_massart}, this implies using the triangle inequality
\begin{align}
    \lim_{\nS \rightarrow \infty}\mathbb P\left[\|\wh \theta_\src\|_2 \ge C_\src\left(\|\theta_\src\|_2+\sigmaS\sqrt{\frac{d}{\nS}}\right)\right]=0,
\end{align}
by adjusting $C_\src$ appropriately.

Let us also define $\mathcal R_\src:=\{\|\wh \theta_\src\|_2 \le C_\src\left(\|\theta_\src\|_2+\sigmaS\sqrt{\frac{d}{\nS}}\right)\}$. From the foregoing discussion, we have
$\mathbb P(\mathcal R^c_\src \cup \mathcal E^c_\src) \to 0$.
Henceforth, we shall operate on $\mathcal R_\src \cap \mathcal E_\src$, where we have
\[
\wh \theta_\src \overset{d}{=} s\,\thetaS+\sigmaS \cdot \mathrm{G}, \quad \mbox{where $\mathrm{G} \sim \dnorm_d(0,\frac{1}{\nS}\Id_d)$.}
\]

Define 
\begin{align}
\label{eq:tilde_mu}
\wt \mu_\src:=\frac{\langle\thetaT,\thetaS\rangle}{\|\thetaT\|_2 \cdot \|\thetaS\|_2}.
\end{align}
By \eqref{eq:alignment}, the above display implies that 
\[
\langle \wh \theta_\src, \thetaT\rangle \overset{d}{=} s\cdot\wt \mu_\src\cdot \|\thetaT\|_2\cdot\|\thetaS\|_2+\sigmaS \cdot \|\thetaT\|_2\cdot \mathrm{Z}, \quad \mbox{where $\mathrm{Z} \sim \dnorm(0,\nS^{-1})$.}
\]
On $\mathcal R_\src$, for any constant $\mathrm M_\src>0$
\begin{align}
    \frac{|\langle \wh \theta_\src, \thetaT\rangle|}{\sigmaT \cdot \|\wh \theta_\src\|_2} \le \mathrm{M}_{\src} \quad \mbox{implies} \quad |\langle \wh \theta_\src, \thetaT\rangle| \le \sigmaT \cdot \mathrm{M}_{\src} \cdot C_\src\left(\|\theta_\src\|_2+\sigmaS\sqrt{\frac{d}{\nS}}\right).
\end{align}
To show that $\langle \wh \theta_\src, \thetaT\rangle/(\sigmaT \cdot \|\wh \theta_\src\|_2) \convp \infty$, we shall show that
\begin{align}
    \limsup_{\nT,\nS \to \infty}\mathbb P\left[\frac{|\langle \wh \theta_\src, \thetaT\rangle|}{\sigmaT \cdot \|\wh \theta_\src\|_2} \le \mathrm{M}_{\src}\right] = 0.
\end{align}
We shall use Lemma~\ref{lem:gaussian_anti_conc}. Define 
\[
\mathrm B_\src:=\left(C_\src\left\{\|\theta_\src\|_2+\sigmaS\sqrt{\frac{d}{\nS}}\right\}\right)^{-1}, \quad \mbox{and} \quad \varrho_\src:=\frac{\sigmaS \cdot \|\thetaT\|_2}{\sqrt{\nS}}.
\]
Then, for any $\mathrm M_\src>0$,
\begin{align}
    \frac{|\langle \wh \theta_\src, \thetaT\rangle|}{\sigmaT \cdot \|\wh \theta_\src\|_2} \le \mathrm{M}_{\src} \quad \mbox{implies} \quad \left|s\cdot\wt \mu_\src\cdot \|\thetaT\|_2\cdot\|\thetaS\|_2+\sigmaS \cdot \|\thetaT\|_2\cdot \mathrm{Z}\right| \le \frac{\mathrm M_\src \cdot \sigmaT}{\mathrm B_\src}.
\end{align}
Since \eqref{eq:lower_bound_3_clust_delt} holds, therefore 
\begin{align}
    \frac{\mathrm M_\src \cdot \sigmaT}{|\wt \mu_\src|\cdot \|\thetaT\|_2\cdot\|\thetaS\|_2\mathrm B_\src} \le \frac{\mathrm M_\src \cdot \sigmaT}{\mu \|\thetaT\|_2\cdot\|\thetaS\|_2\mathrm B_\src} \to 0, \quad \mbox{as $\nT,\nS \to \infty$.}
\end{align}
Consequently, for sufficiently large $\nT$ and $\nS$, we have
\[
\frac{\mathrm M_\src \cdot \sigmaT}{\mathrm B_\src} \le \frac{1}{2}\cdot|\wt \mu_\src|\cdot \|\thetaT\|_2\cdot\|\thetaS\|_2.
\]
Now, we divide into two cases. When $\varrho_\src \le \mu\; \|\thetaT\|_2\;\|\thetaS\|_2$, then using Lemma~\ref{lem:gaussian_anti_conc} (2) we have
\begin{align}
  \limsup_{\nT,\nS \to \infty}\mathbb P\left[\frac{|\langle \wh \theta_\src, \thetaT\rangle|}{\sigmaT \cdot \|\wh \theta_\src\|_2} \le \mathrm{M}_{\src}\right] & \le  \limsup_{\nT,\nS \to \infty} \frac{\mathrm M_\src \cdot \sigmaT}{\mu\;\|\thetaT\|_2\;\|\thetaS\|_2\;\mathrm B_\src}=0. 
\end{align}
Next, suppose $\varrho_\src \ge \mu\; \|\thetaT\|_2\;\|\thetaS\|_2$. Then 
\begin{align}
\label{eq:jack_0}
  \limsup_{\nT,\nS \to \infty}\mathbb P\left[\frac{|\langle \wh \theta_\src, \thetaT\rangle|}{\sigmaT \cdot \|\wh \theta_\src\|_2} \le \mathrm{M}_{\src}\right] & \le  \limsup_{\nT,\nS \to \infty} \frac{\mathrm M_\src \cdot \sigmaT}{\varrho_\src\;\mathrm B_\src}. 
\end{align}
Now we consider two further subcases: (1) $\|\thetaS\|_2 \ge \sigmaS(d/\nS)^{1/2}$, and (2) $\|\thetaS\|_2 \le \sigmaS(d/\nS)^{1/2}$. 
Now, observe that by definition of the first subcase and \eqref{eq:lower_bound_3_clust_delt}
\begin{align}
    \label{eq:jack_3}
    \frac{\varrho_\src\mathrm B_\src}{\sigmaT} \ge \frac{\varrho_\src}{2C_\src\cdot \sigmaT \cdot \|\thetaS\|_2} \ge \frac{\mu \cdot \|\thetaT\|_2}{2C_\src\cdot\sigmaT} = \frac{\mu \cdot \DeltaT}{2C_\src} \to \infty.
\end{align}
Next, if $\|\thetaS\|_2 \le \sigmaS(d/\nS)^{1/2}$, then again using definition of the subcase and \eqref{eq:lower_bound_3_clust_delt}, we get
\begin{align}
\label{eq:jack_2}
    \frac{\varrho_\src\cdot \mathrm B_\src}{\sigmaT} \gtrsim \frac{\varrho_\src}{\sigmaS \cdot \sigmaT}\cdot \sqrt{\frac{\nS}{d}} \ge \frac{\mu \cdot \|\thetaT\|_2 \cdot \|\thetaS\|_2}{\sigmaS \cdot \sigmaT}\cdot \sqrt{\frac{\nS}{d}} \to \infty.
\end{align}
Plugging in \eqref{eq:jack_3} and \eqref{eq:jack_2} in \eqref{eq:jack_0}, we get 
\begin{align}
\label{eq:jack_0_1}
  \limsup_{\nT,\nS \to \infty}\mathbb P\left[\frac{|\langle \wh \theta_\src, \thetaT\rangle|}{\sigmaT \cdot \|\wh \theta_\src\|_2} \le \mathrm{M}_{\src}\right] =0. 
\end{align}
Since $\mathrm M_\src>0$ is arbitrary, the lemma follows when $d \gg \nS$.

\subsection{Proof of Lemma~\ref{lem:projected_diverge_source} when $4 \le d \lesssim \nS$}
\label{sec:proof_s_6_d_ll_n}
When $d \lesssim \nS$, then the estimator $\wh{\theta}_\src$ is constructed using \eqref{eq:hat_q_i_ld}. First, using Lemma~4.4 of \citet{LofflerZhangZhou2021}, we can decompose $\wh v_\calS$ as follows
\begin{align}
    \label{eq:haar_decomp_ld}
    \wh v_\calS = \alpha_\src \cdot \frac{\thetaS}{\|\thetaS\|_2} + \sqrt{1-\alpha^2_\src} \cdot \wh h_\src,
\end{align}
where $\wh h_\src$ is Haar uniformly distributed on the unit sphere within $\sp\{\thetaS\}^\perp$ and is independent of $\alpha_\src$. Similarly, consider the decomposition
\begin{align}
\label{eq:mu_wise_decomp}
    \thetaT:= \wt\mu_\src\cdot\|\thetaT\|_2\cdot \frac{\thetaS}{\|\thetaS\|_2}+\sqrt{1-\wt\mu^2_\src}\cdot \|\thetaT\|_2 \cdot \mathrm H_\trg,
\end{align}
where $\mathrm H_\trg \in \sp\{\thetaS\}^\perp$ and $\|\mathrm H_\trg\|_2=1$ and $\wt \mu_\src$ is defined in \eqref{eq:tilde_mu}.

Next, we prove that $\alpha^2_\src \ge 1/2$ with probability tending to 1, as $\nS \to \infty$. Recall that, in this regime, the source data matrix can be written as
\[
X_{\calS}=\zsrc\thetaS^\top+E_{\calS}\in \mathbb R^{\nS\times d},
\]
where the entries of $E_{\calS}$ are independent
$\dnorm(0,\sigma_\src^2)$. Since multiplication of the data matrix by a scalar does not change its singular vectors, we apply Theorem~3 of \citet{CaiZhang2018} to
\[
\frac{X_{\calS}}{\sigmaS}=\frac{\zsrc\thetaS^\top}{\sigmaS}
+\frac{E_{\calS}}{\sigmaS}.
\]
The signal matrix $\frac{\zsrc\thetaS^\top}{\sigmaS}$
has rank one. Its leading right singular vector is $\thetaS/\|\thetaS\|_2$, and its only nonzero singular value is
\[
\tau_\src=\frac{\sqrt{\nS}\|\thetaS\|_2}{\sigmaS}.
\]
Therefore, applying Theorem~3 of \citet{CaiZhang2018} with
$p_1=\nS$, $p_2=d$, and $r=1$, we obtain
\begin{align}
\label{eq:bound_on_sin_theta}
\mathbb E\left[\sin^2\Theta\left(\wh v_{\calS},\frac{\thetaS}{\|\thetaS\|_2}\right)\right] &\lesssim  \frac{d\left(\tau_\src^2+\nS\right)}{\tau_\src^4}\wedge 1  \lesssim \frac{\sigmaS^4 d\left(\frac{\nS \cdot \|\thetaS\|^2_2}{\sigmaS^2}+\nS\right)}{\nS^2 \|\thetaS\|^4_2}\wedge 1\\
& \lesssim \left(\frac{\sigmaS^2\,d}{\nS\,\|\thetaS\|^2_2}+\frac{\sigmaS^4\,d}{\nS\,\|\thetaS\|^4_2}\right) \wedge 1\\
& \lesssim \left(\frac{d}{\nS \cdot \DeltaS^2}+\frac{d}{\nS \cdot \DeltaS^4}\right) \wedge 1
\end{align}
If $d \le \nS$, then we have
\[
\DeltaS \gg \left(\frac{d\,(\log \nS)^2}{\nS}\right)^{1/4} \ge \sqrt{\frac{d}{\nS}} \cdot(\log \nS)^{1/2}.
\]
Consequently, using \eqref{eq:bound_on_sin_theta}, we can conclude that
\begin{align}
\label{eq:convergenc_sin_theta}
\mathbb E\left[\sin^2\Theta\left(\wh v_{\calS},\frac{\thetaS}{\|\thetaS\|_2}\right)\right] \to 0.
\end{align}
If $d \asymp \nS$, then using $\DeltaS \gg (\log \nS)^{1/2}$ we can conclude \eqref{eq:convergenc_sin_theta}.
Therefore, if $4 \le d \lesssim \nS$, using \eqref{eq:haar_decomp_ld}, we have that 
\[
\mathbb E\left[1-\alpha^2_\src\right]=\mathbb E\left[\sin^2\Theta\left(\wh v_{\calS},\frac{\thetaS}{\|\thetaS\|_2}\right)\right] \to 0, \quad \mbox{as $\nS \to \infty$.}
\]
By Markov's inequality, this implies $\alpha^2_\src \convp 1$. Therefore, by \eqref{eq:haar_decomp_ld} and \eqref{eq:mu_wise_decomp}, we have
\begin{align}
\label{eq:inner_prod_haar}
    \langle \wh v_\calS,\thetaT\rangle = \alpha_\src \cdot \wt\mu_\src \cdot \|\thetaT\|_2+\beta_\src \cdot \|\thetaT\|_2 \cdot e_\src,
\end{align}
where $e_\src$ is a co-ordinate of a random vector Haar uniformly distributed on the unit sphere within $\sp\{\thetaS\}^\perp$ and 
\begin{align}
    \beta_\src:=\sqrt{1-\alpha^2_\src} \times \sqrt{1-\wt\mu_\src^2}.
\end{align}
First, observe that if $\beta_\src=0$, then the Haar term in \eqref{eq:inner_prod_haar} vanishes. Since $\mathrm M_\src$ is a constant and \eqref{eq:lower_bound_3_clust_delt} holds, this implies
\[
\mathbb P\left[\frac{|\langle \wh v_\calS,\thetaT\rangle|}{\sigmaT}\le \mathrm M_\src\right] \to 0.
\]

Therefore, we assume that $\beta_\src>0$. Next, observe that $e_\src$ is independent of $\alpha_\src$ and $\beta_\src$, and hence for any $\mathrm M_\src>0$
\begin{align}
 \mathbb P\left[|\langle \wh v_\calS,\thetaT\rangle| \le \mathrm M_\src \mid \alpha_\src\right] &= \mathbb P\left[|\alpha_\src \cdot \wt\mu_\src \cdot \|\thetaT\|_2+\beta_\src \cdot \|\thetaT\|_2 \cdot e_\src| \le \mathrm M_\src \mid \alpha_\src\right] = \mathbb P\left[e_\src \in \mathcal I_d \mid \alpha_\src\right],
\end{align}
where 
\[
\mathcal I_d:=\left[\frac{-\alpha_\src \,\wt\mu_\src\,\|\thetaT\|_2-\mathrm M_\src\sigmaT}{\beta_\src\,\|\thetaT\|_2},\frac{-\alpha_\src \,\wt\mu_\src\,\|\thetaT\|_2+\mathrm M_\src\sigmaT}{\beta_\src\,\|\thetaT\|_2}\right] \cap [-1,1].
\]
Using the form of the density in \eqref{eq:haar_density}, we have
\begin{align}
     \mathbb P\left[\frac{|\langle \wh v_\calS,\thetaT\rangle|}{\sigmaT} \le \mathrm M_\src \mid \alpha_\src\right] & \lesssim \frac{2\sigmaT\mathrm M_\src \cdot \sqrt{d}}{\beta_\src\|\thetaT\|_2} \cdot \left(1-\frac{(|\alpha_\src\,\wt\mu_\src\,\|\thetaT\|_2|-\mathrm M_\src\sigmaT)^2_+}{\beta^2_\src\|\thetaT\|^2_2}\right)^{(d-4)/2}_+
\end{align}
Let us consider the set
\[
\mathcal E_{\src,\mathrm{ld}}:=\{|\alpha_\src \mu \|\thetaT\|_2| \ge 2\,\mathrm M_\src \cdot \sigmaT\}.
\]
Clearly by \eqref{eq:lower_bound_3_clust_delt} and $\alpha^2_\src \convp 1$, $\mathbb P(\mathcal E^c_{\src,\mathrm{ld}}) \to 0$, as $\nS \to \infty$. Since, $|\wt \mu_\src| \ge \mu$, therefore, on $\mathcal E_{\src,\mathrm{ld}}$, we also have
\(
|\alpha_\src \wt\mu_\src \|\thetaT\|_2| \ge 2\,\mathrm M_\src \cdot \sigmaT.
\)
Therefore, under $\mathcal E_{\src,\mathrm{ld}}$, we have
\begin{align}
     \mathbb P\left[\frac{|\langle \wh v_\calS,\thetaT\rangle|}{\sigmaT} \le \mathrm M_\src \mid \alpha_\src\right] & \lesssim \frac{\sigmaT\cdot \mathrm M_\src \cdot \sqrt{d}}{\beta_\src\|\thetaT\|_2} \cdot \left(1-\frac{\alpha^2_\src \wt\mu^2_\src \|\thetaT\|^2_2}{4\beta^2_\src\|\thetaT\|^2_2}\right)^{(d-4)/2}_+\\
     & \lesssim \frac{\sigmaT\cdot \mathrm M_\src \cdot \sqrt{d}}{\beta_\src\|\thetaT\|_2} \cdot \left(1-\frac{\alpha^2_\src \mu^2 \|\thetaT\|^2_2}{4\beta^2_\src\|\thetaT\|^2_2}\right)^{(d-4)/2}_+\\
     & \lesssim \frac{\mathrm M_\src \cdot \sqrt{d}}{\beta_\src \DeltaT} \cdot \left(1-\frac{\alpha^2_\src \mu^2}{4\beta^2_\src}\right)^{(d-4)/2}_+.
\end{align}
If $(\alpha^2_\src \mu^2)/(4\beta^2_\src)>1$, then clearly
\begin{align}
    \mathbb P\left[\frac{|\langle \wh v_\calS,\thetaT\rangle|}{\sigmaT} \le \mathrm M_\src \mid \alpha_\src\right]=0.
\end{align}
Otherwise, we approximate
\begin{align}
\label{eq:anti_conc_exp_bound}
     \mathbb P\left[\frac{|\langle \wh v_\calS,\thetaT\rangle|}{\sigmaT} \le \mathrm M_\src \mid \alpha_\src\right] & \lesssim \frac{\mathrm M_\src \cdot \sqrt{d}}{\beta_\src \DeltaT} \cdot \left(1-\frac{\alpha^2_\src \mu^2}{4\beta^2_\src}\right)^{(d-4)/2}_+\\
     & \le \frac{\mathrm M_\src \cdot \sqrt{d}}{\beta_\src \DeltaT} \exp\left(-\frac{(d-4)\,\alpha^2_\src \mu^2}{8\beta^2_\src}\right).
\end{align}
To control the right-hand side, define
\[
T_\src:=\frac{(d-4)\alpha_\src^2\mu^2}{8\beta_\src^2}.
\]
Note that if $d=4$, it trivially follows using $(\alpha^2_\src \mu^2)/(4\beta^2_\src) \le 1$ that
\[
\mathbb P\left[\frac{|\langle \wh v_\calS,\thetaT\rangle|}{\sigmaT} \le \mathrm M_\src \mid \alpha_\src\right] \lesssim \frac{\mathrm M_\src}{|\alpha_\src|\,\mu\,\DeltaT} \to 0,
\]
by \eqref{eq:lower_bound_3_clust_delt}.
For $d \ge 5$, using $(\alpha^2_\src \mu^2)/(4\beta^2_\src) \le 1$, we have
\[
\frac{\sqrt d}{\beta_\src}=\sqrt{\frac{8d}{d-4}}\,\frac{\sqrt{T_\src}}{|\alpha_\src\mu|}.
\]
Hence, for \(d\ge 5\), 
\[
\frac{\sqrt d}{\beta_\src}\exp\left\{-\frac{(d-4)\alpha_\src^2\mu^2}{8\beta_\src^2}\right\}
\lesssim \frac{\sqrt{T_\src}e^{-T_\src}}{|\alpha_\src\mu|}\lesssim\frac{1}{|\alpha_\src\mu|},
\]
where we used $\sup_{t\ge0}\sqrt t e^{-t}<\infty$. Substituting this into
\eqref{eq:anti_conc_exp_bound} gives
\begin{align}
\mathbb P\left[\frac{|\langle \wh v_\calS,\thetaT\rangle|}{\sigmaT}\le \mathrm M_\src\,\middle|\, \alpha_\src\right] &\lesssim\frac{\mathrm M_\src}{|\alpha_\src\mu|\DeltaT}.
\label{eq:anti_conc_final_alpha}
\end{align}
In particular, on the event \(\{|\alpha_\src|\ge 1/2\}\),
\begin{align}
\mathbb P\left[\frac{|\langle \wh v_\calS,\thetaT\rangle|}{\sigmaT}\le \mathrm M_\src\,\middle|\, \alpha_\src\right]
&\lesssim\frac{\mathrm M_\src}{\mu\DeltaT}.
\label{eq:anti_conc_final}
\end{align}
This implies
\begin{align}
  \mathbb P\left[\frac{|\langle \wh v_\calS,\thetaT\rangle|}{\sigmaT}\le \mathrm M_\src,\, |\alpha_\src| \ge 1/2,\,\mathcal E_{\src,\mathrm{ld}}\right]  \lesssim\frac{\mathrm M_\src}{\mu\DeltaT}.
\end{align}
Therefore, since $|\alpha_\src| \convp 1$ as $\nS \to \infty$, we get
\begin{align}
 \limsup_{\nT,\nS \to \infty}\mathbb P\left[\frac{|\langle \wh v_\calS,\thetaT\rangle|}{\sigmaT}\le \mathrm M_\src\right] \lesssim \limsup_{\nT,\nS \to \infty}\mathbb P\left[|\alpha_\src|< 1/2\right]+\limsup_{\nT,\nS \to \infty}\mathbb P\left(\mathcal E^c_{\src,\mathrm{ld}}\right)+\limsup_{\nT,\nS \to \infty}\frac{\mathrm M_\src}{\mu\DeltaT}=0,
\end{align}
whence the lemma follows.

\subsection{Proof of Theorem~\ref{thm:choose_correct_adaptive}}
\label{sec:proof_thm_2}
Define the following terms.
\begin{align}
\label{eq:r_n_g_n}
    \mathrm R(\wt z):=\frac{1}{\nT}\sum_{j =1}^{\nT} \wt z_j\ztar_j, \quad \mbox{and} \quad \mathrm G(\wt z):=\frac{1}{\nT}\sum_{j =1}^{\nT} \wt z_j\varepsilon^{(\trg)}_j.
\end{align}
Observe that $\mathrm G(\wt z) \sim \dnorm_d(0,n^{-1}\,\sigmaT^2\,\Id_d)$. Observe that using the definition of $\wh{\rmT}(\wt z)$ in \eqref{eq:validation_statistic}, we get
\begin{align}
\label{eq:triangle_t_n}
    \wh{\rmT}(\wt z):= (\mathrm R(\wt z))^2\cdot\|\thetaT\|^2_2+2\cdot(\mathrm R(\wt z))\cdot\langle \mathrm G(\wt z),\thetaT \rangle + \left(\|\mathrm G(\wt z)\|^2_2-\frac{d\,\sigmaT^2}{\nT}\right).
\end{align}
Note that $\langle \mathrm G(\wt z),\thetaT \rangle \sim \dnorm(0,n^{-1}\sigmaT^2\,\cdot\|\thetaT\|^2_2)$.

We shall prove the following concentration results for $\mathrm R(\wt z)$ and $\mathrm G(\wt z)$.

\begin{lemma}
    \label{lem:uniform_conc_r_n}
    Consider $\mathrm R(\wt z)$ and $\mathrm G(\wt z)$ defined in \eqref{eq:r_n_g_n}. Then with probability greater than $1-e^{-t}$, 
    \begin{align}
        \sup_{\wt z \in \{-1,+1\}^{\nT}}\left|(\mathrm R(\wt z))\cdot\langle \mathrm G(\wt z),\thetaT \rangle\right| \le \sigmaT\,\|\thetaT\|_2\cdot\left(\sqrt{\frac{2}{\pi}}+\sqrt{2}\,\sqrt{\frac{t}{\nT}}\right) .
    \end{align}
\end{lemma}

\begin{lemma}
    \label{lem:uniform_conc_g_n}
    Consider $\mathrm G(\wt z)$ defined in \eqref{eq:r_n_g_n}. Then, with probability at least $1-6\,n^{-1}$,
\begin{align}
    \sup_{\wt z\in\{-1,+1\}^{\nT}}\left|\|\mathrm G(\wt z)\|_2^2-\frac{\sigmaT^2d}{\nT}\right| \le\mathrm C_0\,\sigmaT^2\left(1+\sqrt{\frac{d}{\nT}}\right),
\end{align}
where $\mathrm C_0>0$ is an absolute constant.
\end{lemma}
With the two foregoing lemmas we now prove Theorem~\ref{thm:choose_correct_adaptive}.

\begin{proof}[Proof of Theorem~\ref{thm:choose_correct_adaptive}]
    Plugging in $t=\log\nT$ in Lemma~\ref{lem:uniform_conc_r_n} and using Lemma~\ref{lem:uniform_conc_g_n}, we get using \eqref{eq:triangle_t_n} that
    \begin{align}
    \label{eq:uniform_t_n_boom}
        \sup_{\wt z\in\{-1,+1\}^{\nT}}\left|\wh{\rmT}(\wt z)-(\mathrm R(\wt z))^2\cdot\|\thetaT\|^2_2\right| \le \left(\sigmaT\|\thetaT\|_2\,\sqrt{\frac{2}{\pi}}+\mathrm C_0\,\sigmaT^2\left(1+\sqrt{\frac{d}{\nT}}\right)\right)(1+o_{\mathbb P}(1)).
    \end{align}

\paragraph{Case 1: Strong target signal}
Suppose \eqref{eq:lower_bound_2_clust_delt} holds. Then
\[
\mathcal L(\hatztar,\ztar) \convp 0.
\]
After choosing the optimal global sign, let $\wh z^\circ$ denote the sign-corrected version of $\hatztar$ in $\mathcal L(\hatztar,\ztar)$. then
\[
\mathrm R(\wh z^\circ)=1-2\,\mathcal L(\hatztar,\ztar) \convp 1.
\]
Define
\[
    \rmT_{\nT}:=\wh{\rmT}(\hatztar_{\trg})=\wh{\rmT}(\wh z^\circ),
\]
and recall the definition of $\tau_n$ from \eqref{eq:validation_threshold}.
By construction of the selector, $\hatztar=\hatztar_{\trg}$ on the event $\{\rmT_{\nT}>\tau_{n}\}$. Hence
\[
\mathbb P(\hatztar\neq \hatztar_{\trg})\le \mathbb P(|\rmT_{\nT}|\le \tau_{n}).
\]
Therefore, it suffices to show that $\mathbb P(\rmT_{\nT}\le \tau_{n})\to0$.

By \eqref{eq:uniform_t_n_boom}, we have
\[
|\rmT_{\nT}|=\mathrm R(\hatztar_{\trg})^2\|\thetaT\|_2^2+ O_{\mathbb P}\left\{
\sigmaT\|\thetaT\|_2+\sigmaT^2\left(1+\sqrt{\frac{d}{\nT}}\right)\right\}.
\]
This implies $|\rmT_{\nT}|=\|\thetaT\|_2^2+o_{\mathbb P}(\|\thetaT\|_2^2),$
if
\[
\sigmaT\|\thetaT\|_2+\sigmaT^2\left(1+\sqrt{\frac{d}{\nT}}\right)=o(\|\thetaT\|_2^2).
\]
To verify this, recall that $\DeltaT:=\|\thetaT\|_2/\sigmaT$. Under \eqref{eq:lower_bound_2_clust_delt}, we have $\DeltaT\gg (d/\nT)^{1/4}+1$. Therefore
\[
\frac{\sigmaT\|\thetaT\|_2}{\|\thetaT\|_2^2}=\frac1{\DeltaT}\to0,\qquad
\frac{\sigmaT^2}{\|\thetaT\|_2^2}=\frac1{\DeltaT^2}\to0,
\]
and
\[
\frac{\sigmaT^2\sqrt{d/\nT}}{\|\thetaT\|_2^2}=\frac{\sqrt{d/\nT}}{\DeltaT^2}\to0.
\]
Therefore,
\[
\frac{|\rmT_{\nT}|}{\|\thetaT\|_2^2}\convp 1.
\]
On the other hand,
\[
\frac{\tau_{n}}{\|\thetaT\|_2^2}=C_0\frac{1+\sqrt{d/\nT}}{\DeltaT^2}\to0.
\]
Fix any $\varepsilon\in(0,1)$. For all sufficiently large $\nT$,
\[
\frac{\tau_{n}}{\|\thetaT\|_2^2}\le 1-\varepsilon.
\]
Thus
\[
\mathbb P(|\rmT_{\nT}|\le \tau_{n})
=\mathbb P\left(\frac{|\rmT_{\nT}|}{\|\thetaT\|_2^2}\le\frac{\tau_{n}}{\|\thetaT\|_2^2}\right) \le \mathbb P\left(\frac{|\rmT_{\nT}|}{\|\thetaT\|_2^2}\le1-\varepsilon\right)\to0.
\]
Consequently,
\[
\mathbb P(\hatztar\neq \hatztar_{\trg})\le \mathbb P(|\rmT_{\nT}|\le \tau_{n})\to0.
\]

\paragraph{Case 2: Weak target signal}
Now, suppose
\begin{align}
\label{eq:low_target_signal_cond}
    \DeltaT:=\frac{\|\thetaT\|_2}{\sigmaT}\le \mathrm{D}_0\left\{\left(\frac{d}{\nT}\right)^{1/4}+1\right\},
\end{align}
for some $\mathrm{D}_0>0$.
By the definition of the selector, $\hatztar=\hatztar_{\src}$ whenever $|T_{\nT}|\le \tau_{n}$. Hence
\[
    \mathbb P(\hatztar\neq \hatztar_{\src})\le \mathbb P(|\rmT_{\nT}|>\tau_{n}).
\]
It therefore suffices to prove that $\mathbb P(\rmT_{\nT}>\tau_{n})\to0$ for a sufficiently large choice of $C_0$. By \eqref{eq:uniform_t_n_boom}, probability tending to one,
\[
    |\rmT_{\nT}|\le \|\thetaT\|_2^2+C\left\{\sigmaT\|\thetaT\|_2+\sigmaT^2\left(1+\sqrt{\frac{d}{\nT}}\right)\right\}
\]

Using \eqref{eq:low_target_signal_cond}
\[
    \|\thetaT\|_2^2\le \mathrm{D}_0^2\sigmaT^2((d/\nT)^{1/4}+1)^2,
\]
we have
\[
    \|\thetaT\|_2^2\le 2\mathrm{D}_0^2\sigmaT^2\left(\sqrt{\frac{d}{\nT}}+1\right), \quad \mbox{and} \quad \sigmaT\|\thetaT\|_2\le \mathrm{D}_0\sigmaT^2\left(\left(\frac{d}{\nT}\right)^{1/4}+1\right).
\]
Using the elementary inequality $s+1\le \frac32(s^2+1)$ for all $s\ge0$, with $s=(d/\nT)^{1/4}$, we get
\[
\sigmaT\|\thetaT\|_2\le \frac{3\mathrm{D}_0}{2}\sigmaT^2\left(\sqrt{\frac{d}{\nT}}+1\right).
\]
Combining these bounds, we obtain
\[
    |\rmT_{\nT}|\le \check C\,\sigmaT^2\left(\sqrt{\frac{d}{\nT}}+1\right)
\]
with probability tending to one, where
\[
    \check C\:=2\mathrm D_0^2+\frac{3C\mathrm D_0}{2}+C.
\]
Choose $C_0>2\,\check{C}$ in \eqref{eq:validation_threshold}. Then
\[
\mathbb P(|\rmT_{\nT}|>\tau_{n})\le \mathbb P\left(|\rmT_{\nT}|>\check{C}\,\sigmaT^2\left(\sqrt{\frac{d}{\nT}}+1\right)\right)\to0.
\]
Consequently,
\[
    \mathbb P(\hatztar\neq \hatztar_{\src})\le \mathbb P(|\rmT_{\nT}|>\tau_{\nT})\to0.
\]
Thus, in the low-signal regime, the adaptive selector chooses the source candidate with probability tending to one. The last condition is immediate from the construction in \eqref{eq:grand_def_adaptive}.

\end{proof}

\subsection{Proof of Lemma~\ref{lem:uniform_conc_r_n}}
For each \(i\in[\nT]\), define
\[
\xi_j:=\ztar_j\langle \thetaT,\varepsilon^{(\trg)}_j\rangle .
\]
Since $\ztar_j\in\{-1,+1\}$ and $\varepsilon^{(\trg)}_j\sim \dnorm_d(0,\sigmaT^2 \Id_d)$, we have
\(
\xi_j\simiid \dnorm(0,\sigmaT^2\|\thetaT\|_2^2).
\)
For any $\wt z\in\{-1,+1\}^{\nT}$, set
\[
    a_j:=\wt z_j \ztar_j\in\{-1,+1\}.
\]
Then, by the definitions in \eqref{eq:r_n_g_n},
\[
\mathrm R(\wt z)=\frac{1}{\nT}\sum_{j =1}^{\nT}a_j,\qquad
\langle \mathrm G(\wt z),\thetaT\rangle=\frac{1}{\nT}\sum_{j =1}^{\nT}a_j\,\xi_j .
\]
Since $|\mathrm R(\wt z)|\le 1$, we obtain
\[
\left|\mathrm R(\wt z)\cdot\langle \mathrm G(\wt z),\thetaT\rangle\right|
\le \left|\frac1{\nT}\sum_{j =1}^{\nT}a_j\,\xi_j\right|
\le\frac1{\nT}\sum_{j =1}^{\nT}|\xi_j|.
\]
Taking the supremum over $\wt z\in\{-1,+1\}^{\nT}$ gives
\[
\sup_{\wt z\in\{-1,+1\}^{\nT}}\left|\mathrm R(\wt z)\cdot\langle \mathrm G(\wt z),\thetaT\rangle\right|\le\frac1{\nT}\sum_{j =1}^{\nT}|\xi_j|.
\]

To control the right hand side of the above display, we define $\varsigma:=\sigmaT\|\thetaT\|_2$. If $\varsigma=0$, the conclusion of the lemma immediately follows. Otherwise, we write $\xi_j=\varsigma\,g_j$ where $g_j \simiid \dnorm(0,1)$.
Define $f:\R^{\nT} \to \R$ as
\[
f(x_1,\ldots,x_{\nT}):=\frac1{\nT}\sum_{j=1}^{\nT}|x_j|.
\]
For \(x,y\in\mathbb R^{\nT}\),
\[
|f(x)-f(y)|\le\frac1{\nT}\sum_{j=1}^{\nT}\bigl||x_j|-|y_j|\bigr|\le
\frac1{\nT}\sum_{j=1}^{\nT}|x_j-y_j|\le\frac1{\sqrt{\nT}}\|x-y\|_2.
\]
Thus $f$ is $1/\sqrt{\nT}$-Lipschitz. By the Gaussian concentration inequality for Lipschitz functions, \cite[Theorem 2.26]{wainwright2019high}, for every $u>0$,
\[
\mathbb P\left(|f(g)-\mathbb Ef(g)|\ge u\right)
\le \exp\left(-\frac{\nT u^2}{2}\right).
\]
Moreover,
\[
\mathbb Ef(g)=\mathbb E|g_1|=\sqrt{\frac{2}{\pi}}.
\]
Choosing $u=\sqrt{2t/\nT}$, we get
\[
\mathbb P\left(\frac1{\nT}\sum_{j =1}^{\nT}|g_j|\ge
        \sqrt{\frac{2}{\pi}}+\sqrt{2}\sqrt{\frac{t}{\nT}}\right) \le e^{-t}.
\]
Multiplying by $\varsigma=\sigmaT\|\thetaT\|_2$, with probability at least \(1-e^{-t}\),
\[
\frac1{\nT}\sum_{j =1}^{\nT}|\xi_j|\le\sigmaT\|\thetaT\|_2\left(\sqrt{\frac{2}{\pi}}+\sqrt{2}\sqrt{\frac{t}{\nT}}\right).
\]
Combining this with the preceding deterministic bound proves the lemma.

\subsection{Proof of Lemma~\ref{lem:uniform_conc_g_n}}
Let $E_\calT\in\mathbb R^{\nT\times d}$ denote the target-noise matrix obtained by concatenating $\varepsilon_j^{(\trg)}$ along the rows. Thus
\[
\mathrm G(\wt z)=\frac{1}{\nT}E_\calT^\top \wt z, \quad \mbox{and} \quad \|\mathrm G(\wt z)\|_2^2=\frac{1}{\nT^2}\wt z^\top E_\calT E_\calT^\top\wt z.
\]
For any square matrix $A$, let $\mathcal H(A):=A-\operatorname{diag}(A)$ denote the matrix obtained by zeroing out the diagonal of $A$. Since $\wt z_j^2=1$ for every $j\in\calT$,
\begin{align}
\wt z^\top E_\calT E_\calT^\top\wt z&=
    \sum_{j=1}^{\nT}\|\varepsilon_j^{(\trg)}\|_2^2+
    \wt z^\top
    \mathcal H\!\left(E_\calT E_\calT^\top\right)
    \wt z,
\end{align}
and therefore
\begin{align}
\left|\|\mathrm G(\wt z)\|_2^2-\frac{\sigmaT^2d}{\nT}\right|&\le\left|\frac{1}{\nT^2}\sum_{j=1}^{\nT}\|\varepsilon_j^{(\trg)}\|_2^2-\frac{\sigmaT^2d}{\nT}\right|+\left|\frac{1}{\nT^2}\wt z^\top\mathcal H\!\left(E_\calT E_\calT^\top\right)\wt z\right|.
    \label{eq:target_quad_decomp}
\end{align}
For the second term, since $\|\wt z\|_2^2=\nT$,
\begin{align}
    \sup_{\wt z\in\{-1,+1\}^{\nT}}\left|\frac{1}{\nT^2}\wt z^\top\mathcal H\!\left(E_\calT E_\calT^\top\right)\wt z\right|
    \le\frac{1}{\nT}\left\|\mathcal H\!\left(E_\calT E_\calT^\top\right)\right\|_{\op}.
    \label{eq:target_off_diag_op}
\end{align}

We first control the first term in \eqref{eq:target_quad_decomp}. Since
\[
    \frac{1}{\sigmaT^2}\sum_{j=1}^{\nT}\|\varepsilon_j^{(\trg)}\|_2^2\sim \chi^2_{\nT d},
\]
Lemma~1 of \cite{laurent2000adaptive} gives that with probability at least $1-2e^{-t}$,
\[
\left|\sum_{j=1}^{\nT}\|\varepsilon_j^{(\trg)}\|_2^2-\sigmaT^2\nT d\right|
\le 2\sigmaT^2\sqrt{\nT\,d\,t}+2\sigmaT^2\,t .
\]
Dividing by $\nT^2$, we obtain
\begin{align}
\left|\frac{1}{\nT^2}\sum_{j=1}^{\nT}\|\varepsilon_j^{(\trg)}\|_2^2-
\frac{\sigmaT^2d}{\nT}\right|\le2\sigmaT^2\left(\sqrt{\frac{dt}{\nT^3}}+\frac{t}{\nT^2}\right).
    \label{eq:target_diag_bound}
\end{align}

Next we control the second term in \eqref{eq:target_quad_decomp}. By Remark~4.7.3 of \citet{vershynin2025high}, with probability at least $1-2e^{-t}$,
Equivalently,
\begin{align}
\left\|E_\calT E_\calT^\top-\sigmaT^2\,d\,\Id_{\nT}\right\|_{\op}\le\sigmaT^2
\left[\sqrt{d\,(\nT+t)}+(\nT+t)\right].
\label{eq:target_gram_op_bound}
\end{align}
Moreover, for each $j\in\calT$,
\[
    \frac{1}{\sigmaT^2}
    \|\varepsilon_j^{(\trg)}\|_2^2
    \sim \chi^2_d .
\]
Applying Lemma~1 of \cite{laurent2000adaptive} with $x=t+\log\nT$ and taking a union bound over $j\in\calT$, we get, with probability at least $1-2e^{-t}$,
\begin{align}
\max_{j=1}^{\nT}\left|\|\varepsilon_j^{(\trg)}\|_2^2-\sigmaT^2\,d\right|
\le2\sigmaT^2\sqrt{d(t+\log\nT)}+2\sigmaT^2(t+\log\nT).
\label{eq:target_max_diag_bound}
\end{align}
Now observe that using \eqref{eq:target_gram_op_bound} and \eqref{eq:target_max_diag_bound} and plugging in $t=\log \nT$, we get
\begin{align}
\label{eq:target_hollowed_boom}
\frac{1}{\nT}\left\|\mathcal H\!\left(E_\calT E_\calT^\top\right)\right\|_{\op}
&\le\frac{1}{\nT}\left\|E_\calT E_\calT^\top-\sigmaT^2dI_{\nT}\right\|_{\op} +
\frac{1}{\nT}\cdot\max_{j=1}^{\nT}\left|\|\varepsilon_j^{(\trg)}\|_2^2-\sigmaT^2d\right|\\
& \le \mathrm C_0\,\sigmaT^2\left(1+\sqrt{\frac{d}{\nT}}+\frac{\log \nT}{\nT}\right)
\end{align}
with probability greater than $1-4\,\nT^{-1}$.

Combining \eqref{eq:target_hollowed_boom} with \eqref{eq:target_diag_bound} after plugging in $t=\log \nT$ in the latter, we get
\begin{align}
    \sup_{\wt z\in\{-1,+1\}^{\nT}}\left|\|\mathrm G(\wt z)\|_2^2-\frac{\sigmaT^2d}{\nT}\right| \le\mathrm C_0\,\sigmaT^2\left(1+\sqrt{\frac{d}{\nT}}+\frac{\log \nT}{\nT}\right)+2\sigmaT^2\left(\sqrt{\frac{d\log \nT}{\nT^3}}+\frac{\log \nT}{\nT^2}\right),
\end{align}
with probability greater than $1-6\,\nT^{-1}$. Observing that $\log \nT \ll \nT$, we get the theorem after slightly adjusting the absolute constant $\mathrm C_0>0$ from \eqref{eq:target_hollowed_boom}.

\section{Proofs for Section~\ref{sec:necessary_cond}}
\label{app:proof-sec3}

\begin{lemma}[Loss-domination form of Assouad's lemma]
\label{lem:sec3_loss_domination_assouad}
Let $\{P_\tau:\tau\in\{\pm1\}^k\}$ be a family of distributions, and let
$P_{+j}$ and $P_{-j}$ be the neighboring Assouad marginals defined by
\eqref{eq:sec3_assouad_marginals}. Suppose that, for every estimator
$\hat z$, there is a decoder $\hat\tau=\hat\tau(\hat z)\in\{\pm1\}^k$
such that, for every $\tau\in\{\pm1\}^k$,
\[
    L(\hat z,\tau)
    \ge
    a\sum_{j=1}^k \mathbbm{1}\{\hat\tau_j\neq\tau_j\}
\]
for some $a>0$. Then
\[
    \inf_{\hat z}\sup_{\tau\in\{\pm1\}^k}
    \Ebb_{P_\tau}\bigl[L(\hat z,\tau)\bigr]
    \ge
    \frac{a}{2}\sum_{j=1}^k
    \bigl(1-\TV(P_{+j},P_{-j})\bigr).
\]
Consequently,
\[
    \inf_{\hat z}\sup_{\tau\in\{\pm1\}^k}
    \Ebb_{P_\tau}\bigl[L(\hat z,\tau)\bigr]
    \ge
    \frac{ak}{2}
    \left(1-\max_{1\le j\le k}\TV(P_{+j},P_{-j})\right).
\]
\end{lemma}

\begin{proof}
Fix an estimator $\hat z$ and average the risk uniformly over
$\tau\in\{\pm1\}^k$. The loss-domination assumption gives
\[
    2^{-k}\sum_{\tau}\Ebb_{P_\tau} L(\hat z,\tau)
    \ge
    a\sum_{j=1}^k
    2^{-k}\sum_{\tau}
    P_\tau(\hat\tau_j\neq\tau_j).
\]
For each coordinate $j$,
\[
    2^{-k}\sum_{\tau}
    P_\tau(\hat\tau_j\neq\tau_j)
    =
    \frac12 P_{+j}(\hat\tau_j\neq +1)
    +
    \frac12 P_{-j}(\hat\tau_j\neq -1).
\]
The last display is the testing error of the test $\hat\tau_j$ for
$P_{+j}$ versus $P_{-j}$ with equal priors. Its minimum over all tests is
$\frac12(1-\TV(P_{+j},P_{-j}))$. Hence
\[
    2^{-k}\sum_{\tau}\Ebb_{P_\tau} L(\hat z,\tau)
    \ge
    \frac{a}{2}\sum_{j=1}^k
    \bigl(1-\TV(P_{+j},P_{-j})\bigr).
\]
Since the supremum over $\tau$ is at least the uniform average, taking the
infimum over $\hat z$ proves the first claim. The second follows by lower
bounding each term by
$1-\max_{1\le j\le k}\TV(P_{+j},P_{-j})$.
\end{proof}

\subsection{Proof of Lemma~\ref{lem:sec3_loss_domination}}

\begin{proof}
Let $s^\star\in\{\pm1\}$ attain the minimum in the definition of
$\calL(\hatztar,\ztar)$. If $\hat{\tau}_j\neq\tau_j$, then the two identities
$\hatztar_{2j-1}=s^\star\ztar_{2j-1}$ and $\hatztar_{2j}=s^\star\ztar_{2j}$
cannot both hold, since otherwise multiplying them would give
$\hat{\tau}_j=\tau_j$. Hence each error on $\tau_j$ forces at least one
misclassified target label in the $j$th pair. Summing over $j$ yields
\[
    \nT \calL(\hatztar,\ztar)
    \ge
    \sum_{j=1}^k \mathbbm{1}\{\hat{\tau}_j\neq\tau_j\},
\]
which is the claimed bound after dividing by $\nT$.
\end{proof}

\subsection{Proof of Lemma~\ref{lem:sec3_exact_tv_form}}

\begin{proof}
Fix a pair index $j_0$ and $a\in\{\pm1\}$. By construction,
\[
    P_{aj_0}
    =
    \Ebb_{\tau_{-j_0},\eta,\xi,\kappa,\zsrc}
    \bigl[
        P_{z_a(\tau_{-j_0},\eta),\,\zsrc,\,v_{\trg},\,v_{\src}}
    \bigr].
\]
View the observations as an element of
$\R^{d\times \nT}\times\R^{d\times \nS}$ and write $(x_r,y_r)$ for the $r$th
row pair of the actual observations
\[
    \bigl([\Xtar_1,\ldots,\Xtar_{\nT}],\,[\Xsrc_1,\ldots,\Xsrc_{\nS}]\bigr).
\]
Under the latent configuration
$\bigl(\tau_{-j_0},\eta,\xi,\kappa,\zsrc\bigr)$, the normalized target and source
columns satisfy
\[
    \frac{\Xtar_i}{\sigmaT}=\DeltaT \ztar_i v_{\trg}+g_i,
    \qquad
    \frac{\Xsrc_j}{\sigmaS}=\DeltaS \zsrc_j v_{\src}+h_j,
\]
where $g_i\overset{\mathrm{iid}}{\sim}\dnorm(0,I_d)$ and
$h_j\overset{\mathrm{iid}}{\sim}\dnorm(0,I_d)$, independently. Using
\eqref{eq:construct_signal_adv}, the actual row vectors therefore admit the
representation
\[
\frac{x_r}{\sigmaT}=\frac{\DeltaT}{\sqrt d}\xi_r z_a(\tau_{-j_0},\eta)+\gamma_r,
    \qquad
\frac{y_r}{\sigmaS}=\frac{\DeltaS}{\sqrt d}\kappa_r\xi_r \zsrc+\delta_r,
\]
where $\gamma_r\sim\dnorm(0,I_{\nT})$ and $\delta_r\sim\dnorm(0,I_{\nS})$.
Moreover, across $r\in[d]$ these row pairs are independent.

By definition, $Q$ is the Gaussian measure on
$\R^{d\times \nT}\times\R^{d\times \nS}$ under which all rows are independent
and
\[
    x_r\sim\dnorm(0,\sigmaT^2 I_{\nT}),
    \qquad
    y_r\sim\dnorm(0,\sigmaS^2 I_{\nS}).
\]
For a Gaussian shift $N(m,\sigma^2 I)$ relative to $N(0,\sigma^2 I)$, the
likelihood ratio is
$\exp\{\sigma^{-2}\ip{m}{u}-\|m\|_2^2/(2\sigma^2)\}$. Applying this to the conditional law of
$(x_r,y_r)$ gives the $r$th row likelihood ratio
\[
    \exp\Biggl\{
        \frac{\DeltaT}{\sigmaT\sqrt d}\xi_r
        \ip{z_a(\tau_{-j_0},\eta)}{x_r}
        +
        \frac{\DeltaS}{\sigmaS\sqrt d}\kappa_r\xi_r
        \ip{\zsrc}{y_r}
        -
        \frac{\DeltaT^2\|z_a(\tau_{-j_0},\eta)\|_2^2
        +\DeltaS^2\|\zsrc\|_2^2}{2d}
    \Biggr\}.
\]
Since every target and source label vector has coordinates in $\{\pm1\}$,
\[
    \|z_a(\tau_{-j_0},\eta)\|_2^2=\nT,
    \qquad
    \|\zsrc\|_2^2=\nS.
\]
Multiplying over the independent rows yields the conditional likelihood ratio
\[
    \frac{dP_{z_a(\tau_{-j_0},\eta),\,\zsrc,\,v_{\trg},\,v_{\src}}}{dQ}
    =
    e^{-(\nT\DeltaT^2+\nS\DeltaS^2)/2}
    \prod_{r=1}^d
    \exp\Bigl\{
        \frac{\DeltaT}{\sigmaT\sqrt d}\xi_r
        \ip{z_a(\tau_{-j_0},\eta)}{x_r}
        +
        \frac{\DeltaS}{\sigmaS\sqrt d}\kappa_r\xi_r
        \ip{\zsrc}{y_r}
    \Bigr\}.
\]
Taking expectation over $(\tau_{-j_0},\eta,\xi,\kappa,\zsrc)$ gives the density
of $P_{aj_0}$ with respect to $Q$.

Finally, \eqref{eq:sec3_tv_exact} is the standard identity
\[
    \TV(P_{+j_0},P_{-j_0})
    =
    \frac12 \Ebb_Q\left|
        \frac{dP_{+j_0}}{dQ}-\frac{dP_{-j_0}}{dQ}
    \right|,
\]
valid because both neighboring laws are absolutely continuous with respect to
the common reference measure $Q$.
\end{proof}

\subsection{Proof of Theorem~\ref{thm:sec3_tv_three_routes}}

We prove the four bounds separately. We begin with the target-direction-revealed
route.

\subsubsection{Proof of Item 1: target-direction-revealed route}

\begin{proof}[Proof of Theorem~\ref{thm:sec3_tv_three_routes}, Item 1]
Fix a pair index $j_0$. Let $\widetilde P_{aj_0}$ denote the joint law of the
data together with the revealed target direction $v_{\trg}$ under the
neighboring prior with $\tau_{j_0}=a$. Since $P_{aj_0}$ is the marginal of
$\widetilde P_{aj_0}$ on the data alone, total variation decreases under
marginalization, so
\[
    \TV(P_{+j_0},P_{-j_0})
    \le
    \TV(\widetilde P_{+j_0},\widetilde P_{-j_0}).
\]

Conditional on $v_{\trg}$, only the $j_0$th target pair depends on $\tau_{j_0}$; all
remaining target observations and all source observations have the same law
under the two hypotheses. Therefore
\[
    \TV(\widetilde P_{+j_0},\widetilde P_{-j_0})
    =
    \TV(P_{+j_0,\mathrm{pair}},P_{-j_0,\mathrm{pair}}),
\]
where $P_{\pm k,\mathrm{pair}}$ denotes the law of the target
pair
$\left(
        \Xtar_{2j_0-1},
        \Xtar_{2j_0}
    \right)$
conditional on $v_{\trg}$ under $\tau_{j_0}=\pm1$.

Define
\[
    U_{2j_0-1}:=\ip{v_{\trg}}{\Xtar_{2j_0-1}},
    \qquad
    U_{2j_0}:=\ip{v_{\trg}}{\Xtar_{2j_0}}.
\]
Because $\|v_{\trg}\|_2=1$, under the model one has
\[
    U_i=\sigmaT\DeltaT \ztar_i+\xi_i,
    \qquad
    \xi_i\overset{\mathrm{iid}}{\sim}\dnorm(0,\sigmaT^2).
\]
Recalling $\ztar_{2j_0-1}=\eta_{j_0}$ and $\ztar_{2j_0}=\tau_{j_0}\eta_{j_0}$, and averaging
over $\eta_{j_0}\sim\mathrm{Rad}(1/2)$, the conditional pair laws are
\begin{align*}
    P_{+j_0,\mathrm{pair}}
    &=
    \tfrac12 \dnorm_2\bigl((\sigmaT\DeltaT,\sigmaT\DeltaT),\sigmaT^2 I_2\bigr)
    +
    \tfrac12 \dnorm_2\bigl((-\sigmaT\DeltaT,-\sigmaT\DeltaT),\sigmaT^2 I_2\bigr),\\
    P_{-j_0,\mathrm{pair}}
    &=
    \tfrac12 \dnorm_2\bigl((\sigmaT\DeltaT,-\sigmaT\DeltaT),\sigmaT^2 I_2\bigr)
    +
    \tfrac12 \dnorm_2\bigl((-\sigmaT\DeltaT,\sigmaT\DeltaT),\sigmaT^2 I_2\bigr).
\end{align*}

Introduce the orthogonal change of variables
\[
    S:=\frac{U_{2j_0-1}+U_{2j_0}}{\sqrt2},
    \qquad
    D:=\frac{U_{2j_0-1}-U_{2j_0}}{\sqrt2}.
\]
Under $P_{+j_0,\mathrm{pair}}$, the signal is carried by $S$ and $D$ is
pure noise; under $P_{-j_0,\mathrm{pair}}$, the roles are reversed. Writing
\[
    a_{\trg}:=\sqrt2\,\sigmaT\DeltaT,
    \qquad
    G_{a_{\trg}}:=\tfrac12 \dnorm(a_{\trg},\sigmaT^2)+\tfrac12 \dnorm(-a_{\trg},\sigmaT^2),
    \qquad
    \phi:=\dnorm(0,\sigmaT^2),
\]
the rotated pair laws become
\[
    P_{+j_0,\mathrm{pair}}=G_{a_{\trg}}\otimes \phi,
    \qquad
    P_{-j_0,\mathrm{pair}}=\phi\otimes G_{a_{\trg}}.
\]

By the product formula for Hellinger affinity,
\[
    H^2(G_{a_{\trg}}\otimes\phi,\phi\otimes G_{a_{\trg}})
    =
    1-(1-H^2(G_{a_{\trg}},\phi))^2
    \le
    2H^2(G_{a_{\trg}},\phi).
\]
The density of $G_{a_{\trg}}$ relative to $\phi$ is
\[
    g_{a_{\trg}}(t)=\exp\!\left(-\frac{a_{\trg}^2}{2\sigmaT^2}\right)\cosh\!\left(\frac{a_{\trg}t}{\sigmaT^2}\right).
\]
Since $H^2\le \chi^2$,
\[
    H^2(G_{a_{\trg}},\phi)
    \le
    \chi^2(G_{a_{\trg}}\|\phi)
    =
    \Ebb_\phi[g_{a_{\trg}}^2]-1.
\]
If $Z\sim\dnorm(0,\sigmaT^2)$, then
\begin{align*}
    \Ebb_\phi[g_{a_{\trg}}^2]
    &=
    e^{-a_{\trg}^2/\sigmaT^2}\Ebb\!\left[\cosh^2\!\left(\frac{a_{\trg}Z}{\sigmaT^2}\right)\right]\\
    &=
    e^{-a_{\trg}^2/\sigmaT^2}\Ebb\!\left[\frac{e^{2a_{\trg}Z/\sigmaT^2}+2+e^{-2a_{\trg}Z/\sigmaT^2}}{4}\right]\\
    &=
    e^{-a_{\trg}^2/\sigmaT^2}\frac{e^{2a_{\trg}^2/\sigmaT^2}+1}{2},
\end{align*}
so
\[
    \chi^2(G_{a_{\trg}}\|\phi)=\cosh\!\left(\frac{a_{\trg}^2}{\sigmaT^2}\right)-1.
\]
Therefore
\[
    H^2(P_{+j_0,\mathrm{pair}},P_{-j_0,\mathrm{pair}})
    \le
    2\left(\cosh\!\left(\frac{a_{\trg}^2}{\sigmaT^2}\right)-1\right).
\]
If $a_{\trg}^2/\sigmaT^2\le 1$, then
$\cosh(x)-1\le Cx^2$ on $[0,1]$.  If
$a_{\trg}^2/\sigmaT^2>1$, use instead the trivial bound
$\TV^2\le 1\le a_{\trg}^4/\sigmaT^4$.  Thus, after changing the universal
constant between the two cases,
\[
    \TV(P_{+j_0},P_{-j_0})^2
    \le
    \TV(\widetilde P_{+j_0},\widetilde P_{-j_0})^2
    \le
    H^2(P_{+j_0,\mathrm{pair}},P_{-j_0,\mathrm{pair}})
    \le
    C \frac{a_{\trg}^4}{\sigmaT^4}
    \le
    C'\DeltaT^4,
\]
which is exactly \eqref{eq:sec3_tv_oracle1}.
\end{proof}

\subsubsection{Proof of Item 2: revealed-source-direction route}

\begin{proof}[Proof of Theorem~\ref{thm:sec3_tv_three_routes}, Item 2]
\textit{Proof sketch.}
The idea is to reveal the ``nuisance'' variable
$W:=(v_{\src},\zsrc,\eta_{j_0},\Lambda_{j_0})$ --- comprising the source direction,
all source labels, and every non-tested target sign --- and bound the TV
distance conditional on $W$. Revealing $W$ has two payoffs: (i)~conditional
on $W$ the only remaining difference between the $\tau_{j_0}=+1$ and $\tau_{j_0}=-1$
experiments is the single unresolved bit $\tau_{j_0}$, and (ii)~the rows
$r=1,\ldots,d$ become independent, enabling a row-by-row Hellinger analysis.
The bound proceeds in three steps.
\begin{enumerate}[label=(\roman*)]
\item \emph{Data processing.} $\TV(P_{+j_0},P_{-j_0})\le
      \TV(\widetilde P_{+j_0}^{(\src)},\widetilde P_{-j_0}^{(\src)})$,
      since $P_{\pm j_0}$ are marginals of the joint laws $\widetilde P_{\pm j_0}^{(\src)}$
      that also carry $W$.
\item \emph{Disintegration.} Because $W$ does not involve $\tau_{j_0}$, its
      marginal law is identical under $\tau_{j_0}=+1$ and $\tau_{j_0}=-1$, so the
      disintegration formula gives
      $\TV(\widetilde P_{+j_0}^{(\src)},\widetilde P_{-j_0}^{(\src)})
       =\Ebb\bigl[\TV(P_{+j_0}^{\,W},P_{-j_0}^{\,W})\bigr]$.
\item \emph{Conditional TV bound.}
      Lemma~\ref{lem:sec3_oracle2_conditional_tv} provides a deterministic
      bound on $\TV(P_{+j_0}^{\,W},P_{-j_0}^{\,W})$ for each realization of $W$,
      via Hellinger subadditivity over the now-independent rows together with
      the per-row posterior analysis of
      Lemmas~\ref{lem:sec3_oracle2_posterior}--\ref{lem:sec3_oracle2_conditional_pair}.
\end{enumerate}

\medskip
\noindent\textit{Details.}
Let
\[
    W:=(v_{\src},\zsrc,\eta_{j_0},\Lambda_{j_0}),
    \qquad
    \Lambda_{j_0}:=\bigl(\eta_\ell,\tau_\ell:\ell\neq j_0\bigr),
\]
and let $\widetilde P_{aj_0}^{(\src)}$ denote the joint law of the data together
with the revealed variable $W$ under the neighboring prior with $\tau_{j_0}=a$.
Since $P_{aj_0}$ is the marginal of $\widetilde P_{aj_0}^{(\src)}$ on the data
alone,
\[
    \TV(P_{+j_0},P_{-j_0})
    \le
    \TV(\widetilde P_{+j_0}^{(\src)},\widetilde P_{-j_0}^{(\src)}).
\]

The revealed variable $W$ has the same marginal law under $\tau_{j_0}=+1$ and
$\tau_{j_0}=-1$, because it does not involve the tested sign $\tau_{j_0}$. Therefore,
by the disintegration formula for total variation with a common revealed
marginal,
\[
    \TV(\widetilde P_{+j_0}^{(\src)},\widetilde P_{-j_0}^{(\src)})
    =
    \Ebb\!\left[
        \TV\!\bigl(
            P_{+j_0}^{\,v_{\src},\zsrc,\eta_{j_0},\Lambda_{j_0}},
            P_{-j_0}^{\,v_{\src},\zsrc,\eta_{j_0},\Lambda_{j_0}}
        \bigr)
    \right].
\]
Lemma~\ref{lem:sec3_oracle2_conditional_tv} gives a deterministic bound
for every realization of $W$, hence
\[
    \TV(\widetilde P_{+j_0}^{(\src)},\widetilde P_{-j_0}^{(\src)})
    \le
    C_{\delta_0}^{1/2}\left(
        \tilde\mu^2\DeltaT^2
        +
        \frac{\nT\DeltaT^4}{d}
    \right)^{1/2}.
\]
Combining with the marginalization bound above gives \eqref{eq:sec3_tv_oracle}:
\[
    \TV(P_{+j_0},P_{-j_0})^2
    \le
    C_{\delta_0}\left(
        \tilde\mu^2\DeltaT^2
        +
        \frac{\nT\DeltaT^4}{d}
    \right).\qedhere
\]
\end{proof}

We now prove the supporting lemmas used in the foregoing argument.

\begin{lemma}[One-sided posterior-bias bound away from perfect alignment]
\label{lem:sec3_posterior_bias_one_sided}
Fix $\delta_0\in(0,1)$ and assume $a\in\R$ satisfies
\[
    |\tanh(a)|\le 1-\delta_0.
\]
Let $\varepsilon\in\{\pm1\}$ satisfy
\[
    \Pbb(\varepsilon=+1)=\frac{1+\tanh(a)}{2},
    \qquad
    \Pbb(\varepsilon=-1)=\frac{1-\tanh(a)}{2},
\]
let $G\sim N(0,1)$ be independent of $\varepsilon$, and for $\lambda\ge 0$
define
\[
    m_\lambda:=\tanh(a+\lambda\varepsilon+\sqrt{\lambda}\,G).
\]
Then there is a constant $C_{\delta_0}>0$ such that for all $\lambda\ge 0$,
\[
    \Ebb[m_\lambda^2]
    \le
    \tanh(a)^2
    +
    C_{\delta_0}\bigl(1-\tanh(a)^2\bigr)^2\lambda.
\]
\end{lemma}

\begin{proof}
Set
\[
    F(\lambda):=\Ebb[m_\lambda^2],
    \qquad
    h(x):=\tanh^2(x),
    \qquad
    A_\lambda:=a+\lambda\varepsilon+\sqrt{\lambda}\,G.
\]
Then $F(\lambda)=\Ebb[h(A_\lambda)]$.

Differentiating under the expectation gives
\[
    F'(\lambda)
    =
    \Ebb\!\left[
        h'(A_\lambda)\!\left(
            \varepsilon+\frac{G}{2\sqrt{\lambda}}
        \right)
    \right]
    =
    \Ebb[h'(A_\lambda)\varepsilon]
    +
    \frac{1}{2\sqrt{\lambda}}\Ebb[h'(A_\lambda)G].
\]
Apply Stein's identity to the second term. Conditionally on $\varepsilon$,
with $f(G):=h'(a+\lambda\varepsilon+\sqrt{\lambda}\,G)$, we have
$f'(G)=\sqrt{\lambda}\,h''(A_\lambda)$, hence
\[
    \Ebb[h'(A_\lambda)G\mid\varepsilon]
    =
    \sqrt{\lambda}\,\Ebb[h''(A_\lambda)\mid\varepsilon].
\]
Therefore
\[
    F'(\lambda)
    =
    \Ebb[h'(A_\lambda)\varepsilon]
    +
    \frac12 \Ebb[h''(A_\lambda)].
\]
Now $m_\lambda=\tanh(A_\lambda)=\Ebb[\varepsilon\mid A_\lambda]$, so
\[
    \Ebb[h'(A_\lambda)\varepsilon]
    =
    \Ebb[h'(A_\lambda)m_\lambda].
\]
Since
\[
    h'(x)=2\tanh(x)\bigl(1-\tanh^2(x)\bigr),
    \qquad
    h''(x)=2\bigl(1-\tanh^2(x)\bigr)\bigl(1-3\tanh^2(x)\bigr),
\]
we obtain
\begin{align*}
    F'(\lambda)
    &=
    \Ebb\!\left[
        2m_\lambda^2(1-m_\lambda^2)
        +
        (1-m_\lambda^2)(1-3m_\lambda^2)
    \right]\\
    &=
    \Ebb\!\left[(1-m_\lambda^2)^2\right].
\end{align*}
In particular,
\[
    F'(\lambda)\le 1-F(\lambda).
\]
Since $F(0)=\tanh(a)^2$, Gronwall's inequality gives
\[
    F(\lambda)
    \le
    1-\bigl(1-\tanh(a)^2\bigr)e^{-\lambda}.
\]
Thus
\[
    F(\lambda)-\tanh(a)^2
    \le
    \bigl(1-\tanh(a)^2\bigr)(1-e^{-\lambda})
    \le
    \bigl(1-\tanh(a)^2\bigr)\lambda.
\]
Because $|\tanh(a)|\le 1-\delta_0$,
\[
    1-\tanh(a)^2
    \ge
    1-(1-\delta_0)^2
    =:
    \kappa_{\delta_0}>0.
\]
Hence
\[
    1-\tanh(a)^2
    \le
    \kappa_{\delta_0}^{-1}\bigl(1-\tanh(a)^2\bigr)^2,
\]
which proves the claim.
\end{proof}

\begin{lemma}[Exact conditional density under revealed source direction]
\label{lem:sec3_oracle2_row_density}
Fix a pair index $j_0$ and condition on
\[
    (v_{\src},\zsrc,\eta_{j_0},\Lambda_{j_0}),
    \qquad
    \Lambda_{j_0}:=\bigl(\eta_\ell,\tau_\ell:\ell\neq j_0\bigr).
\]
View the observations $(\Xtar,\Xsrc)$ as an element of
$ \R^{\nT\times d}\times \R^{\nS\times d},$
and let $Q$ denote the Gaussian measure under which the row pairs
$(x_r,y_r)\in\R^{\nT}\times\R^{\nS}$ are independent and satisfy
\[
    x_r\sim\dnorm(0,\sigmaT^2 I_{\nT}),
    \qquad
    y_r\sim\dnorm(0,\sigmaS^2 I_{\nS}).
\]
For $r\in[d]$, write
\begin{equation}\label{eq:def-Br}
    U_{\pm,r}:=\frac{X^{(T)}_{2j_0-1,r}}{\sigmaT}
    \pm
    \frac{X^{(T)}_{2j_0,r}}{\sigmaT},
    \qquad
    B_r:=\sum_{\ell\neq j_0}\eta_\ell\left(
        \frac{X^{(T)}_{2\ell-1,r}}{\sigmaT}
        +
        \tau_\ell\frac{X^{(T)}_{2\ell,r}}{\sigmaT}
    \right).
\end{equation}
Then the conditional law
\[
    P_{\pm j_0}^{\,v_{\src},\zsrc,\eta_{j_0},\Lambda_{j_0}}
    =
    \mathcal L\!\left(
        (\Xtar,\Xsrc)\mid
        \tau_{j_0}=\pm1,\ v_{\src},\zsrc,\eta_{j_0},\Lambda_{j_0}
    \right)
\]
has density with respect to $Q$ given by
\begin{align}
\label{eq:sec3_oracle2_row_density}
    \frac{dP_{\pm j_0}^{\,v_{\src},\zsrc,\eta_{j_0},\Lambda_{j_0}}}{dQ}
    &(\Xtar,\Xsrc)
    =
    \prod_{r=1}^d
    e^{-\nS\DeltaS^2/(2d)}
    \exp\!\left(
        \frac{\DeltaS}{\sigmaS}\,v_{\src,r}
        \sum_{j=1}^{\nS}\zsrc_j X^{(S)}_{j,r}
    \right)
    \cdot\\
    &e^{-\nT\DeltaT^2/(2d)}
    \left[
        \cosh\!\left(
            \frac{\DeltaT}{\sqrt d}(B_r+\eta_{j_0} U_{\pm,r})
        \right)
        +
        \tilde\mu \sqrt d\, v_{\src,r} \sinh\!\left(
            \frac{\DeltaT}{\sqrt d}(B_r+\eta_{j_0} U_{\pm,r})
        \right)
    \right].
\end{align}
\end{lemma}

\begin{proof}
Write $v_{\src}=s/\sqrt d$ with $s\in\{\pm1\}^d$. Under the revealed variables,
\[
    \Xtar_i
    =
    \sigmaT\DeltaT \ztar_i v_{\trg}+\sigmaT g_i,
    \qquad
    \Xsrc_j
    =
    \sigmaS\DeltaS \zsrc_j v_{\src}+\sigmaS h_j,
\]
where $g_i,h_j\sim N(0,I_d)$ independently. Since $v_{\src}=s/\sqrt d$, its
$r$th coordinate is $v_{\src,r}=s_r/\sqrt d$. Also $v_{\trg,r}=\xi_r/\sqrt d$
with
\[
    \kappa_r\sim\mathrm{Rad}\!\left(\frac{1+\tilde\mu}{2}\right),
\]
and after conditioning on $v_{\src}$ we have $\kappa_r=s_r\xi_r$.
Hence, row by row,
\[
    X^{(T)}_{i,r}
    =
    \sigmaT\frac{\DeltaT}{\sqrt d}\,\xi_r \ztar_i + \sigmaT g_{i,r},
    \qquad
    X^{(S)}_{j,r}
    =
    \sigmaS\frac{\DeltaS}{\sqrt d}\, s_r \zsrc_j+\sigmaS h_{j,r},
\]
so the full conditional density with respect to $Q$ factors over rows.

For the source part, the row-$r$ Gaussian likelihood ratio is
\[
    e^{-\nS\DeltaS^2/(2d)}
    \exp\!\left(
        \frac{\DeltaS}{\sigmaS\sqrt d}
        s_r \sum_{j=1}^{\nS}\zsrc_j X^{(S)}_{j,r}
    \right),
\]
which does not depend on $\tau_{j_0}$.

For the target part, conditional on $\xi_r=e$, the row-$r$ shift vector
is $\sigmaT\frac{\DeltaT}{\sqrt d} e\,\ztar$, hence its Gaussian likelihood ratio is
\[
    e^{-\nT\DeltaT^2/(2d)}
    \exp\!\left(
        \frac{\DeltaT}{\sigmaT\sqrt d}
        a\sum_{i=1}^{\nT}\ztar_i X^{(T)}_{i,r}
    \right).
\]
Given $(\eta_{j_0},\Lambda_{j_0},\tau_{j_0})$,
\[
    \sum_{i=1}^{\nT}\ztar_i X^{(T)}_{i,r}
    =
    \sigmaT\bigl(\eta_{j_0} U_{\tau_{j_0},r}+B_r\bigr).
\]
Finally, conditional on $s_r$, the sign $\xi_r$ has distribution
\[
    \Pbb(\xi_r=+1\mid s_r)=\frac{1+\tilde\mu s_r}{2},
    \qquad
    \Pbb(\xi_r=-1\mid s_r)=\frac{1-\tilde\mu s_r}{2},
\]
so averaging over $\xi_r$ yields
\[
    \Ebb\!\left[
        \exp\!\left\{
            \frac{\DeltaT}{\sqrt d}\xi_r(\eta_{j_0} U_{\tau_{j_0},r}+B_r)
        \right\}
        \Bigm| s_r
    \right]
    =
    \cosh\!\left(
        \frac{\DeltaT}{\sqrt d}(B_r+\eta_{j_0} U_{\tau_{j_0},r})
    \right)
    +
    \tilde\mu s_r \sinh\!\left(
        \frac{\DeltaT}{\sqrt d}(B_r+\eta_{j_0} U_{\tau_{j_0},r})
    \right).
\]
Multiplying the source and target factors over $r=1,\ldots,d$ proves
\eqref{eq:sec3_oracle2_row_density}.
\end{proof}

\begin{lemma}[Posterior mean from the non-tested target observations]
\label{lem:sec3_oracle2_posterior}
Under the setup of Lemma~\ref{lem:sec3_oracle2_row_density}, fix a row $r$ and
write
$
    x_i:=\frac{X^{(T)}_{i,r}}{\sigmaT}.
$
Define
$
    S_r:=x_{2j_0-1}+x_{2j_0},
      D_r:=x_{2j_0-1}-x_{2j_0}.
$
For $B_r$ as defined in~\eqref{eq:def-Br} conditionally on $\xi_r$,
\[
    B_r
    \;\overset{d}{=}\;
    (\nT-2)\frac{\DeltaT}{\sqrt d}\,\xi_r+\sqrt{\nT-2}\,G_r,
\]
for $G_r\sim N(0,1)$ independent of $(\xi_r,\zsrc,\eta_{j_0},\Lambda_{j_0})$, and
\begin{equation}
\label{eq:sec3_independence_B_SD}
    B_r \perp (S_r,D_r) \mid (\xi_r,\eta_{j_0},\Lambda_{j_0}).
\end{equation}
Under $\tau_{j_0}=+1$,
\[
    S_r\mid (\xi_r,\eta_{j_0}) \sim N\!\left(2\eta_{j_0}\frac{\DeltaT}{\sqrt d}\,\xi_r,2\right),
    \qquad
    D_r\mid \xi_r \sim N(0,2),
\]
while under $\tau_{j_0}=-1$ these two roles are reversed. Finally, setting
$s_r:=v_{\src,r}\sqrt{d}$, define
\[
    m_r(b,s_r):=\Ebb[\xi_r\mid s_r,B_r=b].
\]
Then
\begin{equation}
\label{eq:sec3_oracle2_posterior_mean}
    m_r(b,s_r)
    =
    \tanh\!\left(
        \operatorname{artanh}(\tilde\mu s_r)+\frac{\DeltaT}{\sqrt d}\,b
    \right).
\end{equation}
\end{lemma}

\begin{proof}
$B_r$ is a sum of $\nT/2-1$ terms, one per non-tested pair.
For each $\ell\neq j_0$, writing $\delta:=\DeltaT/\sqrt d$,
\[
    x_{2\ell-1}+\tau_\ell x_{2\ell}
    =
    2\eta_\ell\delta\,\xi_r+\zeta_{\ell,r},
\]
with $\zeta_{\ell,r}\sim N(0,2)$ independently of everything else, since
$g_{2\ell-1,r}+\tau_\ell g_{2\ell,r}\sim N(0,2)$ for $\tau_\ell\in\{\pm1\}$.
Multiplying by $\eta_\ell$ (with $\eta_\ell^2=1$) and summing over the
$\nT/2-1$ non-tested pairs gives
\[
    B_r
    =
    (\nT-2)\frac{\DeltaT}{\sqrt d}\,\xi_r+\sqrt{\nT-2}\,G_r,
\]
where $G_r\sim N(0,1)$ is independent of $(\xi_r,\zsrc,\eta_{j_0},\Lambda_{j_0})$
because the noise terms $\{\zeta_{\ell,r}\}$ are independent of all these
variables. The independence \eqref{eq:sec3_independence_B_SD} follows because
$B_r$ and $(S_r,D_r)$ are functions of disjoint sets of independent
(given $\xi_r$) observations.

For the tested pair, if $\tau_{j_0}=+1$ then
\[
    x_{2j_0-1}=\eta_{j_0}\frac{\DeltaT}{\sqrt d}\,\xi_r+g_{2j_0-1,r},
    \qquad
    x_{2j_0}=\eta_{j_0}\frac{\DeltaT}{\sqrt d}\,\xi_r+g_{2j_0,r},
\]
so
\[
    S_r\mid (\xi_r,\eta_{j_0}) \sim N\!\left(2\eta_{j_0}\frac{\DeltaT}{\sqrt d}\,\xi_r,2\right),
    \qquad
    D_r\mid \xi_r \sim N(0,2).
\]
If $\tau_{j_0}=-1$, the sign in the second coordinate flips and the two roles
interchange.

It remains to compute the posterior mean. Since $\xi_r\perp(\zsrc,\eta_{j_0},\Lambda_{j_0})\mid v_{\src}$
in the prior (target labels and source data are independent of the target
direction) and $B_r\mid\xi_r$ does not depend on $(\zsrc,\eta_{j_0},\Lambda_{j_0})$,
the posterior satisfies
$m_r(b,s_r)=\Ebb[\xi_r\mid v_{\src},\zsrc,\eta_{j_0},\Lambda_{j_0},B_r=b]
=\Ebb[\xi_r\mid s_r,B_r=b]$,
where the last step uses that, given $v_{\src}$, only the scalar
$s_r=v_{\src,r}\sqrt d$ is relevant for the distribution of $\xi_r$ and $B_r$.
Conditionally on $\xi_r=e$,
the density of $B_r$ is Gaussian with mean $(\nT-2)\frac{\DeltaT}{\sqrt d} e$ and variance
$\nT-2$, so
\[
    \frac{f_{B_r\mid\xi_r=+1}(b)}
         {f_{B_r\mid\xi_r=-1}(b)}
    =
    \exp\!\left(\frac{2\DeltaT}{\sqrt d} b\right).
\]
Also, conditional on $s_r$,
\[
    \log
    \frac{\Pbb(\xi_r=+1\mid s_r)}
         {\Pbb(\xi_r=-1\mid s_r)}
    =
    2\operatorname{artanh}(\tilde\mu s_r).
\]
Bayes' rule combines the prior log-odds and the log-likelihood ratio:
\[
    \log\frac{\Pbb(\xi_r=+1\mid s_r,B_r=b)}{\Pbb(\xi_r=-1\mid s_r,B_r=b)}
    =
    \underbrace{\log\frac{\Pbb(\xi_r=+1\mid s_r)}{\Pbb(\xi_r=-1\mid s_r)}}_{=\,2\operatorname{artanh}(\tilde\mu s_r)}
    +
    \underbrace{\log\frac{f_{B_r\mid\xi_r=+1}(b)}{f_{B_r\mid\xi_r=-1}(b)}}_{=\,\frac{2\DeltaT}{\sqrt d}\,b}
    =
    2\operatorname{artanh}(\tilde\mu s_r)+\frac{2\DeltaT}{\sqrt d}\,b.
\]
For a $\{\pm1\}$-valued variable $\xi_r$ with posterior log-odds equal to $2A$,
the posterior probabilities are $\Pbb(\xi_r=\pm1\mid\cdot)=\tfrac{1\pm\tanh(A)}{2}$,
so the posterior mean is $\tanh(A)$. With $A=\operatorname{artanh}(\tilde\mu s_r)+\frac{\DeltaT}{\sqrt d}\,b$
this gives exactly \eqref{eq:sec3_oracle2_posterior_mean}.
\end{proof}

\begin{lemma}[Conditional tested-pair comparison given $B_r$]
\label{lem:sec3_oracle2_conditional_pair}
Under the setup of Lemma~\ref{lem:sec3_oracle2_row_density}, let $p_{+,r}$
and $p_{-,r}$ denote the row-$r$ marginals of
$P_{+j_0}^{\,v_{\src},\zsrc,\eta_{j_0},\Lambda_{j_0}}$ and
$P_{-j_0}^{\,v_{\src},\zsrc,\eta_{j_0},\Lambda_{j_0}}$ respectively.
Fix $b\in\R$. Then, without an additional restriction on $\DeltaT^2/d$,
\begin{equation}
\label{eq:sec3_oracle2_one_row_conditional}
    H^2\!\bigl(
        p_{+,r}(\cdot\mid B_r=b),
        p_{-,r}(\cdot\mid B_r=b)
    \bigr)
    \le
    C\left(\frac{\DeltaT^4}{d^2}+m_r(b,s_r)^2\frac{\DeltaT^2}{d}\right).
\end{equation}
\end{lemma}

\begin{proof}
Recall from Lemma~\ref{lem:sec3_oracle2_posterior} that
$S_r:=x_{2j_0-1}+x_{2j_0}$ and $D_r:=x_{2j_0-1}-x_{2j_0}$ where
$x_i:=X^{(T)}_{i,r}/\sigmaT$. Set
\[
    U_r:=\frac{S_r}{\sqrt2},\quad
    V_r:=\frac{D_r}{\sqrt2},\quad
    \gamma:=\sqrt{\frac{2\DeltaT^2}{d}},\quad
    \widetilde m_r(b,s_r):=\eta_{j_0}\,m_r(b,s_r).
\]
For $m\in[-1,1]$, define
$f_{m,\gamma}(t):=e^{-\gamma^2/2}[\cosh(\gamma t)+m\sinh(\gamma t)]$
and let $M_{m,\gamma}$ be the law with density $f_{m,\gamma}$ relative to~$\phi$.

The map $(X^{(T)}_{2j_0-1,r},X^{(T)}_{2j_0,r})\mapsto(U_r,V_r)$ is a bijection
(scale by $1/\sigmaT$, then rotate: $(a,b)\mapsto((a+b)/\sqrt2,(a-b)/\sqrt2)$),
so $p_{\pm,r}(\cdot\mid B_r=b)$ and the conditional law of $(U_r,V_r)$ given
$B_r=b$ carry identical information; we work with the latter. Our goal is to show that this
law factors as $M_{\widetilde m_r(b,s_r),\gamma}\otimes\phi$ (resp.\
$\phi\otimes M_{\widetilde m_r(b,s_r),\gamma}$) under $\tau_{j_0}=+1$
(resp.\ $\tau_{j_0}=-1$). We do this by identifying the marginal density of $U_r$
relative to $\phi$, showing $V_r\sim\phi$ unconditionally, and establishing
that $U_r\perp V_r$ after conditioning on $B_r=b$. The hidden variable
throughout is the row label $\xi_r$, which is not observed.

\medskip
\noindent\textit{Step~1: distributions of $U_r$ and $V_r$ given $\xi_r$.}
Lemma~\ref{lem:sec3_oracle2_posterior} states that under $\tau_{j_0}=+1$,
\[
    S_r\mid(\xi_r=e,\eta_{j_0})\sim N\!\left(2\eta_{j_0}\frac{\DeltaT}{\sqrt d}\,e,\,2\right),
    \qquad
    D_r\mid\xi_r\sim N(0,2).
\]
Dividing by $\sqrt2$ gives
\[
    U_r\mid(\xi_r=e,\eta_{j_0})\sim N(\gamma\eta_{j_0} e,1),
    \qquad
    V_r\mid\xi_r\sim N(0,1)=\phi.
\]
Crucially, $V_r\sim\phi$ for every value of $\xi_r$, so $V_r\sim\phi$
unconditionally. Under $\tau_{j_0}=-1$ the roles of $S_r$ and $D_r$ are reversed:
$D_r$ carries the signal and $S_r\sim N(0,2)$, giving $V_r\sim N(\gamma\eta_{j_0} e,1)$
and $U_r\sim\phi$.

\medskip
\noindent\textit{Step~2: effect of conditioning on $B_r=b$.}
By independence~\eqref{eq:sec3_independence_B_SD}, $(U_r,V_r)$ depends only
on the tested-pair observations while $B_r$ depends only on the non-tested
pairs, so $(U_r,V_r)\perp B_r\mid\xi_r$. Therefore the conditional distributions
of $U_r$ and $V_r$ given $\xi_r=e$ are unchanged when we further condition on
$B_r=b$. In particular, $V_r\sim\phi$ still holds after conditioning on $B_r=b$.
Moreover, since $U_r$ and $V_r$ are uncorrelated Gaussians (sum and difference
of two independent normals), they are independent given $\xi_r$, and this
independence also survives conditioning on $B_r=b$.

\medskip
\noindent\textit{Step~3: density of $U_r$ under $p_{+,r}(\cdot\mid B_r=b)$.}
To identify the marginal law of $U_r$, we average its conditional density
(given $\xi_r$) over the posterior of $\xi_r$. Specifically, the density of
$U_r$ under $p_{+,r}(\cdot\mid B_r=b)$ relative to $\phi$ is
$\Ebb[\,dN(\gamma\eta_{j_0}\xi_r,1)/d\phi\,(u)\mid B_r=b,s_r]$,
where the Radon--Nikodym derivative of $N(\gamma\eta_{j_0} e,1)$ w.r.t.\ $\phi$
at $u$ is $e^{\gamma\eta_{j_0} e\,u-\gamma^2/2}$. Using the posterior mean
$m_r(b,s_r)$ from~\eqref{eq:sec3_oracle2_posterior_mean}:
\begin{align*}
    \frac{dp_{+,r}(\cdot\mid B_r=b)}{d\phi}(u)
    &=
    \Ebb\!\left[e^{\gamma\eta_{j_0}\xi_r u-\gamma^2/2}\Bigm|B_r=b,s_r\right]\\
    &=
    e^{-\gamma^2/2}
    \!\left[
        \frac{1+m_r(b,s_r)}{2}\,e^{\gamma\eta_{j_0} u}
        +
        \frac{1-m_r(b,s_r)}{2}\,e^{-\gamma\eta_{j_0} u}
    \right]\\
    &=
    e^{-\gamma^2/2}\bigl[\cosh(\gamma u)+\eta_{j_0} m_r(b,s_r)\sinh(\gamma u)\bigr]
    \;=\;
    f_{\widetilde m_r(b,s_r),\gamma}(u).
\end{align*}
Hence $U_r\mid(B_r=b)\sim M_{\widetilde m_r(b,s_r),\gamma}$ under $\tau_{j_0}=+1$.

\medskip
\noindent\textit{Step~4: product representation.}
Combining Steps~2 and~3: under $\tau_{j_0}=+1$, conditional on $B_r=b$,
$U_r\sim M_{\widetilde m_r(b,s_r),\gamma}$, $V_r\sim\phi$, and $U_r\perp V_r$.
Under $\tau_{j_0}=-1$ the roles are swapped. Therefore
\[
    p_{+,r}(\cdot\mid B_r=b)=M_{\widetilde m_r(b,s_r),\gamma}\otimes\phi,
    \qquad
    p_{-,r}(\cdot\mid B_r=b)=\phi\otimes M_{\widetilde m_r(b,s_r),\gamma}.
\]

\medskip
\noindent\textit{Step~5: Hellinger bound.}
Write $\widetilde m:=\widetilde m_r(b,s_r)$ for short. The product formula for
Hellinger affinity gives
\[
    H^2(M_{\widetilde m,\gamma}\otimes\phi,\,\phi\otimes M_{\widetilde m,\gamma})
    =
    1-\bigl(1-H^2(M_{\widetilde m,\gamma},\phi)\bigr)^2
    \le
    2H^2(M_{\widetilde m,\gamma},\phi).
\]
Since $H^2\le\chi^2$ and $\Ebb_\phi[f_{\widetilde m,\gamma}]=1$,
\[
    H^2(M_{\widetilde m,\gamma},\phi)
    \le
    \chi^2(M_{\widetilde m,\gamma}\|\phi)
    =
    \Ebb_\phi\!\left[f_{\widetilde m,\gamma}^2\right]-1.
\]
Expanding $f_{\widetilde m,\gamma}^2=e^{-\gamma^2}[\cosh(\gamma G)+\widetilde m\sinh(\gamma G)]^2$
for $G\sim N(0,1)$: the cross-term $\cosh(\gamma G)\sinh(\gamma G)$ has zero
expectation by symmetry ($G\overset{d}{=}-G$), while
\[
    \Ebb[\cosh^2(\gamma G)]=\frac{e^{2\gamma^2}+1}{2},
    \qquad
    \Ebb[\sinh^2(\gamma G)]=\frac{e^{2\gamma^2}-1}{2}.
\]
Hence
\[
    \chi^2(M_{\widetilde m,\gamma}\|\phi)
    =
    \cosh(\gamma^2)-1+\widetilde m^{\,2}\sinh(\gamma^2).
\]
If $\gamma\le 1/2$, the last display is at most
$C(\gamma^4+\widetilde m^{\,2}\gamma^2)$ by the elementary small-argument
bounds for $\cosh$ and $\sinh$.  If $\gamma>1/2$, use instead
$H^2\le 1\le 16\gamma^4$.  Substituting
$\widetilde m^{\,2}=m_r(b,s_r)^2$ and
$\gamma^2=2\DeltaT^2/d$ gives
\eqref{eq:sec3_oracle2_one_row_conditional} in both cases.
\end{proof}

\begin{lemma}[Conditional total-variation bound under the revealed-source experiment]
\label{lem:sec3_oracle2_conditional_tv}
Fix $\delta_0\in(0,1)$ and assume $\tilde\mu\in[0,1-\delta_0]$. Under the setup
of Lemma~\ref{lem:sec3_oracle2_row_density},
\begin{equation}
\label{eq:sec3_oracle2_conditional_tv}
    \TV\!\bigl(
        P_{+j_0}^{\,v_{\src},\zsrc,\eta_{j_0},\Lambda_{j_0}},
        P_{-j_0}^{\,v_{\src},\zsrc,\eta_{j_0},\Lambda_{j_0}}
    \bigr)^2
    \le
    C_{\delta_0}\left(
        \tilde\mu^2\DeltaT^2
        +
        \frac{\nT\DeltaT^4}{d}
    \right).
\end{equation}
\end{lemma}

\begin{proof}
Because $B_r$ depends only on the non-tested target observations, its marginal
law is the same under $\tau_{j_0}=+1$ and $\tau_{j_0}=-1$. Therefore
\[
    H^2(p_{+,r},p_{-,r})
    =
    \Ebb\!\left[
        H^2\!\bigl(
            p_{+,r}(\cdot\mid B_r),
            p_{-,r}(\cdot\mid B_r)
        \bigr)
    \right].
\]
Applying Lemma~\ref{lem:sec3_oracle2_conditional_pair},
\[
    H^2(p_{+,r},p_{-,r})
    \le
    C\left(
        \frac{\DeltaT^4}{d^2}+\frac{\DeltaT^2}{d}\,\Ebb[m_r(B_r,s_r)^2]
    \right).
\]

From Lemma~\ref{lem:sec3_oracle2_posterior},
\[
    B_r
    =
    (\nT-2)\frac{\DeltaT}{\sqrt d}\,\xi_r+\sqrt{\nT-2}\,G_r.
\]
Set
\[
    \lambda:=(\nT-2)\frac{\DeltaT^2}{d}.
\]
Then
\[
    \frac{\DeltaT}{\sqrt d}\, B_r
    =
    \lambda\xi_r+\sqrt{\lambda}\,G_r,
\]
and by \eqref{eq:sec3_oracle2_posterior_mean},
\[
    m_r(B_r,s_r)
    =
    \tanh\!\bigl(
        \operatorname{artanh}(\tilde\mu s_r)
        +
        \lambda\xi_r
        +
        \sqrt{\lambda}\,G_r
    \bigr).
\]
Lemma~\ref{lem:sec3_posterior_bias_one_sided} applies with
$a=\operatorname{artanh}(\tilde\mu s_r)$, because $\tanh(a)=\tilde\mu s_r\in[-1+\delta_0,1-\delta_0]$
(using $s_r\in\{\pm1\}$, so $\tanh(a)^2=\tilde\mu^2$, and the bound is symmetric
in the sign of $a$). Since $(1-\tilde\mu^2)^2\le 1$, the Lemma~\ref{lem:sec3_posterior_bias_one_sided}
bound gives
\[
    \Ebb[m_r(B_r,s_r)^2]
    \le
    \tilde\mu^2 + C_{\delta_0}\lambda
    \le
    \tilde\mu^2 + C_{\delta_0}\frac{\nT\DeltaT^2}{d}.
\]
Substituting into the Lemma~\ref{lem:sec3_oracle2_conditional_pair} bound and
using $\DeltaT^4/d^2\le\nT\DeltaT^4/d^2$ gives
\begin{equation}
\label{eq:sec3_oracle2_row_hellinger}
    H^2(p_{+,r},p_{-,r})
    \le
    C_{\delta_0}\left(
        \tilde\mu^2\frac{\DeltaT^2}{d}
        +
        \frac{\nT\DeltaT^4}{d^2}
    \right).
\end{equation}
Conditional on $(v_{\src},\zsrc,\eta_{j_0},\Lambda_{j_0})$, the rows remain
independent, and by Lemma~\ref{lem:sec3_oracle2_row_density} the source row law
is common to both hypotheses. Therefore the full conditional laws are product
measures over $r=1,\ldots,d$, and squared Hellinger distance is subadditive
under products:
\begin{align*}
    H^2\!\bigl(
        P_{+j_0}^{\,v_{\src},\zsrc,\eta_{j_0},\Lambda_{j_0}},
        P_{-j_0}^{\,v_{\src},\zsrc,\eta_{j_0},\Lambda_{j_0}}
    \bigr)
    &\le
    \sum_{r=1}^d H^2(p_{+,r},p_{-,r})\\
    &\le
    C_{\delta_0}\left(
        \tilde\mu^2\DeltaT^2
        +
        \frac{\nT\DeltaT^4}{d}
    \right).
\end{align*}
Finally, $\TV(P,Q)\le H(P,Q)$, so
\eqref{eq:sec3_oracle2_conditional_tv} follows from
\[
    \TV\!\bigl(
        P_{+j_0}^{\,v_{\src},\zsrc,\eta_{j_0},\Lambda_{j_0}},
        P_{-j_0}^{\,v_{\src},\zsrc,\eta_{j_0},\Lambda_{j_0}}
    \bigr)^2
    \le
    H^2\!\bigl(
        P_{+j_0}^{\,v_{\src},\zsrc,\eta_{j_0},\Lambda_{j_0}},
        P_{-j_0}^{\,v_{\src},\zsrc,\eta_{j_0},\Lambda_{j_0}}
    \bigr),
\]
and the bound on the squared Hellinger distance above.
\end{proof}

\subsubsection{Proof of Item 3: doubly subcritical route}

\begin{proof}[Proof of Theorem~\ref{thm:sec3_tv_three_routes}, Item 3]
We establish \eqref{eq:sec3_tv_doubly} via a replica second-moment computation.
The point of the argument is to isolate exactly where the two neighboring
mixtures differ.  After averaging over all nuisance variables, the only
remaining asymmetry between $P_{+j_0}$ and $P_{-j_0}$ comes from the distinguished
target pair $(2j_0-1,2j_0)$; every other target pair and the entire source sample
contribute only background multiplicative factors.  The replica calculation
makes this separation explicit.

\paragraph{Setup.}
Write $\tilde x_{i,r}:=\Xtar_{i,r}/\sigmaT$ and $\tilde y_{j,r}:=\Xsrc_{j,r}/\sigmaS$ for the coordinatewise standardized observations; under $Q$ these are all i.i.d.\ $N(0,1)$. Let $L_a := dP_{aj_0}/dQ$ for $a\in\{+1,-1\}$. By Cauchy--Schwarz,
\[
    \TV(P_{+j_0},P_{-j_0})^2
    \le
    \frac{1}{4}\Ebb_Q[(L_+-L_-)^2].
\]
Under the hypothesis of Item 3, namely
\[
    \nT\DeltaT^4+\nS\DeltaS^4\le c\,d
\]
for a sufficiently small universal constant $c>0$, it therefore suffices to
show
\[
    \Ebb_Q[(L_+-L_-)^2]
    \le
    C\,\frac{\DeltaT^4}{d}.
\]
We will obtain this from the corrected intermediate estimate
\begin{equation}
\label{eq:item3-l2-target}
    \Ebb_Q[(L_+-L_-)^2]
    \le
    C\DeltaT^4
    \Ebb\!\left[
      \rho_d^2\exp\!\left\{
        C\bigl(\nT\DeltaT^4+\nS\tilde\mu^4\DeltaS^4\bigr)\rho_d^2
      \right\}
    \right].
\end{equation}
and then simplify it using the aggregate subcritical condition. We now
derive \eqref{eq:item3-l2-target} in several transparent steps.

\paragraph{Step 1: conditional likelihood ratio.}
Fix $a\in\{+1,-1\}$ and let
\[
    \mathsf{Lat}_a := (\tau_{-j_0},\xi,\kappa,\zsrc,\eta)
\]
with $\tau_{j_0}=a$ fixed. Write $z_{a,i}:= z_a(\tau_{-j_0},\eta)_i$ for the
$i$-th target label under $\tau_{j_0}=a$. Thus the entire target label vector is
determined by $(\tau_{-j_0},\eta)$ together with the forced value $\tau_{j_0}=a$,
while the source labels are $\zsrc$ and the coupled direction variables are
$(\xi,\kappa)$.

Conditional on $\mathsf{Lat}_a$, the standardized observations satisfy
\[
    \tilde x_{i,r}\mid\mathsf{Lat}_a\sim N\!\left(\tfrac{\DeltaT}{\sqrt{d}}\xi_r z_{a,i},1\right),
    \qquad
    \tilde y_{j,r}\mid\mathsf{Lat}_a\sim N\!\left(\tfrac{\DeltaS}{\sqrt{d}}\kappa_r\xi_r\zsrc_j,1\right).
\]
For a single Gaussian shift, the density of $N(m,1)$ relative to $N(0,1)$ is
\[
    \frac{dN(m,1)}{dN(0,1)}(u)
    =
    \exp\!\left(mu-\frac{m^2}{2}\right).
\]
Applying this coordinatewise and multiplying over all target and source entries
gives the conditional likelihood ratio
\[
    L_a(\tilde x,\tilde y\mid\mathsf{Lat}_a)
    =
    \exp\!\Biggl(
        \frac{\DeltaT}{\sqrt{d}}\sum_{r,i}\xi_r z_{a,i}\tilde x_{i,r}-\frac{\nT\DeltaT^2}{2}
    \Biggr)
    \exp\!\Biggl(
        \frac{\DeltaS}{\sqrt{d}}\sum_{r,j}\kappa_r\xi_r\zsrc_j\tilde y_{j,r}-\frac{\nS\DeltaS^2}{2}
    \Biggr).
\]
The constants $\nT\DeltaT^2/2$ and $\nS\DeltaS^2/2$ are the sums of the
coordinatewise $m^2/2$ terms over all target and source observations.

\paragraph{Step 2: replica inner product.}
To compute $\Ebb_Q[L_aL_b]$, draw two independent latent copies
\[
    \mathsf{Lat}_a=(\tau_{-j_0},\xi,\kappa,\zsrc,\eta),
    \qquad
    \mathsf{Lat}_b^*=(\tau_{-j_0}^*,\xi^*,\kappa^*,\zsrc_*,\eta^*),
\]
where $(\zsrc_*)_1,\ldots,(\zsrc_*)_{\nS}$ are the source labels in the second
copy.  Since
$Q$ is a product of standard Gaussians, the expectation over $Q$ factorizes
over all target coordinates $(i,r)$ and source coordinates $(j,r)$.

For each target coordinate $(i,r)$, the Gaussian MGF gives
\[
    \Ebb_Q\!\left[e^{\frac{\DeltaT}{\sqrt{d}}(\xi_r z_{a,i}+\xi_r^* z_{b,i}^*)\tilde x_{i,r}}\right]
    =
    \exp\!\left(\frac{\DeltaT^2}{2d}(\xi_r z_{a,i}+\xi_r^* z_{b,i}^*)^2\right).
\]
Since $(\xi_r)^2=(z_{a,i})^2=1$, expanding the square yields
\[
    (\xi_r z_{a,i}+\xi_r^* z_{b,i}^*)^2
    =
    2+2\xi_r\xi_r^* z_{a,i}z_{b,i}^*.
\]
The constant term $2$ contributes $e^{\DeltaT^2/d}$ per $(i,r)$, which over
all $\nT d$ target coordinates exactly cancels the prefactor
$e^{-\nT\DeltaT^2}$ from the product of the two likelihood ratios, leaving
\[
    \exp\!\left(\frac{\DeltaT^2}{d}\sum_{r=1}^d\xi_r\xi_r^*\sum_{i=1}^{\nT}z_{a,i}z_{b,i}^*\right).
\]
An identical calculation for the source block leaves
\[
    \exp\!\left(\frac{\DeltaS^2}{d}\sum_{r=1}^d\kappa_r\kappa_r^*\xi_r\xi_r^*\sum_{j=1}^{\nS}\zsrc_j (\zsrc_*)_j\right).
\]
Introduce the replica overlap signs
\[
    b_r:=\xi_r\xi_r^*\in\{\pm1\},
    \qquad
    \rho_d:=\frac{1}{d}\sum_{r=1}^d b_r.
\]
Since $\xi_r,\xi_r^*\overset{\mathrm{iid}}{\sim}\mathrm{Rad}(1/2)$, the
products $b_1,\ldots,b_d$ are also i.i.d.\ $\mathrm{Rad}(1/2)$. Combining the
target and source pieces,
\begin{equation}
\label{eq:item3-replica-inner}
    \Ebb_Q[L_aL_b]
    =
    \Ebb\!\left[
        \exp\!\left(\frac{\DeltaT^2}{d}\sum_r b_r\sum_i z_{a,i}z_{b,i}^*\right)
        \exp\!\left(\frac{\DeltaS^2}{d}\sum_r\kappa_r\kappa_r^* b_r\sum_j\zsrc_j (\zsrc_*)_j\right)
    \right].
\end{equation}
We now average the label products in
\eqref{eq:item3-replica-inner} separately over the non-tested target
observations, the source observations, and the distinguished target pair.

\paragraph{Step 3: non-tested target pairs.}
For each pair index $\ell\neq j_0$, the two replicas independently average over
$(\eta_\ell,\tau_\ell)$ and $(\eta_\ell^*,\tau_\ell^*)$. Therefore for every
observation index $i\notin\{2j_0-1,2j_0\}$ the product $z_{a,i}z_{b,i}^*$ is an
independent $\mathrm{Rad}(1/2)$ sign. Equivalently, the non-tested target
observations contribute i.i.d.\ unbiased replica products.

Fixing $(b_r)_{r\le d}$, the corresponding factor in
\eqref{eq:item3-replica-inner} can be written as
\[
    \prod_{i\notin\{2j_0-1,2j_0\}}
    \exp\!\left(\frac{\DeltaT^2}{d}z_{a,i}z_{b,i}^*\sum_{r=1}^d b_r\right)
    =
    \prod_{i\notin\{2j_0-1,2j_0\}}
    \exp\!\left(\DeltaT^2\rho_d\, z_{a,i}z_{b,i}^*\right).
\]
Since each $z_{a,i}z_{b,i}^*$ is $\mathrm{Rad}(1/2)$, averaging one such factor
gives $\cosh(\DeltaT^2\rho_d)$, and independence across the $\nT-2$
non-tested target observations gives
\[
    \Ebb\!\left[\exp\!\left(\frac{\DeltaT^2}{d}\sum_r b_r\sum_{i\notin\{2j_0-1,2j_0\}}z_{a,i}z_{b,i}^*\right)\right]
    =
    \cosh(\DeltaT^2\rho_d)^{\nT-2}.
\]

\paragraph{Step 4: source observations.}
Let $t_r:=\kappa_r\kappa_r^*\in\{\pm1\}$ and
$s_j:=\zsrc_j(\zsrc_*)_j\in\{\pm1\}$. Since
$\kappa_r,\kappa_r^*\overset{\mathrm{iid}}{\sim}\mathrm{Rad}\!\bigl(\frac{1+\tilde\mu}{2}\bigr)$,
we have $\Ebb[t_r]=\tilde\mu^2$. Since $s_j\sim\mathrm{Rad}(1/2)$ and is
independent of $(t_r,b_r)$, averaging over $s_j$ first:
\[
    \Ebb_{s_j}\!\left[\exp\!\left(\frac{\DeltaS^2}{d}s_j\sum_r t_r b_r\right)\right]
    =
    \cosh\!\left(\frac{\DeltaS^2}{d}\sum_r t_r b_r\right).
\]
The signs $t_1,\ldots,t_d$ are shared by all $\nS$ source observations.
They must therefore be averaged only after the product over source
observations is formed.  Conditional on $b=(b_1,\ldots,b_d)$, the correct
source factor is
\begin{equation}
\label{eq:item3-correct-source-factor}
    \Gamma_d(b)
    :=
    \Ebb_t\!\left[
      \cosh\!\left(\frac{\DeltaS^2}{d}\sum_{r=1}^d t_rb_r\right)^{\nS}
    \right],
\end{equation}
not a power of an expectation over $t$.

We next bound this corrected factor.  Put
\[
    \alpha:=\frac{\DeltaS^2}{d},\qquad
    S_b:=\sum_{r=1}^dt_rb_r,\qquad
    m_b:=\Ebb_tS_b=\tilde\mu^2d\rho_d.
\]
The centered variable $U_b:=S_b-m_b$ is conditionally $d$-sub-Gaussian:
\[
    \Ebb_t e^{\lambda U_b}\le e^{d\lambda^2/2}.
\]
Using $\log\cosh x\le x^2/2$, set
$q:=\nS\alpha^2/2$. Then
\[
    \Gamma_d(b)\le \Ebb_t e^{qS_b^2}.
\]
The Gaussian identity $e^{qx^2}=\Ebb_Ge^{\sqrt{2q}Gx}$, with
$G\sim N(0,1)$, and conditional sub-Gaussianity give
\begin{align}
    \Ebb_t e^{qS_b^2}
    &\le
    \Ebb_G\exp\!\left(\sqrt{2q}Gm_b+qdG^2\right)\nonumber\\
    &=
    (1-2qd)^{-1/2}
    \exp\!\left(\frac{qm_b^2}{1-2qd}\right).
    \label{eq:item3-source-gaussian-linearization}
\end{align}
Here $2qd=\nS\DeltaS^4/d$.  Hence, if
$\nS\DeltaS^4/d\le c<1$,
\begin{equation}
\label{eq:item3-correct-source-bound}
    \Gamma_d(b)
    \le
    (1-c)^{-1/2}
    \exp\!\left\{
      \frac{\nS\tilde\mu^4\DeltaS^4}{2(1-c)}\rho_d^2
    \right\}.
\end{equation}

\paragraph{Step 5: distinguished pair.}
Under hypothesis $a$, the labels are $z_{a,2j_0-1}=\eta_{j_0}$ and $z_{a,2j_0}=a\eta_{j_0}$, with the second replica using $z_{b,2j_0-1}^*=\eta_{j_0}^*$ and $z_{b,2j_0}^*=b\eta_{j_0}^*$. The inner product of the two label vectors on this pair is $\eta_{j_0}\eta_{j_0}^*(1+ab)$. If $a=b$, this equals $2\eta_{j_0}\eta_{j_0}^*$, and averaging over $\eta_{j_0}\eta_{j_0}^*\sim\mathrm{Rad}(1/2)$ gives $\cosh(2\DeltaT^2\rho_d)$. If $a\neq b$, this equals $0$, contributing the factor $1$.
This is the only place where the neighboring hypotheses $a=+1$ and $a=-1$
behave differently.  All other target pairs and the source block contribute the
same background factors under both hypotheses.

\paragraph{Step 6: combine.}
Multiplying the contributions from Steps 3--5:
\begin{align*}
    \Ebb_Q[L_+^2] &= \Ebb_b\!\left[\cosh(2\DeltaT^2\rho_d)\cosh(\DeltaT^2\rho_d)^{\nT-2}\Gamma_d(b)\right],\\
    \Ebb_Q[L_+L_-] &= \Ebb_b\!\left[\cosh(\DeltaT^2\rho_d)^{\nT-2}\Gamma_d(b)\right].
\end{align*}
By symmetry $\Ebb_Q[L_-^2]=\Ebb_Q[L_+^2]$, so
\begin{equation}
\label{eq:item3-second-moment}
    \Ebb_Q[(L_+-L_-)^2]
    = 2\Ebb_b\!\left[(\cosh(2\DeltaT^2\rho_d)-1)\cosh(\DeltaT^2\rho_d)^{\nT-2}\Gamma_d(b)\right].
\end{equation}
Equation \eqref{eq:item3-second-moment} is the exact replica formula: the
distinguished pair contributes the hypothesis-sensitive factor
$\cosh(2\DeltaT^2\rho_d)-1$, while all remaining observations only amplify it
through positive multiplicative terms.

\paragraph{Step 7: bound to obtain \eqref{eq:item3-l2-target}.}
Apply $\cosh(2x)-1\le 2x^2e^{x^2}$ and $\log\cosh(x)\le x^2/2$:
\[
    \cosh(2\DeltaT^2\rho_d)-1 \le 2\DeltaT^4\rho_d^2 e^{\DeltaT^4\rho_d^2},
    \qquad
    \cosh(\DeltaT^2\rho_d)^{\nT-2}\le\exp\!\left(\tfrac{\nT-2}{2}\DeltaT^4\rho_d^2\right).
\]
Substituting these estimates and
\eqref{eq:item3-correct-source-bound} into
\eqref{eq:item3-second-moment}, and collecting the exponential factors in
$\rho_d^2$, gives, after changing universal constants,
\[
    \Ebb_Q[(L_+-L_-)^2]
    \le
    C\DeltaT^4
    \Ebb_b\!\left[
      \rho_d^2\exp\!\left\{
        C\bigl(\nT\DeltaT^4+\nS\tilde\mu^4\DeltaS^4\bigr)\rho_d^2
      \right\}
    \right],
\]
which is \eqref{eq:item3-l2-target}.
\paragraph{Step 8: simplify under an aggregate subcritical condition.}
Now assume in addition that
\[
    \nT\DeltaT^4+\nS\DeltaS^4\le c\,d
\]
for a sufficiently small universal constant $c>0$.  Set
\[
    \lambda_d:=C\bigl(\nT\DeltaT^4+\nS\tilde\mu^4\DeltaS^4\bigr),
\]
so $\lambda_d\le Ccd$.  Choosing $c$ sufficiently small ensures
$\lambda_d\le d/4$.  Since $\rho_d=d^{-1}\sum_{r=1}^d b_r$ with
$b_r\overset{\mathrm{iid}}{\sim}\mathrm{Rad}(1/2)$, Hoeffding's inequality gives
\[
    \Pbb(\rho_d^2\ge u)=\Pbb(|\rho_d|\ge \sqrt{u})\le 2e^{-du/2},
    \qquad u\ge 0.
\]
Using the tail-integral identity
\[
    \Ebb[Xe^{\lambda X}]
    =
    \int_0^\infty (1+\lambda u)e^{\lambda u}\,\Pbb(X\ge u)\,du
\]
for a nonnegative random variable $X$ and $\lambda\ge 0$, applied to
$X=\rho_d^2$ and $\lambda=\lambda_d$, we obtain
\begin{align*}
    \Ebb\!\left[\rho_d^2e^{\lambda_d\rho_d^2}\right]
    &\le
    \int_0^\infty (1+\lambda_d u)e^{\lambda_d u}\Pbb(\rho_d^2\ge u)\,du\\
    &\le
    2\int_0^\infty (1+\lambda_d u)e^{-(d/2-\lambda_d)u}\,du
    \;\lesssim\;
    \frac{1}{d},
\end{align*}
because $d/2-\lambda_d\ge d/4$.  Substituting this bound into
\eqref{eq:item3-l2-target} yields
\[
    \Ebb_Q[(L_+-L_-)^2]
    \le
    C\,\frac{\DeltaT^4}{d}.
\]
Dividing by $4$ gives \eqref{eq:sec3_tv_doubly}.
\end{proof}

\subsubsection{Proof of Item 4: product-scale route}

We prove the bound~\eqref{eq:sec3_tv_product}.

\paragraph{Setup.}
Fix a pair index $j_0$ and condition on the full nuisance configuration
\begin{equation}\label{eq:def-Xi_k}
    \Xi_{j_0} := (\eta,\tau_{-j_0},\zsrc),
\end{equation}
where $\eta=(\eta_\ell)_{\ell=1}^k$, $\tau_{-j_0}=(\tau_\ell:\ell\neq j_0)$, and
$\zsrc\in\{\pm1\}^{\nS}$.  Once $\Xi_{j_0}$ and $\tau_{j_0}$ are fixed, the
target label vectors $\ztar_\pm:=\ztar(\tau_{j_0}=\pm1,\tau_{-j_0},\eta)$ are
determined.  Define
\begin{equation}\label{eq:def-l-defs}
    l^{(\trg)}_\pm := \frac{\ztar_\pm}{\sqrt{\nT}}\in\sphere{\nT-1},
    \qquad
    l^{(\src)} := \frac{\zsrc}{\sqrt{\nS}}\in\sphere{\nS-1},
\end{equation}
and set
\begin{equation}\label{eq:def-rT-rS}
    r_T := \DeltaT\sqrt{\frac{\nT}{d}},
    \qquad
    r_S := \DeltaS\sqrt{\frac{\nS}{d}}.
\end{equation}
For each coordinate $r\in[d]$, let
\begin{equation}\label{eq:def-xr-yr}
    x_r := \left(\frac{\Xtar_{i,r}}{\sigmaT}\right)_{i=1}^{\nT}\in\R^{\nT},
    \qquad
    y_r := \left(\frac{\Xsrc_{j,r}}{\sigmaS}\right)_{j=1}^{\nS}\in\R^{\nS},
\end{equation}
and define the inner-product statistics
\begin{equation}\label{eq:def-a-b}
    a_\pm(x_r) := r_T\ip{l^{(\trg)}_\pm}{x_r},
    \qquad
    b(y_r) := r_S\ip{l^{(\src)}}{y_r}.
\end{equation}

\begin{proof}[Proof of Theorem~\ref{thm:sec3_tv_three_routes}, Item~4]
	Lemma~\ref{lem:sec3_product_row_density} gives the row density conditional
	on the shared label configuration $\Xi_{j_0}$.
	Lemmas~\ref{lem:sec3_product_odd_pert}--\ref{lem:sec3_product_approx_error}
	provide the derivative and moment bounds needed for the row-wise KL estimate
	in Lemma~\ref{lem:sec3_product_row_kl}.
	Lemma~\ref{lem:sec3_product_full_kl} then sums over coordinates and converts
	to total variation.
	
	Under the stated assumptions $\nT\DeltaT^2/d\le c^2$ and
	$\tilde\mu^2\nS\DeltaS^2/d\le c^2$, we have $r_T^2\le c^2$ and
	$\tilde\mu^2 r_S^2\le c^2$, so all hypotheses of the preceding lemmas are
	satisfied.  The TV bound~\eqref{eq:sec3_product_tv_cond} is exactly
	\eqref{eq:sec3_tv_product}, with the factor of $2$ absorbed into the
	universal constant $C$.
\end{proof}
\begin{lemma}[Conditional row density]
\label{lem:sec3_product_row_density}
With $\Xi_{j_0}$ as in~\eqref{eq:def-Xi_k}, $l^{(\trg)}_\pm$, $l^{(\src)}$ as in~\eqref{eq:def-l-defs}, $r_T$, $r_S$ as in~\eqref{eq:def-rT-rS}, and $x_r$, $y_r$, $a_\pm$, $b$ as in~\eqref{eq:def-xr-yr}, fix $\Xi_{j_0}$.
The conditional row law of $(x_r,y_r)$ under $\tau_{j_0}=\pm1$
has density relative to the standard Gaussian reference measure given by
\begin{equation}
\label{eq:sec3_product_row_density}
    p_{\pm,r}(x_r,y_r)
    =
    e^{-(r_T^2+r_S^2)/2}\Bigl[
        \cosh\!\bigl(a_\pm(x_r)\bigr)\cosh\!\bigl(b(y_r)\bigr)
        + \tilde\mu\sinh\!\bigl(a_\pm(x_r)\bigr)\sinh\!\bigl(b(y_r)\bigr)
    \Bigr].
\end{equation}
\end{lemma}

\begin{proof}
Each Gaussian shift contributes $dN(m,1)/dN(0,1)=\exp(mu-m^2/2)$.
Summing the quadratic terms over all $\nT$ target observations in row $r$
gives $-\nT\DeltaT^2/(2d)=-r_T^2/2$, and over all $\nS$ source observations
gives $-\nS\DeltaS^2/(2d)=-r_S^2/2$.  Hence, conditional on
$(\Xi_{j_0},\tau_{j_0},\xi_r=e,\kappa_r=\kappa)$, the row density relative to the
standard Gaussian measure is
\[
    e^{-(r_T^2+r_S^2)/2}\exp\!\bigl(e\,a_{\tau_{j_0}}(x_r)+e\kappa\,b(y_r)\bigr).
\]
The prefactor $e^{-(r_T^2+r_S^2)/2}$ does not depend on $(e,\kappa,\tau_{j_0})$
and factors out.  Averaging first over
$\kappa_r\sim\mathrm{Rad}\!\bigl(\tfrac{1+\tilde\mu}{2}\bigr)$:
\[
    \Ebb_\kappa[e^{e\kappa b}]
    =
    \cosh(eb)+\tilde\mu\sinh(eb)
    =
    \cosh(b)+\tilde\mu e\sinh(b),
\]
using that $\cosh$ is even and $\sinh$ is odd.  Then averaging over
$\xi_r=e\sim\mathrm{Rad}(1/2)$ and using
$\Ebb_e[e^{ea}]=\cosh(a)$ and $\Ebb_e[e\,e^{ea}]=\sinh(a)$,
substituting $a=a_{\tau_{j_0}}(x_r)$ gives \eqref{eq:sec3_product_row_density}.
\end{proof}

\begin{lemma}[Odd perturbation bounds for the row log-density]
\label{lem:sec3_product_odd_pert}
Define
\[
    h(a,b):=\log\!\bigl(\cosh(a)\cosh(b)+\tilde\mu\sinh(a)\sinh(b)\bigr).
\]
Set $\Delta_1:=\partial_a h(a,b)-\tanh(a)$ and
$\Delta_2:=\partial_{aa}h(a,b)-\operatorname{sech}^2(a)$.  Then for all
$a,b\in\R$:
\begin{enumerate}[label=(\roman*)]
\item $\displaystyle\partial_a h(a,b)
      =
      \frac{\tanh(a)+\tilde\mu\tanh(b)}{1+\tilde\mu\tanh(a)\tanh(b)}$,
      and $|\partial_a h(a,b)|\le 1$.
\item $\partial_{aa}h(a,b)=1-(\partial_a h(a,b))^2$,
      and $0\le\partial_{aa}h(a,b)\le 1$.
\item $|\partial_b\Delta_1|\le\tilde\mu$ and $|\Delta_1|\le\tilde\mu|b|$.
\item $\Delta_2=-(2\tanh(a)+\Delta_1)\Delta_1$,
      so $|\Delta_2|\le 2\tilde\mu|b|+\tilde\mu^2 b^2$.
\end{enumerate}
\end{lemma}

\begin{proof}
Write $f(a,b):=\cosh(a)\cosh(b)+\tilde\mu\sinh(a)\sinh(b)$.  Since
$f(a,b)=\frac{1+\tilde\mu}{2}\cosh(a+b)+\frac{1-\tilde\mu}{2}\cosh(a-b)\ge1$,
the logarithm is well-defined.

\textit{Part (i).}  Differentiating $f_a=\sinh(a)\cosh(b)+\tilde\mu\cosh(a)\sinh(b)$
and dividing by $f$ gives the formula; $|\partial_a h|\le1$ follows because
$\partial_a h$ is a convex combination of $\tanh(a\pm b)$.

\textit{Part (ii).}  The quotient rule gives
$\partial_{aa}h=f_{aa}/f-(f_a/f)^2=1-(\partial_a h)^2\in[0,1]$,
since $f_{aa}=f$.

\textit{Part (iii).}  Subtracting $\tanh(a)$ from Part~(i):
$\Delta_1=\tilde\mu\tanh(b)\operatorname{sech}^2(a)/(1+\tilde\mu\tanh(a)\tanh(b))$.
Setting $u=\tanh(a)$, $v=\tanh(b)$ and differentiating in $b$ shows
$|\partial_b\Delta_1|\le\tilde\mu$ (the denominator dominates the numerator
since $(1+\tilde\mu uv)^2\ge(1-u^2)(1-v^2)$).  Since $\Delta_1(a,0)=0$,
integrating gives $|\Delta_1|\le\tilde\mu|b|$.

\textit{Part (iv).}  Write $\partial_a h=\tanh(a)+\Delta_1$ and use Part~(ii):
$\partial_{aa}h=\operatorname{sech}^2(a)-2\tanh(a)\Delta_1-\Delta_1^2$,
giving the stated identity for $\Delta_2$.  The bound follows from
$|\tanh(a)|\le1$ and Part~(iii).
\end{proof}

\begin{lemma}[Odd Lipschitz averaging]
\label{lem:sec3_product_odd_lip}
With $r_S$ as in~\eqref{eq:def-rT-rS}, let $\psi:\R\to\R$ be odd and $L$-Lipschitz,
and set $B_t:=ts+r_S h$ where $s\in\{\pm1\}$ with
$\Pbb(s=+1)=\frac{1+\tilde\mu}{2}$ and $h\sim N(0,1)$ is independent of $s$.
Then $\Ebb[\psi(B_t)]=\tilde\mu\,m_\psi(t)$ where $m_\psi$ is odd and
$|m_\psi(t)|\le L|t|$.
\end{lemma}

\begin{proof}
Define $g_\psi(t):=\Ebb_h[\psi(t+r_S h)]$.  Since $\psi$ is odd and
$h\overset{d}{=}-h$, the function $g_\psi$ is odd with $g_\psi(0)=0$ and
$|g_\psi(t)|\le L|t|$.  Conditioning on $s$:
\[
    \Ebb[\psi(B_t)]
    =
    \tfrac{1+\tilde\mu}{2}g_\psi(t)+\tfrac{1-\tilde\mu}{2}g_\psi(-t)
    =
    \tilde\mu g_\psi(t).
\]
Set $m_\psi:=g_\psi$.
\end{proof}

\begin{lemma}[Moment bounds for the rotated coordinates]
\label{lem:sec3_product_moment_bds}
With $r_T$, $r_S$ as in~\eqref{eq:def-rT-rS} and $l^{(\src)}$ as in~\eqref{eq:def-l-defs}, fix
$l^{(\trg)}_+=e_1\in\sphere{\nT-1}$ and let $(x,y)\sim p_{+,r}$, where
$p_{+,r}$ is the row density~\eqref{eq:sec3_product_row_density} with
$a(x)=r_T X_1$ ($X_1$ the first coordinate of $x$) and
$b(y)=r_S\ip{l^{(\src)}}{y}$.
Let $\alpha\in\R$ and define
\[
    Y_\alpha:=\cos\alpha\,X_1+\sin\alpha\,X_2, \quad
    Z_\alpha:=-\sin\alpha\,X_1+\cos\alpha\,X_2, \quad
    B:=r_S\ip{l^{(\src)}}{y}.
\]
Assume $r_T\le1$.  Then there is a representation
\begin{align}
    X_1 &= r_T\eta+g_1, & \eta&\sim\mathrm{Rad}(1/2),\ g_1\sim N(0,1),
    \label{eq:sec3_product_X1_decomp}\\
    B   &= r_S^2 s\eta+r_S h, & s&\sim\mathrm{Rad}\!\bigl(\tfrac{1+\tilde\mu}{2}\bigr),\
    h\sim N(0,1), \label{eq:sec3_product_B_decomp}
\end{align}
where $(\eta,s,g_1,h)$ are independent, and $X_2=g_2\sim N(0,1)$ is
independent of $(X_1,B)$.  Uniformly in $\alpha$:
\begin{gather*}
    \Ebb|Y_\alpha|^j\le C\;\text{ for }j=1,2,3,4,6, \qquad
    \Ebb Z_\alpha^4\le C,\\
    \Ebb|B|^4\le C(r_S^8+r_S^4), \qquad
    \Ebb|r_T Y_\alpha|^4\le Cr_T^4, \qquad
    |\Ebb[Y_\alpha B]|\le C\tilde\mu r_T r_S^2.
\end{gather*}
\end{lemma}

\begin{proof}
The row density~\eqref{eq:sec3_product_row_density} with $l^{(\trg)}_+=e_1$ is
generated by drawing $\eta\sim\mathrm{Rad}(1/2)$, then
$x\mid\eta\sim N(r_T\eta e_1,I_{\nT})$, and drawing
$s\sim\mathrm{Rad}((1+\tilde\mu)/2)$ independently so that
$B\mid(\eta,s)\sim N(r_S^2 s\eta,r_S^2)$, giving
\eqref{eq:sec3_product_X1_decomp}--\eqref{eq:sec3_product_B_decomp} with
$X_2=g_2\sim N(0,1)$ independent.

Conditional on $\eta$, $(Y_\alpha,Z_\alpha)$ is a rotation of $(X_1,X_2)$
with mean bounded by $r_T\le1$ and identity covariance, so moments through
order $6$ are uniformly bounded.  For $B$: the $L^4$ triangle inequality
applied to $B=r_S^2 s\eta+r_S h$ gives $\Ebb|B|^4\le C(r_S^8+r_S^4)$.
For the cross moment: $\Ebb[X_2 B]=0$ since $X_2\perp(s,\eta,h)$, and
\[
    \Ebb[X_1 B]=r_T r_S^2\Ebb[\eta^2 s]=\tilde\mu r_T r_S^2,
\]
using $\eta^2=1$ and $\Ebb[s]=\tilde\mu$.  Hence
$|\Ebb[Y_\alpha B]|=|\cos\alpha\,\Ebb[X_1 B]|\le\tilde\mu r_T r_S^2$.
\end{proof}

\begin{lemma}[Linear approximation error for the row score]
\label{lem:sec3_product_approx_error}
With $r_T$, $r_S$ as in~\eqref{eq:def-rT-rS} and the notation of
Lemma~\ref{lem:sec3_product_moment_bds}, set $a:=r_T Y_\alpha$, $b:=B$,
and $q(a,b):=\partial_a h(a,b)$.  Assume $r_T\le c$ and
$\tilde\mu^2 r_S^2\le c$.  Then
\begin{equation}
\label{eq:sec3_product_approx_error}
    \bigl|\Ebb\bigl[Y_\alpha\bigl(q(a,b)-a-\tilde\mu b\bigr)\bigr]\bigr|
    \le
    C\bigl(r_T^3+\tilde\mu^2 r_T r_S^2\bigr),
\end{equation}
and consequently
$r_T\bigl|\Ebb[Y_\alpha(q-a-\tilde\mu b)]\bigr|\le C(r_T^4+\tilde\mu^2 r_T^2 r_S^2)$.
\end{lemma}

\begin{proof}
Write
\[
    q-a-\tilde\mu b
    =
    \underbrace{(\tanh(a)-a)}_{=:\,e_1(a)}
    +
    \tilde\mu\underbrace{(\tanh(b)-b)}_{=:\,e_2(b)}
    -
    \tilde\mu uv q,
    \qquad u:=\tanh(a),\ v:=\tanh(b).
\]
Denote the resulting expectations by $E_{21}:=\Ebb[Y_\alpha e_1(a)]$,
$E_{22}:=\Ebb[Y_\alpha e_2(b)]$, and $E_{23}:=\Ebb[Y_\alpha uvq]$.

\textit{Target nonlinearity.}  Since $|\tanh(t)-t|\le|t|^3$,
\[
    |E_{21}|\le r_T^3\Ebb|Y_\alpha|^4\le Cr_T^3.
\]

\textit{Source nonlinearity.}  The function $\psi(t):=\tanh(t)-t$ is odd and
$1$-Lipschitz.  Conditional on $(\eta,g_1,g_2)$, $B$ has the form $B_t$ from
Lemma~\ref{lem:sec3_product_odd_lip} with $t=r_S^2\eta$, so
$\Ebb[\psi(B)\mid\eta,g_1,g_2]=\tilde\mu m_\psi(r_S^2\eta)$ with
$|m_\psi(r_S^2\eta)|\le r_S^2$.  Since $m_\psi$ is odd,
$m_\psi(r_S^2\eta)=\eta m_\psi(r_S^2)$, giving
\[
    |\tilde\mu E_{22}|=\tilde\mu^2|m_\psi(r_S^2)||\Ebb[Y_\alpha\eta]|
    \le\tilde\mu^2 r_S^2\cdot r_T|\cos\alpha|
    \le\tilde\mu^2 r_T r_S^2.
\]

\textit{Cross term.}  Decompose $uvq=u^2v+\tilde\mu uv^2-\tilde\mu u^2v^2 q$.
Using Lemma~\ref{lem:sec3_product_odd_lip} for $\psi=\tanh$, $|q|\le1$,
and the moment bounds from Lemma~\ref{lem:sec3_product_moment_bds}:
$|\tilde\mu E_{23}|\le C\tilde\mu^2 r_T r_S^2$.

Collecting gives \eqref{eq:sec3_product_approx_error}; multiplying by $r_T$
gives the second claim.
\end{proof}

\begin{lemma}[Score-square moment bound]
\label{lem:sec3_product_score_sq}
With $r_T$, $r_S$ as in~\eqref{eq:def-rT-rS} and the notation of
Lemma~\ref{lem:sec3_product_moment_bds}, set $a:=r_T Y_\alpha$, $b:=B$,
and $q(a,b):=\partial_a h(a,b)$.  Assume $r_T\le c$.  Then
\begin{equation}
\label{eq:sec3_product_score_sq}
    r_T^2\,\Ebb\!\bigl[Z_\alpha^2\,q(r_T Y_\alpha,B)^2\bigr]
    \le
    C\bigl(r_T^4+\tilde\mu^2 r_T^2 r_S^2\bigr).
\end{equation}
\end{lemma}

\begin{proof}
Differentiating $h(a,b)=\log(\cosh a\cosh b+\tilde\mu\sinh a\sinh b)$ in $a$
and evaluating at $a=0$ gives $q(0,b)=\tilde\mu\tanh(b)$.
Since $\partial_a q=1-q^2\in[0,1]$, the map $a\mapsto q(a,b)$ is $1$-Lipschitz, so
\[
    |q(r_T Y_\alpha,B)-\tilde\mu\tanh(B)|\le r_T|Y_\alpha|.
\]
Hence $q^2\le 2r_T^2 Y_\alpha^2+2\tilde\mu^2\tanh^2(B)$, giving
\[
    r_T^2\Ebb[Z_\alpha^2 q^2]
    \le 2r_T^4\Ebb[Z_\alpha^2 Y_\alpha^2]
       +2\tilde\mu^2 r_T^2\Ebb[Z_\alpha^2\tanh^2(B)].
\]
By Lemma~\ref{lem:sec3_product_moment_bds} and Cauchy--Schwarz,
$\Ebb[Z_\alpha^2 Y_\alpha^2]\le(\Ebb Z_\alpha^4)^{1/2}(\Ebb Y_\alpha^4)^{1/2}\le C$,
so the first term is $\le Cr_T^4$.

For the second term, apply Cauchy--Schwarz:
$$\Ebb[Z_\alpha^2\tanh^2(B)]\le(\Ebb Z_\alpha^4)^{1/2}(\Ebb\tanh^4(B))^{1/2}\le C(\Ebb\tanh^4(B))^{1/2}.$$
Since $|\tanh t|\le1$ and $|\tanh t|\le|t|$, we have $\Ebb[\tanh^4(B)]\le\min(1,\Ebb[B^4])$.
By Lemma~\ref{lem:sec3_product_moment_bds}, $\Ebb|B|^4\le C(r_S^4+r_S^8)$, so
$\min(1,\Ebb[B^4])\le\min(1,C(r_S^4+r_S^8))\le Cr_S^4$,
where the last bound holds for all $r_S\ge0$
(for $r_S\le1$: $r_S^4+r_S^8\le2r_S^4$; for $r_S>1$: $\min(1,\cdot)=1\le r_S^4$).
Hence $(\Ebb\tanh^4(B))^{1/2}\le Cr_S^2$, giving $\Ebb[Z_\alpha^2\tanh^2(B)]\le Cr_S^2$,
and the second term is $\le C\tilde\mu^2 r_T^2 r_S^2$.
\end{proof}

\begin{lemma}[Conditional row KL bound]
\label{lem:sec3_product_row_kl}
With $\Xi_{j_0}$ as in~\eqref{eq:def-Xi_k}, $r_T$, $r_S$ as in~\eqref{eq:def-rT-rS}, and $l^{(\trg)}_\pm$ as in~\eqref{eq:def-l-defs},
for each fixed $\Xi_{j_0}$ let $p_{\pm,r}$ be the conditional row densities from
Lemma~\ref{lem:sec3_product_row_density}.  Whenever $r_T\le c$ and
$\tilde\mu^2 r_S^2\le c$,
\begin{equation}
\label{eq:sec3_product_row_kl}
    \KL(p_{+,r}\,\|\,p_{-,r})
    \le
    C\,r_T^2(r_T^2+\tilde\mu^2 r_S^2)\,\|l^{(\trg)}_+-l^{(\trg)}_-\|^2.
\end{equation}
\end{lemma}

\begin{proof}
\textit{Reduction by rotation.}
KL divergence is invariant under bijective measurable maps of the sample
space: for any bijection $T$, $\KL(P\|Q)=\KL(T_\#P\|T_\#Q)$.  We exploit
this with the map $T:(x_r,y_r)\mapsto(Ox_r,y_r)$, where $O\in O(\nT)$ is the
orthogonal transformation satisfying
\[
    Ol^{(\trg)}_+ = e_1, \qquad Ol^{(\trg)}_- = l_\alpha := \cos\alpha\,e_1+\sin\alpha\,e_2,
    \qquad \alpha := \arccos\!\ip{l^{(\trg)}_+}{l^{(\trg)}_-}\in[0,\pi].
\]
(Such an $O$ exists because $l^{(\trg)}_\pm$ are unit vectors; $l_\alpha$ is the
unique unit vector in $\mathrm{span}\{e_1,e_2\}$ with $\ip{e_1}{l_\alpha}=\ip{l^{(\trg)}_+}{l^{(\trg)}_-}$.)

\textit{Why only $x_r$ is rotated.}
The two hypotheses $\tau_{j_0}=\pm1$ differ only through the target statistic
$a_\pm(x_r)=r_T\ip{l^{(\trg)}_\pm}{x_r}$.
The source statistic $b(y_r)=r_S\ip{l^{(\src)}}{y_r}$ is the same under
both hypotheses (since $l^{(\src)}=\zsrc/\sqrt{\nS}$ is fixed by $\Xi_{j_0}$),
so no rotation of $y_r$ is needed or helpful.  We therefore act on
$x_r$ alone, leaving $y_r$ unchanged.

\textit{Effect of $T$ on the row distributions.}
The pushforward $T_\#P_{\pm,r}$ is the law of $(Ox_r,y_r)$ when
$(x_r,y_r)\sim P_{\pm,r}$.  Since $p_{\pm,r}$
(Lemma~\ref{lem:sec3_product_row_density}) depends on $x_r$ only through
$a_\pm(x_r)=r_T\ip{l^{(\trg)}_\pm}{x_r}$, and
$\ip{l^{(\trg)}_\pm}{x_r}=\ip{Ol^{(\trg)}_\pm}{Ox_r}$, setting
$\tilde x_r:=Ox_r$ shows that $T_\#P_{+,r}$ has target statistic
$r_T\ip{e_1}{\tilde x_r}$ and $T_\#P_{-,r}$ has target statistic
$r_T\ip{l_\alpha}{\tilde x_r}$, while both retain the same source statistic
$b(y_r)$.  Hence $T_\#P_{+,r}=P_{e_1}$ and $T_\#P_{-,r}=P_{l_\alpha}$,
where the subscripts now denote the (rotated) target label directions.
For the densities to be expressed against the same reference measure
$Q=N(0,I_{\nT})\otimes N(0,I_{\nS})$ after the change of variables, we
need $T_\#Q=Q$.  The source factor $N(0,I_{\nS})$ is unchanged since $T$
does not touch $y_r$.  The target factor $N(0,I_{\nT})$ is $O$-invariant
because $\|Ox_r\|=\|x_r\|$.  Therefore $T_\#Q=Q$, and
\[
    \KL(P_{+,r}\|P_{-,r})
    =
    \KL(T_\#P_{+,r}\|T_\#P_{-,r})
    =
    \KL(P_{e_1}\|P_{l_\alpha})
    =:
    F(\alpha).
\]

\textit{Definitions in rotated coordinates.}
Write $X_i:=\ip{e_i}{x_r}$ for the $i$-th coordinate of the (now rotated)
standardized column, and define (as in Lemma~\ref{lem:sec3_product_moment_bds})
\[
    Y_\alpha := \ip{l_\alpha}{x_r} = \cos\alpha\,X_1+\sin\alpha\,X_2,
    \qquad
    Z_\alpha := \partial_\alpha Y_\alpha = -\sin\alpha\,X_1+\cos\alpha\,X_2,
\]
so that $\partial_\alpha Z_\alpha=-Y_\alpha$.  After the rotation,
$a_+(x_r)=r_T X_1$ and $a_-(x_r)=r_T Y_\alpha$.  Also set
\[
    B := b(y_r) = r_S\ip{l^{(\src)}}{y_r},
\]
the source statistic from~\eqref{eq:def-xr-yr}, which is the same under
both hypotheses.  Because
$h(a,b)=\log(\cosh a\cosh b+\tilde\mu\sinh a\sinh b)$ is the log-density
of each row (Lemma~\ref{lem:sec3_product_row_density}), and the Gaussian
prefactor $e^{-(r_T^2+r_S^2)/2}$ cancels in the KL ratio,
\[
    F(\alpha)
    =
    \Ebb_{p_{e_1}}\!\bigl[h(r_T X_1,B)-h(r_T Y_\alpha,B)\bigr].
\]

\textit{Law under $p_{e_1}$.}
From the model (see proof of Lemma~\ref{lem:sec3_product_row_density}):
conditional on $\Xi_{j_0}$ and $\tau_{j_0}=+1$, the target column (after rotation) is distributed as
$x_r \overset{d}{=} \xi_r\,r_T\,e_1+\varepsilon$, where $\xi_r\sim\mathrm{Rad}(1/2)$
is the coordinate direction sign and $\varepsilon\sim N(0,I_{\nT})$ is independent.
Thus $X_2=\ip{e_2}{x_r}=\varepsilon_2\sim N(0,1)$ is independent of $(\xi_r,X_1,B)$.
More generally, $(Y_\alpha,Z_\alpha)$ is an orthogonal rotation of $(X_1,X_2)$, and
since under $p_{e_1}$ we have $\Ebb[X_1^2]=r_T^2+1$ and $\Ebb[X_2^2]=1$:
\begin{equation}\label{eq:rotated-moments}
    \Ebb[Y_\alpha^2]=r_T^2\cos^2\alpha+1, \qquad \Ebb[Z_\alpha^2]=r_T^2\sin^2\alpha+1.
\end{equation}

\textit{$F(0)=F'(0)=0$.}
$F(0)=0$ since $Y_0=X_1$.  Differentiating in $\alpha$ and writing
$q:=\partial_a h(r_T Y_\alpha,B)$ (Lemma~\ref{lem:sec3_product_odd_pert}(i)):
\[
    F'(\alpha) = -r_T\,\Ebb_{p_{e_1}}\!\bigl[q(r_T Y_\alpha,B)\,Z_\alpha\bigr].
\]
At $\alpha=0$, $Z_0=X_2$ is independent of $(X_1,B)$ and $\Ebb[X_2]=0$,
so $F'(0)=-r_T\,\Ebb[q(r_T X_1,B)]\cdot\Ebb[X_2]=0$.

\textit{Computing $F''$.}
Differentiating once more, using $\partial_\alpha Z_\alpha=-Y_\alpha$ and
$\partial_a q=\partial_{aa}h=1-q^2$ (Lemma~\ref{lem:sec3_product_odd_pert}(ii)):
\begin{align*}
    F''(\alpha)
    &=
    -r_T^2\Ebb\!\bigl[(1-q^2)Z_\alpha^2\bigr]
    +
    r_T\Ebb\!\bigl[q\,Y_\alpha\bigr]\\
    &=
    r_T\Ebb[q\,Y_\alpha]
    -r_T^2\Ebb[Z_\alpha^2]
    +r_T^2\Ebb[Z_\alpha^2 q^2].
\end{align*}
Writing $q = (r_T Y_\alpha+\tilde\mu B)+(q-r_T Y_\alpha-\tilde\mu B)$
in the first term $r_T\Ebb[qY_\alpha]$ and expanding gives the decomposition
\[
    F''(\alpha) = T_0+T_1+T_2+T_3,
\]
where
\begin{align*}
    T_0 &:= r_T^2\Ebb[Y_\alpha^2-Z_\alpha^2], &
    T_1 &:= \tilde\mu r_T\Ebb[Y_\alpha B], \\
    T_2 &:= r_T\Ebb\!\bigl[Y_\alpha(q-r_T Y_\alpha-\tilde\mu B)\bigr], &
    T_3 &:= r_T^2\Ebb[Z_\alpha^2 q^2].
\end{align*}

\textit{Bound on $T_0$.}
By~\eqref{eq:rotated-moments},
$\Ebb[Y_\alpha^2]-\Ebb[Z_\alpha^2]=r_T^2(\cos^2\alpha-\sin^2\alpha)$,
so $|T_0|=r_T^4|\cos^2\alpha-\sin^2\alpha|\le r_T^4$.

\textit{Bound on $T_1$.}
By Lemma~\ref{lem:sec3_product_moment_bds}, $|\Ebb[Y_\alpha B]|\le\tilde\mu r_T r_S^2$,
so $|T_1|=\tilde\mu r_T|\Ebb[Y_\alpha B]|\le\tilde\mu^2 r_T^2 r_S^2$.

\textit{Bound on $T_2$.}
This is the approximation error; Lemma~\ref{lem:sec3_product_approx_error} gives
$|T_2|\le C(r_T^4+\tilde\mu^2 r_T^2 r_S^2)$.

\textit{Bound on $T_3$.}
Lemma~\ref{lem:sec3_product_score_sq} gives
$|T_3|=r_T^2\Ebb[Z_\alpha^2 q^2]\le C(r_T^4+\tilde\mu^2 r_T^2 r_S^2)$.

\textit{Conclusion.}
Collecting the four bounds,
$|F''(\alpha)|\le Cr_T^2(r_T^2+\tilde\mu^2 r_S^2)$ uniformly in $\alpha$.
Since $F(0)=F'(0)=0$, Taylor's theorem in integral form gives
\[
    F(\alpha) = \int_0^\alpha(\alpha-t)F''(t)\,dt
    \le Cr_T^2(r_T^2+\tilde\mu^2 r_S^2)\alpha^2.
\]
Finally, $\|l^{(\trg)}_+-l^{(\trg)}_-\|^2=2(1-\cos\alpha)=4\sin^2(\alpha/2)$,
and $|\sin\theta|\ge 2|\theta|/\pi$ for $\theta\in[-\pi/2,\pi/2]$, so
$\alpha^2\le\frac{\pi^2}{4}\|l^{(\trg)}_+-l^{(\trg)}_-\|^2$,
proving~\eqref{eq:sec3_product_row_kl}.
\end{proof}

\begin{lemma}[Conditional full-sample KL and TV]
\label{lem:sec3_product_full_kl}
With $\Xi_{j_0}$ as in~\eqref{eq:def-Xi_k} and $r_T$, $r_S$ as in~\eqref{eq:def-rT-rS}, for each fixed $\Xi_{j_0}$,
whenever $r_T\le c$ and $\tilde\mu^2 r_S^2\le c$,
\begin{equation}
\label{eq:sec3_product_kl_bound}
    \KL\!\bigl(P_{+j_0}\mid\Xi_{j_0}\,\big\|\,P_{-j_0}\mid\Xi_{j_0}\bigr)
    \le
    \frac{C\bigl(\nT\DeltaT^4+\tilde\mu^2\nS\DeltaT^2\DeltaS^2\bigr)}{d}.
\end{equation}
Consequently,
\begin{equation}
\label{eq:sec3_product_tv_cond}
    \TV(P_{+j_0},P_{-j_0})
    \le
    \sqrt{\frac{C\bigl(\nT\DeltaT^4+\tilde\mu^2\nS\DeltaT^2\DeltaS^2\bigr)}{2d}}.
\end{equation}
\end{lemma}

\begin{proof}
Fix $\Xi_{j_0}$.  The rows $(x_r,y_r)$ (as defined in~\eqref{eq:def-xr-yr}) are conditionally independent because
$(\xi_r,\kappa_r)$ are i.i.d.\ across $r$.  By additivity of KL for product
measures and Lemma~\ref{lem:sec3_product_row_kl},
\[
    \KL\!\bigl(P_{+j_0}\mid\Xi_{j_0}\,\big\|\,P_{-j_0}\mid\Xi_{j_0}\bigr)
    =
    \sum_{r=1}^d \KL(p_{+,r}\|p_{-,r})
    \le
    d\cdot Cr_T^2(r_T^2+\tilde\mu^2 r_S^2)\|l^{(\trg)}_+-l^{(\trg)}_-\|^2.
\]
The target label vectors $\ztar_+$ and $\ztar_-$ differ only at index $2j_0$,
so $\|l^{(\trg)}_+-l^{(\trg)}_-\|^2=\|\ztar_+-\ztar_-\|^2/\nT=4/\nT$.  Substituting
$r_T^2=\nT\DeltaT^2/d$ and $\tilde\mu^2 r_S^2=\tilde\mu^2\nS\DeltaS^2/d$
proves \eqref{eq:sec3_product_kl_bound}.

By Pinsker's inequality,
\[
    \TV(P_{+j_0}\mid\Xi_{j_0},P_{-j_0}\mid\Xi_{j_0})
    \le
    \sqrt{\tfrac12\KL(P_{+j_0}\mid\Xi_{j_0}\|P_{-j_0}\mid\Xi_{j_0})}.
\]
Since the bound \eqref{eq:sec3_product_kl_bound} is deterministic and uniform
in $\Xi_{j_0}$, TV convexity under mixtures yields
\[
    \TV(P_{+j_0},P_{-j_0})
    \le
    \Ebb_{\Xi_{j_0}}\bigl[\TV(P_{+j_0}\mid\Xi_{j_0},P_{-j_0}\mid\Xi_{j_0})\bigr]
    \le
    \sqrt{\frac{C(\nT\DeltaT^4+\tilde\mu^2\nS\DeltaT^2\DeltaS^2)}{2d}}.
    \qedhere
\]
\end{proof}

\subsection{Proof of Theorem~\ref{thm:necessary_condition}}
\label{app:proof-necessary-cond}

\begin{lemma}[Conditioning the surrogate prior onto the exact alignment class]
\label{lem:sec3_conditioning}
Let $\Pi_{\tilde\mu}$ denote the parameter prior induced by
\eqref{eq:construct_signal_adv} together with the pair-product/label
construction from Section~\ref{sec:necessary_cond}, and define
\[
    E_\mu
    :=
    \left\{|\ip{v_{\trg}}{v_{\src}}|\ge\mu\right\}
    =
    \left\{\left|\frac{1}{d}\sum_{r=1}^d \kappa_r\right|\ge\mu\right\}.
\]
Then the conditional prior $\Pi_{\tilde\mu}(\cdot\mid E_\mu)$ is supported on
the parameter class $\Omega(\DeltaT,\DeltaS,\mu)$. Moreover, for the clustering
loss $\calL\in[0,1]$,
\[
    \inf_{\hatztar}\Ebb_{\Pi_{\tilde\mu}(\cdot\mid E_\mu)}
    \bigl[\calL(\hatztar,\ztar)\bigr]
    \ge
    \frac{
        \inf_{\hatztar}\Ebb_{\Pi_{\tilde\mu}}[\calL(\hatztar,\ztar)]
        -
        \Pi_{\tilde\mu}(E_\mu^c)
    }{
        \Pi_{\tilde\mu}(E_\mu)
    }.
\]
\end{lemma}

\begin{proof}
The support claim is immediate from the definition of $E_\mu$: on this event
the realized directions satisfy the deterministic constraint
$|\ip{v_{\trg}}{v_{\src}}|\ge\mu$, while all other coordinates of the parameter
construction are unchanged, so the conditioned prior is supported on
$\Omega(\DeltaT,\DeltaS,\mu)$.

For the risk comparison, fix any estimator $\hatztar$ and decompose the prior
expectation according to whether $E_\mu$ occurs:
\[
    \Ebb_{\Pi_{\tilde\mu}}[\calL]
    =
    \Pi_{\tilde\mu}(E_\mu)\,
    \Ebb_{\Pi_{\tilde\mu}(\cdot\mid E_\mu)}[\calL]
    +
    \Pi_{\tilde\mu}(E_\mu^c)\,
    \Ebb_{\Pi_{\tilde\mu}(\cdot\mid E_\mu^c)}[\calL].
\]
Because $\calL\in[0,1]$, the second term is bounded by
$\Pi_{\tilde\mu}(E_\mu^c)$. Therefore
\[
    \Ebb_{\Pi_{\tilde\mu}}[\calL]
    \le
    \Pi_{\tilde\mu}(E_\mu)\,
    \Ebb_{\Pi_{\tilde\mu}(\cdot\mid E_\mu)}[\calL]
    +
    \Pi_{\tilde\mu}(E_\mu^c).
\]
Rearranging and then taking the infimum over $\hatztar$ yields
\[
    \inf_{\hatztar}\Ebb_{\Pi_{\tilde\mu}(\cdot\mid E_\mu)}[\calL(\hatztar,\ztar)]
    \ge
    \frac{
        \inf_{\hatztar}\Ebb_{\Pi_{\tilde\mu}}[\calL(\hatztar,\ztar)]
        -
        \Pi_{\tilde\mu}(E_\mu^c)
    }{
        \Pi_{\tilde\mu}(E_\mu)
    }.
\]
\end{proof}

\begin{lemma}[Prescribed bad-event bound for exact-alignment conditioning]
\label{lem:sec3_conditioning_bad_event}
Fix $q_0\in(0,1/8)$. There exist constants
$c_2=c_2(q_0),c_3=c_3(q_0),C_*=C_*(q_0)>0$ and
$d_0=d_0(q_0)\in\mathbb N$ such that the following holds for all
$d\ge d_0$. Assume $\mu\in[0,\tfrac14]$, and define
\[
    \tilde\mu
    :=
    \begin{cases}
        2\mu, & \text{if }\mu\sqrt d+d^{-1/2}\le c_2
        \text{ or }\mu^2 d\ge c_3,\\[4pt]
        2\sqrt{c_3/d}, & \text{otherwise.}
    \end{cases}
\]
Then $\mu\le\tilde\mu\le C_*\mu$, $\tilde\mu\le1/2$, and for
\[
    E_\mu
    =
    \left\{\left|\frac{1}{d}\sum_{r=1}^d \kappa_r\right|\ge\mu\right\}
\]
under the surrogate prior $\Pi_{\tilde\mu}$ one has
\[
    \Pi_{\tilde\mu}(E_\mu^c)\le q_0.
\]
More precisely:
\begin{enumerate}[label=\emph{(\alph*)}]
    \item if $\mu\sqrt d+d^{-1/2}\le c_2$, then
    \[
        \Pi_{\tilde\mu}(E_\mu^c)
        \le
        C\!\left(\mu\sqrt d+\frac{1}{\sqrt d}\right);
    \]
    \item if $\mu^2 d\ge c_3$, then
    \[
        \Pi_{\tilde\mu}(E_\mu^c)
        \le
        \exp\!\left(-\frac{\mu^2 d}{2}\right);
    \]
    \item in the intermediate branch
    $\mu\sqrt d+d^{-1/2}>c_2$ and $\mu^2 d<c_3$, one has
    \[
        \Pi_{\tilde\mu}(E_\mu^c)
        \le
        \exp\!\left(-\frac{c_3}{2}\right).
    \]
\end{enumerate}
\end{lemma}

\begin{proof}
Let $C_{\mathrm{bin}}$ be a universal constant such that, for every
$d\ge 1$ and every $p\in[1/2,3/4]$,
\[
    \sup_{0\le j\le d}\Pbb(\mathrm{Bin}(d,p)=j)\le \frac{C_{\mathrm{bin}}}{\sqrt d}.
\]
Choose
\[
    c_2\le \frac{q_0}{2C_{\mathrm{bin}}},
    \qquad
    c_3\ge 2\log(1/q_0),
    \qquad
    d_0\ge \max\left\{\frac{4}{c_2^2},16c_3\right\},
    \qquad
    C_*:=\max\!\left\{2,\frac{4\sqrt{c_3}}{c_2}\right\}.
\]
We verify that these choices work.

By construction, $\tilde\mu\ge\mu$ in all branches. In branches (a) and (b) we
have $\tilde\mu=2\mu$, so certainly $\tilde\mu\le C_*\mu$. In the intermediate
branch, since $d\ge d_0\ge 4/c_2^2$, we have $d^{-1/2}\le c_2/2$, and the
condition $\mu\sqrt d+d^{-1/2}>c_2$ therefore implies $\mu\sqrt d>c_2/2$.
Hence
\[
    \frac{\tilde\mu}{\mu}
    =
    \frac{2\sqrt{c_3/d}}{\mu}
    <
    \frac{4\sqrt{c_3}}{c_2}
    \le C_*.
\]
Thus $\mu\le\tilde\mu\le C_*\mu$. Moreover,
$d_0\ge16c_3$ and $\mu\le1/4$ imply $\tilde\mu\le1/2$ in every branch.

\medskip\noindent\textit{Branch (a): small $\mu\sqrt d$.}
Here $\tilde\mu=2\mu$, so
$\kappa_r\sim\mathrm{Rad}((1+2\mu)/2)$ and
$S_d:=\sum_{r=1}^d\kappa_r=2B_d-d$ with
$B_d\sim\mathrm{Bin}(d,(1+2\mu)/2)$. The event $E_\mu^c$ is
\[
    E_\mu^c=\{|S_d|<\mu d\}.
\]
The interval $(-\mu d,\mu d)$ contains at most $2\mu d+2$ integers of the
correct parity. Since $\mu\le 1/4$, the binomial parameter lies in
$[1/2,3/4]$, so by the definition of $C_{\mathrm{bin}}$,
\[
    \Pi_{\tilde\mu}(E_\mu^c)
    \le
    (2\mu d+2)\frac{C_{\mathrm{bin}}}{\sqrt d}
    \le
    2C_{\mathrm{bin}}\!\left(\mu\sqrt d+\frac{1}{\sqrt d}\right).
\]
Under the branch assumption
$\mu\sqrt d+d^{-1/2}\le c_2$, the right-hand side is at most
$2C_{\mathrm{bin}}c_2\le q_0$.

\medskip\noindent\textit{Branch (b): large $\mu^2 d$.}
Again $\tilde\mu=2\mu$, so
\[
    \Ebb\!\left[\frac{1}{d}\sum_{r=1}^d\kappa_r\right]=2\mu.
\]
Hence
\[
    E_\mu^c
    \subseteq
    \left\{\frac{1}{d}\sum_{r=1}^d \kappa_r-2\mu\le-\mu\right\}.
\]
Since each $\kappa_r\in[-1,1]$, Hoeffding's inequality gives
\[
    \Pi_{\tilde\mu}(E_\mu^c)
    \le
    \exp\!\left(-\frac{\mu^2 d}{2}\right).
\]
Since $\mu^2 d\ge c_3$ and $c_3\ge 2\log(1/q_0)$, this is at most $q_0$.

\medskip\noindent\textit{Branch (c): intermediate $\mu$.}
Here $\tilde\mu=2\sqrt{c_3/d}$, so
\[
    \Ebb\!\left[\frac{1}{d}\sum_{r=1}^d\kappa_r\right]
    =
    2\sqrt{\frac{c_3}{d}}.
\]
Because $\mu^2 d<c_3$, we have $\mu<\sqrt{c_3/d}$, hence
\[
    E_\mu^c
    \subseteq
    \left\{\frac{1}{d}\sum_{r=1}^d \kappa_r-2\sqrt{\frac{c_3}{d}}
    \le
    -\sqrt{\frac{c_3}{d}}\right\}.
\]
Hoeffding therefore yields
\[
    \Pi_{\tilde\mu}(E_\mu^c)
    \le
    \exp\!\left(-\frac{c_3}{2}\right).
\]
Since $c_3\ge 2\log(1/q_0)$, this is at most $q_0$.
\end{proof}

\begin{proof}[Proof of Theorem~\ref{thm:necessary_condition}]
Write
\[
    R_T:=\frac{\nT\DeltaT^4}{d},\qquad
    R_S:=\frac{\nS\DeltaS^4}{d},\qquad
    A:=\mu^2\DeltaT^2,\qquad
    P:=\frac{\mu^2\nS\DeltaT^2\DeltaS^2}{d}.
\]
We first record a fixed-constant obstruction for the target signal. Let
$C_{\mathrm T}$ denote the universal constant in Item~1 of
Theorem~\ref{thm:sec3_tv_three_routes}, and choose $c_0>0$ sufficiently small
that $C_{\mathrm T}c_0^4\le 1/4$. Consider the fully aligned subclass, which is
contained in $\Omega(\DeltaT,\DeltaS,\mu)$ for every $\mu\in[0,1]$. If
$\DeltaT\le c_0$, Item~1 gives
\[
    \TV(P_{+j},P_{-j})\le\frac12
\]
for every neighboring pair. The Assouad reduction in
\eqref{eq:sec3_assouad_display} therefore gives a fixed positive lower bound
on the minimax risk. Consequently, vanishing minimax risk requires
\begin{equation}
\label{eq:necessary-deltaT-fixed}
    \DeltaT>c_0
\end{equation}
eventually.

We now show that, for sufficiently small universal constants
$c_T,c_S,c_A,c_P>0$, vanishing minimax risk also requires
\begin{equation}
\label{eq:necessary-fixed-threshold}
    R_T>c_T
    \qquad\text{or}\qquad
    \bigl(R_S>c_S,\ A>c_A,\ P>c_P\bigr)
\end{equation}
eventually. Together with \eqref{eq:necessary-deltaT-fixed}, this will imply
the two alternatives stated in the theorem.

We next fix the constants used below. Set $q_0:=1/32$ in
Lemma~\ref{lem:sec3_conditioning_bad_event}, and let $C_*$ be the resulting
universal comparison constant. If every neighboring TV distance is at most
$1/2$, \eqref{eq:sec3_assouad_display} gives the surrogate Bayes lower bound
\[
    b_0:=\frac14\left(1-\frac12\right)=\frac18.
\]
For $\mu\le1/4$, Lemma~\ref{lem:sec3_conditioning} then transfers this to the
positive lower bound
\begin{equation}
\label{eq:conditioned-positive-bayes-bound}
    \frac{b_0-q_0}{1-q_0}=\frac3{31}
\end{equation}
over $\Omega(\DeltaT,\DeltaS,\mu)$. For $\mu>1/4$, we instead use the fully
aligned subclass, which is contained in that same parameter class.

Let $c_{\mathrm D}$ be the small constant in Item~3 of
Theorem~\ref{thm:sec3_tv_three_routes}, and denote the constants on the
right-hand sides of Items~2 and~4 by $C_{\mathrm R}$ (with
$\delta_0=1/2$) and $C_{\mathrm P}$, respectively. Let $c_{\mathrm{side}}>0$
denote the constant in the two side conditions of Item~4. Choose
$c_T,c_S,c_A,c_P>0$ sufficiently small that
\begin{equation}
\label{eq:necessary-choice-constants}
\begin{gathered}
    c_T+c_S\le c_{\mathrm D},\qquad
    C_{\mathrm R}(c_T+C_*^2c_A)\le\frac14,\qquad
    C_{\mathrm P}(c_T+C_*^2c_P)\le\frac14,\qquad
    C_{\mathrm P}(c_T+16c_P)\le\frac14,\\
    c_A<\frac{c_0^2}{16},\qquad
    \frac{c_T}{c_0^2}\le c_{\mathrm{side}}^2,\qquad
    \frac{C_*^2c_P}{c_0^2}\le c_{\mathrm{side}}^2,\qquad
    \frac{16c_P}{c_0^2}\le c_{\mathrm{side}}^2.
\end{gathered}
\end{equation}

Suppose, toward a contradiction, that the minimax risk tends to zero but
\eqref{eq:necessary-fixed-threshold} fails infinitely often. Along a
subsequence, $R_T\le c_T$ and at least one of
\[
    R_S\le c_S,\qquad A\le c_A,\qquad P\le c_P
\]
holds. Passing to a further subsequence, we may assume that the same one of
these three inequalities holds throughout.

\paragraph{Doubly subcritical case: $R_S\le c_S$.}
The first inequality in \eqref{eq:necessary-choice-constants} verifies the
small-constant hypothesis in Item~3. Hence, uniformly in the tested pair,
\[
    \TV(P_{+j},P_{-j})^2
    \le C\frac{\DeltaT^4}{d}
    =C\frac{R_T}{\nT}\to0.
\]
In particular, all neighboring TV distances are eventually at most $1/2$.

\paragraph{Revealed-source case: $A\le c_A$.}
If $\mu\le1/4$, the conditioning construction gives
$\tilde\mu\le C_*\mu$ and $\tilde\mu\le1/2$. Item~2 and
\eqref{eq:necessary-choice-constants} therefore yield
\[
    \TV(P_{+j},P_{-j})^2
    \le C_{\mathrm R}(\tilde\mu^2\DeltaT^2+R_T)
    \le C_{\mathrm R}(C_*^2c_A+c_T)\le\frac14.
\]
If $\mu>1/4$, this case cannot occur eventually, because
\eqref{eq:necessary-deltaT-fixed} gives
$A=\mu^2\DeltaT^2>c_0^2/16>c_A$.

\paragraph{Product-scale case: $P\le c_P$.}
First suppose $\mu\le1/4$. Then the surrogate product scale is at most
$C_*^2P$. Moreover, using \eqref{eq:necessary-deltaT-fixed} and
\eqref{eq:necessary-choice-constants},
\[
    \frac{\nT\DeltaT^2}{d}=\frac{R_T}{\DeltaT^2}
    \le\frac{c_T}{c_0^2}\le c_{\mathrm{side}}^2,
    \qquad
    \frac{\tilde\mu^2\nS\DeltaS^2}{d}
    =\frac{\tilde\mu^2\nS\DeltaT^2\DeltaS^2/d}{\DeltaT^2}
    \le\frac{C_*^2c_P}{c_0^2}\le c_{\mathrm{side}}^2.
\]
Thus both side conditions in Item~4 hold, and
\[
    \TV(P_{+j},P_{-j})^2
    \le C_{\mathrm P}(R_T+C_*^2P)\le\frac14.
\]
If $\mu>1/4$, use the fully aligned subclass. Its unweighted product scale
satisfies
\[
    \frac{\nS\DeltaT^2\DeltaS^2}{d}=\frac{P}{\mu^2}\le16P,
\]
and \eqref{eq:necessary-deltaT-fixed} and
\eqref{eq:necessary-choice-constants} give
\[
    \frac{\nT\DeltaT^2}{d}
    \le\frac{c_T}{c_0^2}\le c_{\mathrm{side}}^2,
    \qquad
    \frac{\nS\DeltaS^2}{d}
    \le\frac{16c_P}{c_0^2}\le c_{\mathrm{side}}^2.
\]
Item~4, now with perfect alignment, gives
\[
    \TV(P_{+j},P_{-j})^2
    \le C_{\mathrm P}(R_T+16P)\le\frac14.
\]

In every possible case the neighboring TV distances are at most $1/2$.
For $\mu\le1/4$, the conditioned Bayes risk is therefore at least $3/31$ by
\eqref{eq:conditioned-positive-bayes-bound}; for $\mu>1/4$, Assouad applied
directly to the fully aligned subclass gives a positive constant lower bound.
Both conclusions contradict the assumed vanishing minimax risk. Hence
\eqref{eq:necessary-fixed-threshold} holds eventually.

Finally, combine \eqref{eq:necessary-deltaT-fixed} with
\eqref{eq:necessary-fixed-threshold}. If $R_T>c_T$, then
\[
    \DeltaT>c_0,
    \qquad
    \DeltaT>c_T^{1/4}\left(\frac d{\nT}\right)^{1/4},
\]
which gives condition~\emph{(A)} after adjusting a universal constant. If the
second alternative in \eqref{eq:necessary-fixed-threshold} holds, then
\[
    \DeltaS>c_S^{1/4}\left(\frac d{\nS}\right)^{1/4},
    \qquad
    \mu\DeltaT>\sqrt{c_A},
    \qquad
    \mu\DeltaS\DeltaT>\sqrt{c_P}\sqrt{\frac d{\nS}},
\]
which is condition~\emph{(B)}. This proves
Theorem~\ref{thm:necessary_condition}.
\end{proof}

\section{Theoretical properties of Algorithm~\ref{alg:tclust}}
\label{sec:theo_t_clust}
In this section, we study the theoretical properties of Algorithm~\ref{alg:tclust} described in Section~\ref{sec:ts_clust_intro} for clustering observations from Gaussian mixture models with multiple communities. In that direction, assume that the data points \(\{\mathsf{X}_j:j\in[n]\}\subseteq\R^p\), satisfy
\begin{align}
    \mathsf{X}_j = \nu_{\mathsf{z}_j} + \epsilon_j, \quad \mbox{for $j=1,\ldots,n$,}
\end{align}
where $\{\mathsf{z}_1,\ldots,\mathsf{z}_n\} \in \{1,\ldots,K\}, \{\nu_1,\ldots,\nu_j\} \subseteq \R^p$ and $\epsilon_j \simiid \dnorm_p(0_p,\Id_p)$. Define
\begin{align}
\label{eq:global_delta}
\Delta:= \min_{a \neq b \in [K]} \|\nu_a-\nu_b\|_2.
\end{align}
Assume that $n/K \to \infty$ and the following balance condition holds.
\[
\min_{a \in [K]} \sum_{j=1}^n \mathbbm 1\{\mathsf{z}_j=a\} \ge \frac{\alpha n}{K},
\]
for some $\alpha<1$. 

Consider the following theorem characterizing the misclustering loss of the iterates $\widehat{\mathsf{z}}^{(t)} \in \{1,\ldots,K\}^n$ of Algorithm~\ref{alg:tclust}.

\begin{theorem}
\label{thm:tclust_iterates}
Suppose that $n/K\to\infty$ and
\begin{align}
\label{eq:k_means_condition}
    \frac{\Delta^2}{(\varepsilon+1)K^2(Kp/n+1)}\to\infty .
\end{align}
Let $\widehat{\mathsf z}^{(0)},\widehat{\mathsf z}^{(1)},\ldots,
\widehat{\mathsf z}^{(T)}$ be the sequence of labels produced by
Algorithm~\ref{alg:tclust}.
Then, for any $C'_1>0$, there exist constants $C'_2>0$, depending only on
$\alpha$ and $C'_1$, such that with probability at least
\[
    1-2\exp\{-C'_1(n+p)\}-\exp\{-\Delta\}-n^{-1},
\]
there exists a permutation $\pi_0\in\Pi_K$ such that the initializer satisfies
\[
    \sum_{j=1}^n\mathbbm 1\{\pi_0(\widehat{\mathsf z}_j^{(0)})\neq \mathsf z_j\}\le
    \frac{C'_2(\varepsilon+1)K(n+p)}
    {\Delta^2},
\]
and, for every $t=1,\ldots,T_0$,
\[
    \frac{1}{n}\sum_{j=1}^n\mathbbm 1\{\pi_0(\widehat{\mathsf z}_j^{(t)})\neq \mathsf z_j\}\le\exp\left\{-(1+o(1))\frac{\Delta^2}{8}\right\}
    +2^{-t}.
\]
Consequently, if $T_0\ge 3\log n$, then
\begin{align}
\label{eq:init_false}
    \min_{\pi\in\Pi_K}\frac{1}{n}\sum_{j=1}^n
    \mathbbm 1\{\pi(\widehat{\mathsf z}_j^{(T_0)})\neq \mathsf z_j\}\le\exp\left\{-(1+o(1))\frac{\Delta^2}{8}\right\}.
\end{align}
\end{theorem}

\begin{proof}
Using Proposition~4.1 of \citet{gao2022iterative}, we get that for any $C'_1>0$, there exists a constant
$C'_2>0$, depending only on $\alpha$ and $C'_1$, such that with probability
at least $1-\exp\{-C'_1(n+p)\}$, the estimates $\wh{\mathsf{z}}^{(0)}$ obtained in Step 5 of Algorithm~\ref{alg:tclust} applied with $\mathsf{X}_1,\ldots,\mathsf{X}_n$ satisfies
\[
\min_{\pi\in\Pi_K}\sum_{j=1}^n\left\|
\nu_{\pi(\widehat{\mathsf z}_j^{(0)})}
-\nu_{\mathsf z_j}\right\|_2^2\le C'_2(\varepsilon+1)K(n+p).
\]
By \eqref{eq:global_delta}, we get
\begin{align}
\label{eq:first_init}
\min_{\pi\in\Pi_K}\sum_{j=1}^n\mathbbm 1\{\pi(\widehat{\mathsf z}_j^{(0)})& \neq \mathsf z_j^*\}\le \frac{1}{\Delta^2}\sum_{j=1}^n\left\|
\nu_{\pi(\widehat{\mathsf z}_j^{(0)})}
-\nu_{\mathsf z_j}\right\|_2^2\\
& \le \frac{C'_2(\varepsilon+1)K(n+p)}{\Delta^2},
\end{align}
with probability
at least $1-\exp\{-C'_1(n+p)\}$.

Let $\pi_0$ be the minimizing permutation in \eqref{eq:first_init}. By \eqref{eq:k_means_condition}, we also have
\[
(\varepsilon+1)K(n+p)=o\left(\frac{n\Delta^2}{K}\right),
\]
and hence, using Corollary 4.1 of \cite{gao2022iterative}, we get
\[
\frac{1}{n}\sum_{j=1}^n\mathbbm 1\{\pi_0(\widehat{\mathsf z}_j^{(t)})\neq\mathsf z_j\}
\le\exp\left\{-(1+o(1))\frac{\Delta^2}{8}\right\}+2^{-t},
\]
with probability $1-\exp(-\Delta)-2\exp\{-C'_1(n+p)\}-n^{-1}$, where the probability is derived using union bound.
Finally, if $T_0\ge 3\log n$, then $2^{-T_0}=o(n^{-3})$. Since the
misclustering proportion takes values in \(\{0,1/n,\ldots,1\}\), the
algorithmic error term is negligible, and hence
\[
\frac{1}{n}\sum_{j=1}^n\mathbbm 1\{\pi_0(\widehat{\mathsf z}_j^{(T_0)})\neq\mathsf z_j\}
\le\exp\left\{-(1+o(1))\frac{\Delta_{\min}^2}{8}\right\}.
\]
The foregoing inequality automatically implies \eqref{eq:init_false}.
\end{proof}

\section{Proofs of results from Section~\ref{sec:mult_clust_mult_source}}
To prove Theorem~\ref{thm:consistency_mult_clust_final}, we shall use Theorem~\ref{thm:tclust_iterates} and the following lemmas.

\begin{lemma}
\label{lem:high_dim_mult_consistent_cluster_2}
Consider the estimators
$\hatThetaS{1},\ldots,\hatThetaS{m}$
constructed according to \eqref{eq:construction_mult_src_d_gg_m} or \eqref{eq:construction_mult_d_ll_m} depending on the aspect ratios $d/\nSm{i}$ for $i \in [m]$. Let
$\Pest^{(\src)}$ denote the orthogonal projector onto
\[
\mathcal R:=\col(\hatThetaS{1})+\cdots+\col(\hatThetaS{m}).
\]
Suppose that, for some $i_\star\in[m]$
\begin{align}
\label{eq:mult_cluster_source_signal_strength_high}
&\min_{a \neq c}\frac{|\ip{\thetamcT{a}-\thetamcT{c}}{\thetamSS{i_\star}{a}-\thetamSS{i_\star}{c}}|}
    {\|\thetamcT{a}-\thetamcT{c}\|_2\,\|\thetamSS{i_\star}{a}-\thetamSS{i_\star}{c}\|_2}
    \ge \mu, \quad \DeltaSm{i_\star}\gg \left(\frac{d(\beta K)}{\nSm{i_\star}}\right)^{1/4}\sqrt{(\beta K)\vee \log\nSm{i_\star}},\\
&\mu\cdot\DeltaT \gg K^3\sqrt{\log K} \quad \text{and}\quad
\mu\cdot\DeltaSm{i_\star}\cdot\DeltaT\gg K^3\sqrt{\log K}\cdot
\sqrt{\frac{K(d+\log K)}{\nSm{i_\star}}}.
\end{align}
Then
\begin{align}
\label{eq:projected_distance_diverge}
\frac{\displaystyle \min_{a\neq c}\|\Pest^{(\src)}(\thetamcT{a}-\thetamcT{c})\|_2^2}{K^2(\log K)\,\sigmaT^2}\convp \infty.
\end{align}
\end{lemma}

\begin{lemma}
\label{lem:validation_stat_candidate_partition}
Let $X_\calT=M_\calT+E_\calT\in\R^{\nT\times d}$, where
$M_\calT:=\Ztar\ThetaT^\top$ and the rows of $E_\calT$ are i.i.d.
$\dnorm_d(0,\sigmaT^2\Id_d)$. Fix a candidate partition
$\wt Z\in\mathcal Z_{\nT,K}$ and recall the definition $\wt n_a(\wt Z)$ and $\wt{\theta}^{(\trg)}_a(\wt Z)$ from \eqref{eq:candiate_cluster_size_cluster_mean}. Assume that $\wt n_a(\wt Z)>0$ for $a \in [K]$.
Define
\begin{align}
\label{eq:pi_z_t}
    \mathrm D(\wt Z):=\operatorname{diag}\{\wt n_1(\wt Z),\ldots,\wt n_K(\wt Z)\}, \quad
    \mathrm P(\wt Z):=\wt Z\mathrm D(\wt Z)^{-1}\wt Z^\top, \quad
    \Pi(\wt Z):=\mathrm P(\wt Z)-\frac{1}{\nT}\mathbbm 1_{\nT}\mathbbm 1_{\nT}^\top .
\end{align}
Then $\Pi(\wt Z)$ is an orthogonal projector of rank $K-1$, and
\begin{align}
\label{eq:candidate_partition_projection_identity}
    \sum_{1\le a\neq b\le K}
    \frac{\wt n_a(\wt Z)\wt n_b(\wt Z)}{2\nT^2}
    \|\wt{\theta}^{(\trg)}_a(\wt Z)-\wt{\theta}^{(\trg)}_b(\wt Z)\|_2^2
    =
    \frac{1}{\nT}\|\Pi(\wt Z)X_\calT\|_F^2.
\end{align}
Consequently, the validation statistic in \eqref{eq:validation_stat_mult_clust}
can be written as
\begin{align}
\label{eq:candidate_partition_validation_projection_form}
    \wh{\mathsf S}(\wt Z)
    =
    \frac{1}{\nT}\|\Pi(\wt Z)X_\calT\|_F^2
    -\frac{\sigmaT^2d(K-1)}{\nT}.
\end{align}
Moreover, for every $u>0$, there exists an absolute constant $C>0$ such that
\begin{align}
\label{eq:candidate_partition_validation_concentration}
    \mathbb P\left(
    \left|\wh{\mathsf S}(\wt Z)-\frac{1}{\nT}\|\Pi(\wt Z)M_\calT\|_F^2\right|
    \ge C\left[
    \frac{\sigmaT}{\nT}\|\Pi(\wt Z)M_\calT\|_F\sqrt{u}
    +\frac{\sigmaT^2}{\nT}\{\sqrt{d(K-1)u}+u\}
    \right]\right)\le 4e^{-u}.
\end{align}
\end{lemma}

\subsection{Proof of Theorem~\ref{thm:consistency_mult_clust_final}}
Suppose the oracle informs us that \eqref{eq:trg_mult_c_1} holds, that is,
\begin{align}
    \DeltaT \gg \sqrt K\cdot
    \max\left\{1,\left(\frac{d(\beta K)}{\nT}\right)^{1/4}\right\}.
\end{align}
In this case, we estimate $\Ztar$ using the relaxed \texttt{K-means} procedure
of \cite{giraud2019partial}. The above signal-strength condition, together
with Theorem~1 of \cite{giraud2019partial}, implies that
\begin{align}
    \mathcal L_{\mathrm{mult}}(\hatZtar_{\mathrm{orc}},\Ztar)\convp 0.
\end{align}
Since $\mathcal L_{\mathrm{mult}}(\hatZtar_{\mathrm{orc}},\Ztar)$ is bounded, it further
follows that
\begin{align}
    \mathbb E\left[\mathcal L_{\mathrm{mult}}(\hatZtar_{\mathrm{orc}},\Ztar)\right]\to 0.
\end{align}

Next, suppose that \eqref{eq:trg_mult_c_1} does not hold, and we
use the source-based estimator in Algorithm~\ref{alg:projected_clustering_mult}.
Let $\mathcal F_\src$ denote the sigma-field generated by the source
observations in $\calS_1,\ldots,\calS_m$. Then the source subspace
$\mathcal R$ and its orthogonal projector $\Pest^{(\src)}$ are
$\mathcal F_\src$-measurable. Let $\widehat{\mathrm U}_\src\in\R^{d\times r}$
be an orthonormal basis matrix for $\mathcal R$, where
$r=\operatorname{rank}(\mathcal R)$. Since the source and target datasets are
mutually independent, $\widehat{\mathrm U}_\src$ is independent of
$\{\Xtar_j:j\in[\nT]\}$. Therefore, conditionally on $\mathcal F_\src$, the
projected observations satisfy
\begin{align}
\label{eq:conditional_projected_distribution}
    \hatXtar_j
    \overset{\mathrm{indep}}{\sim}
    \dnorm_r\left((\widehat{\mathrm U}_\src)^\top\thetamcT{z_j},
    \sigmaT^2\,\Id_r\right),
    \quad \mbox{where $(\Ztar_j)_{z_j}=1$ for $j\in[\nT]$.}
\end{align}
Moreover, since $\mathcal R=\col(\hatThetaS{1})+\cdots+\col(\hatThetaS{m})$,
we have $r\le mK$. Define
\begin{align}
\label{eq:delta_min_b}
    \Delta_{\min}^2
    &:=
    \frac{\displaystyle \min_{a\neq b}
    \|(\widehat{\mathrm U}_\src)^\top(\thetamcT{a}-\thetamcT{b})\|_2^2}
    {\sigmaT^2}
    =
    \frac{\displaystyle \min_{a\neq b}
    \|\Pest^{(\src)}(\thetamcT{a}-\thetamcT{b})\|_2^2}
    {\sigmaT^2}.
\end{align}
Since the clustering module is invariant under a common positive rescaling of
all observations, applying it to $\{\hatXtar_j:j\in[\nT]\}$ is equivalent to
applying it to $\{\hatXtar_j/\sigmaT:j\in[\nT]\}$. By
\eqref{eq:conditional_projected_distribution}, the latter observations have
identity noise covariance and minimum squared separation $\Delta_{\min}^2$.

Since $m=O(1)$ and $r\le mK$, we have
\begin{align}
\label{eq:rank_term_vanish}
    \frac{Kr}{\nT}\le \frac{mK^2}{\nT}\to 0, \quad \mbox{since $K^2/\nT \to 0$ by assumption.}
\end{align}
Furthermore, by \eqref{eq:trg_mult_c_1} and
Lemma~\ref{lem:high_dim_mult_consistent_cluster_2}, for any
$\varepsilon=O(\log K)$,
\begin{align}
\label{eq:good_projected_separation_event}
    \frac{\Delta_{\min}^2}
    {(\varepsilon+1)K^2\left(K^2m/\nT+1\right)}
    \convp \infty.
\end{align}
Consequently, all assumptions of Theorem~\ref{thm:tclust_iterates} hold
conditionally on $\mathcal F_\src$ with probability tending to one. Hence,
running the Lloyd iterations in Algorithm~\ref{alg:tclust} for
$T_0=3\log\nT$ steps on the projected observations
$\{\hatXtar_1,\ldots,\hatXtar_{\nT}\}$ yields
\begin{align}
\label{eq:foldwise_misclustering_bound}
    \mathcal L_{\mathrm{mult}}(\hatZtar_{\mathrm{orc}},\Ztar)
    \le
    \exp\left\{-(1+o(1))\frac{\Delta_{\min}^2}{8}\right\}
    \convp 0
\end{align}
with conditional probability tending to one. Therefore, the same conclusion
holds unconditionally. Finally, since $\mathcal L_{\mathrm{mult}}$ is bounded,
dominated convergence gives
\begin{align}
    \mathbb E\left[\mathcal L_{\mathrm{mult}}(\hatZtar_{\mathrm{orc}},\Ztar)\right]\to 0.
\end{align}
\subsection{Proof of Theorem~\ref{thm:consistency_mult_clust_final_adaptive}}
In this section, we focus on the proof of the adaptive selection of the target-based and source-based estimators in \eqref{eq:mult_clust_adaptive}. Recall the definition of $\Pi(\wt Z)$ for any candidate partition in $\mathcal Z_{\nT,K}$ from \eqref{eq:pi_z_t} such that $\wt n_a(\wt Z)>0$ for all $a \in [K]$. Define
\begin{align}
\label{eq:target_oracle_energy_mult_re_write}
    \mathfrak D_{\trg,K}^2:=\frac{1}{\nT}\|\Pi(\Ztar)M_\calT\|_F^2,
    \quad \mbox{where} \quad M_\calT:=\Ztar\ThetaT^\top.
\end{align}
and the validation noise floor
\begin{align}
\label{eq:validation_noise_floor_mult}
    \mathfrak F_{\nT,d,K}:=\sigmaT^2\, K\,\left(\sqrt{\frac{dK}{\nT}}+1\right).
\end{align}
We shall use the following consequence of
Lemma~\ref{lem:validation_stat_candidate_partition}. For every fixed $\eta\in(0,1)$ there exists $C_\eta>0$ such that with probability tending to one,
\begin{align}
\label{eq:uniform_validation_bound_mult}
\left|\wh{\mathsf S}(\wt Z)-\frac{1}{\nT}\|\Pi(\wt Z)M_\calT\|_F^2\right|
\le\eta\cdot\frac{1}{\nT}\|\Pi(\wt Z)M_\calT\|_F^2
+C_\eta\mathfrak F_{\nT,d,K},\quad\mbox{uniformly over $\wt Z\in\mathcal Z_{\nT,K}$,}
\end{align}
where $\wt n_a(\wt Z)>0$ for all $a \in [K]$.
To see this, set $u_{\nT}:=\nT\log K+\log\nT.$
Since $|\mathcal Z_{\nT,K}|\le K^{\nT}$, applying
\eqref{eq:candidate_partition_validation_concentration} to each
$\wt Z\in\mathcal Z_{\nT,K}$ and taking a union bound shows that, with probability at least $1-\frac{4}{\nT}$, we have, simultaneously for every $\wt Z\in\mathcal Z_{\nT,K}$,
\begin{align}
\label{eq:uniform_validation_bound_mult_intermediate}
\left|\wh{\mathsf S}(\wt Z)-\frac{1}{\nT}\|\Pi(\wt Z)M_\calT\|_F^2\right|
\le C\left[\frac{\sigmaT}{\nT}\|\Pi(\wt Z)M_\calT\|_F\sqrt{u_{\nT}}+\frac{\sigmaT^2}{\nT}\left\{\sqrt{d(K-1)u_{\nT}}+u_{\nT}\right\}\right].
\end{align}
For the first term on the right-hand side, the inequality $2ab\le\eta a^2+\eta^{-1}b^2$ gives
\begin{align}
\frac{C\sigmaT}{\nT}\|\Pi(\wt Z)M_\calT\|_F\sqrt{u_{\nT}}\le\eta\cdot\frac{1}{\nT}\|\Pi(\wt Z)M_\calT\|_F^2+C_\eta\sigmaT^2\frac{u_{\nT}}{\nT}.
\end{align}
Moreover, $u_{\nT}/\nT=\log K+(\log\nT/\nT) \lesssim \log K$, where we used $K\ge2$ and
$\sup_{x\ge1}(\log x)/x=1/e<\log2$. Consequently,
\begin{align}
&\frac{\sigmaT^2}{\nT}\sqrt{d(K-1)u_{\nT}}
=\sigmaT^2\sqrt{\frac{d(K-1)}{\nT}\left(\log K+\frac{\log\nT}{\nT}\right)}\lesssim\sigmaT^2\sqrt{\frac{d(K-1)\log K}{\nT}}, \\
&\mbox{and} \quad \sigmaT^2\frac{u_{\nT}}{\nT}\lesssim \sigmaT^2\log K.
\end{align}
Substituting these bounds into
\eqref{eq:uniform_validation_bound_mult_intermediate} and recalling the
definition of $\mathfrak F_{\nT,d,K}$ yields
\eqref{eq:uniform_validation_bound_mult}.

Observe that $M_\calT=\Ztar\ThetaT^\top$ is piecewise constant with respect to the true partition $\Ztar$, we have $\mathrm P(\Ztar)M_\calT=M_\calT$. Hence
\begin{align}
\Pi(\Ztar)M_\calT=M_\calT-\frac{1}{\nT}\mathbbm 1_{\nT}\mathbbm 1_{\nT}^\top M_\calT.
\end{align}
Moreover, since $\Pi(\wt Z)\mathbbm 1_{\nT}=0$,
\begin{align}
\Pi(\wt Z)M_\calT=\Pi(\wt Z)\Pi(\Ztar)M_\calT.
\end{align}
Since $\Pi(\wt Z)$ is an orthogonal projector, it is a contraction in Frobenius
norm. Therefore,
\begin{align}
\|\Pi(\wt Z)M_\calT\|_F=\|\Pi(\wt Z)\Pi(\Ztar)M_\calT\|_F\le\|\Pi(\Ztar)M_\calT\|_F,
\end{align}
which implies
\begin{align}
\label{eq:oracle_maximality_validation_mult}
    \frac{1}{\nT}\|\Pi(\wt Z)M_\calT\|_F^2\le \mathfrak D_{\trg,K}^2,
    \quad \mbox{for all $\wt Z\in\mathcal Z_{\nT,K}$.}
\end{align}

\paragraph{Strong target signal.}
Suppose \eqref{eq:trg_mult_c_1} holds. Then observe that, under \eqref{eq:approximate_balance}, using Theorem~1 of \cite{giraud2019partial}, we can conclude 
\[
\limsup_{\nT \to \infty}\mathbb P\left[\mbox{there exists}\; a\in [K]\;\mbox{such that}\;\wt n_a(\hatZtar_T)=0\right] = 0.
\]
Therefore, henceforth we shall assume that $\wt n_a(\hatZtar_T) \neq 0$ for all $a \in [K]$.
Let
$\mathrm C_{\nT}:=\Id_{\nT}-\nT^{-1}\mathbbm 1_{\nT}\mathbbm 1_{\nT}^\top$
and write $Q_\calT:=\mathrm C_{\nT}M_\calT=\Pi(\Ztar)M_\calT$. After applying
the optimal permutation in the definition of $\mathcal L_{\mathrm{mult}}$, let
$z_j$ and $\hat z_j$ denote the true and estimated labels, and define
$\mathcal M:=\{j:\hat z_j\neq z_j\}$. Let $\widetilde M_\calT$ be the matrix
whose $j$-th row is $\thetamcT{\hat z_j}^{\top}$ and set
$\widetilde Q_\calT:=\mathrm C_{\nT}\widetilde M_\calT$. Since
$\widetilde Q_\calT$ is centered and piecewise constant with respect to
$\hatZtar_\trg$, it belongs to the range of $\Pi(\hatZtar_\trg)$. Also,
$\Pi(\hatZtar_\trg)M_\calT=\Pi(\hatZtar_\trg)Q_\calT$, since
$\Pi(\hatZtar_\trg)\mathbbm 1_{\nT}=0$. Therefore, by the Pythagorean identity
and the optimality of orthogonal projection onto the range of
$\Pi(\hatZtar_\trg)$,
\begin{align}
\label{eq:target_energy_loss_projection}
    0\le \mathfrak D_{\trg,K}^2-\frac{1}{\nT}\|\Pi(\hatZtar_\trg)M_\calT\|_F^2
    &=\frac{1}{\nT}\|(\Id_{\nT}-\Pi(\hatZtar_\trg))Q_\calT\|_F^2
    \le \frac{1}{\nT}\|Q_\calT-\widetilde Q_\calT\|_F^2.
\end{align}
Furthermore, since $\mathrm C_{\nT}$ is an orthogonal projector,
$\|Q_\calT-\widetilde Q_\calT\|_F
\le \|M_\calT-\widetilde M_\calT\|_F$. The matrices $M_\calT$ and
$\widetilde M_\calT$ differ only on the misclustered indices in $\mathcal M$.
Hence,
\begin{align}
\label{eq:target_energy_loss_misclassification}
    \frac{1}{\nT}\|Q_\calT-\widetilde Q_\calT\|_F^2
    \le \frac{1}{\nT}\|M_\calT-\widetilde M_\calT\|_F^2
    \le \frac{|\mathcal M|}{\nT}\max_{a\neq b}\|\thetamcT{a}-\thetamcT{b}\|_2^2.
\end{align}
Combining \eqref{eq:target_energy_loss_projection} and
\eqref{eq:target_energy_loss_misclassification}, we obtain
\begin{align}
    0\le \mathfrak D_{\trg,K}^2-\frac{1}{\nT}\|\Pi(\hatZtar_\trg)M_\calT\|_F^2
    \le \frac{|\mathcal M|}{\nT}\max_{a\neq b}\|\thetamcT{a}-\thetamcT{b}\|_2^2.
\end{align}
By \eqref{eq:approximate_balance}, \eqref{eq:candidate_partition_projection_identity} and \eqref{eq:target_oracle_energy_mult_re_write}, we have
\begin{align}
    \max_{a\neq b}\|\thetamcT{a}-\thetamcT{b}\|_2^2
    \le 4\beta K\,\mathfrak D_{\trg,K}^2.
\end{align}
Since $|\mathcal M|/\nT=\mathcal L_{\mathrm{mult}}(\hatZtar_\trg,\Ztar)$ after
the optimal relabeling, it follows that
\begin{align}
    0\le 1-\frac{\|\Pi(\hatZtar_\trg)M_\calT\|_F^2}{\|\Pi(\Ztar)M_\calT\|_F^2}
    \le 4\beta K\,\mathcal L_{\mathrm{mult}}(\hatZtar_\trg,\Ztar).
\end{align}
Together with \eqref{eq:target_k_loss_vanish} and $\beta=O(1)$, this proves
\begin{align}
\label{eq:target_energy_consistency_mult}
\frac{1}{\nT}\|\Pi(\hatZtar_\trg)M_\calT\|_F^2=\mathfrak D_{\trg,K}^2\{1+o_p(1)\}.
\end{align}

Observe that \eqref{eq:trg_mult_c_1} implies
\begin{align}
    \mathfrak F_{\nT,d,K}= \sigmaT^2\,K\cdot\left(\sqrt{\frac{dK}{\nT}}+1\right) \le 2\,\sigmaT^2\,K\cdot\max\left\{\sqrt{\frac{dK}{\nT}},1\right\} \ll \sigmaT^2\,\DeltaT^2,
\end{align}
since $\beta\ge 1$ is a constant. Moreover, using the weighted pairwise representation of
$\mathfrak D_{\trg,K}^2$ and the balance condition \eqref{eq:approximate_balance},
\begin{align}
    \mathfrak D_{\trg,K}^2
    =
    \sum_{1\le a\neq b\le K}\frac{n_a n_b}{2\nT^2}
    \|\thetamcT{a}-\thetamcT{b}\|_2^2
    \gtrsim \sigmaT^2\DeltaT^2,
\end{align}
where $n_a:=\sum_{j=1}^{\nT}\mathbbm 1\{\Ztar_j=e_a\}$. Hence
\eqref{eq:trg_mult_c_1} and \eqref{eq:validation_noise_floor_mult} implies
\begin{align}
\label{eq:target_energy_dominates_noise_floor}
    \frac{\mathfrak D_{\trg,K}^2}{\mathfrak F_{\nT,d,K}}\to\infty.
\end{align}
Combining \eqref{eq:uniform_validation_bound_mult},
\eqref{eq:target_energy_consistency_mult}, and
\eqref{eq:target_energy_dominates_noise_floor}, we obtain
\begin{align}
    |\wh{\mathsf S}(\hatZtar_\trg)|
    \ge (1-\eta)\mathfrak D_{\trg,K}^2\{1+o_p(1)\}-C_\eta\mathfrak F_{\nT,d,K}.
\end{align}
Since $\mathfrak D_{\trg,K}^2/\mathfrak F_{\nT,d,K}\to\infty$, the right-hand
side exceeds $\mathfrak t_n=D_0\mathfrak F_{\nT,d,K}$ with probability tending
to one. Therefore,
\begin{align}
\label{eq:adaptive_select_target_high_signal}
    \mathbb P\left(\hatZtar=\hatZtar_\trg\right)\to1.
\end{align}
By Theorem~1 of \citet{giraud2019partial}, with
$s_{\trg}^2:=\DeltaT^2\wedge m_{\trg}\DeltaT^4/d$ and
$m_{\trg}:=\min_{a\in[K]}\sum_{j=1}^{\nT}\mathbbm 1\{\Ztar_j=e_a\}$, the
relaxed \texttt{K-means} estimator satisfies
$\mathcal L_{\mathrm{mult}}(\hatZtar_\trg,\Ztar)\le \exp\{-c s_{\trg}^2\}$
with probability tending to one whenever $s_{\trg}^2\gtrsim \nT/m_{\trg}$.
By \eqref{eq:approximate_balance} and the target signal condition \eqref{eq:trg_mult_c_1}, we have
$m_{\trg}\ge \nT/(\beta K)$ and
$s_{\trg}^2\ge \DeltaT^2\wedge \nT\DeltaT^4/(\beta Kd)\gg \beta K$; hence
\begin{align}
\label{eq:target_k_loss_vanish}
    K\,\mathcal L_{\mathrm{mult}}(\hatZtar_\trg,\Ztar)\convp0.
\end{align}
Consequently, using \eqref{eq:adaptive_select_target_high_signal} we get 
\begin{align}
    \mathbb E\left[\mathcal L_{\mathrm{mult}}(\hatZtar,\Ztar)\right]
    \le \mathbb P(\hatZtar\neq\hatZtar_\trg)
    +\mathbb E\left[\mathcal L_{\mathrm{mult}}(\hatZtar_\trg,\Ztar)\right]\to0.
\end{align}

\paragraph{Weak target signal.}
Next, consider the source-transfer regime when \eqref{eq:trg_mult_c_1} does not hold. In other words, there exists $\mathrm F_0>0$ such that
\begin{align}
    \label{eq:negation_of_forty}
   \DeltaT \le  \mathrm F_0\sqrt{K}\;\max\left\{1,
    \left(\frac{d \cdot (\beta K)}{\nT}\right)^{1/4} \right\}. 
\end{align}
Observe that by \eqref{eq:approximate_balance} and \eqref{eq:non_diverging_energy_inter_cluster}, we have another constant $\mathrm F_1$
\begin{align}
\label{eq:low_target_energy_condition_mult}
    \mathfrak D_{\trg,K}^2\le \max_{a \neq b}\|\thetamcT{a}-\thetamcT{b}\|^2_2 \le \kappa\,\sigmaT^2\, \DeltaT^2 \le \kappa\,\mathrm{F}_1\,\mathfrak F_{\nT,d,K}.
\end{align}
On the event \eqref{eq:uniform_validation_bound_mult}, using
\eqref{eq:oracle_maximality_validation_mult} and
\eqref{eq:low_target_energy_condition_mult}, we have
\begin{align}
    |\wh{\mathsf S}(\hatZtar_\trg)|
    \le (1+\eta)\frac{1}{\nT}\|\Pi(\hatZtar_\trg)M_\calT\|_F^2
    +C_\eta\mathfrak F_{\nT,d,K}
    \le \{2(1+\eta)\kappa\,\mathrm{F}_1+C_\eta\}\mathfrak F_{\nT,d,K},
\end{align}
with probability tending to one.
Choosing $D_0>4(1+\eta)\kappa\,\mathrm{F}^2_0+2\,C_\eta$ in
\eqref{eq:validation_threshold__mult_clust} gives
\begin{align}
\label{eq:adaptive_select_source_low_signal}
    \mathbb P\left(\hatZtar=\hatZtar_\src\right)\to1.
\end{align}
By the source part of the proof in \eqref{eq:foldwise_misclustering_bound}, we have
\begin{align}
\label{eq:source_estimator_consistency_mult}
    \mathbb E\left[\mathcal L_{\mathrm{mult}}(\hatZtar_\src,\Ztar)\right]\to0.
\end{align}
Therefore,
\begin{align}
    \mathbb E\left[\mathcal L_{\mathrm{mult}}(\hatZtar,\Ztar)\right]
    \le \mathbb P(\hatZtar\neq\hatZtar_\src)
    +\mathbb E\left[\mathcal L_{\mathrm{mult}}(\hatZtar_\src,\Ztar)\right]\to0.
\end{align}
This completes the proof of the adaptive selection step.

\subsection{Proof of Lemma~\ref{lem:high_dim_mult_consistent_cluster_2} when $d\gg \nSm{i_\star}$}

Throughout this subsection, we assume that the conditions in \eqref{eq:mult_cluster_source_signal_strength_high} hold for some $i_\star\in[m]$ and for the same $i_\star$, we also have $d\gg \nSm{i_\star}$. For notational convenience, define
\begin{align}
\label{eq:source_high_dim_aux_quantities}
    L_K:=K\sqrt{\log K}, \quad q_{i_\star}:=\sqrt{\frac{K}{\nSm{i_\star}}}, \quad r_{i_\star}:=\sqrt{\frac{K(d+\log K)}{\nSm{i_\star}}}.
\end{align}
The assumptions in \eqref{eq:mult_cluster_source_signal_strength_high} imply
\begin{align}
\label{eq:source_high_dim_uniform_pair_condition}
    K^2\left\{\frac{L_K}{\mu\cdot\DeltaT}+\frac{L_K\cdot r_{i_\star}}{\mu\cdot\DeltaSm{i_\star}\cdot\DeltaT}\right\}\to 0.
\end{align}

Since $d\gg \nSm{i_\star}$ and the source signal strength condition in \eqref{eq:mult_cluster_source_signal_strength_high} holds, Theorem~2 of \citet{giraud2019partial} implies that the relaxed \texttt{K-means} estimator $\wh Z^{(\src_{i_\star})}$ achieves exact recovery with probability tending to one. More precisely, define
\begin{align}
\label{eq:source_high_dim_exact_recovery_event}
    \mathcal E_{\src_{i_\star}}:=\left\{\mbox{there exists $\pi^{(\src_{i_\star})}\in\mathcal P_K$ such that $\wh Z^{(\src_{i_\star})}_j=\pi^{(\src_{i_\star})}\circ\Zsrcm{i_\star}_j$ for all $j=1,\ldots,\nSm{i_\star}$}\right\}.
\end{align}
Then $\mathbb P(\mathcal E_{\src_{i_\star}})\to 1$. On $\mathcal E_{\src_{i_\star}}$, after applying the corresponding permutation to the columns of $\wh Z^{(\src_{i_\star})}$, we may regard $\wh Z^{(\src_{i_\star})}$ as equal to $\Zsrcm{i_\star}$.

For $a\in[K]$, define
\begin{align}
\label{eq:source_high_dim_cluster_sizes}
    m_a^{(\src_{i_\star})}:=\sum_{j=1}^{\nSm{i_\star}}\mathbbm 1\{\Zsrcm{i_\star}_{ja}=1\}.
\end{align}
By \eqref{eq:approximate_balance_src},
\begin{align}
\label{eq:source_high_dim_balance}
    \frac{\nSm{i_\star}}{\beta K}\le m_a^{(\src_{i_\star})}\le \frac{\beta\nSm{i_\star}}{K}, \quad \mbox{for all $a\in[K]$.}
\end{align}
For every $a\neq c$, define
\begin{align}
\label{eq:source_high_dim_contrasts}
    \delta_{ac}^{(\trg)}:=\thetamcT{a}-\thetamcT{c}, \quad \delta_{ac}^{(\src_{i_\star})}:=\thetamSS{i_\star}{a}-\thetamSS{i_\star}{c}, \quad \widehat\delta_{ac}^{(\src_{i_\star})}:=\wh{\theta}_{\src_{i_\star},a}-\wh{\theta}_{\src_{i_\star},c}.
\end{align}
Also define the oracle source noise contrast
\begin{align}
\label{eq:source_high_dim_noise_contrast}
    \xi_{ac}^{(\src_{i_\star})}:=\frac{1}{m_a^{(\src_{i_\star})}}\sum_{j=1}^{\nSm{i_\star}}\Zsrcm{i_\star}_{ja}\epsm{i_\star}_j-\frac{1}{m_c^{(\src_{i_\star})}}\sum_{j=1}^{\nSm{i_\star}}\Zsrcm{i_\star}_{jc}\epsm{i_\star}_j.
\end{align}
Then
\begin{align}
\label{eq:source_high_dim_noise_contrast_dist}
    \xi_{ac}^{(\src_{i_\star})}\sim\dnorm_d\left(0,\sigmaSm{i_\star}^2\left\{\frac{1}{m_a^{(\src_{i_\star})}}+\frac{1}{m_c^{(\src_{i_\star})}}\right\}\Id_d\right).
\end{align}
On $\mathcal E_{\src_{i_\star}}$, we have
\begin{align}
\label{eq:source_high_dim_empirical_contrast_decomp}
    \widehat\delta_{ac}^{(\src_{i_\star})}=\delta_{ac}^{(\src_{i_\star})}+\xi_{ac}^{(\src_{i_\star})}.
\end{align}

Since $\widehat\delta_{ac}^{(\src_{i_\star})}\in\operatorname{range}(\Pest^{(\src)})$, we have
\begin{align}
    \left\langle\widehat\delta_{ac}^{(\src_{i_\star})},\delta_{ac}^{(\trg)}\right\rangle=\left\langle\widehat\delta_{ac}^{(\src_{i_\star})},\Pest^{(\src)}\delta_{ac}^{(\trg)}\right\rangle.
\end{align}
Consequently,
\begin{align}
\label{eq:source_high_dim_projection_lower_bound_raw}
    \|\Pest^{(\src)}\delta_{ac}^{(\trg)}\|_2\ge \frac{\left|\left\langle\widehat\delta_{ac}^{(\src_{i_\star})},\delta_{ac}^{(\trg)}\right\rangle\right|}{\|\widehat\delta_{ac}^{(\src_{i_\star})}\|_2}.
\end{align}
Define
\begin{align}
\label{eq:source_high_dim_scalar_defs}
    \nu_{ac}^{(\src_{i_\star})}:=\left\langle\delta_{ac}^{(\src_{i_\star})},\delta_{ac}^{(\trg)}\right\rangle, \quad G_{ac}^{(\src_{i_\star})}:=\left\langle\xi_{ac}^{(\src_{i_\star})},\delta_{ac}^{(\trg)}\right\rangle.
\end{align}
Then \eqref{eq:source_high_dim_empirical_contrast_decomp} gives
\begin{align}
\label{eq:source_high_dim_inner_product_decomp}
    \left\langle\widehat\delta_{ac}^{(\src_{i_\star})},\delta_{ac}^{(\trg)}\right\rangle=\nu_{ac}^{(\src_{i_\star})}+G_{ac}^{(\src_{i_\star})}.
\end{align}
By the alignment condition in \eqref{eq:mult_cluster_source_signal_strength_high},
\begin{align}
\label{eq:source_high_dim_pairwise_alignment}
    |\nu_{ac}^{(\src_{i_\star})}|\ge \mu\|\delta_{ac}^{(\src_{i_\star})}\|_2\|\delta_{ac}^{(\trg)}\|_2.
\end{align}
Moreover,
\begin{align}
\label{eq:source_high_dim_scalar_noise_dist}
    G_{ac}^{(\src_{i_\star})}\sim\dnorm\left(0,\sigmaSm{i_\star}^2\left\{\frac{1}{m_a^{(\src_{i_\star})}}+\frac{1}{m_c^{(\src_{i_\star})}}\right\}\|\delta_{ac}^{(\trg)}\|_2^2\right).
\end{align}
Let $s_{ac}^{(\src_{i_\star})}:=\sqrt{\operatorname{Var}(G_{ac}^{(\src_{i_\star})})}$. By \eqref{eq:source_high_dim_balance}, there exist constants $c_{\src,i_\star},C_{\src,i_\star}>0$ such that, uniformly over $a\neq c$,
\begin{align}
\label{eq:source_high_dim_scalar_noise_sd}
    c_{\src,i_\star}\sigmaSm{i_\star}q_{i_\star}\|\delta_{ac}^{(\trg)}\|_2\le s_{ac}^{(\src_{i_\star})}\le C_{\src,i_\star}\sigmaSm{i_\star}q_{i_\star}\|\delta_{ac}^{(\trg)}\|_2.
\end{align}

Using Lemma~1 of \citet{laurent2000adaptive} and a union bound over all pairs $a\neq c$, there exists a constant $C'_{\src,1}>0$ such that the event
\begin{align}
\label{eq:source_high_dim_noise_norm_event}
    \mathcal N_{\src_{i_\star}}:=\left\{\max_{a\neq c}\|\xi_{ac}^{(\src_{i_\star})}\|_2\le C'_{\src,1}\sigmaSm{i_\star}\sqrt{\frac{K(d+\log K)}{\nSm{i_\star}}}\right\}
\end{align}
satisfies $\mathbb P(\mathcal N_{\src_{i_\star}}^c)=o(1)$. Hence, on $\mathcal E_{\src_{i_\star}}\cap\mathcal N_{\src_{i_\star}}$,
\begin{align}
\label{eq:source_high_dim_empirical_contrast_norm_bound}
    \|\widehat\delta_{ac}^{(\src_{i_\star})}\|_2\le \|\delta_{ac}^{(\src_{i_\star})}\|_2+C'_{\src,1}\sigmaSm{i_\star}r_{i_\star}
\end{align}
simultaneously for all $a\neq c$. Combining \eqref{eq:source_high_dim_projection_lower_bound_raw}, \eqref{eq:source_high_dim_inner_product_decomp}, and \eqref{eq:source_high_dim_empirical_contrast_norm_bound}, we obtain, on $\mathcal E_{\src_{i_\star}}\cap\mathcal N_{\src_{i_\star}}$,
\begin{align}
\label{eq:source_high_dim_projection_lower_bound}
    \|\Pest^{(\src)}\delta_{ac}^{(\trg)}\|_2\ge \frac{|\nu_{ac}^{(\src_{i_\star})}+G_{ac}^{(\src_{i_\star})}|}{\|\delta_{ac}^{(\src_{i_\star})}\|_2+C'_{\src,1}\sigmaSm{i_\star}r_{i_\star}}.
\end{align}

Fix $\mathrm M>0$. On $\mathcal E_{\src_{i_\star}}\cap\mathcal N_{\src_{i_\star}}$, the event
\begin{align}
\label{eq:source_high_dim_small_projection_event}
    \frac{\|\Pest^{(\src)}\delta_{ac}^{(\trg)}\|_2}{L_K\sigmaT}\le \mathrm M
\end{align}
implies
\begin{align}
\label{eq:source_high_dim_small_ball_radius}
    |\nu_{ac}^{(\src_{i_\star})}+G_{ac}^{(\src_{i_\star})}|\le D_{ac}^{(\src_{i_\star})}:=\mathrm M\sigmaT L_K\left\{\|\delta_{ac}^{(\src_{i_\star})}\|_2+C'_{\src,1}\sigmaSm{i_\star}r_{i_\star}\right\}.
\end{align}
We now bound this small-ball probability uniformly over $a\neq c$.

\medskip
\noindent
\emph{Case 1.} Suppose that $\mu\|\delta_{ac}^{(\src_{i_\star})}\|_2\ge C_{\src,i_\star}\sigmaSm{i_\star}q_{i_\star}$. Then \eqref{eq:source_high_dim_pairwise_alignment} and \eqref{eq:source_high_dim_scalar_noise_sd} imply $|\nu_{ac}^{(\src_{i_\star})}|\ge s_{ac}^{(\src_{i_\star})}$. Furthermore, for some constant $C_{\mathrm M}>0$,
\begin{align}
\label{eq:source_high_dim_mean_dominant_radius_ratio}
    \frac{D_{ac}^{(\src_{i_\star})}}{|\nu_{ac}^{(\src_{i_\star})}|}\le C_{\mathrm M}\left\{\frac{L_K}{\mu\DeltaT}+\frac{L_Kr_{i_\star}}{\mu\DeltaSm{i_\star}\DeltaT}\right\}=o(1).
\end{align}
Indeed, the first term follows from $|\nu_{ac}^{(\src_{i_\star})}|\ge \mu\sigmaT\DeltaT\|\delta_{ac}^{(\src_{i_\star})}\|_2$, whereas the second follows from $|\nu_{ac}^{(\src_{i_\star})}|\ge \mu\sigmaSm{i_\star}\sigmaT\DeltaSm{i_\star}\DeltaT$. Therefore, for all sufficiently large $\nT$ and $\nSm{i_\star}$, $D_{ac}^{(\src_{i_\star})}\le |\nu_{ac}^{(\src_{i_\star})}|/2$. By Lemma~\ref{lem:gaussian_anti_conc}(2),
\begin{align}
\label{eq:source_high_dim_mean_dominant_small_ball}
    \mathbb P\left(|\nu_{ac}^{(\src_{i_\star})}+G_{ac}^{(\src_{i_\star})}|\le D_{ac}^{(\src_{i_\star})}\right)\lesssim \frac{D_{ac}^{(\src_{i_\star})}}{|\nu_{ac}^{(\src_{i_\star})}|}\le C_{\mathrm M}\left\{\frac{L_K}{\mu\DeltaT}+\frac{L_Kr_{i_\star}}{\mu\DeltaSm{i_\star}\DeltaT}\right\}.
\end{align}

\medskip
\noindent
\emph{Case 2.} Suppose that $\mu\|\delta_{ac}^{(\src_{i_\star})}\|_2<C_{\src,i_\star}\sigmaSm{i_\star}q_{i_\star}$. By Lemma~\ref{lem:gaussian_anti_conc}(1) and \eqref{eq:source_high_dim_scalar_noise_sd},
\begin{align}
\label{eq:source_high_dim_noise_dominant_initial}
    \mathbb P\left(|\nu_{ac}^{(\src_{i_\star})}+G_{ac}^{(\src_{i_\star})}|\le D_{ac}^{(\src_{i_\star})}\right)\lesssim \frac{D_{ac}^{(\src_{i_\star})}}{s_{ac}^{(\src_{i_\star})}}\lesssim \frac{D_{ac}^{(\src_{i_\star})}}{\sigmaSm{i_\star}\sigmaT q_{i_\star}\DeltaT}.
\end{align}
If $\|\delta_{ac}^{(\src_{i_\star})}\|_2\ge C'_{\src,1}\sigmaSm{i_\star}r_{i_\star}$, then
\begin{align}
    D_{ac}^{(\src_{i_\star})}\le 2\mathrm M\sigmaT L_K\|\delta_{ac}^{(\src_{i_\star})}\|_2,
\end{align}
and the defining inequality of Case~2 gives
\begin{align}
\label{eq:source_high_dim_noise_dominant_large_contrast}
    \frac{D_{ac}^{(\src_{i_\star})}}{\sigmaSm{i_\star}\sigmaT q_{i_\star}\DeltaT}\le C_{\mathrm M}\frac{L_K\|\delta_{ac}^{(\src_{i_\star})}\|_2}{\sigmaSm{i_\star}q_{i_\star}\DeltaT}\le C_{\mathrm M}\frac{L_K}{\mu\DeltaT}.
\end{align}
On the other hand, if $\|\delta_{ac}^{(\src_{i_\star})}\|_2<C'_{\src,1}\sigmaSm{i_\star}r_{i_\star}$, then
\begin{align}
    D_{ac}^{(\src_{i_\star})}\le 2\mathrm M C'_{\src,1}\sigmaT\sigmaSm{i_\star}L_Kr_{i_\star},
\end{align}
and hence
\begin{align}
\label{eq:source_high_dim_noise_dominant_small_contrast_initial}
    \frac{D_{ac}^{(\src_{i_\star})}}{\sigmaSm{i_\star}\sigmaT q_{i_\star}\DeltaT}\le C_{\mathrm M}\frac{L_Kr_{i_\star}}{q_{i_\star}\DeltaT}.
\end{align}
Moreover, since $\DeltaSm{i_\star}\le \|\delta_{ac}^{(\src_{i_\star})}\|_2/\sigmaSm{i_\star}$ and we are in Case~2, we have $\mu\DeltaSm{i_\star}\le C_{\src,i_\star}q_{i_\star}$. Therefore,
\begin{align}
\label{eq:source_high_dim_noise_dominant_small_contrast}
    \frac{L_Kr_{i_\star}}{q_{i_\star}\DeltaT}\le C\frac{L_Kr_{i_\star}}{\mu\DeltaSm{i_\star}\DeltaT}.
\end{align}
Combining \eqref{eq:source_high_dim_noise_dominant_initial}, \eqref{eq:source_high_dim_noise_dominant_large_contrast}, and \eqref{eq:source_high_dim_noise_dominant_small_contrast}, we get
\begin{align}
\label{eq:source_high_dim_noise_dominant_small_ball}
    \mathbb P\left(|\nu_{ac}^{(\src_{i_\star})}+G_{ac}^{(\src_{i_\star})}|\le D_{ac}^{(\src_{i_\star})}\right)\le C_{\mathrm M}\left\{\frac{L_K}{\mu\DeltaT}+\frac{L_Kr_{i_\star}}{\mu\DeltaSm{i_\star}\DeltaT}\right\}.
\end{align}

Combining \eqref{eq:source_high_dim_mean_dominant_small_ball} and \eqref{eq:source_high_dim_noise_dominant_small_ball}, we obtain, uniformly over all $a\neq c$,
\begin{align}
\label{eq:source_high_dim_uniform_pair_small_ball}
    \mathbb P\left(|\nu_{ac}^{(\src_{i_\star})}+G_{ac}^{(\src_{i_\star})}|\le D_{ac}^{(\src_{i_\star})}\right)\le C_{\mathrm M}\left\{\frac{L_K}{\mu\DeltaT}+\frac{L_Kr_{i_\star}}{\mu\DeltaSm{i_\star}\DeltaT}\right\}.
\end{align}
Let $\mathcal H_{\src_{i_\star}}:=\mathcal E_{\src_{i_\star}}\cap\mathcal N_{\src_{i_\star}}$. Since $\mathrm M>0$ is a constant, using \eqref{eq:source_high_dim_small_projection_event}, \eqref{eq:source_high_dim_uniform_pair_small_ball}, and a union bound over all pairs $a\neq c$, we have
\begin{align}
\label{eq:source_high_dim_all_pairs_union_bound}
    \mathbb P\left(\min_{a\neq c}\frac{\|\Pest^{(\src)}\delta_{ac}^{(\trg)}\|_2}{L_K\sigmaT}\le \mathrm M\right)\le \mathbb P(\mathcal H_{\src_{i_\star}}^c)+C_{\mathrm M}K(K-1)\left\{\frac{L_K}{\mu\DeltaT}+\frac{L_Kr_{i_\star}}{\mu\DeltaSm{i_\star}\DeltaT}\right\}=o(1),
\end{align}
where the final equality follows from \eqref{eq:source_high_dim_uniform_pair_condition}. Since $\mathrm M>0$ is arbitrary,
\begin{align}
    \min_{a\neq c}\frac{\|\Pest^{(\src)}(\thetamcT{a}-\thetamcT{c})\|_2}{K\sqrt{\log K}\,\sigmaT}\convp\infty.
\end{align}
Equivalently,
\begin{align}
    \frac{\displaystyle \min_{a\neq c}\|\Pest^{(\src)}(\thetamcT{a}-\thetamcT{c})\|_2^2}{K^2(\log K)\,\sigmaT^2}\convp\infty,
\end{align}
which proves Lemma~\ref{lem:high_dim_mult_consistent_cluster_2} in the regime $d\gg \nSm{i_\star}$.

\subsubsection{Proof of Lemma~\ref{lem:high_dim_mult_consistent_cluster_2} when $ \max\{K+3,2K\}\le d\lesssim \nSm{i_\star}$}

Throughout this subsection, assume that $mK+3\le d\lesssim \nSm{i_\star}$ and that $\calS_{i_\star}$ satisfies \eqref{eq:mult_cluster_source_signal_strength_high}. Observe that \eqref{eq:mult_cluster_source_signal_strength_high} implies
\begin{align}
\label{eq:source_low_dim_necessary_condition}
    \DeltaSm{i_\star}\gg\left(\frac{d(\beta K)^2}{\nSm{i_\star}}\right)^{1/4}\sqrt{(\beta K) \vee \log \nSm{i_\star}}, \quad \text{and}\quad \mu\cdot\DeltaT\gg K\sqrt{\log K},
\end{align}
as $\beta K\ge 1$.

Recall the construction of $\hatThetaS{i_\star}$ from \eqref{eq:construction_mult_d_ll_m}. We may write
\begin{align}
X_{\calS_{i_\star}}=Z_{\calS_{i_\star}}\ThetamS{i_\star}^{\top}+E_{\calS_{i_\star}},
\end{align}
where $Z_{\calS_{i_\star}}\in\R^{\nSm{i_\star}\times K}$ is obtained by stacking $\Zsrcm{i_\star}_j$ for $j\in[\nSm{i_\star}]$ along the rows, and $E_{\calS_{i_\star}}\in\R^{\nSm{i_\star}\times d}$ is obtained by stacking the noise vectors $\epsm{i_\star}_j$ along the rows. By \eqref{eq:approximate_balance_src}, the matrix $Z_{\calS_{i_\star}}$ has full column rank $K$. Therefore,
\begin{align}
\row\left(Z_{\calS_{i_\star}}\ThetamS{i_\star}^{\top}\right)=\row\left(\ThetamS{i_\star}^{\top}\right)=\col\left(\ThetamS{i_\star}\right).
\end{align}
Thus, estimating the column space of $\ThetamS{i_\star}$ is equivalent to estimating the row space of $Z_{\calS_{i_\star}}\ThetamS{i_\star}^{\top}$. This row space is estimated by the column space of $\wh V_{\calS_{i_\star}}$. Let $\mathrm P_{\src_{i_\star}}$ denote the orthogonal projector onto $\col(\ThetamS{i_\star})$, and let $\widehat{\mathrm P}_{\src_{i_\star}}$ denote the orthogonal projector onto $\col(\wh V_{\calS_{i_\star}})$. Applying Lemma~1 and Theorem~3 of \citet{CaiZhang2018}, after accounting for the noise variance $\sigmaSm{i_\star}^2$, there exists a constant $\mathrm C'_{\src,i_\star}>0$ such that
\begin{align}
    \mathbb E\left[\|\widehat{\mathrm P}_{\src_{i_\star}}-\mathrm P_{\src_{i_\star}}\|_{\op}^2\right]\le \mathrm C'_{\src,i_\star}\cdot(\mathrm R_{i_\star}\wedge 1),
\end{align}
where
\begin{align}
\label{eq:source_low_dim_risk_term}
    \mathrm R_{i_\star}:=\frac{d\,\sigmaSm{i_\star}^2\left\{\wt\sigma_K^2+\nSm{i_\star}\sigmaSm{i_\star}^2\right\}}{\wt\sigma_K^4},
\end{align}
and $\wt\sigma_K$ is the $K$-th singular value of $Z_{\calS_{i_\star}}\ThetamS{i_\star}^{\top}$. By \eqref{eq:singular_value_lower} and \eqref{eq:approximate_balance_src},
\begin{align}
    \wt\sigma_K\ge \sqrt{\frac{\nSm{i_\star}}{\beta K}}\cdot\DeltaSm{i_\star}\cdot\sigmaSm{i_\star}.
\end{align}
Plugging this bound into \eqref{eq:source_low_dim_risk_term}, we obtain
\begin{align}
\label{eq:source_low_dim_risk_bound}
    \mathrm R_{i_\star}\lesssim \frac{d(\beta K)}{\nSm{i_\star}\DeltaSm{i_\star}^2}+\frac{d(\beta K)^2}{\nSm{i_\star}\DeltaSm{i_\star}^4}.
\end{align}
Using \eqref{eq:source_low_dim_necessary_condition}, \eqref{eq:source_low_dim_risk_bound}, and Markov's inequality we have
\begin{align}
\label{eq:source_low_dim_projector_consistency}
    \|\widehat{\mathrm P}_{\src_{i_\star}}-\mathrm P_{\src_{i_\star}}\|_{\op}=o_p(1).
\end{align}
Let
\begin{align}
\label{eq:source_low_dim_good_projector_event}
    \mathcal E_{\src_{i_\star}}:=\left\{\rho_{\src_{i_\star}}\le \frac12\right\}, \quad \mbox{where}\quad \rho_{\src_{i_\star}}:=\|\widehat{\mathrm P}_{\src_{i_\star}}-\mathrm P_{\src_{i_\star}}\|_{\op}.
\end{align}
By \eqref{eq:source_low_dim_projector_consistency}, $\mathbb P(\mathcal E_{\src_{i_\star}}^c)\to0$.

Let $\mathrm V_{\src_{i_\star}}\in\R^{d\times K}$ be an orthonormal basis matrix for $\col(\ThetamS{i_\star})$, so that $\mathrm P_{\src_{i_\star}}=\mathrm V_{\src_{i_\star}}\mathrm V_{\src_{i_\star}}^\top$. Let $\mathrm V_{\src_{i_\star},\perp}\in\R^{d\times(d-K)}$ be an orthonormal basis matrix for $\col(\ThetamS{i_\star})^\perp$. By Lemma~4.4 of \citet{LofflerZhangZhou2021}, the source singular vector matrix admits the decomposition
\begin{align}
\label{eq:source_low_dim_haar_decomp}
    \wh V_{\calS_{i_\star}}=\mathrm V_{\src_{i_\star}}\mathrm A_{\src_{i_\star}}+\mathrm V_{\src_{i_\star},\perp}\mathrm H_{\src_{i_\star}}\left(\Id_K-\mathrm A_{\src_{i_\star}}^\top\mathrm A_{\src_{i_\star}}\right)^{1/2},
\end{align}
where $\mathrm A_{\src_{i_\star}}:=\mathrm V_{\src_{i_\star}}^\top\wh V_{\calS_{i_\star}}$, and $\mathrm H_{\src_{i_\star}}\in\R^{(d-K)\times K}$ has orthonormal columns and is independent of $\mathrm A_{\src_{i_\star}}$. Moreover, $\mathrm H_{\src_{i_\star}}$ is Haar distributed on the corresponding Stiefel manifold.

Observe that
\begin{align}
    \mathrm A_{\src_{i_\star}}\mathrm A_{\src_{i_\star}}^\top=\mathrm V_{\src_{i_\star}}^\top\widehat{\mathrm P}_{\src_{i_\star}}\mathrm V_{\src_{i_\star}}.
\end{align}
Since $\mathrm V_{\src_{i_\star}}^\top\mathrm P_{\src_{i_\star}}\mathrm V_{\src_{i_\star}}=\Id_K$, on $\mathcal E_{\src_{i_\star}}$ we have
\begin{align}
    \|\mathrm A_{\src_{i_\star}}\mathrm A_{\src_{i_\star}}^\top-\Id_K\|_{\op}=\left\|\mathrm V_{\src_{i_\star}}^\top(\widehat{\mathrm P}_{\src_{i_\star}}-\mathrm P_{\src_{i_\star}})\mathrm V_{\src_{i_\star}}\right\|_{\op}\le \|\widehat{\mathrm P}_{\src_{i_\star}}-\mathrm P_{\src_{i_\star}}\|_{\op}\le \frac12.
\end{align}
Consequently,
\begin{align}
\label{eq:source_low_dim_A_min_singular}
    \sigma_{\min}(\mathrm A_{\src_{i_\star}})\ge \frac{1}{\sqrt2}, \quad \mbox{on $\mathcal E_{\src_{i_\star}}$.}
\end{align}

For every $a\neq c$, define
\begin{align}
    \delta_{ac}^{(\trg)}:=\thetamcT{a}-\thetamcT{c}, \quad \delta_{ac}^{(\src_{i_\star})}:=\thetamSS{i_\star}{a}-\thetamSS{i_\star}{c}.
\end{align}
Decompose the target contrast relative to the source signal subspace as
\begin{align}
\label{eq:source_low_dim_target_decomp}
    \delta_{ac}^{(\trg)}=\mathrm V_{\src_{i_\star}}x_{ac,i_\star}+\mathrm V_{\src_{i_\star},\perp}y_{ac,i_\star},
\end{align}
where $x_{ac,i_\star}:=\mathrm V_{\src_{i_\star}}^\top\delta_{ac}^{(\trg)}$ and $y_{ac,i_\star}:=\mathrm V_{\src_{i_\star},\perp}^\top\delta_{ac}^{(\trg)}$. Since $\delta_{ac}^{(\src_{i_\star})}\in\col(\ThetamS{i_\star})$, the alignment condition in \eqref{eq:mult_cluster_source_signal_strength_high} and Cauchy--Schwarz imply
\begin{align}
\label{eq:source_low_dim_oracle_projection_lower}
    \|x_{ac,i_\star}\|_2=\|\mathrm P_{\src_{i_\star}}\delta_{ac}^{(\trg)}\|_2\ge \frac{|\ip{\delta_{ac}^{(\src_{i_\star})}}{\delta_{ac}^{(\trg)}}|}{\|\delta_{ac}^{(\src_{i_\star})}\|_2}\ge \mu\|\delta_{ac}^{(\trg)}\|_2\ge \mu\sigmaT\DeltaT.
\end{align}

Using \eqref{eq:source_low_dim_haar_decomp} and \eqref{eq:source_low_dim_target_decomp},
\begin{align}
\label{eq:source_low_dim_projected_target_decomp}
    \wh V_{\calS_{i_\star}}^\top\delta_{ac}^{(\trg)}=\mathrm A_{\src_{i_\star}}^\top x_{ac,i_\star}+\left(\Id_K-\mathrm A_{\src_{i_\star}}^\top\mathrm A_{\src_{i_\star}}\right)^{1/2}\mathrm H_{\src_{i_\star}}^\top y_{ac,i_\star}.
\end{align}
Since $\wh V_{\calS_{i_\star}}$ has orthonormal columns,
\begin{align}
    \|\widehat{\mathrm P}_{\src_{i_\star}}\delta_{ac}^{(\trg)}\|_2=\|\wh V_{\calS_{i_\star}}^\top\delta_{ac}^{(\trg)}\|_2.
\end{align}
On $\mathcal E_{\src_{i_\star}}$, \eqref{eq:source_low_dim_A_min_singular} and \eqref{eq:source_low_dim_oracle_projection_lower} give
\begin{align}
\label{eq:source_low_dim_deterministic_component_lower}
    \|\mathrm A_{\src_{i_\star}}^\top x_{ac,i_\star}\|_2\ge \sigma_{\min}(\mathrm A_{\src_{i_\star}})\|x_{ac,i_\star}\|_2\ge \frac{1}{\sqrt2}\mu\sigmaT\DeltaT.
\end{align}
Define
\begin{align}
    e_{ac,i_\star}:=\frac{\mathrm A_{\src_{i_\star}}^\top x_{ac,i_\star}}{\|\mathrm A_{\src_{i_\star}}^\top x_{ac,i_\star}\|_2}.
\end{align}
Then $\|e_{ac,i_\star}\|_2=1$, and hence $\|\wh V_{\calS_{i_\star}}e_{ac,i_\star}\|_2=1$. Therefore, if $\|\widehat{\mathrm P}_{\src_{i_\star}}\delta_{ac}^{(\trg)}\|_2\le t$, then
\begin{align}
    \left|e_{ac,i_\star}^\top\wh V_{\calS_{i_\star}}^\top\delta_{ac}^{(\trg)}\right|\le t.
\end{align}
By \eqref{eq:source_low_dim_projected_target_decomp},
\begin{align}
\label{eq:source_low_dim_scalar_projection_decomp}
    e_{ac,i_\star}^\top\wh V_{\calS_{i_\star}}^\top\delta_{ac}^{(\trg)}=\|\mathrm A_{\src_{i_\star}}^\top x_{ac,i_\star}\|_2+\left\langle \mathrm H_{\src_{i_\star}}\left(\Id_K-\mathrm A_{\src_{i_\star}}^\top\mathrm A_{\src_{i_\star}}\right)^{1/2}e_{ac,i_\star},y_{ac,i_\star}\right\rangle.
\end{align}
Thus,
\begin{align}
\label{eq:source_low_dim_projection_small_ball_reduction}
    \mathbb P\left(\|\widehat{\mathrm P}_{\src_{i_\star}}\delta_{ac}^{(\trg)}\|_2\le t\right)\le \mathbb P\left(\left|\|\mathrm A_{\src_{i_\star}}^\top x_{ac,i_\star}\|_2+\left\langle \mathrm H_{\src_{i_\star}}\left(\Id_K-\mathrm A_{\src_{i_\star}}^\top\mathrm A_{\src_{i_\star}}\right)^{1/2}e_{ac,i_\star},y_{ac,i_\star}\right\rangle\right|\le t\right).
\end{align}

We now apply Haar anti-concentration. Since $\mathrm A_{\src_{i_\star}}$ is independent of $\mathrm H_{\src_{i_\star}}$, whenever $\left(\Id_K-\mathrm A_{\src_{i_\star}}^\top\mathrm A_{\src_{i_\star}}\right)^{1/2}e_{ac,i_\star}\neq0$, the vector
\begin{align}
    \frac{\mathrm H_{\src_{i_\star}}\left(\Id_K-\mathrm A_{\src_{i_\star}}^\top\mathrm A_{\src_{i_\star}}\right)^{1/2}e_{ac,i_\star}}{\left\|\left(\Id_K-\mathrm A_{\src_{i_\star}}^\top\mathrm A_{\src_{i_\star}}\right)^{1/2}e_{ac,i_\star}\right\|_2}
\end{align}
is Haar distributed on the unit sphere in $\R^{d-K}$. Define
\begin{align}
\label{eq:source_low_dim_haar_coordinate}
    \xi_{ac,i_\star}:=\frac{\left\langle \mathrm H_{\src_{i_\star}}\left(\Id_K-\mathrm A_{\src_{i_\star}}^\top\mathrm A_{\src_{i_\star}}\right)^{1/2}e_{ac,i_\star},y_{ac,i_\star}\right\rangle}{\|y_{ac,i_\star}\|_2\cdot\left\|\left(\Id_K-\mathrm A_{\src_{i_\star}}^\top\mathrm A_{\src_{i_\star}}\right)^{1/2}e_{ac,i_\star}\right\|_2}.
\end{align}
Then $\xi_{ac,i_\star}$ is distributed as one coordinate of a Haar vector on the unit sphere in $\R^{d-K}$. If $\left(\Id_K-\mathrm A_{\src_{i_\star}}^\top\mathrm A_{\src_{i_\star}}\right)^{1/2}e_{ac,i_\star}=0$, then the random term in \eqref{eq:source_low_dim_scalar_projection_decomp} is zero and the following bound is trivial whenever $t\le \|\mathrm A_{\src_{i_\star}}^\top x_{ac,i_\star}\|_2/2$.

Let
\begin{align}
    \alpha_{ac,i_\star}:=\|\mathrm A_{\src_{i_\star}}^\top x_{ac,i_\star}\|_2, \quad \beta_{ac,i_\star}:=\|y_{ac,i_\star}\|_2\cdot\left\|\left(\Id_K-\mathrm A_{\src_{i_\star}}^\top\mathrm A_{\src_{i_\star}}\right)^{1/2}e_{ac,i_\star}\right\|_2.
\end{align}
Conditionally on $(\alpha_{ac,i_\star},\beta_{ac,i_\star})$, the variable $\xi_{ac,i_\star}$ remains distributed as one coordinate of a Haar vector on the unit sphere in $\R^{d-K}$. Its density on $[-1,1]$ is
\begin{align}
    f_{d-K}(x)=c_{d-K}(1-x^2)_+^{(d-K-3)/2}, \quad \mbox{where}\quad c_{d-K}:=\frac{\Gamma((d-K)/2)}{\sqrt{\pi}\,\Gamma((d-K-1)/2)}\le C\sqrt{d-K}.
\end{align}
Therefore, for $\beta_{ac,i_\star}>0$,
\begin{align}
\label{eq:source_low_dim_conditional_small_ball}
    &\mathbb P\left(\left|\alpha_{ac,i_\star}+\left\langle \mathrm H_{\src_{i_\star}}\left(\Id_K-\mathrm A_{\src_{i_\star}}^\top\mathrm A_{\src_{i_\star}}\right)^{1/2}e_{ac,i_\star},y_{ac,i_\star}\right\rangle\right|\le t\;\middle|\;\alpha_{ac,i_\star},\beta_{ac,i_\star}\right)\\
    &\qquad\le \frac{2t\sqrt{d-K}}{\beta_{ac,i_\star}}\left(1-\left(\frac{(\alpha_{ac,i_\star}-t)_+}{\beta_{ac,i_\star}}\right)^2\right)_+^{(d-K-3)/2}.
\end{align}

Fix $\mathrm M>0$ and set
\begin{align}
    t:=\sigmaT K\sqrt{\log K}\,\mathrm M.
\end{align}
By \eqref{eq:source_low_dim_deterministic_component_lower} and \eqref{eq:source_low_dim_necessary_condition}, on $\mathcal E_{\src_{i_\star}}$ we have $t\le \alpha_{ac,i_\star}/2$ for all sufficiently large $\nT$ and $\nSm{i_\star}$. Hence, on $\mathcal E_{\src_{i_\star}}$,
\begin{align}
\label{eq:source_low_dim_conditional_bound_after_t}
    &\mathbb P\left(\left|\alpha_{ac,i_\star}+\left\langle \mathrm H_{\src_{i_\star}}\left(\Id_K-\mathrm A_{\src_{i_\star}}^\top\mathrm A_{\src_{i_\star}}\right)^{1/2}e_{ac,i_\star},y_{ac,i_\star}\right\rangle\right|\le \sigmaT K\sqrt{\log K}\,\mathrm M\;\middle|\;\alpha_{ac,i_\star},\beta_{ac,i_\star}\right)\\
    &\qquad\le \frac{2\sigmaT\mathrm M K\sqrt{\log K}\sqrt{d-K}}{\beta_{ac,i_\star}}\left(1-\frac{\alpha_{ac,i_\star}^2}{4\beta_{ac,i_\star}^2}\right)_+^{(d-K-3)/2}.
\end{align}
If $\alpha_{ac,i_\star}^2\ge 4\beta_{ac,i_\star}^2$, the right-hand side of \eqref{eq:source_low_dim_conditional_bound_after_t} is zero. On the other hand, if $\alpha_{ac,i_\star}^2\le 4\beta_{ac,i_\star}^2$, then \eqref{eq:source_low_dim_deterministic_component_lower} implies $\mu^2\DeltaT^2\sigmaT^2\le 8\beta_{ac,i_\star}^2$. Therefore,
\begin{align}
\label{eq:source_low_dim_conditional_exp_bound}
    &\mathbb P\left(\left|\alpha_{ac,i_\star}+\left\langle \mathrm H_{\src_{i_\star}}\left(\Id_K-\mathrm A_{\src_{i_\star}}^\top\mathrm A_{\src_{i_\star}}\right)^{1/2}e_{ac,i_\star},y_{ac,i_\star}\right\rangle\right|\le \sigmaT K\sqrt{\log K}\,\mathrm M\;\middle|\;\alpha_{ac,i_\star},\beta_{ac,i_\star}\right)\\
    &\qquad\le \frac{2\sigmaT\mathrm M K\sqrt{\log K}\sqrt{d-K}}{\beta_{ac,i_\star}}\exp\left(-\frac{(d-K-3)\mu^2\sigmaT^2\DeltaT^2}{16\beta_{ac,i_\star}^2}\right).
\end{align}
Define
\begin{align}
    s_{ac,i_\star}:=\frac{\sqrt{d-K}\,\mu\DeltaT\sigmaT}{\beta_{ac,i_\star}}.
\end{align}
Then, for an absolute constant $c>0$, \eqref{eq:source_low_dim_conditional_exp_bound} gives
\begin{align}
\label{eq:source_low_dim_pair_small_ball_final}
    \mathbb P\left(\left|\alpha_{ac,i_\star}+\left\langle \mathrm H_{\src_{i_\star}}\left(\Id_K-\mathrm A_{\src_{i_\star}}^\top\mathrm A_{\src_{i_\star}}\right)^{1/2}e_{ac,i_\star},y_{ac,i_\star}\right\rangle\right|\le \sigmaT K\sqrt{\log K}\,\mathrm M\;\middle|\;\alpha_{ac,i_\star},\beta_{ac,i_\star}\right)\lesssim \frac{\mathrm M K\sqrt{\log K}}{\mu\DeltaT}.
\end{align}
Combining \eqref{eq:source_low_dim_projection_small_ball_reduction} and \eqref{eq:source_low_dim_pair_small_ball_final}, and then taking a union bound over all ordered pairs $a\neq c$, we obtain
\begin{align}
\label{eq:source_low_dim_all_pairs_union_bound}
    \mathbb P\left(\frac{\min_{a\neq c}\|\widehat{\mathrm P}_{\src_{i_\star}}\delta_{ac}^{(\trg)}\|_2}{K\sqrt{\log K}\,\sigmaT}\le \mathrm M,\;\mathcal E_{\src_{i_\star}}\right)\le C\,\frac{\mathrm M K^3\sqrt{\log K}}{\mu\DeltaT}\to 0,
\end{align}
where the convergence follows from \eqref{eq:source_low_dim_necessary_condition}. Finally, since $\col(\wh V_{\calS_{i_\star}})\subseteq\mathcal R$, we have
\begin{align}
    \|\Pest^{(\src)}\delta_{ac}^{(\trg)}\|_2\ge \|\widehat{\mathrm P}_{\src_{i_\star}}\delta_{ac}^{(\trg)}\|_2, \quad \mbox{for all $a\neq c$.}
\end{align}
Combining this observation with \eqref{eq:source_low_dim_all_pairs_union_bound} and $\mathbb P(\mathcal E_{\src_{i_\star}}^c)\to0$, we get, for every fixed $\mathrm M>0$,
\begin{align}
    \mathbb P\left(\frac{\min_{a\neq c}\|\Pest^{(\src)}\delta_{ac}^{(\trg)}\|_2}{K\sqrt{\log K}\,\sigmaT}\le \mathrm M\right)\to0.
\end{align}
Since $\mathrm M>0$ is arbitrary, this implies
\begin{align}
    \frac{\displaystyle \min_{a\neq c}\|\Pest^{(\src)}(\thetamcT{a}-\thetamcT{c})\|_2^2}{K^2(\log K)\,\sigmaT^2}\convp\infty,
\end{align}
which proves Lemma~\ref{lem:high_dim_mult_consistent_cluster_2} in the regime $\max\{K+3,2K\}\le d\lesssim \nSm{i_\star}$.

\subsection{Proof of Lemma~\ref{lem:validation_stat_candidate_partition}}
To prove Lemma~\ref{lem:validation_stat_candidate_partition}, we consider the following concentration lemmas.

\begin{lemma}
\label{lem:validation_linear_gaussian_concentration}
Let $E_\calT\in\R^{\nT\times d}$ have independent rows distributed as
$\dnorm_d(0,\sigmaT^2\Id_d)$. Then, for any fixed matrix
$A\in\R^{\nT\times d}$ and any $u>0$,
\begin{align}
    \mathbb P\left(\left|\frac{2}{\nT}\langle A,E_\calT\rangle_F\right|
    \ge \frac{2\sigmaT}{\nT}\|A\|_F\sqrt{2u}\right)\le 2e^{-u},
\end{align}
where $\langle A,E_\calT\rangle_F:=\operatorname{Tr}(AE^\top_\calT)$.
\end{lemma}

\begin{proof}
Since the entries of $E_\calT$ are independent centered Gaussians with variance
$\sigmaT^2$, we have
\begin{align}
    \langle A,E_\calT\rangle_F\sim \dnorm\left(0,\sigmaT^2\|A\|_F^2\right).
\end{align}
The claim follows from the standard Gaussian tail bound.
\end{proof}

\begin{lemma}
\label{lem:validation_projected_chisq_concentration}
Let $E_\calT\in\R^{\nT\times d}$ have independent rows distributed as
$\dnorm_d(0,\sigmaT^2\Id_d)$, and let $\Pi\in\R^{\nT\times \nT}$ be a fixed
orthogonal projector of rank $r$. Then, for every $u>0$, there exists an
absolute constant $C>0$ such that
\begin{align}
    \mathbb P\left(\left|\frac{1}{\nT}\left\{\|\Pi E_\calT\|_F^2-\sigmaT^2dr\right\}\right|
    \ge \frac{C\sigmaT^2}{\nT}\left\{\sqrt{dru}+u\right\}\right)\le 2e^{-u}.
\end{align}
\end{lemma}

\begin{proof}
Since $\Pi$ is an orthogonal projector of rank $r$, for each column
$(E_\calT)_{*\ell}$ we have
$\|\Pi(E_\calT)_{\cdot \ell}\|_2^2/\sigmaT^2\sim\chi_r^2$. Summing over the
$d$ independent columns gives
\begin{align}
    \frac{\|\Pi E_\calT\|_F^2}{\sigmaT^2}\sim \chi^2_{dr}.
\end{align}
The result follows from the standard centered chi-square concentration inequality (see, Lemma~1 \cite{laurent2000adaptive}).
\end{proof}

\begin{proof}[Proof of Lemma~\ref{lem:validation_stat_candidate_partition}]
Since $\wt n_a(\wt Z)>0$ for all $a\in[K]$, the matrix $\wt Z$ has full column rank and $\wt Z^\top\wt Z=\mathrm D(\wt Z)$. Hence
\begin{align}
    \mathrm P(\wt Z)^2=\wt Z\mathrm D(\wt Z)^{-1}\wt Z^\top\wt Z\mathrm D(\wt Z)^{-1}\wt Z^\top=\wt Z\mathrm D(\wt Z)^{-1}\wt Z^\top=\mathrm P(\wt Z),
    \qquad
    \mathrm P(\wt Z)^\top=\mathrm P(\wt Z).
\end{align}
Thus $\mathrm P(\wt Z)$ is the orthogonal projector onto $\col(\wt Z)$ and has
rank $K$. Also $\mathbbm 1_{\nT}\in\col(\wt Z)$, and therefore
$\mathrm P(\wt Z)\mathbbm 1_{\nT}=\mathbbm 1_{\nT}$. It follows that
\begin{align}
    \Pi(\wt Z):=\mathrm P(\wt Z)-\frac{1}{\nT}\mathbbm 1_{\nT}\mathbbm 1_{\nT}^\top
\end{align}
is an orthogonal projector onto the centered subspace of $\col(\wt Z)$. Consequently, $\operatorname{rank}\{\Pi(\wt Z)\}=K-1$.

Next, observe that the $j$-th row of $\mathrm P(\wt Z)X_\calT$ is equal to $\wt{\theta}^{(\trg)}_a(\wt Z)^\top$ whenever $\wt Z_{ja}=1$. Therefore,
\begin{align}
    \frac{1}{\nT}\|\Pi(\wt Z)X_\calT\|_F^2
    =
    \frac{1}{\nT}\sum_{a=1}^K \wt n_a(\wt Z)
    \|\wt{\theta}^{(\trg)}_a(\wt Z)-\bar X_\calT\|_2^2,
    \quad
    \bar X_\calT:=\frac{1}{\nT}\sum_{j=1}^{\nT}\Xtar_j.
\end{align}
Using the weighted variance identity,
\begin{align}
    \sum_{a=1}^K \wt n_a(\wt Z)\|\wt{\theta}^{(\trg)}_a(\wt Z)-\bar X_\calT\|_2^2
    =
    \frac{1}{2\nT}\sum_{1\le a\neq b\le K}\wt n_a(\wt Z)\wt n_b(\wt Z)
    \|\wt{\theta}^{(\trg)}_a(\wt Z)-\wt{\theta}^{(\trg)}_b(\wt Z)\|_2^2.
\end{align}
Dividing by $\nT$ proves \eqref{eq:candidate_partition_projection_identity}.
The representation \eqref{eq:candidate_partition_validation_projection_form}
then follows immediately from the definition of
$\wh{\mathsf S}(\wt Z)$ in \eqref{eq:validation_stat_mult_clust}.

It remains to prove the concentration bound. Since $X_\calT=M_\calT+E_\calT$,
we have
\begin{align}
    \|\Pi(\wt Z)X_\calT\|_F^2
    =
    \|\Pi(\wt Z)M_\calT\|_F^2
    +2\langle \Pi(\wt Z)M_\calT,\Pi(\wt Z)E_\calT\rangle_F
    +\|\Pi(\wt Z)E_\calT\|_F^2.
\end{align}
Since $\Pi(\wt Z)$ is symmetric and idempotent,
\begin{align}
    \langle \Pi(\wt Z)M_\calT,\Pi(\wt Z)E_\calT\rangle_F
    =
    \langle \Pi(\wt Z)M_\calT,E_\calT\rangle_F.
\end{align}
Therefore,
\begin{align}
\label{eq:candidate_partition_validation_decomp}
    \wh{\mathsf S}(\wt Z)-\frac{1}{\nT}\|\Pi(\wt Z)M_\calT\|_F^2
    =
    \frac{2}{\nT}\langle \Pi(\wt Z)M_\calT,E_\calT\rangle_F
    +\frac{1}{\nT}\left\{\|\Pi(\wt Z)E_\calT\|_F^2-\sigmaT^2d(K-1)\right\}.
\end{align}
Applying Lemma~\ref{lem:validation_linear_gaussian_concentration} with
$A=\Pi(\wt Z)M_\calT$, and applying
Lemma~\ref{lem:validation_projected_chisq_concentration} with
$\Pi=\Pi(\wt Z)$ and $r=K-1$, we obtain that with probability at least
$1-4e^{-u}$,
\begin{align}
    \left|\wh{\mathsf S}(\wt Z)-\frac{1}{\nT}\|\Pi(\wt Z)M_\calT\|_F^2\right|
    \le
    C\left[
    \frac{\sigmaT}{\nT}\|\Pi(\wt Z)M_\calT\|_F\sqrt{u}
    +\frac{\sigmaT^2}{\nT}\{\sqrt{d(K-1)u}+u\}
    \right],
\end{align}
after increasing the absolute constant $C>0$ if necessary. This proves
\eqref{eq:candidate_partition_validation_concentration}.

\end{proof}

\section{Proofs of results in Section~\ref{sec:two_clust_mult_source}}
To prove Theorem~\ref{thm:two_clust_multi_source}, we shall use the following lemma characterizing the alignment of $\wh v_{\trg}\in\R^{r_{\src}}$, a unit-norm leading eigenvector of $\wh{\Sigma}_{\calT}$ defined in \eqref{eq:hat_sigma_t} with $\rmqv:=\wh Q^\top_\src\thetaT \in \R^{r_\src}$ where $\wh Q_\src$ is a basis matrix of $\spclust^{(\src)}$ defined in \eqref{eq:combined_source_space}.

\begin{lemma}
    \label{lem:alignment_v_trg}
    Consider $\wh v_{\trg}\in\R^{r_{\src}}$ and $\rmqv$ introduced above. If there exists $i \in [m]$ satisfying \eqref{eq:correct_choose_alignment} and \eqref{eq:necessary_condition_mult_source},
then
    \begin{align}
\label{eq:relative_perturbation}
    \frac{\|\wh\Sigma_\calT-\Ebb[\wh\Sigma_{\calT}\mid\wh Q_\src]\|_{\op}}{\|\rmqv\|_2^2}&\lesssim\frac{1}{\rho^{(\src)}}\cdot\left(\sqrt{\frac{m}{\nT}}+\sqrt{\frac{\log\nT}{\nT}}\right)+\frac{1}{(\rho^{(\src)})^2}\cdot\left(\sqrt{\frac{m}{\nT}}+\sqrt{\frac{\log\nT}{\nT}}\right)\\
    &~~~+\frac{1}{(\rho^{(\src)})^2}\cdot\left(\frac{m}{\nT}+\frac{\log\nT}{\nT}\right),
\end{align}
with probability greater than $1-4n^{-1}_T$, where
\begin{align}
\label{eq:def_rho_s}
\rho^{(\src)}:=\frac{\|\wh{\mathrm P}^{(\src)}\thetaT\|_2}{\sigmaT}.
\end{align}
In particular, if $\rho^{(\src)} \to \infty$,
\[
\limsup_{\nT,\nSm{i} \to \infty} \Pbb\left[\sin\Theta\left(\wh v_\trg,\frac{\rmqv}{\|\rmqv\|_2}\right)>1/2\right]=0.
\]
\end{lemma}

\subsection{Proof of Theorem~\ref{thm:two_clust_multi_source}}
First, observe that if \eqref{eq:lower_bound_2_clust_delt} holds, then using Theorem 7 of \cite{Ndaoud2022}, it holds that $\hatztar$ constructed using the target data $\calT$ in \eqref{eq:oracle_target_estimate} satisfies
\begin{align}
    \mathbb E[\calL(\hatztar,\ztar)] \to 0, \quad \mbox{as $\nT \to \infty$.}
\end{align}

Now suppose \eqref{eq:lower_bound_2_clust_delt} does not hold but \eqref{eq:necessary_condition_mult_source} holds for some $\calS_i$ for $i \in [m]$ satisfying \eqref{eq:correct_choose_alignment}. 
In this setup, we use $\hatztar$ defined using \eqref{eq:oracle_source_estimate_mult_source}. Assume that $(\rmqv)^\top \wh v_\trg \ge 0$. We shall show that
\begin{align}
\label{eq:consistent_nec_cond}
    \frac{1}{\nT}\sum_{j =1}^{\nT}\mathbbm 1\{\sgn(\wh v^\top_\trg\wh{X}^{(\trg)}_j \neq \ztar\} \convp 0.
\end{align}
If $(\rmqv)^\top \wh v_\trg < 0$, then we repeat the same analysis with $-\wh v_\trg$. The resulting estimated cluster labels is then $-\hatztar \in \{-1,+1\}^{\nT}$. By definition of the loss in \eqref{eq:loss_func_2}, we have
\begin{align}
    \mathcal L(\hatztar,\ztar) \le \frac{1}{\nT}\sum_{j =1}^{\nT}\mathbbm 1\{\sgn(-\wh v^\top_\trg\wh{X}^{(\trg)}_j) \neq \ztar\}.
\end{align}
Therefore, we can still conclude that
\[
\mathcal L(\hatztar,\ztar) \convp 0, \quad \mbox{if}\quad \frac{1}{\nT}\sum_{j =1}^{\nT}\mathbbm 1\{\sgn(-\wh v^\top_\trg\wh{X}^{(\trg)}_j) \neq \ztar\} \convp 0,
\]
and one can show the above conclusion using the same analysis.

Let us define
\[
\delta:=\sin\Theta\left(\wh v_\trg,\frac{\rmqv}{\|\rmqv\|_2}\right).
\]
Consider the orthogonal decomposition
\begin{equation}
\label{eq:eigenvector_decomposition}
\wh v_\trg=c\cdot\frac{\rmqv}{\|\rmqv\|_2}+w, \qquad w\perp \rmqv, \qquad \|w\|_2=\delta, \qquad c=\sqrt{1-\delta^2}.
\end{equation}
By Lemma~\ref{lem:alignment_v_trg}, we have $\delta\le1/2$, with probability tending to 1. This implies we have $|c| \ge 1/2$ with probability tending to 1, if the projected cluster center $\rho^{(\src)}$ defined in \eqref{eq:def_rho_s} satisfies $\rho^{(\src)} \to \infty$, with probability tending to 1.

Since $\ztar_j \wh{X}^{(\trg)}_j=\rmqv+\sigmaT h_j$, misclustering occurs for observation $j \in [\nT]$ if
\begin{equation}
\label{eq:margin_decomposition}
    \ztar_j \wh v^\top_\trg\wh{X}^{(\trg)}_j
    =\wh v^\top_\trg \rmqv+\sigmaT\wh v^\top_\trg h_j
    =c\,\|\rmqv\|_2+\sigmaT \cdot c \cdot \frac{(\rmqv)^\top h_j}{\|\rmqv\|_2}+\sigmaT w^\top h_j \le 0.
\end{equation}
Recall the definition of $\rho^{(\src)}$ from \eqref{eq:def_rho_s} to observe that $\|\rmqv\|_2=\sigmaT\cdot \rho^{(\src)}$.
Therefore, for any $\gamma \in (0,1)$, the foregoing event can be decomposed as
\begin{align}
\label{eq:error_split}
\mathbbm 1\{\ztar_j \wh v^\top_\trg\wh{X}^{(\trg)}_j \le 0\} \le \mathbbm 1\left\{\frac{(\rmqv)^\top h_j}{\|\rmqv\|_2}\le-(1-\gamma)\rho^{(\src)}\right\} + \mathbbm 1\left\{|w^\top h_j| \ge c\,\gamma\,\rho^{(\src)}\right\}
\end{align}
Using $c\ge1/2$, define
\[
A_n:=\frac{1}{\nT}\sum_{j =1}^{\nT}\mathbbm 1\left\{\frac{(\rmqv)^\top h_j}{\|\rmqv\|_2}\le-(1-\gamma)\rho^{(\src)}\right\},
\]
and
\[
B_n:=\frac{1}{\nT}\sum_{j =1}^{\nT}\mathbbm 1\left\{|w^\top h_j| \ge c\,\gamma\,\rho^{(\src)}\right\}.
\]
Then \eqref{eq:loss_func_2} and \eqref{eq:error_split} imply
\begin{equation}
\label{eq:loss_split}
    \mathcal L(\hatztar,\ztar)\le A_n+B_n.
\end{equation}

We begin with providing a high probability upper bound on $A_n$. Observe that $((\rmqv)^\top h_j)/\|\rmqv\|_2 \simiid \dnorm(0,1)$. Therefore, conditioned on the source datasets, the indicator variables 
\[
\mathbbm 1\left\{\frac{(\rmqv)^\top h_j}{\|\rmqv\|_2}\le-(1-\gamma)\rho^{(\src)}\right\} \simiid \mathrm{Bernoulli}(\Phi(-(1-\gamma)\rho^{(\src)})).
\]
Therefore, by Hoeffding's inequality, we have
\begin{equation}
\label{eq:hello_bound_1_a_n}
A_n \le \Phi\bigl(-(1-\gamma)\rho^{(\src)}\bigr)+\sqrt{\frac{\log \nT}{2\nT}},
\end{equation}
with probability greater than $1-n^{-1}$.

Now, consider the second term $B_n$. By the elementary inequality $\mathbbm 1\{|x|\ge a\}\le x^2/a^2$ and the definition of $\mathrm H$ from \eqref{eq:def_h_bar_rm_h},
\begin{align}
B_n\le\frac4{\gamma^2\rho^2_\src}\frac{1}{\nT}\sum_{j =1}^{\nT}(w^\top h_j)^2& \le \frac{4}{\gamma^2\rho^2_\src}w^\top \mathrm Hw \le\frac{4\|w\|_2^2}{\gamma^2\rho^2_\src}\|\mathrm H\|_{\op}\lesssim\frac{4\delta^2(1+q_{\nT})^2}{\gamma^2\rho^2_\src},
\label{eq:dependent_bound}
\end{align}
where 
\[
q_{\nT}:=\sqrt{\frac{m}{\nT}}+\sqrt{\frac{\log\nT}{\nT}},
\]
since the dimension $r_{\src} \lesssim m$.
By Lemma~\ref{lem:alignment_v_trg}, we have $\delta \le 1/2$ with probability converging to 1 if $\rho^{(\src)} \to \infty$. Therefore, if $\rho^{(\src)} \to \infty$, we have $B_n \to 0$ with probability tending to 1. Define
\begin{align}
    \mathcal C_{\mathrm{mult}}:=\{\delta \le 1/2,\;\mbox{and}\;\eqref{eq:hello_bound_1_a_n}\; \mbox{hold}\}.
\end{align}
Observe that for any $\eta>0$, 
\begin{align}
\label{eq:write_out_in_full_mult_src}
    &\Pbb\left[\frac{1}{\nT}\sum_{j =1}^{\nT}\mathbbm 1\{\sgn(\wh v^\top_\trg\wh{X}^{(\trg)}_j) \neq \ztar\} > \eta ~\Big|~ \rho^{(\src)}\right] \\
    & \le \Pbb\left[A_n > \frac{\eta}{4},\mathcal C_{\mathrm{mult}}~\Big|~ \rho^{(\src)}\right]+\Pbb\left[B_n > \frac{\eta}{4},\mathcal C_{\mathrm{mult}}~\Big|~ \rho^{(\src)}\right]+\Pbb[\mathcal C^c_{\mathrm{mult}} \mid \rho^{(\src)}].
\end{align}
Taking $\nT, \rho^{(\src)} \to \infty$, the three terms on the right hand side of the above display converges to 0. Now, observe that since all the terms in \eqref{eq:write_out_in_full_mult_src} are bounded, therefore 
\begin{align}
\label{eq:consistency_mult_source_spectral}
\frac{1}{\nT}\sum_{j =1}^{\nT}\mathbbm 1\{\sgn(\wh v^\top_\trg\wh{X}^{(\trg)}_j) \neq \ztar\} \convp 0,
\end{align}
unconditionally if $\rho^{(\src)} \convp \infty$. 

Again, from the definition of $\rho^{(\src)}$ in \eqref{eq:def_rho_s}, we get that 
\begin{align}
    \rho^{(\src)}:=\frac{\|\wh{\mathrm P}^{(\src)}\thetaT\|_2}{\sigmaT} \ge \frac{|\langle \wh \theta_{\src_i},\thetaT\rangle|}{\|\wh \theta_{\src_i}\|_2\,\sigmaT}.
\end{align}
Now, since \eqref{eq:correct_choose_alignment} and \eqref{eq:necessary_condition_mult_source} holds for $\calS_i$, if $d \gg \nSm{i}$, we can retrace the arguments presented in Section~\ref{sec:proof_s_6_d_gg_n} to show that
\begin{align}
    \frac{\|\wh{\mathrm P}^{(\src)}\thetaT\|_2}{\sigmaT} \ge \frac{|\langle \wh \theta_{\src_i},\thetaT\rangle|}{\|\wh \theta_{\src_i}\|_2\,\sigmaT} \convp \infty.
\end{align}
Similarly, if $4 \le d \lesssim \nSm{i}$, we can retrace the arguments presented in Section~\ref{sec:proof_s_6_d_ll_n} to show that
\begin{align}
    \frac{\|\wh{\mathrm P}^{(\src)}\thetaT\|_2}{\sigmaT} \ge \frac{|\langle \wh \theta_{\src_i},\thetaT\rangle|}{\|\wh \theta_{\src_i}\|_2\cdot\sigmaT} \convp \infty.
\end{align}
Therefore, we have $\rho^{(\src)} \convp \infty$ which implies \eqref{eq:consistency_mult_source_spectral}. This further implies the theorem.

\subsection{Proof of Theorem~\ref{thm:two_clust_multi_source_lb}}
\label{app:proof-multisource-lb}
\begin{proof}
Fix $i^\star\in[m]$ and suppose
$\mathfrak R_{\mathrm{or}}^{[i^\star]}\to0$.  We reduce this oracle-indexed
experiment to the one-source model involving the designated aligned source
$i^\star$.

For every $j\neq i^\star$, choose once and for all a deterministic unit vector
$w_j\in\R^d$ and deterministic labels
$\bar z^{(j)}\in\{-1,1\}^{\nSm{j}}$, and define
\[
    \bar{\theta}_{S,j}:=\sigmaSm{j}\DeltaSm{j}\,w_j.
\]
Consider the hard subclass
\begin{align*}
    \mathcal H_{i^\star}
    :=
    \Bigl\{
    \bigl(\thetaT,\{\thetamS{\ell}\}_{\ell=1}^m,\ztar,\{\zsrcm{\ell}\}_{\ell=1}^m\bigr):
    &\,(\thetaT,\thetamS{i^\star},\ztar,\zsrcm{i^\star})
    \in \Omega(\DeltaT,\DeltaSm{i^\star},\mu),\\
    &\,\thetamS{j}=\bar{\theta}_{S,j},\;
    \zsrcm{j}=\bar z^{(j)}
    \text{ for every } j\neq i^\star
    \Bigr\}.
\end{align*}
The signal-to-noise constraints for every fixed nuisance source follow from
$\|\bar\theta_{S,j}\|_2/\sigmaSm{j}=\DeltaSm{j}$. The varying one-source
tuple satisfies the target and $i^\star$-th source constraints and has alignment at
least $\mu$. Therefore
\begin{equation}
\label{eq:hard-mult-subclass-inclusion}
    \mathcal H_{i^\star}\subseteq\Omega_{\mathrm{mult}}^{[i^\star]}.
\end{equation}

On $\mathcal H_{i^\star}$, the joint law factors as
\[
    P_{\vartheta}^{(i^\star)}\otimes R_{-i^\star},
\]
where $P_{\vartheta}^{(i^\star)}$ is exactly the one-source law corresponding
to the varying parameters
$\vartheta=(\thetaT,\thetamS{i^\star},\ztar,\zsrcm{i^\star})\in
\Omega(\DeltaT,\DeltaSm{i^\star},\mu)$, and $R_{-i^\star}$ is the common law
of the remaining fixed nuisance sources.  Revealing $i^\star$ gives no
additional information about the varying one-source parameter on this
subclass.  Moreover, in the Assouad construction for the hard one-source
subclass, the neighboring laws satisfy
\[
    P_{+k}^{\mathrm{mult}} = P_{+k}^{(i^\star)}\otimes R_{-i^\star},
    \qquad
    P_{-k}^{\mathrm{mult}} = P_{-k}^{(i^\star)}\otimes R_{-i^\star}.
\]
Since total variation is unchanged by tensoring with a common independent law,
\[
    \TV(P_{+k}^{\mathrm{mult}},P_{-k}^{\mathrm{mult}})
    =
    \TV(P_{+k}^{(i^\star)},P_{-k}^{(i^\star)}).
\]
Thus the additional sources do not alter any neighboring TV distance in the
one-source lower-bound experiment.  The same observation applies to the
target-direction oracle argument in the proof of
Theorem~\ref{thm:necessary_condition}.  Consequently, by
\eqref{eq:hard-mult-subclass-inclusion}, vanishing oracle-indexed minimax risk
would imply vanishing risk in the embedded one-source experiment.
Theorem~\ref{thm:necessary_condition}, applied with
$\nS=\nSm{i^\star}$ and $\DeltaS=\DeltaSm{i^\star}$, therefore gives either
\[
    \DeltaT\gtrsim
    \max\!\left\{1,\left(\frac d{\nT}\right)^{1/4}\right\},
\]
or
\[
    \DeltaSm{i^\star}\gtrsim
    \left(\frac d{\nSm{i^\star}}\right)^{1/4},
    \qquad
    \mu\DeltaT\gtrsim1,
    \qquad
    \mu\DeltaSm{i^\star}\DeltaT
    \gtrsim\sqrt{\frac d{\nSm{i^\star}}}.
\]
These are precisely the alternatives in
\eqref{eq:multi_source_lb_condition}, which proves the theorem.
\end{proof}

\subsection{Proof of Lemma~\ref{lem:alignment_v_trg}}
Recall the definition of $\wh{\Sigma}_{\calT}$ from \eqref{eq:hat_sigma_t} and $\wh Q_\src$ be an orthonormal basis of $\spclust^{(\src)}$ defined in \eqref{eq:combined_source_space}. If \eqref{eq:necessary_condition_mult_source} holds for some $i \in [m]$ and \eqref{eq:correct_choose_alignment} holds, then by definition in \eqref{eq:def_hat_x_tar_q_tar}
\begin{align}
    \wh X^{(\trg)}_j = \wh Q^\top_\src \Xtar_j= \ztar_j\,\rmqv+\sigmaT\cdot\rmqg_j \in \R^{r_\src}, \quad \mbox{for all $j \in \calT$,}
\end{align}
where $\rmqv:=\wh Q^\top_\src\thetaT \in \R^{r_\src}$ and $\rmqg_j:=\wh Q^\top_\src\varepsilon^{(\trg)}_j/\sigmaT \in \R^{r_\src}$.
Observe that conditioned on the source datasets 
\[
\rmqg_j \simiid N_{r_\src}(0,\Id_{r_\src}), \quad \mbox{and} \quad \|\rmqv\|_2=\|\wh{\mathrm P}^{(\src)}\thetaT\|_2.
\]
Recall the definition of $\rho^{(\src)}$ from \eqref{eq:def_rho_s}. Let $\wh v_\trg$ be a unit leading eigenvector of $\wh\Sigma_{\calT}$. The estimated labels are
\[
\hatztar_j:=\operatorname{sgn}(\wh v^\top_\trg \wh X^{(\trg)}_j), \qquad i=1,\ldots,\nT.
\]
Conditionally on the source data, the matrix
\[
\Ebb[\wh\Sigma_{\calT}\mid\wh Q_\src]:= \rmqv(\rmqv)^\top+\sigmaT^2\Id_{r_\src}.
\]
Observe that $\rmqv$ is the leading eigenvector of $\Ebb[\wh\Sigma_{\calT}\mid\wh Q_\src]$ with eigenvalue $(\sigmaT^2+\|\rmqv\|^2_2)$. The eigen gap between the leading eigenvalue of $\Ebb[\wh\Sigma_{\calT}\mid\wh Q_\src]$ and the second largest eigenvalue is $\|\rmqv\|^2_2$.

Define $h_j:=\ztar_j\cdot\rmqg_j$ for $j \in [\nT]$ and define
\begin{align}
\label{eq:def_h_bar_rm_h}
    \bar{h}:=\frac{1}{\nT}\sum_{j =1}^{\nT}h_j \quad \mbox{and} \quad \mathrm H:=\frac{1}{\nT}\sum_{j =1}^{\nT}h_jh_j^\top.
\end{align}
Moreover, $h_jh_j^\top=\rmqg_j(\rmqg_j)^\top$. Expanding $\wh\Sigma_{\calT}$, we get
\begin{align}
    \wh\Sigma_{\calT}
    &=\frac1n\sum_{i \in \calT}(\ztar_j\rmqv+\sigmaT \rmqg_j)(\ztar_j\rmqv+\sigmaT \rmqg_j)^\top\\
    &=\rmqv(\rmqv)^\top+\sigmaT\bigl(\rmqv\bar{ h}^\top+\bar{h} (\rmqv)^\top\bigr)+\sigmaT^2\,\mathrm H,
\end{align}
Therefore
\begin{align}
   \wh\Sigma_{\calT}- \Ebb[\wh\Sigma_{\calT}\mid\wh Q_\src]:=\sigmaT\bigl(\rmqv\bar{ h}^\top+\bar{h} (\rmqv)^\top\bigr)+\sigmaT^2\,(\mathrm H-\Id_{r_\src}).
\end{align}
Observe that since $\sqrt \nT\,\bar h\sim N_{r_\src}(0,\Id_{r_\src})$, we have
\[
    \nT\|\bar h\|_2^2 \sim \chi_{r_\src}^2 .
\]
By the Lemma~1 of \cite{laurent2000adaptive}, for every $t>0$,
\[
\mathbb P\left(\nT\|\bar h\|_2^2\ge r_\src+2\sqrt{r_\src t}+2t\right) \le e^{-t}.
\]
Consequently, with probability at least $1-\nT^{-1}$,
\begin{equation}
\label{eq:hbar_bound}
\|\bar h\|_2 \lesssim \sqrt{\frac{r_\src+\log \nT}{\nT}}.
\end{equation}
In particular, using $r_\src\le 2m$, with probability at least $1-\nT^{-1}$,
\[
\|\bar h\|_2 \le \frac{\sqrt{m}+\sqrt{\log \nT}}{\sqrt{\nT}}.
\]
Consequently, with probability at least $1-2\nT^{-1}$, we have
\[
\left\|\sigmaT\bigl(\rmqv\bar{ h}^\top+\bar{h} (\rmqv)^\top\bigr)\right\|_{\op} \lesssim \frac{\sigmaT \cdot \|\rmqv\|_2}{\sqrt{\nT}}\left(\sqrt{m}+\sqrt{\log \nT}\right).
\]
Next, define
\[
G=[h_1,\ldots,h_{\nT}]\in\R^{r_\src\times \nT},
\]
By Corollary 5.35 \cite{vershynin2012introduction}, applied to $G^\top\in\R^{\nT\times r}$, with probability at least
$1-2\nT^{-1}$,
\[
    s_{\max}(G)\le \sqrt{\nT}+\sqrt r_\src+\sqrt{2t},
    \qquad
    s_{\min}(G)\ge \bigl(\sqrt{\nT}-\sqrt r_\src-\sqrt{2t}\bigr)_+ ,
\]
where $s_{\max}$ and $s_{\min}$ reflect the maximum and minimum singular values.
Consequently, we have
\begin{align}
\label{eq:covariance_bounds}
    \|\mathrm H-\Id_{r_\src}\|_{\op}\lesssim \left(\sqrt{\frac{r_\src}{\nT}}+\sqrt{\frac{\log\nT}{\nT}}\right)+\left(\frac{r_\src}{\nT}+\frac{\log\nT}{\nT}\right).
\end{align}

On the intersection of the events \eqref{eq:hbar_bound} and \eqref{eq:covariance_bounds}, we have
\begin{align}
    \frac{\|\wh\Sigma_\calT-\Ebb[\wh\Sigma_{\calT}\mid\wh Q_\src]\|_{\op}}{\|\rmqv\|_2^2}&\le\frac{1}{\rho^{(\src)}}\cdot\left(\sqrt{\frac{r_\src}{\nT}}+\sqrt{\frac{\log\nT}{\nT}}\right)+\frac{1}{(\rho^{(\src)})^2}\cdot\left(\sqrt{\frac{r_\src}{\nT}}+\sqrt{\frac{\log\nT}{\nT}}\right)\\
    &~~~+\frac{1}{(\rho^{(\src)})^2}\cdot\left(\frac{r_\src}{\nT}+\frac{\log\nT}{\nT}\right).
\end{align}

Using Theorem~1 of \cite{yu2015useful} and the foregoing display, we get
\begin{align}
\label{eq:davis_kahan}
\sin\Theta\left(\wh v_\trg,\frac{\rmqv}{\|\rmqv\|_2}\right)&\lesssim \frac{\|\wh\Sigma_\calT-\Ebb[\wh\Sigma_{\calT}\mid\wh Q_\src]\|_{\op}}{\|\rmqv\|_2^2}\\
&\le\frac{1}{\rho^{(\src)}}\cdot\left(\sqrt{\frac{r_\src}{\nT}}+\sqrt{\frac{\log\nT}{\nT}}\right)+\frac{1}{(\rho^{(\src)})^2}\cdot\left(\sqrt{\frac{r_\src}{\nT}}+\sqrt{\frac{\log\nT}{\nT}}\right)\\
    &~~~+\frac{1}{(\rho^{(\src)})^2}\cdot\left(\frac{r_\src}{\nT}+\frac{\log\nT}{\nT}\right).
\end{align}
Since, $r_\src \le m$, we can conclude \eqref{eq:relative_perturbation} using \eqref{eq:hbar_bound} and \eqref{eq:covariance_bounds}. This automatically implies $\delta \le 1/2$ with probability tending to 1, for all large values of $\nT$ and $\nSm{i}$ tends to infinity, as long as $\rho^{(\src)} \to \infty$.

\end{document}

%% file: numerical_section.tex

\section{Numerical Experiments}
\label{sec:numerical_methods}
In this section, we benchmark the empirical performance of Algorithms~\ref{alg:meta_oracle_transfer_clustering_two_community},~\ref{alg:adaptive_transfer_clustering_two_community} and~\ref{alg:projected_clustering_mult} in synthetic data in terms of recovering the target cluster labels in a finite sample framework.

\begin{figure}[!t]
    \centering
    \includegraphics[width=\textwidth]
    {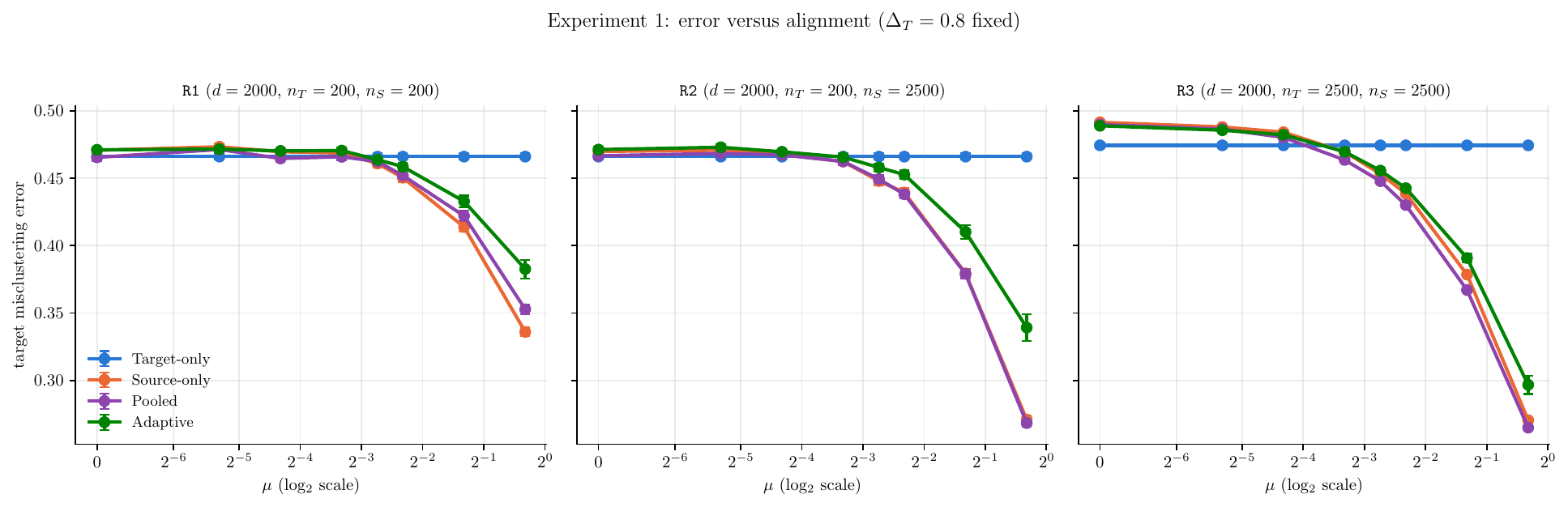}
    \caption{Target misclustering error as a function of the alignment parameter $\mu$ under three aspect-ratio regimes. The left, middle, and right panels report the results for regimes \texttt{R1}, \texttt{R2}, and \texttt{R3}, respectively. Error bars represent Monte Carlo standard errors.}
    \label{fig:error_vs_alignment}
\end{figure}

\subsection{Dependence of performance on alignment strength}
\label{sec:sim_align}
In this experiment, we consider the two-community Gaussian mixture model described in \eqref{eq:target_sample} and \eqref{eq:source_sample}, and study the performances of the target and source branches of Algorithm~\ref{alg:meta_oracle_transfer_clustering_two_community}, along with the adaptive clustering procedure in Algorithm~\ref{alg:adaptive_transfer_clustering_two_community} as a function of the alignment parameter $\mu$ between the target and source directions $\thetaT$ and $\thetaS$, respectively.

We consider three regimes of ambient dimension and sample sizes: (1) \texttt{R1}: $d \gg \max\{\nS,\nT\}$, (2) \texttt{R2}: $d \gg \nT$ but $d \le \nS$, and (3) \texttt{R3}: $d \le \min\{\nT,\nS\}$. Across all three regimes, we set $d=2000$. In \texttt{R1}, we choose $(\nT,\nS)=(200,200)$; in \texttt{R2}, we choose $(\nT,\nS)=(200,2500)$; and in \texttt{R3}, we choose $(\nT,\nS)=(2500,2500)$.

In each of the three regimes, we generate the target labels by first fixing the cluster sizes to be as balanced as possible, with $\lfloor \nT/2\rfloor$ labels equal to $+1$ and $\lceil \nT/2\rceil$ labels equal to $-1$, and then assigning these labels to the observations through a uniformly random permutation. The source labels were similarly generated ensuring $\lfloor \nS/2\rfloor$ labels equal to $+1$ and $\lceil \nS/2\rceil$ labels equal to $-1$.
Next, to generate the target and source directions, we draw $u,v\in\R^d$ as a Haar-random orthonormal pair. For a prescribed triple $(\DeltaT,\DeltaS,\mu)$, the direction vectors are then defined as $\thetaT := \DeltaT\,u$ and $\thetaS := \DeltaS\bigl(\mu\,u+\sqrt{1-\mu^2}\,v\bigr)$,
which ensures that $\|\thetaT\|_2=\DeltaT$, $\|\thetaS\|_2=\DeltaS$, and $\langle\thetaT,\thetaS\rangle/(\|\thetaT\|_2\|\thetaS\|_2)=\mu$. In this experiment, we set $\DeltaT=0.8$ and $\DeltaS=3$. Thus, the target-only branch is not sufficiently informative to recover the cluster labels reliably, whereas the source becomes useful only when $\mu$ is sufficiently large. 
We vary $\mu\in\{0.000,\ 0.025,\ 0.05,\ 0.10,\ 0.15,\ 0.20,\ 0.40,\ 0.80\}$.

For each configuration, the target and source datasets are generated independently according to \eqref{eq:target_sample} and \eqref{eq:source_sample}. Since $\thetaT$ depends only on $u$, the target dataset $\calT$ does not depend on $\mu$. It is therefore generated once within each Monte Carlo replication and reused across the entire $\mu$-grid. By contrast, since $\thetaS$ depends on $\mu$, a fresh source dataset $\calS$ is generated at each grid point.

We compare four methods in each regime and for each value of $\mu$: the target-only estimator, obtained by applying the target branch of Algorithm~\ref{alg:meta_oracle_transfer_clustering_two_community}; the source-only estimator, obtained by applying the source branch of Algorithm~\ref{alg:meta_oracle_transfer_clustering_two_community}; the adaptive estimator proposed in Algorithm~\ref{alg:adaptive_transfer_clustering_two_community}; and the pooled estimator proposed in Algorithm~\ref{alg:pooled_transfer_clustering}. In the adaptive procedure, the absolute constant $C_0$ in \eqref{eq:validation_threshold} is obtained using the bootstrap procedure from Section~\ref{sec:bootstrap_calibration}, with $50$ bootstrap replications and $\alpha=0.5$. Performance is measured in terms of the misclustering loss $\mathcal L$ defined in \eqref{eq:loss_func_2}.

We repeat the entire experiment over $100$ independent Monte Carlo replications and report the average target misclustering error as a function of $\mu$ for all three regimes in Figure~\ref{fig:error_vs_alignment}. In regime \texttt{R1}, the source-only estimator performs best when the source is sufficiently well aligned with the target. The adaptive and pooled estimators improve upon the target-only estimator, although the adaptive estimator exhibits slightly worse performance as a result of the small target sample size. 
In regimes \texttt{R2} and \texttt{R3}, the pooled estimator almost matches the source-only estimator.
In the small-$\nT$, large-$d$ regimes the target-only estimator and the estimators using source information are comparable for small $\mu$.
By contrast, in regime \texttt{R3}, 
the target-only estimator substantially outperforms the estimators using source information for small values of $\mu$. This illustrates the phenomenon of \emph{negative transfer}. 

An numerical experiment exploring the dependence of misclustering on the target and source signal strengths $(\DeltaT,\DeltaS)$ is provided in Supplement~\ref{sec:sim_sig_strength}.

\begin{figure}[tb]
    \centering
    \includegraphics[scale=0.6]
    {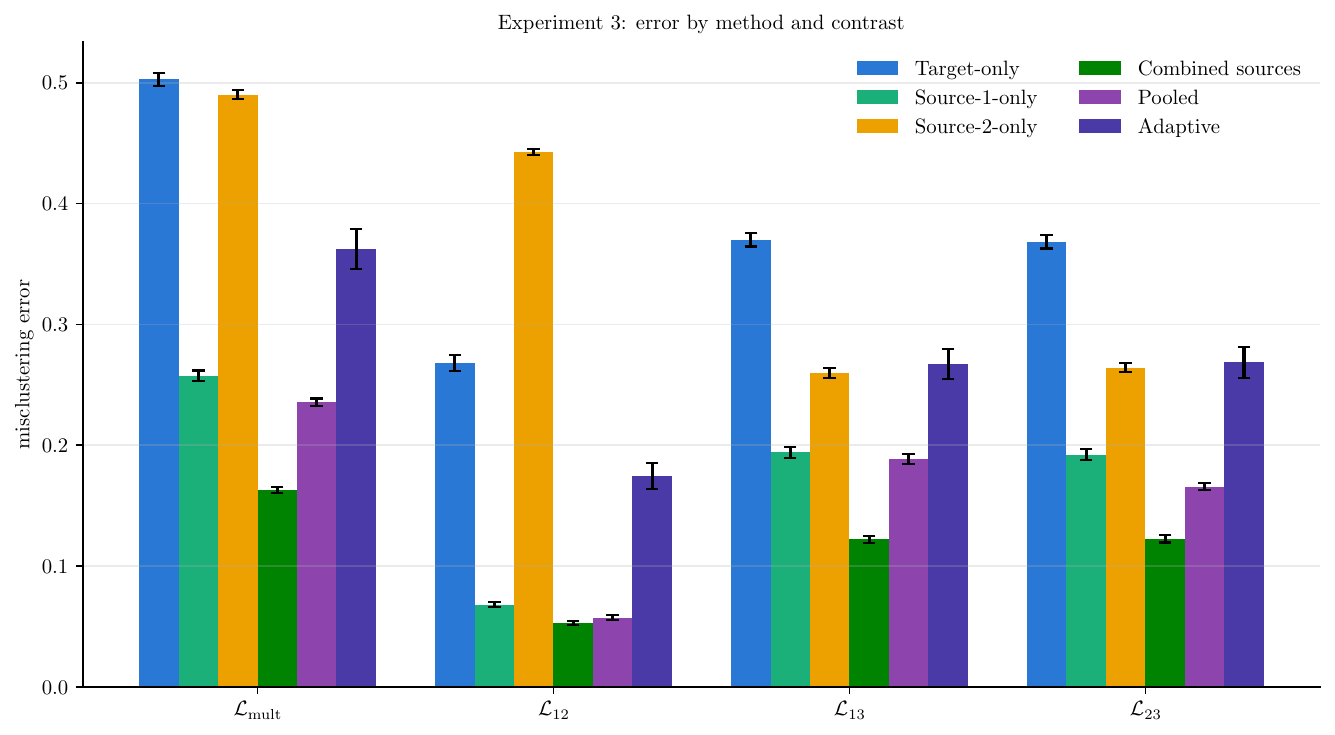}
    \caption{Overall and pairwise target misclustering errors for the
    target-only, source-1-only, source-2-only, combined-source, pooled, and
    adaptive estimators. The quantities $\calL_{\mathrm{mult}}$,
    $\calL_{12}$, $\calL_{13}$, and $\calL_{23}$ denote the overall
    permutation-invariant misclustering error and the pairwise
    misclustering errors for the $(1,2)$, $(1,3)$, and $(2,3)$ contrasts,
    respectively. Error bars represent Monte Carlo standard errors over
    $100$ independent replications.}
    \label{fig:complementary_sources}
\end{figure}

\subsection{Performance in multi-source and multi-cluster settings}
\label{sec:mult_clust_sim}

In this experiment, we study the finite-sample performance of Algorithm~\ref{alg:projected_clustering_mult}. This setting is more complicated than the two-community frameworks studied in Sections~\ref{sec:sim_align}--\ref{sec:sim_sig_strength}, because, when multiple signal directions are present, the difficulty of the problem is determined by the ability of the target and source datasets to separate all pairwise contrasts between the cluster means.

In this direction, we consider a setting with $K=3$ clusters and $m=2$ source datasets. For each dataset, namely, the target and the two sources, the cluster sizes are fixed to be as balanced as possible (each cluster receiving either $\lfloor n/K\rfloor$ or $\lceil n/K\rceil$ observations for both target and the two source datasets).
To generate the cluster centers, three Haar-random orthonormal directions $u,v,r\in\R^d$ are drawn. The target cluster centers are defined as $\theta_{T,1} = 2u$, $\theta_{T,2} = -2u$ and $\theta_{T,3} = 2v$.
The $(1,2)$ contrast is the strongest for the target-only procedure, while the contrasts involving cluster $3$ are comparatively weaker. The dataset $\calS_1$ is designed to separate cluster $1$ from clusters $2$ and $3$, while providing only weak information for separating clusters $2$ and $3$. Its cluster centers are defined as $\theta_{S_1,1} = 2.75\,u$, $\theta_{S_1,2} = -2.75\,u$ and $\theta_{S_1,3} = -2.75\,u + 0.25\,r$.
The dataset $\calS_2$ is designed to recover the missing $(2,3)$ contrast, with cluster centers given by $\theta_{S_2,1} = 0$, $\theta_{S_2,2} = 1.6\,v$, and $\theta_{S_2,3} = -1.6\,v$.
Using these cluster centers, the datasets $\calT$, $\calS_1$, and $\calS_2$ are generated independently according to \eqref{eq:mult_clust_target_sample}--\eqref{eq:mult_clust_source_sample}, with $\nT=200$, $\nSm{1}=\nSm{2}=1500$, and $d=1000$. Observe that the setting considered in this experiment provides weaker signal in the source datasets than those required by our theoretical results in Theorem~\ref{thm:consistency_mult_clust_final} and hence provides a stress test for our source based and pooled estimators in the multi-cluster setting.

We compare six procedures: the target-only clustering procedure obtained by applying the target-side estimator of Algorithm~\ref{alg:projected_clustering_mult}; two source-only procedures obtained by applying the source-side estimator using $\calS_1$ and $\calS_2$, respectively; the combined-source estimator obtained by applying the source-side estimator using both $\calS_1$ and $\calS_2$; the pooled estimator from Algorithm~\ref{alg:pooled_transfer_clustering}; and the full adaptive procedure in Algorithm~\ref{alg:projected_clustering_mult}. For the adaptive procedure, the constant $D_0$ is estimated using the bootstrap calibration described in Section~\ref{sec:bootstrap_calibration}, with $10$ bootstrap samples and $\alpha=0.5$. We measure the overall permutation-invariant target misclustering error $\mathcal L_{\mathrm{mult}}$ defined in \eqref{eq:loss_func_2m}, together with the pairwise misclustering errors obtained after restricting attention to the $(1,2)$, $(1,3)$, and $(2,3)$ contrasts, denoted by $\mathcal L_{12}$, $\mathcal L_{13}$, and $\mathcal L_{23}$, respectively, computed using \eqref{eq:loss_func_2}. Each loss, averaged over 100 independent replicates is presented in Figure~\ref{fig:complementary_sources}. Our findings illustrate the benefit of combining sources that contain complementary information about the target cluster contrasts. The target-only estimator has the largest overall misclustering error, while source-$1$-only estimator substantially improves performance, especially for the $(1,2)$ contrast. In the $(1,2)$, using source-$2$-only estimator leads to \emph{negative transfer} since it has no information about $u$. The combined-source estimator achieves the smallest overall and pairwise errors because the two sources jointly recover the directions $u$ and $v$, thereby preserving all pairwise target contrasts. The pooled estimator and the adaptive also substantially improves upon the target-only estimator.
Overall, the experiment demonstrates that multiple sources can provide genuinely complementary information and that combining their estimated signal subspaces can markedly improve target clustering even when no single source is sufficiently informative.

\begin{figure}[tb]
    \centering
    \includegraphics[width=\textwidth]
    {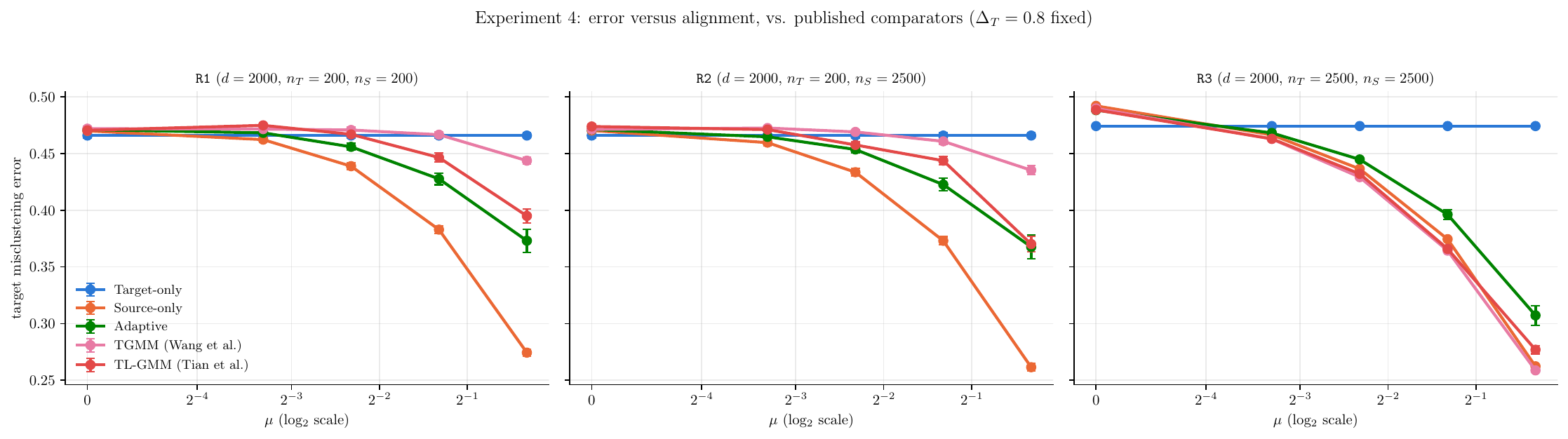}
    \caption{Target misclustering error as a function of the alignment
    parameter $\mu$ for the proposed estimators and the published
    comparators TGMM \citep{wang_zhou} and TL-GMM \citep{tian2025robust}. The left, middle, and right panels correspond to regimes
    \texttt{R1}, \texttt{R2}, and \texttt{R3}, respectively. The target
    signal strength is fixed at $\DeltaT=0.8$. Error bars represent Monte
    Carlo standard errors.}
    \label{fig:comparison_vs_alignment_external}
\end{figure}

\subsection{Comparison with benchmark methods}
\label{sec:comparator_benchmarks}

In this experiment, we compare the performance of our procedure with two benchmark methods from the literature, namely, the TGMM procedure of \cite[Algorithm~1]{wang_zhou} and the TL-GMM procedure of \cite[Algorithm~7]{tian2025robust}. The simulation setup is the same as that in Section~\ref{sec:sim_align}, except that we vary $\mu\in\{0.0,0.1,0.2,0.4,0.8\}$. We compare the three estimators considered in Section~\ref{sec:sim_align}, namely, the target-only, source-only, and adaptive estimators from Algorithm~\ref{alg:adaptive_transfer_clustering_two_community}, with TGMM and TL-GMM across the specified values of $\mu$. For the two comparator procedures we needed to fix some hyperparameters. In TGMM, we set the shrinkage parameter to $\lambda=0.5$. For TL-GMM, we used isotropic component covariances with batch-specific variances estimated using \eqref{eq:consistent_plug_in}. The isotropic assumption aligns with the data generating process and helps to stabilize variance estimation particularly for small sample sizes. Furthermore, we set the hyperparameters set $\kappa=1/3$ and $C_{\lambda_0}=1.7$. For both methods, we used $50$ EM iterations.

The results are presented in Figure~\ref{fig:comparison_vs_alignment_external}. We observe that, when $\nT\ll d$, both TGMM and TL-GMM are outperformed by the source-based and adaptive estimators. In contrast, when $\nT\gtrsim d$, the two benchmark procedures become competitive with the source-based estimator. Overall, this experiment illustrates that our procedure provides a tangible improvement in transfer-assisted clustering when the ambient dimension is substantially larger than the target sample size.

%% file: lung_atlas.tex
\section{Transfer-assisted clustering of human lung scRNA-seq data}
\label{sec:lung_atlas}
In this section, we analyzed the single cell RNA sequencing data from \cite{Vieira_Braga2019-nl} discussed in the introduction. The original data was produced by profiling the whole transcriptome of freshly resected human lung tissue using \emph{Dropseq}, a droplet-based sequencing platform. 
The dataset consists of four patient samples (\texttt{ASK428}, \texttt{ASK440}, \texttt{ASK452}, \texttt{ASK454}), all processed with the same sequencing technology. 
The raw count matrix comprised $10{,}360$ cells and $16{,}327$ genes, resolved into $13$ cell types. After applying standard QC and pre-processing steps outlined in Supplement~\ref{sec:data_preprocess}, we obtained a data matrix
$X \in \R^{9941 \times 5000}$ that could be approximately modeled using \eqref{eq:mult_clust_target_sample}-\eqref{eq:mult_clust_source_sample} (similar models were adopted in \cite{zhong2022empirical}, and \cite{cao2025modah}). In this data, the number of cells in \texttt{ASK428}, \texttt{ASK440}, \texttt{ASK452} and \texttt{ASK454} were $2{,}098$, $3{,}183$, $2{,}387$, and $2{,}273$, respectively. Every patient's batch contained all $13$ cell types.

Our objective was to use each patient in turn as the target and the other three as sources, apply Algorithm~\ref{alg:projected_clustering_mult} (and its non-adaptive special cases) to cluster the target's cells into $K=13$ groups, and compare the recovered clusters against the annotated cell types. 
We compared the target-only estimator, the combined source-only estimator, the adaptive estimator (with $D_0$ chosen using bootstrap calibration with $\alpha=0.5$ and 30 bootstrap replicates), and the pooled estimator (Algorithm~\ref{alg:pooled_transfer_clustering}). Furthermore, we also compared against the TL-GMM \citep{tian2025robust}, and two single-cell specific transfer methods, namely GDEC \citep{gdec_paper}, an autoencoder based pipeline with a a deep embedded clustering (DEC) head, and the non-negative matrix factorization (NMF) based pipeline from \cite{nmf_paper} \footnote{Further details about the implementation of the comparator methods are provided in Supplement~\ref{sec:comparator_methods}}. 

\begin{table}[t]
\centering
\resizebox{\textwidth}{!}{%
\begin{tabular}{lcccccccccccc}
\toprule
 & \multicolumn{3}{c}{\texttt{ASK428}} & \multicolumn{3}{c}{\texttt{ASK440}} & \multicolumn{3}{c}{\texttt{ASK452}} & \multicolumn{3}{c}{\texttt{ASK454}} \\
\cmidrule(lr){2-4} \cmidrule(lr){5-7} \cmidrule(lr){8-10} \cmidrule(lr){11-13}
Method & ARI & V-measure & $\mathcal{L}_{\mathrm{mult}}$ & ARI & V-measure & $\mathcal{L}_{\mathrm{mult}}$ & ARI & V-measure & $\mathcal{L}_{\mathrm{mult}}$ & ARI & V-measure & $\mathcal{L}_{\mathrm{mult}}$ \\
\midrule
Target-only & \textbf{0.456} & \textbf{0.681} & 0.459 & 0.363 & 0.618 & 0.509 & \textbf{0.440} & \textbf{0.661} & \textbf{0.395} & 0.208 & \textbf{0.469} & 0.624 \\
Multi-source pooled & 0.432 & 0.653 & 0.507 & 0.346 & 0.604 & 0.534 & 0.354 & 0.606 & 0.504 & 0.159 & 0.418 & 0.675 \\
Pooled  & 0.451 & 0.667 & 0.450 & \textbf{0.440} & \textbf{0.630} & \textbf{0.431} & 0.362 & 0.615 & 0.504 & \textbf{0.242} & 0.454 & 0.509 \\
Adaptive multi-source & \textbf{0.456} & \textbf{0.681} & 0.459 & 0.363 & 0.618 & 0.509 & \textbf{0.440} & \textbf{0.661} & \textbf{0.395} & 0.208 & \textbf{0.469} & 0.624 \\
TL-GMM & 0.443 & 0.655 & \textbf{0.445} & 0.421 & 0.618 & 0.498 & 0.298 & 0.573 & 0.509 & 0.208 & 0.418 & 0.613 \\
NMF & 0.000 & 0.000 & 0.660 & 0.000 & 0.000 & 0.609 & 0.000 & 0.000 & 0.627 & 0.000 & 0.000 & \textbf{0.494} \\
GDEC & 0.285 & 0.480 & 0.642 & 0.288 & 0.398 & 0.602 & 0.242 & 0.495 & 0.610 & 0.098 & 0.286 & 0.723 \\
\bottomrule
\end{tabular}%
}
\caption{ARI, V-measure, and misclustering error ($\mathcal{L}_{\mathrm{mult}}$) on the lung atlas (13 fine-grained cell types, K=13), leave-one-batch-out. ARI/V-measure: higher is better; $\mathcal{L}_{\mathrm{mult}}$: lower is better. Best value per subcolumn in bold.}
\label{tab:lung_k13_combined}
\end{table}

The performance of each method was evaluated using the permutation-invariant misclustering loss $\mathcal L_{\mathrm{mult}}$, the Adjusted Rand Index \citep{hubert1985comparing}, and the V-measure \citep{rosenberg-hirschberg-2007-v}. The results are reported in Table~\ref{tab:lung_k13_combined}. First, the adaptive estimator matches the target only estimator across all batches. This behavior is consistent with the performance of the multi-source pooled estimator which is dominated by the target-only estimator across all three metrics and all four batches.
Second, no single method dominates uniformly across all target batches. The target-only and adaptive method achieve the best ARI and V-measure on \texttt{ASK428} and \texttt{ASK452}.
The pooled estimator performs best across all three metrics on
\texttt{ASK440} and is competitive on \texttt{ASK428} where TL-GMM achieves the best $\mathcal L_{\mathrm{mult}}$.
Third, the performance of NMF in \texttt{ASK454} demonstrates that $\mathcal L_{\mathrm{mult}}$ might sometimes provide misleading assessment. In \texttt{ASK454}, NMF achieves smallest $\mathcal L_{\mathrm{mult}}$ but its ARI is equal to 0. This pattern emerges because NMF favored a highly degenerate and uninformative partition in a imbalanced data. Overall, the target-only, pooled, adaptive multi-source, and TL-GMM procedures were substantially more effective at recovering the fine-grained $13$-cell-type structure of the lung atlas than the NMF and GDEC baselines.

\section{Discussion}
\label{sec:discussion}

In this paper, we have characterized the regimes when an auxiliary source dataset can be used to improve clustering in a target dataset generated from
high-dimensional Gaussian mixture models. The key finding is that transfer
utility is governed by the target signal strength $\DeltaT$, the source
signal strength $\DeltaS$, the source-target alignment $\mu$, and the aspect ratio
$d/\nS$. We provide sufficient conditions for consistent clustering when either the target is sufficiently informative on its own or the source signal is sufficiently strong and well aligned with the target, as summarized in \eqref{eq:overview_transfer_regime}, up to logarithmic factors. We also establish complementary information-theoretic lower bounds. In the two-community setting, these results identify the same qualitative transfer regimes, although a logarithmic gap remains between the sufficient and necessary signal-strength conditions. The theoretical conclusions are supported by simulation studies and an application to an scRNA-seq dataset.

%
The logarithmic gap in the upper and lower bounds traces to the requirement of exactly
recovering the source labels when $d \gg \nS$ (via Theorem~8 of \citet{Ndaoud2022}).
We believe it is a proof artifact rather than a genuine statistical cost, and expect that a sharper source-label recovery argument could close it. However, we relegate such closure to future research since in most applications such a logarithmic gap is negligible. 

The natural extensions of our analysis are the extensions to the settings where the number of sources $m$ grows with $\nT$. Furthermore, pinning down the precise dependence of the number of communities $K$ in the multi cluster setting by proving a lower bound to the misclustering risk also remains an open problem. Other potentially interesting areas of extension are incorporating spatial relations between the subjects as in spatial transcriptomics or such analysis when the observed features are sparse. We hope to explore such areas in future research.

